\documentclass[11pt,
]{amsart}
\usepackage{%amsmath,
amssymb, graphicx
}
 \usepackage[foot]{amsaddr}
\usepackage{amsthm}
\usepackage{graphicx,color}
\newtheorem{theorem}{Theorem}
\newtheorem{lemma}{Lemma}
\newtheorem{proposition}{Proposition}
\newtheorem{definition}{Definition}
\newtheorem{corollary}{Corollary}
\newtheorem{remark}{Remark}
\newtheorem{problem}{Problem}

\theoremstyle{remark}

\DeclareMathOperator{\supp}{supp}

\newcommand{\R}{{\bf R}}
\newcommand{\C}{{\bf C}}
\newcommand{\dB}{{\mathrm{d}^\mathcal{B}}}
\newcommand{\deB}{{\delta^\mathcal{B}}}
\newcommand{\LaB}{{\Delta^\mathcal{B}}}
\newcommand{\sym}{\hbox{Sym}} 
\newcommand{\Id}{\mbox{Id}}

\newcommand{\be}[1]{\begin{equation}\label{#1}} 
\newcommand{\ee}{\end{equation}} 

\newcommand{\N}{\mathbf{N}} 
\newcommand{\Z}{\mathbf{Z}}

\newcommand{\p}{\partial}

\title[Mixed ray transform]{Generic uniqueness and stability for the mixed ray transform}


\author[M. V. de Hoop]{Maarten V. de Hoop}
\address{M. V. de Hoop: Department of Computational and Applied Mathematics, Rice University, Houston,
    TX 77005, USA (\tt{mdehoop@rice.edu}).}

\author[T. Saksala]{Teemu Saksala}
\address{T. Saksala: Department of Computational and Applied Mathematics,
  Rice University, Houston, TX, 77005, USA
   (\tt{teemu.s.saksala@gmail.com})}

  \author[G. Uhlmann]{Gunther Uhlmann}
\address{G. Uhlmann: Department of Mathematics, University of Washington, Seattle, WA 98195, USA; Institute for Advanced Study, 
The Hong Kong University of Science and Technology, Kowloon, Hong Kong, China (\tt{gunther@math.washington.edu})
}
  
  \author[J. Zhai]{Jian Zhai}
\address{J. Zhai: Institute for Advanced Study,
  The Hong Kong University of Science and Technology, Kowloon, Hong Kong, China
  (\tt{iasjzhai@ust.hk}, \tt{jian.zhai@outlook.com}).}

\begin{document}

\begin{abstract}
We consider the mixed ray transform of tensor fields on a three-dimensional compact simple Riemannian manifold with boundary. We prove the injectivity of the transform, up to natural obstructions, and establish stability estimates for the normal operator on generic three dimensional simple manifold in the case of $1+1$ and $2+2$ tensors fields. 

We show how the anisotropic perturbations of 
averaged
isotopic travel-times of 
$qS$-polarized
elastic waves provide partial information about the mixed ray transform of $2+2$ tensors fields. If in addition we include the measurement of the shear wave amplitude, the complete mixed ray transform can be recovered. We also show how one can obtain the mixed ray transform from an anisotropic perturbation of the \text{Dirichlet-to-Neumann} map of an isotropic elastic wave equation on a smooth and bounded domain in three dimensional Euclidean space. 
\color{black}
\end{abstract}

\maketitle 
%\tableofcontents
\section{Introduction}
\label{Se:intro}

%\textbf{[Euclidean version of the intro. Gunther suggested to have something like this] (TS)}

%\color{blue}
In this paper we study an inverse problem of recovering a $4$-tensor field from a family of certain line integrals. 
%\textbf{This problem arises from the elastic $qS$-wave tomography in weakly anisotropic elastic medium, which we will show  in the appendix. [JZ: I would just use ``tomography" rather than ``travel-time tomography".]}
%In this section we recall the definition of a linear transform, originally given in \cite[Chapter 7]{Shara}, that maps a $4$-tensor field to the this family of line integrals and call it the \textit{mixed ray transform}. 
This family 
%of line integrals 
is called the \textit{mixed ray transform}, and it was first considered in \cite[Chapter 7]{Shara}.
%Our aim is to study the injectivity of the transform on $3$-dimensional Riemannian manifolds. 
We characterize the kernel of the mixed ray transform for $1+1$ and $2+2$ tensor fields for generic simple $3$-dimensional Riemannian manifolds and provide a stability estimate for the corresponding $L^2$-normal operator. 

We begin by introducing the mixed ray transform in the Euclidean space.
% \textbf{[JZ: we can consider a tensor $f\in C_0^\infty(\R^3)$. Therefore do not need to introduce a bounded domain and the exit time. ][(TS) That is true, but then the range of MRT is different. This would be a space of tensor fields on the space of lines in $\R^3$]}\textbf{[JZ: this part is to get people acquainted with mixed ray transform, therefore should be made as accessible as possible. Avoid any unnecessary technicalities here. We can write this part similar to how X-ray transform is introduced in many literatures.] (TS)[I think that this is a valid point]} 
%Let $\Omega \subset \R^3$ be a compact and convex domain with smooth boundary and $f$ a smooth two tensor field on $\Omega$. For each boundary point $x \in \p \Omega$ and each inward pointing unit vector $\xi$ at $x$ we use the notation $\tau(x,\xi) \in (0,\infty)$ for the length of the line segment $\{x+t\xi: t\geq 0\}$ contained in $\Omega$. Then we choose a vector $\eta$, that is orthogonal to $\xi$. The mixed ray transform $L_{1,1}f$ of $f$ at $(x,\xi,\eta)$ is given by
%\begin{equation}
%\label{eq:euc_mrt}
%L_{1,1}f(x,\xi,\eta):=\int_0^{\tau(x,\xi)}f_{ij}(x+t\xi)\eta^i\xi^j \; \mathrm{d}t.
%\end{equation}
Let $f$ be a smooth compactly supported two tensor field on $\R^3$. We choose a point $x \in \R^3$ and a unit vector $\xi$. Thus $x$ and $\xi$ define a line  $\{x+t\xi \in \R^3: t\in \R\}$. Then we choose a vector $\eta$, that is orthogonal to $\xi$. The mixed ray transform $L_{1,1}f$ of $f$ for $(x,\xi,\eta)$ is given by
\begin{equation}
\label{eq:euc_mrt}
L_{1,1}f(x,\xi,\eta):=\int_{-\infty}^\infty f_{ij}(x+t\xi)\eta^i\xi^j \; \mathrm{d}t.
\end{equation}
%This equation implies that the mixed ray transform is a linear operator mapping a tensor field to a real function over the triplets  $(x,\xi,\eta)$. 
We note that if we had chosen $\eta=\xi$ in \eqref{eq:euc_mrt}, then we would have obtained the (longitudinal) ray transform of $f$. 
%The formula \eqref{eq:euc_mrt} gives hints about the kernel of this transform. First we note that if $f$ is conformal to the Euclidean metric then the integrand in \eqref{eq:euc_mrt} vanishes pointwise. We say that this type of tensor fields from the \textit{trivial kernel} of $L_{1,1}$. 
%Let be the $\mu$ be the trace operator. If we write 
%\begin{equation}
%\label{eq:euc_decomp}
%f=\left(I-I\frac{\mu }{3}\right)f+\frac{\mu f}{3}I=:\mathcal B f+\frac{\mu f}{3}I,
%\end{equation}
%we notice that the first term in the right hand side of \eqref{eq:euc_decomp} is \textit{trace free} or in the other words in the kernel of $\mu$ and the second term is conformal to $e$. 
%
We recall that any $2$-tensor $f$ has a unique decomposition $f_{ij}(x)=f_{ij}^\mathcal{B}(x)+c(x)\delta_{ij}$, with a zero \textit{trace} $\mu f^\mathcal{B}:=\sum_{i=1}^3f^\mathcal{B}_{ii}=0$. Since $\xi$ and $\eta$ were chosen to be orthogonal to each other we get from \eqref{eq:euc_mrt} that $L_{1,1}(c(x)\delta_{ij})=0$. Thus for the mixed ray transform, the only relevant tensor fields are the \textit{trace-free} ones, for which $\mu f=0$. 
%Let $\mathcal{B}$ be the projection onto the trace free class.
 %We note that the first order differential operator $\mathcal B \nabla$ maps a one form on the trace free class of two tensor.
%
Notice that if $f=(\nabla v)^\mathcal{B}$ for some $1$-form $v\in C_0^\infty(\R^3)$ then the fundamental theorem of calculus implies that $L_{1,1}f=0.$ Therefore  $L_{1,1}$ always has a non-trivial \textit{natural kernel}, consisting of \textit{potential tensor fields} $(\nabla v)^\mathcal B, \: v\in C_0^\infty(\R^3)$.
In this paper, we will consider the mixed ray transform on certain Riemannian manifolds  and study its injectivity up to the natural obstruction.\\

Let $(M,g)$ be a simple $3$-dimensional Riemannian manifold with boundary $\partial M$. 
%\color{blue}
We recall that a compact Riemannian manifold is simple if it has  a strictly convex boundary and 
any two points $x,y \in M$ can be connected by a unique geodesic, contained in $M$, depending smoothly on $x$ and $y$. We use
the notation $TM$ for the tangent bundle of $M$, 
%\color{blue}
$T^\ast M$ for the contangent bundle,
\color{black}
and  $S M$ for the unit sphere bundle, defined as
$SM=\{(x,\xi)\in TM;\,|\xi|_g=1\}$. Let $\partial_+(SM)=\{(x,\xi)\in
SM;\, x\in\partial M,\,\langle \xi,\nu\rangle_g < 0\}$ be the inward
pointing unit sphere bundle on $\partial M$, where $\nu$ is the
outward pointing unit normal vector field to the boundary. We use the
notation $S^k \tau'_M \otimes S^\ell \tau'_M , \: k,\ell\geq 1$ for
the space of $k+\ell$ tensor fields on $M$ that are symmetric with respect to
first $k$ and last $\ell$ indices.
%\color{blue}
Note that \textit{a priori} we do not pose any regularity properties for the tensor fields. To emphasize the regularity we use the standard notations $C^m, C^\infty, L^2,$ or $H^m$ in front of the vector
space of the corresponding tensor fields. 
%\color{red}
%Through out the paper we will use the Einstein summing convention, for the repeated indices in lower and upper positions, and musical isomorphisms 
%\[
%v^\flat=v_i:= g_{ij}(x)v^j, \quad \xi^\sharp=\xi^i:=g^{ij}\xi_j, \quad v \in T_xM, \: \xi \in T^\ast_xM,
%\]
%to raise and lower the indices. 

The mixed ray transform $L_{k,\ell}f $
of a smooth tensor field $f\in C^\infty(S^k \tau'_M \otimes S^\ell \tau'_M)$ is given by the following formula
\begin{equation}
\label{eq:mixed_ray_trans_act}
%\color{red}
\left(L_{k,\ell}f\right)(x,\xi,\eta)=\int_{0}^{\tau(x,\xi)}  f_{i_1\ldots i_kj_1\ldots j_\ell}(\gamma(t))\eta(t)^{i_1} \cdots\eta(t)^{i_k} \dot\gamma(t)^{j_1}\cdots\dot\gamma(t)^{j_\ell}\mathrm{d}t,
%\quad \eta\perp\xi \in \p T_xM,
\end{equation}
where $(x,\xi)\in\p_+(SM)$ and $\gamma(t)=\gamma_{x,\xi}(t)$ is the unit-speed geodesic
given by the initial conditions $(x,\xi)$. 
%\color{blue}
The vector $\eta\in T_xM$ is perpendicular to $\xi$, and $\eta(t)$ is the parallel translation of $\eta$ along the geodesic $\gamma(t)$. We note that $\eta(t)\perp\dot{\gamma}(t)$ for any $t$. By $\tau(x,\xi)$ we mean the exit time of $\gamma$, which is the first positive time in which $\gamma$ hits the boundary again. Since $(M,g)$ is simple the exit time function $\tau$ is smooth on $\p_+SM$ \cite[Lemma 4.1.1.]{Shara}.

%\textbf{[Which of the following figures to choose? The referee asked for a figure of the set up. The first contains the geodesic, normal field and the tensor, but it is quite messy (TS)]}
%\begin{figure}[h]
%\begin{picture}(250,125)
%\put(0,5){\includegraphics[width=10cm]{MRT.pdf}}
%\put(-20,45){$\p M$}
%%\put(94,92){{\color{blue} $\gamma$}}
%\put(103,0){$x$}
%%%\put(50,0){\color{blue}$q_1$}
% \end{picture}
%\caption{In this figure we illustrate the mixed ray transform of a tensor field $f \in S^1 \tau'_M \otimes S^1 \tau'_M$, which is represented by the pairs of dashed blue and red arrows. We choose an initial point $x \in \p M$ and an initial velocity $\xi \in S_xM$. The blue line is the geodesic $\gamma$ given by  these initial conditions. Next we choose $\eta \in T_xM, \: \eta \perp \xi$ and compute its  parallel translation along $\gamma$, this is illustrated by red arrows on $\gamma$. In \eqref{eq:mixed_ray_trans_act} we first compute the dot products of the arrows of the same color on $\gamma$ and then compute their products. Finally we integrate this over the interval $[0,\tau(x,\xi)]$.
%}
%\end{figure}

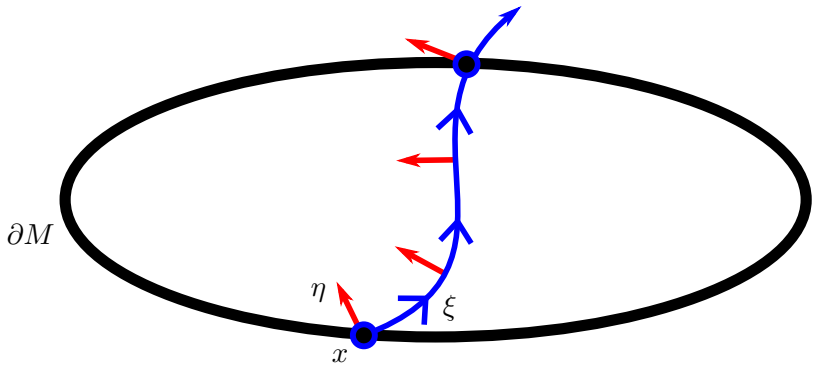
\begin{figure}[h]
\begin{picture}(250,125)
\put(0,5){\includegraphics[width=10cm]{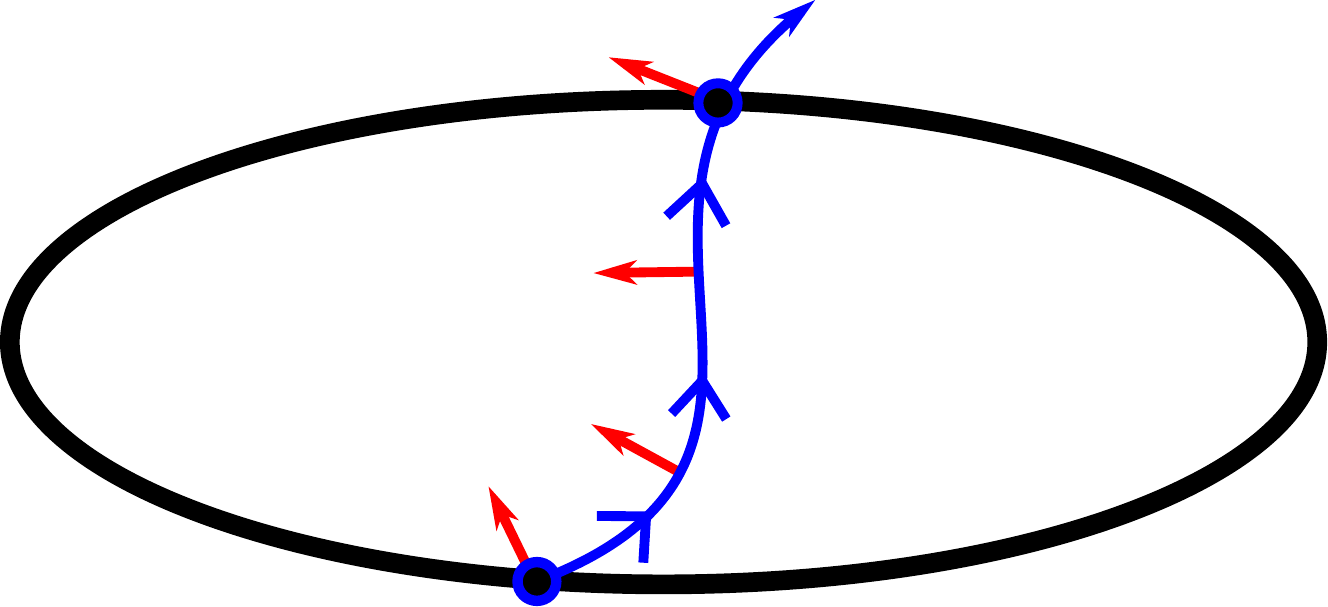}}
\put(-20,45){$\p M$}
%\put(94,92){{\color{blue} $\gamma$}}
\put(103,0){$x$}
\put(95,25){$\eta$}
\put(145,18){$\xi$}
%%\put(50,0){\color{blue}$q_1$}
 \end{picture}
\caption{In this figure we illustrate the notations used in the definition of the mixed ray transform  \eqref{eq:mixed_ray_trans_act}. We choose an initial point $x \in \p M$ and an initial velocity $\xi \in S_xM$, blue arrow. The blue line is the geodesic $\gamma$ given by these initial conditions. Finally we choose $\eta \in T_xM, \: \eta \perp \xi$ and compute its parallel translation along $\gamma$, this is illustrated by red arrows on $\gamma$. 
%\textbf{[JZ: The vector $\xi$ should be a vector tangent to the geodesic.],[(TS) Sorry I don't understand your comment, can you please make the edit you have in your mind?. Here $\xi$ is the initial direction of $\gamma$, which is tangential to $\gamma$]}
}
\end{figure}

%\bigskip
%\textbf{[Should we recall the boundary rigidity problem and its linearization i.e. the geodesic ray transform here? This would give more context to the reader and a small cap between technical definitions of MRT and its kernel? (TS). ]
%[I suggest just to give some more references later when we mention the boundary rigidity problem (JZ).]
%[After reading the referee comments again I noticed that he suggested that we recall the more familiar examples earlier and give a picture for the setup. Thus I would favour this order (TS)]}
%\bigskip

%\noindent
%{\color{blue}
%The mixed ray transform $L_{2,2}$ of a tensor field in $S^2 \tau'_M \otimes S^2 \tau'_M$, that is $k=\ell=2$, arises from the linearization of the elastic \textit{qS}-wave travel-time tomography in weakly anisotropic medium. 
%\textbf{[I would omit this and give all the discussion in the the paragraph (above) where we talk about elastic wave equation. (TS)]}
%For more details, see the Appendix of the paper. We also refer to \cite[Chapter 7]{Shara} for a discussion. 
%\textbf{[I moved the mention about the appendix after the main theorem (TS)]}
%}

%\color{blue}
If $k=0$ in (\ref{eq:mixed_ray_trans_act}), the operator
$L_{0,\ell}$ is the (longitudinal) \textit{geodesic ray transform} $I_\ell$ for a symmetric
$\ell$-tensor field $f$. 
%The $s$-injectity of $I_\ell$ has been studied extensively. In this context, $s$-injectivity means that $I_\ell f=0$ implies $f=\sym(j_1,\ldots,j_\ell)\nabla u$ with $u\in S^{\ell-1}\tau'_M$ and $u\vert_{\partial M}=0$. 
%
%\color{red}
The most interesting case is $\ell=2$ which arises from the linearization of the \textit{boundary rigidity problem},
% in which the goal is to reconstruct the Riemannian metric from the boundary distance function, \cite{michel1981rigidite}. 
that concerns the
recovery of the Riemannian metric from its boundary distance
function. 
%A Riemannian manifold $M$
%is called \textit{boundary rigid}, if the boundary distance function
%determines the isometry class of $M$ 
It was conjectured by Michel \cite{michel1981rigidite} that simple metrics are \textit{boundary rigid}, which means that they are uniquely determined, up to a diffeomorphism fixing the boundary, by the boundary distance function.
%It is well known that this problem is not uniquely solvable
%for two reasons. 
%The boundary distance function is invariant under any
%diffeomorphism that is an identity at the boundary. This is a natural
%gauge of the problem. Thus the best result would be recovery of the
%metric up to the diffeomorphism invariance.
% A Riemannian manifold $M$
%is called \textit{boundary rigid}, if the boundary distance function
%determines the isometry class of $M$. 
%It was proposed by Michel
%\cite{michel1981rigidite} to study the boundary rigidity problem under
%the simplicity assumption. 
 Significant
progress has been made in studying this problem
\cite{burago2010boundary, croke1991rigidity, lassas2003semiglobal,michel1981rigidite,
  otal1990spectre,pestov2005two,stefanov2016boundary,
  stefanov2017local}. 
%  However, it is still not known whether all simple
%manifolds are boundary rigid. 
%It was proven by Pestov and Uhlmann
%\cite{pestov2005two} that a simple surface is always boundary
%rigid. More recently, boundary rigidity results have been established
%on manifolds of dimension 3 and higher that satisfy certain global
%convex foliation conditions \cite{stefanov2016boundary,
%  stefanov2017local}. 
The linearization of the boundary rigidity problem leads to an integral geometry problem of recovering a symmetric $2$-tensor field $f$ from its geodesic ray transform $I_2f$ (see, for instance, \cite{sharafutdinov2004ray}). 
%where
%\[
%   I_2f(x,\xi) = \int_0^{\tau(x,\xi)}
%        f_{ij}(\gamma(t))\dot{\gamma}^i(t)\dot{\gamma}^j(t)
%             \, \mathrm{dt}, \quad (x,\xi)\in \p_+SM, \: \gamma=\gamma_{x,\xi}.
%\]
%where $\gamma$ is the unit speed geodesic given by initial conditions
%$(x,\xi)\in \p TM, \: |\xi|_g=1$ and $\tau(x,\xi)$ is the exit time of
%$\gamma$. This transform maps a smooth symmetric $2$-tensor field $f$
%to a smooth function on the unit sphere bundle $\p SM$ on the
%boundary. If a correct measure is given on $\p SM$ the
%transform $I_2$ extends to a bounded linear operator between the
%corresponding $L^2$-spaces \cite{Shara}.
%
%As the original boundary rigidity problem also its linearization has a natural
%gauge. Due to Helmholtz decomposition, we can write any $L^2$-regular
%tensor field as a sum of a solenoidal tensor (divergence free) and a
%potential tensor (symmetric differential of some $1$-tensor vanishing
%at the boundary). Using the fundamental theorem of calculus it is
%straightforward to prove that $I_2$ anhilates any potential
%fields. This property is in line with the gauge of the boundary
%rigidity problem as a $1$-tensor field vanishing at the boundary is
%an infinitesimal generator of a flow of diffeomorphisms that keeps the
%boundary fixed.
% Therefore the best to hope for is to prove that $I_2$ is injective
% over solenoidal tensors (s-injective).
%
\noindent 
The problem of reconstructing a symmetric $4$-tensor field $f$ from $I_4f$ arises from the linearization of elastic \textit{qP}-wave travel-times \cite{vcerveny1982linearized, de2019inverting}. 
%\textbf{[This is a good addittion! Maybe we should compare this linearization to the qS-linearization in the appendix? (TS)] [JZ: That might be a good idea. I think you can easily see what we can obtain from linearizing the \textit{qP}-wave travel-time tomography in the appendix.]}

Using the fundamental theorem of calculus, it is straightforward to see that if $f=\sym\nabla u$ with $u\in S^{\ell-1}\tau'_M$ and $u\vert_{\partial M}=0$, then $I_\ell f=0$. Here $\sym$ is the symmetrization operator and $\nabla$ is the Levi-Civita connection. We recall that the  operator $I_\ell$ is called $s$-injective if its kernel coincides with the image of the operator $\sym\nabla\colon H^1_0(S^{\ell-1}\tau'_M) \to L^2(S^{\ell}\tau'_M)$.
% \textbf{[JZ: This statement does not consider the condition $v\vert_{\partial M}=0$]}. 
We list here some cases where $s$-injectivity of $I_\ell$ is known:

\begin{itemize}
\item $(M,g)$ simple, $\mathrm{dim}\geq 2$, $\ell=0$ \cite{Muk2,Muk1},
  $\ell=1$ \cite{AR}.
\item $(M,g)$ simple, $\mathrm{dim}\geq 2$, $\ell\geq 2$ under curvature conditions
  \cite{Dair, Pestov, PS, Shara}.
\item $(M,g)$ simple, $\mathrm{dim}=2$, $\ell$ arbitrary \cite{PSU}.
\item $(M,g)$ simple, $\mathrm{dim}\geq 2$, $\ell=2$: generic $s$-injectivity
  \cite{SU}.
\item $(M,g)$ admits a strictly convex foliation, $\mathrm{dim}\geq
  3$, $\ell=0$ \cite{UV}, $\ell=1,2$ \cite{SUV2}, $\ell=4$
  \cite{de2019inverting}.
\end{itemize}

\medskip
%\color{blue}
%In contrast to the geodesic ray transform the target space of the mixed ray transform in three dimensional case is more ``exotic" than a function space over the inward pointing bundle at the boundary. This is due to the freedom in choosing the vector $\eta$ of \eqref{eq:mixed_ray_trans_act} in the orthocomplement of the initial velocity $\xi$ of the geodesic $\gamma$. We postpone the rigorous definition of the mixed ray transform to Section \ref{Se:normal_op}, and move to discuss the natural gauge in the problem of determining a tensor field from its mixed ray transform. 

%\textbf{[(TS) I really do like the way you re-wrote the end of this subsection]}
In this paper we consider the mixed ray transform $L_{k,\ell}$ as a generalization of the  geodesic ray transform $I_\ell$, and study its kernel. As for the Euclidean case, we only need to consider $L_{k,\ell}$ acting on ``trace-free" tensors. First, we introduce the operator (\textit{symmetrized tensor product with the metric})
%\color{red}
$\lambda\colon S^{k-1} \tau'_M
\otimes S^{\ell-1} \tau'_M \to S^{k} \tau'_M \otimes S^{\ell} \tau'_M$
defined by
\begin{equation}
\label{eq:map_lambda}
(\lambda w)_{i_1\ldots i_kj_1\ldots j_\ell}:= \sym(i_1\ldots i_k)\sym(j_1\ldots j_\ell)(g_{i_kj_l}w_{i_1\ldots i_{k-1}j_1\ldots j_{\ell-1}}),
\end{equation}
where  $\sym(\cdot)$ is the symmetrization with respect to indices listed in the argument. The algebraic dual of the operator $\lambda$, is the \textit{trace operator}
\begin{equation}
\label{eq:map_mu}
\mu\colon S^{k} \tau'_M \otimes S^{\ell} \tau'_M
\to S^{k-1} \tau'_M \otimes S^{\ell-1} \tau'_M, \quad (\mu u)_{i_1\ldots i_{k-1}j_1\ldots j_{\ell-1}}:=u_{i_1\ldots i_{k}j_1\ldots j_{\ell}}g^{i_kj_\ell}.
\end{equation}
Therefore we see that  
\[
 S^{k}\tau'_M\otimes S^{\ell} \tau'_M=\ker \mu \oplus \hbox{Im}\lambda.
\]
The tensors in $\ker \mu$ are called \textit{trace free}. We use the notation $\mathcal B$ for the projection onto the trace-free class and write
\[
 S^k\tau'_M\otimes^\mathcal{B} S^\ell \tau'_M:= \mathcal{B}(\{S^k\tau'_M\otimes S^\ell \tau'_M\})=\ker \mu.
\]
We note here that $L_{k,\ell}(\hbox{Im}\lambda)=0$ and $L_{k,\ell}(\mathcal{B}f)=L_{k,\ell}(f)$. Therefore, from now on we assume $f\in S^k\tau'_M\otimes^\mathcal{B} S^\ell \tau'_M$.

% In the following we define two operators, the ranges of which always lie in the kernel.
%\color{red}
%These operators are 
%\color{blue}
%the \textit{symmetrized tensor product with the metric}
%\color{red}
%$\lambda\colon S^{k-1} \tau'_M
%\otimes S^{\ell-1} \tau'_M \to S^{k} \tau'_M \otimes S^{\ell} \tau'_M$
%defined by
%\begin{equation}
%\label{eq:map_lambda}
%(\lambda w)_{i_1\ldots i_kj_1\ldots j_\ell}:= \sym(i_1\ldots i_k)\sym(j_1\ldots j_\ell)(g_{i_kj_l}w_{i_1\ldots i_{k-1}j_1\ldots j_{\ell-1}})
%\end{equation}
%and
%\color{blue}
To describe the natural kernel of $L_{k,\ell}$, acting on $S^k\tau'_M\otimes^\mathcal{B} S^\ell \tau'_M$, we introduce
the \textit{symmetrized gradient operator} $\mathrm{d}'\colon S^{k}\tau'_M\otimes S^{\ell-1 }\tau'_M \to  S^{k}\tau'_M\otimes S^{\ell}\tau'_M$ defined by
%\color{red}
\begin{equation}
\label{eq:map_d'}
(\mathrm{d}'v)_{i_1\ldots i_kj_1\ldots j_\ell}:=\sym(j_1\ldots j_\ell) v_{i_1\ldots i_kj_1\ldots j_{\ell-1};j_\ell}.
\end{equation}
In \eqref{eq:map_d'} the index after the semicolon, stands for the corresponding index of the covariant derivative
%\color{blue}  
%$\nabla v$ 
%\color{red}
of a tensor field $v$. It was shown in \cite[Chapter 7]{Shara} that
\[
L_{k,\ell}(\mathcal{B}\mathrm{d}'u)=L_{k,\ell}(\mathrm{d}'u)=0, \quad  \hbox{for } u\in S^{k}\tau'_M\otimes^\mathcal{B} S^{\ell-1
}\tau'_M, \quad \hbox{ with } u\vert_{\partial M}=0.
\]
After these preparations we are ready set the following definition of solenoidal-injectivity ($s$-injectivity) for the mixed ray transform:
 $L_{k,\ell}$ is called $s$-injective if $L_{k,\ell}f=0$ and $f\in L^2(S^k\tau'_M\otimes^\mathcal{B} S^\ell \tau'_M)$ imply $f=\mathrm{d}^\mathcal{B}v:=\mathcal{B}\mathrm{d}'v$ with some tensor field $v\in H^1_0(S^k\tau'_M\otimes^\mathcal{B} S^{\ell-1} \tau'_M)$. 
 \noindent Here
\[
\dB\colon H^1(S^{k}\tau'_M\otimes^\mathcal{B} S^{\ell-1 }\tau'_M) \to L^2(S^{k}\tau'_M\otimes^\mathcal{B} S^{\ell}\tau'_M).
\]
% where the zero boundary
%condition means that each coordinate function of $u$ is zero at the
%boundary.
%\color{red}
%\label{de:s_inj_1}
We also introduce the formal adjoint of $\mathrm{d}'$
\[
-\delta'\colon S^{k}\tau'_M\otimes S^{\ell }\tau'_M \to S^{k}\tau'_M\otimes S^{\ell-1}\tau'_M,
\]
where $\delta'$ is the \textit{divergence operator}
\[
(\delta'u)_{i_1\ldots i_kj_1\ldots j_{\ell-1}}:=g^{j_\ell j_{\ell+1}}u_{i_1\ldots i_kj_1\ldots j_\ell;j_{\ell+1}}.
\]
We define $\deB:=\delta'\vert_{S^{k}\tau'_M\otimes^\mathcal{B} S^{\ell}\tau'_M}$. One can readily check that $\hbox{Im}(\deB)\subset S^{k}\tau'_M\otimes^\mathcal{B} S^{\ell-1}\tau'_M$, and therefore
%=\delta'-\frac{1}{a_1}
%\frac{k\cdot\ell}{(k+\ell+1)\cdot 1}
%\delta'\lambda\mu=\delta'\colon\quad
\[
\deB\colon S^{k}\tau'_M\otimes^\mathcal{B} S^{\ell}\tau'_M \to S^{k}\tau'_M\otimes^\mathcal{B} S^{\ell-1}\tau'_M.
\]
%We call the operators $\dB$ and $\deB$ the \textit{gradient}- and the \textit{divergence}- operator on the trace-free class. 
However we will verify later in Lemma \ref{Le:dif_ops_B} that $\mathrm{d}^\mathcal{B}$ and $-\deB$ are well defined and formally adjoint to each other.
%\textbf{[(TS) Note that the range of $\deB$ is not \textit{a priori} clear. Could you please motivate Theorem 1 a bit, before its statement?]}
% The first main result of this paper is \textbf{[JZ: I would not use "first main result", "second main result". They are not two parallel results. The decomposition is not of independent interest.]}

%\bigskip
%\textbf{[Also Gunther agreed that we should emphasize the new decomposition even more](TS)}
The following tensor decomposition plays an essential role in the analysis of the mixed ray transform.
\begin{theorem}\label{Th:tensor_decomp}
For any $f\in H^m(S^k\tau'_M\otimes^\mathcal{B} S^\ell \tau'_M)$, $k=\ell\in \{1,2\}, \: m \in \{0,1,2,\ldots\}$, there exists a unique decomposition 
\begin{equation}\label{decomp_higher}
f=f^s+\mathrm{d}^\mathcal{B}v,
\end{equation}
with $f^s\in H^m(S^k\tau'_M\otimes^\mathcal{B} S^\ell \tau'_M)$, $\delta^\mathcal{B}f^s=0$, and $v \in  H^{m+1}( S^k\tau'_M\otimes^\mathcal{B} S^{\ell-1} \tau'_M)$, $v|_{\p M}=0$. In addition there exists a constant $C>0$ such that
\begin{equation}
\label{eq:estimates_for_decomp}
\|f^s\|_{H^m}\leq C\|f\|_{H^m}, \quad \|v\|_{H^{m+1}}\leq C\|\deB f\|_{H^{m-1}}.
\end{equation}
\end{theorem}
%\textbf{[JZ: I think the blue $C$ could be $1$, would you please check? It is some projection.]}
%\textbf{[Should we drop the first $s$-injectivity definition as we want to emphasize the trace free class? (TS)]}
%We say that $L_{k,\ell}$ is $s$-injective in $M$ if $L_{k,\ell}f=0$,
%$f\in L^2(S^k\tau'_M\otimes S^\ell \tau'_M)$ implies that
%$f=\mathrm{d}'v+\lambda w$ with some tensor fields $v\in
%H^1_0(S^k\tau'_M\otimes S^{\ell-1} \tau'_M)$ and $w\in
%L^2(S^{k-1}\tau'_M\otimes S^{\ell-1} \tau'_M)$. 
%{\color{blue}

This theorem will be proved in Section \ref{sec:Unique_dec}. We note that a decomposition equivalent to \eqref{decomp_higher} has been provided earlier by Sharafutdinov \cite[Lemma 7.2.1]{Shara}: for any $f\in
L^2(S^k\tau'_M\otimes S^\ell \tau'_M), \: k, \ell \geq 1$, there is a
decomposition
\begin{equation}\label{decomp_threeparts}
   f = f^s + \mathrm{d}'v + \lambda w,
\end{equation}
with $ \mu f^s=0, \: \delta'f^s=0,\: \hbox{ for some } v\in
H^1_0(S^{k}\tau'_M\otimes S^{\ell-1} \tau'_M),
\: w\in L^2(S^{k-1}\tau'_M\otimes S^{\ell-1} \tau'_M)$. Moreover if
$\ell \geq 2$ we can choose $v$ such that $\mu v=0$.
The equivalence of the decompositions \eqref{decomp_higher} and \eqref{decomp_threeparts} can be observed by noticing that
\[
\mathrm{d}^\mathcal{B}v-\mathrm{d}'v
%=-\frac{k\cdot\ell}{(k+\ell+1)\cdot 1}\lambda\mu\mathrm{d}'v
\in \mathrm{Im}{\lambda},\quad \hbox{ and }\quad \mathcal B f-f\in \mathrm{Im}{\lambda},\quad v \in L^2(S^k\tau'_M\otimes^\mathcal{B} S^{\ell-1}\tau'_M), \:  f \in L^2(S^k\tau'_M\otimes S^\ell \tau'_M)
\]
and rearranging terms. We remark that it was shown in \cite{Shara} that the solenoidal part $f^s$, in \eqref{decomp_threeparts} is uniquely determined by $f$, but the uniqueness of $v$ and $w$ was not proven. However the uniqueness of a quite similar decomposition has been proved in \cite{dairbekov2011conformal}.

\subsection*{Main result}
%\textbf{[New(TS)]}
%{\color{blue}
The main purpose of this paper is to establish the $s$-injectivity of $L_{1,1}$ and $L_{2,2}$ for $g$ in a generic subset of all simple metrics on $M$. We also provide a stability estimate for the corresponding normal operators. The analogous result for $I_{2}=L_{0,2}$ is given in \cite[Theorem 1.5]{SU}.
We will present a detailed proof for $L_{1,1}$. The proof is similar for $L_{2,2}$, modulo some key calculations which we will also provide.

We then introduce some necessary notations in order to state the main result of this paper.  
\color{black}
We write $L_{g}=L_{k,\ell}$ to emphasize the dependence on the metric
$g$. We denote 
%\color{blue}
the $L^2$-normal operator $L_g^\ast L_g$ 
\color{black}
of the mixed ray transform by
$\mathcal{N}_g$ (see Section \ref{Se:normal_op} for the rigorous
definitions).  Since $(M,g)$ is simple we can without loss of
generality assume that $M \subset \R^3$ with a simple metric $g$ that is smoothly extended in whole $\R^3$. 
%
%\color{blue}
%
\color{black}
Thus we can find a small open neighborhood $M_1$ of $M$, such that $(\overline{M_1},g)$ is a simple, (see \cite[page 454]{stefanov2004stability}). A tensor
field $f$ defined on $M$ will be extended by a zero field to
$M_1\setminus M$ . We note that this creates jumps at the boundary $\p
M$. To tackle this, the $\tilde{H}_2$-norm was introduced in
\cite{stefanov2004stability} (see also Section
\ref{Se:Parametrix}). 
%{\color{blue} 
As the decomposition \eqref{decomp_higher} depends on the domain, we use the notation $f_M^s$ for the solenoidal part of $f$ on $M$ to emphasize this. Our main result is

\begin{theorem}\label{maintheorem}
Let $(k,\ell)=(1,\,1)$ or $(2,\,2)$. There exists an integer $m_0$
such that for each $m\geq m_0$, the set $\mathcal{G}^m(M)$ of simple
$C^m$-regular metrics in $M$, for which $L_{g}$ is s-injective, is open and
dense in the $C^m$-topology. Moreover, for any $g\in\mathcal{G}^m$,
\[
  \|f^s_M\|_{L^2(M)} \leq C\|\mathcal{N}_gf\|_{\tilde{H}_2(M_1)},
  \quad\quad
  f \in H^1(S^k\tau'_M\otimes^\mathcal{B} S^\ell \tau'_M),
\]
with a constant $C>0$ that can be chosen locally uniformly in
$\mathcal{G}^m(M)$ in the $C^m(M)$-topology.
\end{theorem}
%\textbf{[Note that in the earlier version we had essentially the wrong $f^s_M$ in this theorem](TS)}
%\medskip

%\color{blue}
The $s$-injectivity of $L_{k,\ell}$, $k,\ell\geq 1$, has been proved
for two-dimensional simple manifolds in \cite{de2018mixed}. On higher dimensional manifolds, the $s$-injective was
established in \cite[Theorem 7.2.2]{Shara} under restrictions on the sectional
curvature of $(M,g)$. In both of these aforementioned papers, the sharper tensor decomposition \eqref{decomp_higher} is not needed. However, in this paper the decomposition \eqref{decomp_higher} is a key component of the proof of Theorem \ref{maintheorem}.

We also refer to \cite{krishnan2018s} for the study of a related problem in dimension two. %A similar natural obstruction for injectivity exists also for  \textit{the light ray transform} on Lorentzian manifolds. 
In a recent paper \cite{feizmohammadi2019light} the authors showed that on globally hyperbolic stationary Lorentzian manifolds, \textit{the light ray transform} is injective up to a similar natural obstruction that $L_{1,1}$ has.
In the Appendix \ref{Se:app_1} we relate the mixed ray transform $L_{2,2}$ to the averaged travel-times of $qS$-polarized elastic waves. 
However we note that the travel-time data alone only gives us partial information about the mixed ray transform. If in addition we include the measurement of the shear wave amplitude, the complete mixed ray transform can be recovered. In the Appendix \ref{dnmap}, we will show how one can obtain the mixed ray transform from a linearization of the \text{Dirichlet-to-Neumann} map of an elastic wave equation on a smooth bounded domain $M\subset \R^3$. Here we relay on the observation that both the travel-time and the amplitude are encoded in the Dirichlet-to-Neumann map. 
%We also refer to \cite[Chapter 7]{Shara} for an alternative derivation of the mixed ray transform.
%we will recall how the mixed ray transform is related to elastic \textit{qS}-wave travel-time tomography. We also connect the operator $L_{2,2}$ to the Dirichlet-to-Neumann map of the elastic wave equation.
%we recall a linearization scheme for weakly anisotropic travel-time measurements of $qS$ waves that leads to an integral geometric problem related to the mixed ray transform. 
We refer to \cite[Chapter 7]{Shara} for an alternative approach to obtain $L_{2,2}$.
% The $s$-injectivity result of this linear problem is a key step towards solving the nonlinear travel-time tomography problem.
%\textbf{[Here we should recall something about \cite{SU}, to justify the claim. It might be too much to say it in this way as we don't have the result yet (TS)]}

\subsection*{Outline of the proof}
%\color{blue}
%\color{red}
In the beginning of the Section \ref{sec:Unique_dec} we find an explicit representation for the projection $\mathcal B$ onto the space of trace-free tensors.
%$f=f^\mathcal{B}+\lambda v$ for
%unique $f^\mathcal{B}\in S^{k} \tau'_M \otimes S^{\ell} \tau'_M$ and
%$v \in S^{k-1} \tau'_M \otimes S^{\ell-1} \tau'_M$ such that the \textit{trace-free condition} $\mu f^\mathcal{B}=0$ holds. 
%Under the assumption $k\geq \ell\geq 1$  we derive explicit formulas for the projections onto the subspaces $\ker \mu$ and Im$\lambda$ and 
%{\color{red} 
Then we prove Theorem \ref{Th:tensor_decomp} in the case $k=\ell=2$. The rest of the paper is devoted to prove Theorem \ref{maintheorem}. We give detailed proof for the case
$k=\ell=1$ in Sections \ref{sec:Unique_dec}-\ref{Se:generic_s_inj},
and discuss the required modifications for the $k=\ell=2$ in the
final section. 
%that, on the null space }of $\mu$ any $L^2$-regular tensor field admits a
%{\color{red}Helmholtz decomposition with unique solenoidal and potential components. Using this new tensor decomposition, we redefine the $s$-injectivity of the mixed ray transform for the trace free class of $L^2(S^{k}
%\tau'_M \otimes S^{\ell} \tau'_M), \: k=\ell \in \{1,2\}$,.}

In sections \ref{Se:normal_op}-\ref{Se:generic_s_inj} we study
the mixed ray transform on $1+1$ tensor fields $f$ satisfying the
trace-free condition. Section \ref{Se:normal_op} is dedicated to the
study of the normal operator 
%{\color{red}
$\mathcal{N}_g$ of the mixed ray
transform on $1+1$ tensor fields. First we show that $\mathcal{N}_g$ is an integral operator and find its Schwartz kernel. In the second part of the section we prove that the normal operator is
a pseudo-differential operator ($\Psi$DO) of order $-1$. We also give
an explicit coordinate-dependent formula for the principal symbol of
this operator. 
%If the metric $g$ has constant coefficients, then this
%formula gives the full symbol of the corresponding normal operator.

%Since any potential tensor field is annihilated by {\color{red}$\mathcal{N}_g$},
%the normal operator of the mixed ray transform is not an elliptic
%$\Psi$DO. 
%{\color{red}
Since in Theorem \ref{maintheorem} we assumed that the metric is only finitely smooth we start Section \ref{Se:Parametrix} by recalling some basics of the theory of $\Psi$DO's whose amplitudes are only finitely smooth. This is needed to establish the continuity of $\mathcal{N}_g$, and several other operators, with respect to metric $g$ in $C^m$-topology. We prove that $\mathcal{N}_g$ is elliptic
acting on the solenoidal tensor fields. 
%Here a tensor field is called solenoidal if $f^s=f$. 
This manifests the Fredholm nature of the normal operator on some extended simple manifold. Then we can recover the solenoidal part (on the extended manifold) of the tensor field $f$ from
$\mathcal{N}_gf$ 
%{\color{red}
modulo a finitely smooth term.
 %We note that the Helmholtz
%decompostion depends on the domain where the corresponding second
%order boundary value problem is defined. 
In the second part of Section
\ref{Se:Parametrix} we compare the solenoidal parts 
%{\color{red}
on original manifold and on the extended manifold. Then we establish a
reconstruction formula for the solenoidal part of tensor fields 
%{\color{red}
on the original manifold. 
%{\color{red}
We also give a stability estimate for the normal operator (see Theorem \ref{th:stability}).

In Section \ref{Se:analytic_s_injective}, we prove the $s$-injectivity of the mixed ray transform on analytic simple
Riemannian manifolds (see Theorem
\ref{sinjectivity_analytic}). Since analytic metrics are $C^m$-dense
in the space of all simple metrics, Theorem
\ref{sinjectivity_analytic} can be used to prove Theorem \ref{maintheorem}  in Section
\ref{Se:generic_s_inj}.
%To do this we show that parametrix $(\LaB)^{-1}$ of
%$\LaB$ depends $C^k$-continuously on the metric. We note that this
%requires a refining for $(\LaB)^{-1}$ since the space $\ker \mu$,
%where this operator is defined, is metric dependent.

\section{Decomposition of the trace-free tensor fields}
\label{sec:Unique_dec}
%In this section we refine the decomposition \eqref{decomp_threeparts} \color{blue} to suit better the analysis of the mixed ray transform. 
We begin this section by finding an explicit formula for the projection $\mathcal{B}$ from $S^{k}\tau'_M\otimes S^{\ell} \tau'_M$ onto $\ker \mu=S^{k}\tau'_M\otimes^{\mathcal{B}} S^{\ell} \tau'_M$.
% which we then use to find coordinate representation for the corresponding projections $\dB$ and $\deB$ of the differential operators $\mathrm{d}'$ and $\delta'$. 
%In the second part of this section we prove Theorem \ref{Th:tensor_decomp}.
 \color{black}
  
\subsection{Domain of the mixed ray transform}
%\color{red}
We choose some $f\in S^k\tau'_M\otimes S^\ell \tau'_M$  and write
\begin{equation}\label{decomp_B}
f=\mathcal{B}f+\lambda w, \quad w\in S^{k-1}\tau'_M\otimes S^{\ell-1} \tau'_M.
\end{equation}
In the following lemma we find a representation for $w$ in \eqref{decomp_B} under the assumption $k\geq \ell \geq 1$.  

\begin{lemma}\label{Le:mu_comp_B=0}
Any tensor field $f\in S^k\tau'_M\otimes S^\ell \tau'_M, \: k\geq \ell \geq  1$ admits the decomposition \eqref{decomp_B}, where $w\in S^{k-1}\tau'_M\otimes S^{\ell-1} \tau'_M$ is given by the formula 
\[
w=\mathcal{A}f:=\left(\sum_{K=1}^{\ell}\left((-1)^{K+1}\frac{\prod_{h=1}^{K-1}b_h}{\prod_{h=1}^Ka_h}\right)\lambda^{K-1}\mu^K\right)f,
\]
and 
\[
\mathcal{B}f=(\mathrm{Id}-\lambda \mathcal{A})f, 
\]
where
\[
a_h= \frac{h(k+\ell+2-h)}{k\ell}, \quad b_h =\frac{(k-h)(\ell-h)}{k\ell}.
\]
\end{lemma}

%\textbf{(The lemma used to be after the proof (TS))}

\color{black}

\begin{proof}
To begin we derive  the following formula for the commutator of  $\lambda$ and  $\mu $
\begin{equation}
\label{eq:commutator_formula}
\begin{split}
\mu\lambda w=&\left(\sym(i_1\ldots  i_k)\sym(j_1\ldots  j_\ell)(w_{i_1\ldots i_{k-1}j_1\ldots j_{\ell-1}})g_{i_kj_l}\right)g^{i_kj_l}
\\
=&\frac{1}{k\ell}(w_{i_1\ldots i_{k-1}j_1\ldots j_{\ell-1}})g_{i_kj_l}g^{i_kj_l}
\\
+&\frac{k-1}{k\ell}w_{i_1\ldots i_{k-2},i_kj_1\ldots j_{\ell-1}}g_{i_{k-1}j_\ell}g^{i_kj_\ell}
\\
+&\frac{\ell-1}{k\ell}w_{i_1\ldots i_{k-1}j_1\ldots j_{\ell-2}j_\ell}g_{i_{k}j_{\ell-1}}g^{i_kj_\ell}
\\
+&\sym(i_1\ldots  i_{k-1})\sym(j_1\ldots  j_{\ell-1})w_{i_1\ldots i_{k-2}i_kj_1\ldots j_{\ell-2}j_\ell}g_{i_{k-1}j_{\ell-1}}g^{i_kj_\ell}
\\
=&\frac{k+\ell+1}{k\ell}w+\frac{(k-1)(\ell-1)}{k\ell}\lambda\mu w.
\end{split}
\end{equation}
In the case $k=\ell=1$, $w$ is a function, and formulas \eqref{decomp_B} and \eqref{eq:commutator_formula} imply
\begin{equation}
\label{eq:w_in_1,1_case}
w=\frac{\mu f}{3}.
\end{equation}

To proceed for the higher order tensors we assume $\max\{k, \ell\}\geq 2$ and use the commutator formula \eqref{eq:commutator_formula} to prove that for $m\in \{2, \ldots, \min\{k,\ell\}\}$ we have 
\begin{equation}
\label{eq:commutator_higher_order}
\mu^{m-1}\lambda w =  a_{h}\mu^{m-2}w+b_{h}\mu^{m-1-h}\lambda\mu^{h} w,  \quad h\in \{1, \ldots, m -1\},
\end{equation}
where
\begin{equation}
\label{eq:coef_of_commutator}
a_{h}=\frac{1}{k\ell}\sum_{r=1}^{h}x_r, \quad x_r:=k+\ell+3-2r, \quad b_{h}=\frac{(k-h)(\ell-h)}{k\ell}.
\end{equation}
We note that the case $m=2$ is the same as \eqref{eq:commutator_formula}. If $m>2$ we do an induction over $h$. 
%we have due to \eqref{eq:commutator_formula}
%\[
%\begin{split}
%\mu^{m-1}\lambda w &
%\\
%=&\mu^{m-2}\left(\frac{k+\ell+1}{k\ell}w+\frac{(k-1)(\ell-1)}{k\ell}\lambda\mu w\right)
%\\
%=&\frac{k+\ell+1}{k\ell}\mu^{m-2}w+\frac{(k-1)(\ell-1)}{k\ell}\mu^{m-2}\lambda\mu w
%%\\
%%=&\frac{k+\ell+1}{k\ell}\mu^{m-2}w+\frac{(k-1)(\ell-1)}{k\ell}\mu^{m-3}\left(\frac{k+\ell-1}{(k-1)(\ell-1)}\mu w+\frac{(k-2)(\ell-2)}{(k-1)(\ell-1)}\lambda\mu^2 w\right)\\
%%=&\left(\frac{k+\ell+1}{k\ell}+\frac{k+\ell-1}{k\ell}\right)\mu^{m-2}w+\frac{(k-2)(\ell-2)}{k\ell}\mu^{m-3}\lambda\mu^2 w.
%\\
%\end{split}
%\]
The initial step of the induction follows from  \eqref{eq:commutator_formula}.  For the induction we note that  
\\
$\mu^h w \in S^{k-h-1}\tau'_M\otimes S^{\ell-h-1} \tau'_M$. Due to \eqref{eq:commutator_formula} we have
\[
\mu^{m-1-h}\lambda\mu^{h} w=\mu^{m-2-h} \left(\frac{k+\ell-2h+1}{(k-h)(\ell-h)}\mu^h w+\frac{(k-h-1)(\ell-h-1)}{(k-h)(\ell-h)}\lambda\mu^{h+1}w\right).
\]
Therefore if \eqref{eq:coef_of_commutator} holds for $h \in \{1,\ldots,  m-2\}$, it also holds for $h+1$.

Next we note that  for any $m \leq \min\{k,\ell\}$  the formulas $(\ref{decomp_B})$ and \eqref{eq:commutator_higher_order} imply
\[
\mu^{m-1}f=b_{m-1}\lambda\mu^{m-1} w+a_{m-1}\mu^{m-2}w.
\]
We denote $K=m-1$. Thus for any $K\in \{1, \ldots, \min\{k,\ell\}- 1\}$, it holds
\begin{equation}
\label{eq:recursive_formula_for_w}
\mu^{K}f=b_K\lambda\mu^{K} w+a_K \mu^{K-1}w, \quad a_K= \frac{K(k+\ell+2-K)}{k\ell}, \quad b_K =\frac{(k-K)(\ell-K)}{k\ell}.
\end{equation}

Now we assume that $\ell\leq k$ since we are mostly interested in the case $k=\ell=2$. The case $\ell>k$ can be dealt with similarly. 

%\textbf{Maybe we should not mention the case $\ell>k$?}

We choose $K=\ell-1$ and apply $\mu$ to both sides of the equation \eqref{eq:recursive_formula_for_w} to  get
\[
\mu^{\ell}f=b_{\ell-1}\mu \lambda\mu^{\ell-1}w+a_{\ell-1}\mu^{\ell-1}w.
\]
%Using \eqref{eq:commutator_higher_order} and \eqref{eq:coef_of_commutator} we obtain
%\begin{equation}\label{mul}
%\begin{split}
%\mu^{\ell-1} f=\mu^{\ell-1}\lambda w=
%%=&\mu^{\ell-2}\left(\frac{k+\ell+1}{k\ell}w+\frac{(k-1)(\ell-1)}{k\ell}\lambda\mu w\right)\\
%%=&\frac{k+\ell+1}{k\ell}\mu^{\ell-2}w+\frac{(k-1)(\ell-1)}{k\ell}\mu^{\ell-2}\lambda\mu w\\
%%=&\frac{k+\ell+1}{k\ell}\mu^{\ell-2}w+\frac{(k-1)(\ell-1)}{k\ell}\mu^{\ell-3}\left(\frac{k+\ell-1}{(k-1)(\ell-1)}\mu w+\frac{(k-2)(\ell-2)}{(k-1)(\ell-1)}\lambda\mu^2 w\right)\\
%%=&\left(\frac{k+\ell+1}{k\ell}+\frac{k+\ell-1}{k\ell}\right)\mu^{\ell-2}w+\frac{(k-2)(\ell-2)}{k\ell}\mu^{\ell-3}\lambda\mu^2 w\\
%%=&\left(\frac{k+\ell+1}{k\ell}+\frac{k+\ell-1}{k\ell}+\cdots+\frac{k-\ell+5}{k\ell}\right)\mu^{\ell-2}w+\frac{k-\ell+1}{k\ell}\lambda\mu^{\ell-1}w\\
%%=&
%\frac{(k+3)(\ell-1)}{k\ell}\mu^{\ell-2}w+\frac{k-\ell+1}{k\ell}\lambda\mu^{\ell-1}w.
%\end{split}
%\end{equation}
%%\[
%%\frac{k+\ell+1}{k\ell}+\frac{k+\ell-1}{k\ell}+\cdots+\frac{k-\ell+5}{k\ell}=\frac{1}{k\ell}\sum_{m=0}^{\ell-2}x_m, \quad x_m:=k+\ell+1-2m.
%%\]
%%Since $x_m$ is an arithmetic sequence we  have
%%\[
%%\sum_{m=0}^{\ell-1}x_m=\frac{\ell(x_0+x_{\ell-1})}{2}=\frac{\ell(k+\ell+1+k-\ell+3)}{2}=(\ell-1)(k+3).
%%\]
%%
%%\[
%%a_m\mu^{\ell-2}w+b_m\mu^{\ell-m-1}\lambda\mu^{m} w, \quad  a_0=0, \: b_0=1.
%%\]
We note that for any $v\in S^{m}\tau'_M$,
$
%\begin{split}
\mu \lambda v=
%&(\sym(i_1\ldots  i_{m+1})v_{i_1\ldots  i_m}g_{i_{m+1}j})g^{i_{m+1}j}
%\\
%=&\frac{1}{m+1}\bigg(v_{i_1\ldots  i_m}g_{i_{m+1}j}g^{i_{m+1}j}+m v_{i_1\ldots  i_{m-1},i_{m+1}}g_{i_{m}j}g^{i_{m+1}j}\bigg)
%\\
%=&\frac{1}{m+1}\bigg(3v_{i_1\ldots  i_m}+m v_{i_1\ldots  i_{m}}\bigg)
%\\
%=&
\frac{m+3}{m+1}v.%_{i_1\ldots  i_m}.
%\end{split}
$
This implies
\[
\mu \lambda\mu^{\ell-1}w=\frac{k-\ell+3}{k-\ell+1}\mu^{\ell-1}w,
\]
and we have found the formula
\[
\begin{split}
\mu^{\ell-1}w=
\left( b_{\ell-1}\frac{k-\ell+3}{k-\ell+1}+a_{\ell-1}\right)^{-1}\mu^\ell f
%=\left( \frac{(k-\ell+1)}{k\ell}\frac{k-\ell+3}{k-\ell+1}+\frac{(\ell-1)(k+3)}{k\ell} \right)^{-1}\mu^\ell f
%\\
%=\left( \frac{k-\ell+3}{k\ell}+\frac{(\ell-1)(k+3)}{k\ell} \right)^{-1}\mu^\ell f
%\\
=\frac{k}{k+2}\mu^\ell f=\frac{1}{a_\ell}\mu^\ell f. 
\end{split}
\]

%Applying $\mu$ to \eqref{mul}, we get
%\begin{equation}\label{wl1}
%\mu^{\ell} f=\frac{k+2}{k}\mu^{\ell-1}w. 
%\end{equation}
%
%We use \eqref{wl1} to write \eqref{mul} in the form
%\[
%\frac{(k+3)(\ell-1)}{k\ell}\mu^{\ell-2}w=\mu^{\ell-1} f-\frac{k-\ell+1}{k\ell}\frac{k}{k+2}\lambda\mu^\ell f,
%\]
%i.e.,
%\[
%\mu^{\ell-2}w=\frac{k\ell}{(k+3)(\ell-1)}\mu^{\ell-1}f-\frac{(k-\ell+1)k}{(k+2)(k+3)(\ell-1)}\lambda\mu^\ell f.
%\]
%Similar to $(\ref{mul})$, we can show that for $K<\ell$
%\[
%\frac{(k+\ell-K+2)K}{k\ell}\mu^{K-1}w=\mu^Kf-\frac{(k-K)(\ell-K)}{k\ell}\lambda\mu^K w,
%\]
%which gives
%\[
%\mu^{K-1}w=\frac{k\ell}{(k+\ell-K+2)K}\mu^Kf-\frac{(k-K)(\ell-K)}{(k+\ell-K+2)K}\lambda\mu^K w.
%\]

By the recursion formula \eqref{eq:recursive_formula_for_w} we get
\begin{equation}
\label{eq_v}
\begin{split}
w=&\frac{\mu f}{a_{1}}-\frac{b_{1}}{a_{1}} \lambda \mu w=\frac{\mu f}{a_{1}}-\frac{b_{1}}{a_{1}}\lambda \left(\frac{\mu^2 f}{a_{2}}- \frac{b_{2}}{a_{2}} \lambda \mu ^2w \right)
\\
=&\frac{\mu f}{a_{1}}- \frac{b_{1}}{a_{1}a_{2}}\lambda\mu^2 f+ \frac{b_{1}b_{2}}{a_{1}a_{2}} \lambda^2 \mu ^2w
\\
=& \cdots
\\
=&\sum_{K=1}^{\ell-1}\left((-1)^{K+1}\frac{\prod_{h=1}^{K-1}b_h}{\prod_{h=1}^Ka_h}\right)\lambda^{K-1}\mu^K f+\left((-1)^{\ell+1}\frac{\prod_{h=1}^{\ell-1}b_h}{\prod_{h=1}^{\ell-1}a_h}\right)\lambda^{\ell-1}\mu^{\ell-1}w
\\
=&\left(\sum_{K=1}^{\ell}\left((-1)^{K+1}\frac{\prod_{h=1}^{K-1}b_h}{\prod_{h=1}^Ka_h}\right)\lambda^{K-1}\mu^K\right)f.
\end{split}
\end{equation}
The last row is the representation we were looking for.
\end{proof}

We recall that at the Section \ref{Se:intro} we had given the formal definitions for the gradient operator
\[
\dB:=\mathcal{B}\mathrm{d}'\colon H^1(S^{k}\tau'_M\otimes^\mathcal{B} S^{\ell-1 }\tau'_M) \to L^2(S^{k}\tau'_M\otimes^\mathcal{B} S^{\ell}\tau'_M),
\]
and divergence operator
\[
-\deB:=-\delta'
%=\delta'-\frac{1}{a_1}
%\frac{k\cdot\ell}{(k+\ell+1)\cdot 1}
%\delta'\lambda\mu=\delta'\colon\quad 
\colon H^1(S^{k}\tau'_M\otimes^\mathcal{B} S^{\ell}\tau'_M) \to L^2(S^{k}\tau'_M\otimes^\mathcal{B} S^{\ell-1}\tau'_M)
\]
on the trace-free class.
\color{black}

\begin{lemma}
\label{Le:dif_ops_B}
The differential operators $\dB$ and $-\deB$ are well defined, formally adjoint to each other and 
\begin{equation}
\label{eq:dB}
\mathrm{d}^\mathcal{B} v=\mathrm{d}'v-\frac{1}{a_1} \lambda \mu  \mathrm{d}'v, \quad v \in S^{k}\tau'_M\otimes^\mathcal{B} S^{\ell-1 }\tau'_M, \quad \hbox{ when } k\geq \ell\geq 1.
\end{equation}
\end{lemma}
\begin{proof}
The operator $\dB$ is clearly well defined by its definition, and the operator $\deB$ is well defined  since $\mu$ and $\delta'$ commute.

%\begin{align*}
%(\mu \delta' f)_{i_1,\ldots i_{k-1}j_1\ldots  j_{\ell-2}}=&\left( (\nabla_h f_{i_1,\ldots i_{k}j_1\ldots  j_{\ell}})g^{hj_{\ell}}\right) g^{i_k j_{\ell-1}}=\left( (\nabla_h f_{i_1,\ldots i_{k}j_1\ldots  j_{\ell}})g^{i_k j_{\ell}}\right) g^{hj_{\ell-1}}
%\\
%=&\left( \nabla_h( f_{i_1,\ldots i_{k}j_1\ldots  j_{\ell}}g^{i_k j_{\ell}})\right) g^{hj_{\ell-1}}=(\delta' \mu f)_{i_1,\ldots i_{k-1}j_1\ldots  j_{\ell-2}}.
%\end{align*}
We note that for any $u \in S^{k}\tau'_M\otimes^\mathcal{B} S^{\ell-1 }\tau'_M$ we have $\mu^2 \mathrm{d}' u=0$.
%\begin{align*}
%(\mu^2 d' u)=&\left(\sym(j_1\ldots  j_{\ell-1},h)\nabla_h u_{i_1,\ldots i_kj_1,\ldots j_{\ell-1}}\right)g^{i_kh}g^{i_{k-1}j_{\ell-1}}
%\\
%=& \frac{1}{\ell} \left( \left(\nabla_h u_{i_1,\ldots i_kj_1,\ldots j_{\ell-1}}\right)g^{i_{k-1}j_{\ell-1}}g^{i_kh} + \sum_{m=1}^{\ell-1}\left(\nabla_{j_m} u_{i_1,\ldots i_kj_1\ldots j_{m-1}hj_{m+1}, \ldots j_{\ell-1}}\right)g^{i_{k-1}j_m}g^{i_kh} \right)
%\\
%=& \frac{1}{\ell} \left( \nabla_h \left( u_{i_1,\ldots i_kj_1,\ldots j_{\ell-1}}g^{i_{k}j_{\ell-1}} \right)g^{i_{k-1}h}
%%\\
%%+& 
%+\sum_{m=1}^{\ell-1} \nabla_{j_m} \left( u_{i_1,\ldots i_kj_1\ldots j_{m-1} h j_{m+1} \ldots j_{\ell-1}}g^{i_k h}\right)g^{i_{k-1}j_m }\right) 
%\\
%=&0.
%\end{align*}
Therefore operator $\dB$ has the representation \eqref{eq:dB}. The proof of this claim is a direct consequence of the fact that the Levi-Civita connection commutes with any contraction. 

The operators $\dB$ and $-\deB$ are formal adjoints to each other since $\mathrm{d}'$ and $-\delta'$, and also  $\lambda$ and $\mu$, are formal adjoints respectively.
\end{proof}

\subsection{Tensor decomposition in the kernel of $\mu$}
%\color{red}
In the $L^2$-space of $m$-tensor fields on $M$ we use the standard
definition of the inner product
\[
\langle f,h\rangle_{g}=\int_M f_{i_1\cdots i_m}\overline{h}_{j_1\cdots j_m}g^{i_1j_1}\cdots g^{i_mj_m}(\det g)^{1/2}\mathrm{d}x.
\]
%\color{blue}
Assuming the result of the following Lemma we are ready to present the proof of Theorem \ref{Th:tensor_decomp}.
\color{black}
\begin{lemma}
\label{existence2}
Let $(M,g)$ be a smooth Riemannian manifold. There exists a unique solution 
\[
u\in H^m(S^k\tau'_M\otimes^\mathcal{B} S^{\ell-1} \tau'_M),
%\cap H^1_0(S^k\tau'_M\otimes^\mathcal{B} S^{\ell-1} \tau'_M), 
\:k=\ell\in \{1,2\}, \: m \in 1,2,\ldots
\] 
to the boundary value problem
\begin{equation}\label{eqnu1}
\LaB u:=\delta^\mathcal{B}\mathrm{d}^\mathcal{B}u=h\quad\text{in}~M^{int}, \quad u\vert_{\partial M}=w,
\end{equation}
for any $h\in H^{m-2}(S^k\tau'_M\otimes^\mathcal{B} S^{\ell-1} \tau'_M)$ and $w \in H^{m-\frac{1}{2}}(S^k\tau'_{M}\otimes^\mathcal{B} S^{\ell-1} \tau'_{ M}\vert_{\partial M})$. Moreover there exists $C>0$ such that the following energy estimate is valid
\begin{equation}
\label{eq:energy_estimate}
\|u\|_{H^m(M) }\leq C \left( \|h\|_{H^{m-2}(M) }+ \|w\|_{H^{m-\frac{1}{2}}(\p M)}\right).
\end{equation}
%If $\ell=1$, the (not well defined) condition $\mu u=0$ is redundant.
\end{lemma}
%\textbf{Double check the proof of this Lemma in the case $(M,g)$ is not simple!}

\begin{proof}[Proof of Theorem \ref{Th:tensor_decomp}]
%\color{blue}
We consider the boundary value problem \eqref{eqnu1} with the zero boundary value $w\equiv 0$. Let $(\LaB)^{-1}$ be the corresponding solution operator. We denote  $v:=(\LaB)^{-1}\deB f$. Thus $v$ solves the problem 
\begin{equation}\label{eq_Blaplacian}
\LaB v=\delta^\mathcal{B}f,\quad\quad v\vert_{\partial M}=0,
\end{equation}
and the energy estimate \eqref{eq:energy_estimate} implies that the  \textit{projection operator onto the potential fields} $\mathcal{P}_M:=\dB(\LaB)^{-1}\deB $ is a bounded operator in $H^{m}(S^k\tau'_M\otimes^\mathcal{B} S^{\ell} \tau'_M)$. We define a second bounded operator by setting $\mathcal S_M:=I-\mathcal P_M$ and call this the  \textit{projection operator onto the solenoidal tensor fields}. Finally we denote $f^s:=\mathcal S_M f$ and obtain
\[
f=f^s+\dB v, \quad \hbox{ with } \quad  \deB f^s=0.
\]
The estimate \eqref{eq:estimates_for_decomp} follows from the boundedness of the operators $\mathcal S_M$ and $(\LaB)^{-1}$.
\color{black}
%We suppose first that \eqref{decomp_higher} holds and derive a second order elliptic boundary value problem to find $v$. We apply $\delta^\mathcal{B}$ to $(\ref{decomp_higher})$, and get an equation
%\begin{equation}\label{eq_Blaplacian}
%\delta^\mathcal{B}\mathrm{d}^\mathcal{B}v=\delta^\mathcal{B}f,\quad\quad v\vert_{\partial M}=0,
%\end{equation}
%or equivalently
%\[
%\delta'\mathrm{d}'v-\frac{1}{a_1}\delta'\lambda\mu\mathrm{d}'v=\delta' f, \quad\quad v\vert_{\partial M}=0.
%\]
%Therefore to find $v $ we only need to show the existence and the uniqueness of the solution to this boundary value problem. This is done in the following lemma. However we only give an explicit proof under the assumption $k=\ell=2$, in which case the equation \eqref{eq_Blaplacian} has the form
%\[
%\delta'\mathrm{d}' v-\frac{4}{5}\delta'\lambda\mu\mathrm{d}'v=\delta'f , \quad\quad v\vert_{\partial M}=0.
%\]
\end{proof}
%The proof for the case $k=\ell=1$ is similar but simpler.
%\color{red}
\begin{remark}
\label{Re:projections}
The operators $\mathcal{S}_M$ and $\mathcal{P}_M$ are both projections, i.e., $\mathcal{S}_M^2=\mathcal{S}_M$, $\mathcal{P}_M^2=\mathcal{P}_M$. 
These projections are formally self-adjoint since $\LaB$ is formally self-adjoint and thus its inverse $(\LaB)^{-1}$ is also formally self-adjoint, see \cite[Theorem 10.2-2]{kreyszig1978introductory}.
\end{remark}
\color{black}

%\color{blue}
The rest of this section is devoted to the proof of Lemma \ref{existence2}. We first recall some facts about the solvability of boundary value problems for elliptic systems. See for instance \cite[Section 9]{wloka1995boundary} for a thorough review. 
%\color{blue} 
Recall that we can without loss of generality assume that $M^{int}\subset \R^3$ is a domain with a smooth boundary. We use the notations $T^\ast M$ and $T^\ast \R^3$ for the cotangent bundles of $M$ and $\R^3$ respectively.

Let $\alpha \in \N^3$ be a multi-index and 
$
D^{\alpha}=(-\mathrm{i})^{|\alpha|}\p_{x_1}^{\alpha_1}\p_{x_2}^{\alpha_2}\p_{x_3}^{\alpha_3}.
$
We say that a differential operator $A=(\sum_{|\alpha| \leq 2}A^\alpha_{ij}(x)D^\alpha)_{i,j=1}^3$ is a second order (homogeneously) elliptic operator if the order of the operator $\sum_{|\alpha| \leq 2}A^\alpha_{ij}(x)D^\alpha$ is two for any $i,j$ and the\textit{ characteristic polynomial} $\chi(x,\xi)$ of the operator $L$ does not vanish outside the set $\R^3\times \{\xi =0\} \subset T^\ast \R^3$. Recall that the characteristic polynomial of $A$ is defined by
\[
\chi(x,\xi):=\det \left(\sigma_A(x,\xi)\right), \quad 
\sigma_A(x,\xi): =\left(\sum_{|\alpha|=2} A_{ij}^\alpha(x) \xi^\alpha\right)_{i,j=1}^3, \quad (x,\xi) \in T^\ast \R^3 .
%\det \left(\sum_{|\alpha|=2} L_{ij}^\alpha(x) \xi^\alpha\right)
%\quad  L_{ij}(x):=\sum_{|\alpha|\leq 2} L_{ij}^\alpha(x) D^\alpha
\]
We note that this is equivalent for the \textit{principal symbol} $\sigma_A(x,\xi)$ of $A$ to be a bijective linear operator for every cotangent vector $(x,\xi)\in T^\ast \R^3 \setminus \{0\}$. 

Next we define the Lopatinskij condition. Let $z \in \p M$ and $(x',t)$ be boundary coordinates near $z$, that is $t^{-1}\{0\}\subset \p M$.
%In these coordinates we write the principal symbol of $L$ at $(z,\xi) \in \p T^\ast M$ as
%\[
%\sigma_L\left(z,0,\xi',-i\frac{\mathrm{d}}{\mathrm{d} t}\right), \quad \xi' \hbox{ is the projection of $\xi$ on } T^\ast_z\p M.
%\]
\begin{definition}
\label{De:Lopa}
We say that the operator $A$ satisfies the Lopatinskij at a point $z\in \p M$ if the constant coefficient initial value problem 
\[
\sigma_A\left(z,0,\xi',-\mathrm{i}\frac{\mathrm{d}}{\mathrm{d}t}\right)v(t)=0, \quad t \in \R_+, \quad v(0)=0\in \R^3,  \quad \xi'\in T^\ast_z\p M\setminus \{0\}.
\]
has only the trivial solution in  
$
%\mathcal M_+:=
\left\lbrace u\in C^2(\R_+): 
% \sigma_L\left(x',0,\xi',-i\frac{\mathrm{d}}{\mathrm{d} t}\right) u(t) =0, \: 
 u(0)=0, \: \lim_{t\to \infty} u(t)=0\right\rbrace.
$
\end{definition}
\noindent For the rest of the paper we use the notations $\sigma(A)$ to denote the principal symbol of an operator $A$. Often we do not emphasize the point in which the principal symbol is evaluated. 
\begin{definition}
\label{de:elliptic_prob}
We say that the boundary value problem 
\begin{equation}
\label{eq:BVB_for_L}
A u =f, \hbox{ in } M, \quad u|_{\p M}=w, \quad u\in H^{m}(M), \: f \in H^{m-2}(M), \: w \in H^{m-\frac{1}{2}}(\p M), \: m\in 1,2\ldots 
\end{equation}
is \textit{elliptic}  if:
\[
\hbox{(I) The operator $A$ is elliptic.  \quad (II) The Lopatinskij condition holds for any $z \in \p M$. } 
\]
\end{definition}
%~\\
%\textbf{I have commented the definition of the Lopatinskij condition!}
%\begin{remark}
%Since operator $L$ has a constant order $2$ we can set the $DN$-numbers $t_i=s_i=1$ for any $i=1\ldots p$. As our boundary operator is just the plain trace operator it holds that the number $\ell_0$ in \cite[Section 9, equation (33)]{wloka1995boundary} equals zero. Therefore the boundary value problem \eqref{eq:BVB_for_L} is of the same form as in \cite[Theorem 9.32]{wloka1995boundary}.
%\end{remark}
\color{black}

%\color{blue} 
We aim to use techniques for the elliptic problems to prove Lemma \ref{existence2}. To do so we first find the principal symbols of the operators $\dB$, $\deB$ and $\LaB$.  
%\color{red}
We introduce the notation 
\[
S^k T'_xM \otimes^\mathcal{B} S^\ell  T'_xM :=\{f(x)\vert \: f\in S^k \tau'_M \otimes^\mathcal{B} S^\ell \tau'_M\},
\]
for the evaluation of $f\in S^k \tau'_M \otimes S^\ell \tau'_M$ at $x\in M$. This is just the space of all tensors acting on the fiber $T_xM$ that are symmetric with respect to the first $k$ and the last $\ell$ indices and trace free. 
\color{black}

\begin{lemma}
\label{Le:symbols}
Let $(x,\xi) \in T^\ast M$. Define operators
\begin{equation}
\label{eq:i_xi}
\begin{split}
&i^{\mathcal{B}}_\xi\colon  S^k T'_xM \otimes^{\mathcal{B}} S^{\ell-1}  T'_xM \to S^k T'_xM \otimes^{\mathcal{B}} S^{\ell}  T'_xM,\quad i^{\mathcal{B}}_\xi:=\mathcal B  i_\xi,  
\\
&(i_\xi  u)_{i_1\ldots i_k j_1\ldots  j_\ell}:= \mathrm{Sym}( j_1\ldots  j_\ell) u_{i_1\ldots  i_k j_1\ldots  j_{\ell-1}}\xi_{j_\ell}
\end{split}
\end{equation}
and
\begin{equation}
\label{eq:j_xi}
j_\xi^{\mathcal{B}}\colon  S^k T'_xM \otimes^{\mathcal{B}} S^{\ell}  T'_xM \to S^k T'_xM \otimes^{\mathcal{B}} S^{\ell-1}  T'_xM,\quad (j^\mathcal{B}_\xi  v)_{i_1\ldots i_k j_1,\ldots ,j_\ell}:= v_{i_1\ldots i_k j_1\ldots  j_{\ell}}\xi^{j_\ell}
\end{equation}

In the case $k=\ell=2$ the principal symbols of $\frac{1}{\mathrm{i}}\dB$ and $\frac{1}{\mathrm{i}}\deB$ are $i_\xi^{\mathcal{B}}$ and $j_\xi^{\mathcal{B}}$ respectively. The principal symbol of $\LaB $ is
\[
\sigma(\LaB)\colon  S^2 T'_xM \otimes^{\mathcal{B}} S^{1}T'_xM  \to S^2 T'_xM \otimes^{\mathcal{B}} S^{1}T'_xM
\] 
such that
\begin{equation}
\label{eq:prin_symb_of_B_Lap}
\begin{split}
-(\sigma(\LaB)u)_{i_1i_2j_1}=&\frac{1}{2}  \left( u_{i_1i_2j_1}|\xi|_g^2+ u_{i_1i_2h}\xi^{h}\xi_{j_1}\right)
\\
 -&\frac{1}{10} \left( u_{ri_2j_1}\xi^{r}\xi_{i_1}+u_{ri_1j_1}\xi^{r}\xi_{i_2}+ g_{i_1j_1}\xi^{r}\xi^{h}u_{ri_2h}+g_{i_2j_1}\xi^{r}\xi^{h}u_{ri_1h}\right).
\end{split}
\end{equation}
\end{lemma}
\begin{proof}
We refer to \cite[Section 8]{wloka1995boundary} for the definitions of the principal symbols of $\Psi$DOs over vector bundles. Recall that in local coordinates differential $d'u$ of a tensor field $u \in S^2 T'_xM \otimes S^{1}T'_xM $ has a representation 
\[
(\mathrm{d}'u)_{i_1i_2j_1j_2}=\sym(j_1,j_2)\left(\frac{\p u_{i_1i_2j_1}}{\p x_{j_2}}-\left(u_{ri_2j_1}\Gamma^r_{i_1 j_{2}}+u_{i_1rj_1}\Gamma^r_{i_2 j_{2}}+u_{i_1i_2r}\Gamma^r_{j_1 j_{2}}\right)\right).
\]
Here $\Gamma^k_{ij}$ are the Christoffel symbols 
%\color{blue} 
of the metric $g$\color{black}. Therefore the principal symbol of $d'$ is exactly the map $(x,\xi)\mapsto i_{\xi}$.
%, which in this case acts as
%\[
%i_{\xi}u=\frac{1}{2}\left(u_{i_1i_2j_1}\xi_{j_2}+u_{i_1i_2j_2}\xi_{j_1}\right).
%\]
Since $\mu ^2 i_\xi u=0$ for any $u \in S^2 T'_xM \otimes^{\mathcal{B}} S^{1}T'_xM $ we have due to \eqref{eq:dB} that
\[
\frac{1}{\mathrm{i}}\sigma(\dB)=i_\xi-\frac{4}{5}\lambda \mu i_\xi=i_\xi^\mathcal{B}.
\]

Similarly for $\delta' u, \: u \in S^2 T'_xM \otimes S^{2}T'_xM $ we have
\[
\begin{split}
(\delta' u)_{i_1i_2j_1}=&\left(\frac{\p u_{i_1i_2j_1j_2}}{\p x_{h}}-\left(u_{ri_2j_1j_2}\Gamma^r_{i_1 h}+u_{i_1rj_1j_2}\Gamma^r_{i_2 h}+u_{i_1i_2rj_2}\Gamma^r_{j_1h}+u_{i_1i_2j_1r}\Gamma^r_{j_2h}\right)\right)g^{hj_2}.
%\\
%=&\frac{\p u_{i_1i_2j_1j_2}}{\p x_{h}}-\left(u_{ri_2j_1j_2}\Gamma^r_{i_1 h}+u_{i_1rj_1j_2}\Gamma^r_{i_2 h}+u_{i_1i_2rj_2}.\Gamma^r_{j_1h}+u_{i_1i_2j_1r}\Gamma^r_{j_2h}\right)g^{hj_2}
\end{split}
\]
Thus the principal symbol of $\delta'$ is given by
\[
\frac{1}{\mathrm{i}}\sigma(\delta')=\frac{1}{\mathrm{i}}\sigma(\deB)=j_\xi=j_\xi^\mathcal{B}.
\]

By \cite[Theorem 8.44]{wloka1995boundary} we have $-\sigma(\LaB)=j_\xi^{\mathcal{B}}i_\xi^{\mathcal{B}}$. The proof of \eqref{eq:prin_symb_of_B_Lap} is a direct computation recalling that $\mu u=0$. However we give it here as we need the computations later.
\[
\begin{split}
-\sigma(\LaB u)_{i_1i_2j_1}=&(j_\xi^{\mathcal{B}}i_\xi^{\mathcal{B}}u)_{i_1i_2j_1}=
\frac{1}{2} \xi^{h}\left( u_{i_1i_2j_1}\xi_{h}+u_{i_1i_2h}\xi_{j_1} -\frac{4}{5}\lambda \mu \left( u_{i_1i_2j_1}\xi_{h}+u_{i_1i_2h}\xi_{j_1}\right)\right)
\\
=&\frac{1}{2} \xi^{h}\left( u_{i_1i_2j_1}\xi_{h}+u_{i_1i_2h}\xi_{j_1} -\frac{4}{5}\lambda \left( u_{ri_2t}\xi_{h}+u_{ri_2h}\xi_{t}\right)g^{rt}\right)
\\
=&\frac{1}{2} \xi^{h}\left( u_{i_1i_2j_1}\xi_{h}+u_{i_1i_2h}\xi_{j_1} -\frac{4}{5}\sym(i_1i_2)\sym(j_1h) \left(g_{i_1j_1}u_{ri_2h}\xi^{r}\right)\right)
\\
=&\frac{1}{2} \xi^{h}\left( u_{i_1i_2j_1}\xi_{h}+u_{i_1i_2h}\xi_{j_1} -\frac{1}{5} \left(g_{i_1j_1}u_{ri_2h}+g_{i_1h}u_{ri_2j_1}+g_{i_2j_1}u_{ri_1h}+g_{i_2h}u_{ri_1j_1}\right)\xi^{r}\right)
\\
%=&\frac{1}{2}  \left( u_{i_1i_2j_1}|\xi|_g^2+ u_{i_1i_2h}\xi^{h}\xi_{j_1}\right) -\frac{1}{10} \left(g_{i_1j_1}u_{ri_2h}+g_{i_1h}u_{ri_2j_1}+g_{i_2j_1}u_{ri_1h}+g_{i_2h}u_{ri_1j_1}\right)\xi^{r}\xi^{h}
%\\
%=&\frac{1}{2}  \left( |\xi|_g^2u+ j_\xi^{\mathcal{B}}i_\xi^{\mathcal{B}}u\right)_{i_1i_2j_1}
=&\frac{1}{2}  \left( u_{i_1i_2j_1}|\xi|_g^2+ u_{i_1i_2h}\xi^{h}\xi_{j_1}\right)
\\
 &-\frac{1}{10} \left( u_{ri_2j_1}\xi^{r}\xi_{i_1}+u_{ri_1j_1}\xi^{r}\xi_{i_2}+ g_{i_1j_1}\xi^{r}\xi^{h}u_{ri_2h}+g_{i_2j_1}\xi^{r}\xi^{h}u_{ri_1h}\right)
\end{split}
\]
\end{proof}

%\color{blue}
In the following lemma we show that the problem \eqref{eqnu1} is elliptic.

\begin{lemma}
\label{le:elliptic_prob}
The problem \eqref{eqnu1} is elliptic in the sense of Definition \ref{de:elliptic_prob}.
\end{lemma}
%\color{red}
%textbf{This used to be the first two parts of the proof of lemma \ref{existence2}. (TS)}
\begin{proof}
We check first the ellipticity of the principal symbol of $\LaB$. If we denote 
\[
\begin{split}
&(p_1(\xi)u)_{i_1i_2j_1}= u_{i_1i_2h}\xi^{h}\xi_{j_1}, \quad (p_2(\xi)u)_{i_1i_2j_1}=|\xi|_g^2u_{i_1i_2j_1}- u_{i_1i_2h}\xi^{h}\xi_{j_1},
\\
&(p_3(\xi)u)_{i_1i_2j_1}=|\xi|_g^2u_{i_1i_2j_1}-g_{i_2j_1}\xi^{r}\xi^{h}u_{ri_1h},
\end{split}
\]
then straightforward calculations show that $p_\alpha(\xi)$, $\alpha=1,2,3$ are non-negative in the sense of
\[
\langle p_\alpha(\xi)u,u\rangle_{L^2}\geq 0, \quad u \in L^2\left(S^2 T'_xM \otimes^{\mathcal{B}} S^{1}T'_xM  \right).
\]

Equation \eqref{eq:prin_symb_of_B_Lap} implies
\[
%&\left\langle\frac{2}{5}|\xi|_g^2u_{ijk}+\frac{1}{2}\xi_k\xi^\ell u_{ij\ell}-\frac{1}{10}\left(\xi_j\xi^mu_{imk}+\xi_i\xi^mu_{jmk}+g_{jk}\xi^\ell\xi^mu_{im\ell}+g_{ik}\xi^m\xi^\ell u_{jm\ell}\right),u\right\rangle_{L^2}
\\
\left\langle-\sigma(\LaB)u-\frac{1}{10}|\xi|_g^2u,u\right\rangle_{L^2} \geq\frac{1}{2}\langle p_1(\xi)u,u\rangle_{L^2}+\frac{1}{5}\langle p_2(\xi)u,u\rangle_{L^2}+\frac{1}{5}\langle p_3(\xi)u,u\rangle_{L^2}\geq 0.
\]
Hence we obtain
\[
\langle -\sigma(\LaB)u,u\rangle_{L^2} \geq \frac{1}{10}|\xi|_g^2\langle u,u\rangle_{L^2},
\]
which proves the ellipticity of $\sigma(\LaB)$. %Actually $\sigma(\LaB)$ is homogeneously elliptic as every component of the principal symbol is of order two.

%%%%%%%%%%%%%%
%Verifycation of non-negativity of $p_\alpha$
%To conclude the ellipticity proof we note that
%\[
%(p_1(\xi)u)_{i_1i_2j_1}u^{i_1i_2j_1}=|w(\xi)|^2, \quad w_{i_1i_2}(\xi):=u_{i_2i_2j}\xi^j.
%\]
%
%Due to Cauchy-Schwartz inequality we get
%\[
%(p_2(\xi)u)_{i_1i_2j_1}u^{i_1i_2j_1}=(|\xi|^2u_{i_1i_2j_1}- u_{i_1i_2h}\xi^{h}\xi_{j_1})u^{i_1i_2j_1}=|\xi|^2|u|^2-|w(\xi)|^2\geq 0.
%\]
%
%And finally since $\mu u=0$ we get
%\[
%(p_3(\xi)u)_{i_1i_2j_1}u^{i_1i_2j_1}=(|\xi|^2u_{i_1i_2j_1}-g_{i_2j_1}\xi^{r}\xi^{h}u_{ri_1h})u^{i_1i_2j_1}=|\xi|^2|u|^2.
%\]

\medskip
Next, we verify the Lopatinskij condition. For that, we choose local coordinates $(x^1,x^2,x^3=t)$, $t\geq 0$ in a neighborhood of a point $x_0\in \partial M$, such that the boundary $\partial M$ is locally represented by $t=0$, and $g_{ij}(x_0)=\delta_{ij}$. We set a differential operator 
%\color{blue}
$D=(D_j)_{j=1}^3$, $D_j=-\mathrm{i}\frac{\partial}{\partial x^j}, \: j\in \{1,2\}$ and $D_3=D_t=-\mathrm{i}\frac{\mathrm{d}}{\mathrm{d}t}$.
%\color{red} 
Then we denote
\[
\mathrm{d}_0^\mathcal{B}(D)=\sigma(\mathrm{d}^\mathcal{B})(x_0,D),\quad \delta_0^\mathcal{B}(D)=\sigma(\delta^\mathcal{B})(x_0,D).
\]
We need to show that the only solution for the system of ordinary differential equations 
\begin{equation}\label{equation0}
\begin{split}
\delta_0^\mathcal{B}(\xi',D_t)\mathrm{d}_0^\mathcal{B}(\xi',D_t)v(t)&=0,\quad t\in \R_+
\\
v(0)&=0,
\end{split}
\end{equation}
which satisfies $v(t)\rightarrow 0$ as $t\rightarrow +\infty$, is the zero field. 

Let 
\[
u \in \mathcal{S}(S^2\tau'_{\R_+}\otimes^{\mathcal{B}_0} S^{2} \tau'_{\R_+}), \quad v \in \mathcal{S}(S^2\tau'_{\R_+}\otimes^{\mathcal{B}_0} S^{1} \tau'_{\R_+}),
\] 
where $ \mathcal{S}$ means that the function $u(t)$ has a rapid decrease when $t$ tends to $+\infty$ and $\mathcal{B}_0$ means that $u$ belongs to the kernel of the operator $\mu$ associated to the Euclidean metric. If $v(0)=0$ then  Lemma \ref{Le:dif_ops_B} implies the following Green's formula
\begin{equation}\label{Green0}
\int_0^\infty\langle\delta_0^\mathcal{B}(\xi',D_t)u,v\rangle\mathrm{d}t=-\int_0^\infty\langle u,\mathrm{d}_0^\mathcal{B}(\xi',D_t)v\rangle\mathrm{d}t.
\end{equation}
Due to denseness of rapidly decreasing tensor fields, the formula \eqref{Green0} holds for any $u$ and $v$ both vanishing in the infinity and $v(0)=0$.

\medskip
Let $v(t)$ be a solution of \eqref{equation0}. Taking $u(t):=\mathrm{d}_0^\mathcal{B}(\xi',D_t)v(t)$ in \eqref{Green0} we obtain
\begin{equation}
\label{eq:ODE}
\mathrm{d}_0^\mathcal{B}(\xi',D_t)v(t)=0, \quad v(0)=0.
\end{equation}
Finally we show that only zero field solves this initial value problem.

%$S^2\tau'_{\R_+}\otimes^{\mathcal{B}_0} S^{1} \tau'_{\R_+}$ with the space of $m$-dimensional vector fields over $\R_+$ and show that \eqref{eq:ODE} implies a homogeneous system of linear initial value problems
%\[
%\dot{v}(t)=A(\xi')v(t), \quad v(0)=0,
%\]
%where $A(\xi') \in \R^{m\times m}$ is a square matrix. Thus the  Lopatinskij condition is verified due to general existence and uniqueness theory of first order initial value problems.  

We note that \eqref{eq:ODE} implies the following equation in coordinates 
\[
\begin{split}
&\frac{1}{\mathrm{i}}(\mathrm{d}_0^\mathcal{B}(\xi)v)_{i_1i_2j_1j_2}=
\\
&\frac{1}{2} v_{i_1i_2j_1}\xi_{j_2}+\frac{1}{2} v_{i_1i_2j_2}\xi_{j_1} -\frac{1}{10} \left(\delta_{i_1j_1}v_{ri_2j_2}+\delta_{i_1j_2}v_{ri_2j_1}+\delta_{i_2j_1}v_{i_1rj_2}+\delta_{i_2j_2}v_{i_1rj_1}\right)\xi^{r}=0.
\end{split}
\]
In the previous formula we set $j_2=3$ and $\xi_{3}=-\mathrm{i}\frac{\mathrm{d}}{\mathrm{d}t}$. Then we obtain the following system of ordinary differential equations.

\begin{equation}\label{eq_vijk1}
\begin{split}
\frac{1}{\mathrm{i}}\left(\mathrm{d}_0^\mathcal{B}(\xi',D_t)v(t)\right)_{i_1i_233}
%=&
%D_{t}v_{i_1i_23}
%-\frac{1}{5} \left(v_{i_1r3}\delta_{i_23}+v_{ri_23}\delta_{i_13}\right)\xi^{r}
=&D_{t}v_{i_1i_23}-\frac{1}{5}\left(D_t v_{i_133}\delta_{i_23}+D_tv_{i_233}\delta_{i_13}\right)
\\
&-\frac{1}{5}\sum_{r\neq 3}(v_{i_1r3}\xi^r\delta_{i_23}+v_{i_2r3}\xi^r\delta_{i_13})
\\
=&0
\end{split}
\end{equation}
and, for $j_1\neq 3$,
\begin{equation}\label{eq_vijk2}
\begin{split}
&\frac{1}{\mathrm{i}}\left(\mathrm{d}_0^\mathcal{B}(\xi',D_t)v(t)\right)_{i_1i_2j_13}
\\
=&\frac{1}{2}D_{t}v_{i_1i_2j_1}-\frac{1}{10}\left(D_t v_{i_13j_1}\delta_{i_23}+D_tv_{i_23j_1}\delta_{i_13}+D_t v_{i_133}\delta_{i_2j_1}+D_tv_{i_233}\delta_{i_1j_1}\right)
\\
&+\frac{1}{2}v_{i_1i_23}\xi_{j_1}-\frac{1}{10}\sum_{m\neq 3}(v_{i_1mj_1}\xi^m\delta_{i_23}+v_{i_2mj_1}\xi^m\delta_{i_13}+v_{i_1m3}\xi^m\delta_{i_2j_1}+v_{i_2m3}\xi^m\delta_{i_1j_1})
\\
=&0.
\end{split}
\end{equation}

Finally we solve these equations with initial value $v(0)=0$, which is done in the following sequence:
\begin{itemize}
\item
If  $i_1,i_2 \neq 3$, the equation \eqref{eq_vijk1} gives
\[
D_{t}v_{i_1i_23}=0, t>0, \quad v(0)=0.
\]
This implies $v_{i_1i_23}=0$ for $i_1,i_2\neq 3$. 
\item
If  $i_1\neq 3, i_2=3$ or $i_2\neq 3, i_1=3$, the equation \eqref{eq_vijk1} gives
\[
\frac{4}{5}D_{t}v_{i_133}=\frac{4}{5}D_{t}v_{3i_23}=0, t>0, \quad v(0)=0,
\]
which implies $v_{i33}$ and $v_{3i3}$ for $i\neq 3$. 
\item 
Finally the equation \eqref{eq_vijk1} gives
\[
\frac{3}{5}D_{t}v_{333}=0, t>0, \quad v(0)=0,
\]
and thus $v_{333}=0$. 
\end{itemize}

We have proved that $v_{i_1i_23}=0$ and therefore the equation \eqref{eq_vijk2} simplifies to 
\begin{equation}\label{eq_vijk3}
\begin{split}
&D_{t}v_{i_1i_2j_1}-\frac{1}{5}\left(D_t v_{i_13j_1}\delta_{i_23}+D_tv_{i_23j_1}\delta_{i_13}-\sum_{m\neq 3}(v_{i_1mj_1}\xi^m\delta_{i_23}+v_{i_2mj_1}\xi^m\delta_{i_13})\right)=0,
\\
&j_1\neq 3.
\end{split}
\end{equation}

\begin{itemize}
\item 
We take $i_1,i_2\neq 3$ in \eqref{eq_vijk3}, and get
\[
D_{t}v_{i_1i_2j_1}=0, t>0, \quad v(0)=0.
\]
Thus $v_{i_1i_2j_1}=0$ for $i_1,i_2,j_1\neq 3$. 
\item
The choices  $i_1\neq 3, i_2=3$ and $i_2\neq 3, i_1=3$ in \eqref{eq_vijk3}, imply
\[
\frac{4}{5}D_{t}v_{i_13j_1}=\frac{4}{5}D_{t}v_{3i_2j_1}=0, t>0, \quad v(0)=0,
\]
and then $v_{i_1j_2j_1}=0$ if at most one of $i_1,i_2$ equals $3$.
\item 
Finally we take $i_1=i_2=3$ in \eqref{eq_vijk3}, to get
\[
\frac{3}{5}D_tv_{33j_2}=0, t>0, \quad v(0)=0,
\]
and then $v_{33j_2}=0$ for $j_2\neq 3$. 
\end{itemize}
Thus we have proven that the zero field is the only solution of \eqref{eq:ODE} and we have verified the Lopatsinkij condition for \eqref{eqnu1}. Alas we have proved the ellipticity of the boundary value problem \eqref{eqnu1}.
\end{proof}
\color{black}
Now we are ready to give a proof for Lemma \ref{existence2}.

\begin{proof}[Proof of Lemma \ref{existence2}]
%\color{blue}
We use the notation $\hbox{Tr}\colon H^{m}(M)\to H^{m-\frac{1}{2}}(\p M)$ for the trace operator. 
%\color{red}
As we have verified in Lemma \ref{le:elliptic_prob} that the problem \eqref{eqnu1} is elliptic, it holds due to \cite[Theorem 9.32]{wloka1995boundary} that the operator $(\LaB,\hbox{Tr})\colon H^{m}(M)\to (H^{m-2}(M)\times H^{m-\frac{1}{2}}(\p M))$ is a Fredholm operator (a bounded operator with finite dimensional kernel and co-kernel). Moreover there exists a uniform constant $C>0$ such that the following \textit{a priori} estimate holds for any $u \in H^{m}(M)$
\begin{equation}
\label{eq:apriori_estimate_1}
\|u\|_{H^{m}(M)}\leq C\left(\|\LaB u\|_{H^{m-2}(M)}+\|\hbox{Tr} u\|_{H^{m-\frac{1}{2}}(\p M)} +\|u\|_{H^{m-1}(M)}\right).
\end{equation}
As the embedding $H^{m}(M) \hookrightarrow H^{m-1}(M)$ is compact it holds due to \cite[Lemma 2]{stefanov2004stability}  that we can write \eqref{eq:apriori_estimate_1} in the form
\[
\|u\|_{H^{m}(M)}\leq C\left(\|\LaB u\|_{H^{m-2}(M)}+\|\hbox{Tr} u\|_{H^{m-\frac{1}{2}}(\p M)}\right)
\]
for some uniform constant $C$, if \eqref{eqnu1} is uniquely solvable. This is the estimate \eqref{eq:energy_estimate}.
% that the solution operator $(\LaB)^{-1}$ of \eqref{eqnu1} is bounded, if it exists.
\color{black}
In order to verify the unique solvability of the boundary value problem \eqref{eqnu1}, and to conclude the proof, we show that $(\LaB, \hbox{Tr})$ has a trivial kernel and co-kernel. 

%We first claim
%\[
%\mathrm{d}'(\mu\mathrm{d}'u)^T=\delta' u^T.
%\]
%If the claim is true, then
%\[
%\Delta u-\frac{1}{3}[\mathrm{d}'(\mu\mathrm{d}'u)^T]^T=0
%\]
%implies
%\[
%\int_M \langle\mathrm{d}'u,\mathrm{d}'u\rangle-\frac{1}{3}\langle \delta' u^T, \delta' u^T\rangle\mathrm{d}V=0.
%\]
%To show the claim, we calculate
%\[
%\mathrm{d}'u=\sym(j_1\ldots j_\ell)u_{i_1\ldots i_{k}j_1\ldots j_{\ell-1};j_\ell};
%\]
%\[
%\mu\mathrm{d}'u=u_{i_1\ldots i_{k}j_1\ldots j_{\ell-1};j_\ell}g^{i_kj_\ell};
%\]
%\[
%(\mu\mathrm{d}'u)^T=u_{j_1\ldots j_{k}i_1\ldots i_{\ell-1};j_\ell}g^{j_kj_\ell}
%\]
%and
%\[
%u^T=u_{j_1\ldots j_{k}i_1\ldots i_{\ell-1}};
%\]
%\[
%\delta'u^T=u_{j_1\ldots j_{k}i_1\ldots i_{\ell-1};j_{k+1}}g^{j_kj_{k+1}}.
%\]
%
We show first the kernel of $(\LaB, \hbox{Tr})$ is trivial. Let $u$ solve \eqref{eqnu1} with a source $h\equiv 0$ and boundary value $w\equiv 0$. Since we have proved that $\LaB$ is elliptic, it holds that $u$ is smooth. As $\dB$ and $-\deB$ are formally adjoint we get
\[
 \int_M \langle\mathrm{d}^\mathcal{B}u,\mathrm{d}^\mathcal{B}u\rangle_g \mathrm{d}V=\int_M \langle -\LaB u, u\rangle_g \mathrm{d}V=0,
\]
%We can verify
%\[
%\begin{split}
%&\langle \mathrm{d}'u-\frac{4}{5}\lambda\mu\mathrm{d}'u,\,\mathrm{d}'u-\frac{4}{5}\lambda\mu\mathrm{d}'u\rangle\\
%=&\langle\mathrm{d}'u,\mathrm{d}'u\rangle -\frac{8}{5}\langle\mu\mathrm{d}'u,\mu\mathrm{d}'u\rangle+\frac{16}{25}\langle\lambda\mu\mathrm{d}'u,\lambda\mu\mathrm{d}'u\rangle\\
%=&\langle\mathrm{d}'u,\mathrm{d}'u\rangle -\frac{8}{5}\langle\mu\mathrm{d}'u,\mu\mathrm{d}'u\rangle+\frac{16}{25}\times\frac{5}{4}\langle\mu\mathrm{d}'u,\mu\mathrm{d}'u\rangle\\
%=&\langle\mathrm{d}'u,\mathrm{d}'u\rangle -\frac{4}{5}\langle\mu\mathrm{d}'u,\mu\mathrm{d}'u\rangle.
%\end{split}
%\]
%Thus
%\[
%\int_M\langle \mathrm{d}'u-\frac{4}{5}\lambda\mu\mathrm{d}'u,\mathrm{d}'u-\frac{4}{5}\lambda\mu\mathrm{d}'u\rangle\mathrm{d}V=0.
%\]
and
\[
\mathrm{d}^\mathcal{B}u=\mathrm{d}'u-\frac{4}{5}\lambda\mu\mathrm{d}'u=0.
\]
Next we note that for any $x\in M$ and $v \in S^2\tau'_M\otimes^\mathcal{B} S^{1} \tau'_M$ holds
\[
(\dB v)_{ijkl}\eta^{i}\eta^{j}\xi^{k}\xi^{l}
%=(\mathrm{d}'v-\frac{4}{5}\lambda\mu\mathrm{d}'v)_{i_1i_2j_1j_2}\eta^{i_1}\eta^{i_2}\xi^{j_1}\xi^{j_2}
=(\mathrm{d}'v)_{ijkl}\eta^{i}\eta^{j}\xi^{k}\xi^{l}
\]
if $\xi,\eta \in T_xM$ are orthogonal. 
%\color{blue}
In the following we use the notation $\gamma$ for the geodesic $\gamma_{x,\xi}, \: \xi \in S_xM$ and $\eta(t)$ stands for the parallel transport of $\eta$ along $\gamma$.
\color{black}
By straightforward computations we obtain
\begin{equation}\label{dd22}
\begin{split}
\frac{\mathrm{d}}{\mathrm{d}t}[v_{ijk}(\gamma(t))\eta^i(t)\eta^j(t)\dot{\gamma}^k(t)]
%=&\left(\frac{D v}{\mathrm{d}t}\right)_{ijk}(\gamma(t))\eta^i(t)\eta^j(t)\dot{\gamma}^k(t)\\
=&(\mathrm{d}^\mathcal{B}v)_{ijkl}\eta^i(t)\eta^j(t)\dot{\gamma}^k(t)\dot{\gamma}^l(t).
\end{split}
\end{equation}

Let $x_0\in M\setminus \partial M$ and $z_{0}$ be a closest boundary point to $x_0$. We use the notation $\xi_0 \in S_{x_0}M \setminus \{0\}$ for the direction of the unit speed geodesic $\gamma_0$ connecting $x_0$ to $z_0$.  
%Since the boundary of  $(M,g)$ is strictly convex, 
%\color{blue}
Since the geodesic $\gamma_0$ intersects $\p M$ transversally 
\color{black}
there exists a neighborhood $U\subset S_{x_0}M$ of $\xi_0$ such that the exit time function  $\tau$ is finite and smooth in $U$. Then for any $\xi \in U$ and $\eta\perp \xi$, equation \eqref{dd22} and the fundamental theorem of calculus imply
\[
u_{ijk}(x_0)\eta^i\eta^j\xi^k=-
\int^{\tau(\xi)}_0[\mathrm{d}^\mathcal{B}u]_{ijkl}(\gamma_{x_0,\xi}(t))\eta^i(t)\eta^j(t)\xi^k(t)\xi^l(t)\mathrm{d}t
=0.
\]
Here $\eta(t)$ is a parallel field along the geodesic $\gamma_{x_0,\xi}$ with $\eta(0)=\eta$ and $\xi(t)=\dot{\gamma}_{x_0,\xi}(t)$. In the following we use a short hand notation $u=u(x_0)$.
Therefore 
\begin{equation}
\label{eq:u_in_basis}
u_{ijk}\eta^{i}\eta^{j}\xi^k=0, \quad  \xi\in U, \: \eta\perp \xi.
\end{equation}
We choose an orthonormal basis $B=\{\xi,\eta, \widetilde \eta\}$ for the three dimensional space $T_{x_0}M$, where $\xi \in U$.
%and in the following we prove that, 
%\begin{equation}
%\label{eq:u_in_basis_2}
%u_{ijk}v_{\ell_1}^iv^j_{\ell_2}v^k_{\ell_3}=0, \quad \hbox{ for any } v_{\ell_1},v_{\ell_2},v_{\ell_3} \in B.
%\end{equation}
%This implies  that $u(x_0)=0$ and since $x_0\in M$ was an arbitrary point we conclude that $u$ is a zero field.
By polarization, \eqref{eq:u_in_basis} implies
\begin{equation}
\label{eq:u_in_polari}
u_{ijk}\eta^i\widetilde \eta^j\xi^k
%=u_{ijk}\widetilde\eta^i \eta^j\xi^k
=0.
\end{equation}
For any $\epsilon >0$ that is small enough, the equation \eqref{eq:u_in_basis} gives
\begin{equation}
\label{eq:u_in_perturbed_basis}
u_{ijk}(\eta+\epsilon \xi)^{i}(\eta+\epsilon \xi)^{j}(\xi-\epsilon \eta)^k=0.
\end{equation}
Therefore the coefficients of the $\epsilon^3, \epsilon^2, \epsilon, 1$ of  the expansion of \eqref{eq:u_in_perturbed_basis} have to vanish.  Clearly the same holds if $\eta$ is replaced by $\widetilde \eta$ in \eqref{eq:u_in_perturbed_basis}. Now we have proven
\begin{equation}
\label{eq:u_in_perturbed_basis_2}
\begin{split}
u_{ijk}\xi^{i}\xi^{j}\eta^k=0,
& \quad 
u_{ijk}\left(\xi^i\xi^j\xi^k-2\eta^{i}\xi^{j}\eta^k\right)=0,
\quad 
u_{ijk}\left(2\eta^i\xi^j\xi^k-\eta^i\eta^j\eta^k\right)=0,
\\
u_{ijk}\xi^{i}\xi^{j}\widetilde \eta^k=0,
& \quad 
u_{ijk}\left(\xi^i\xi^j\xi^k-2\widetilde\eta^{i}\xi^{j}\widetilde\eta^k\right)=0,
\quad 
u_{ijk}\left(2\widetilde\eta^i\xi^j\xi^k-\widetilde\eta^i\widetilde\eta^j\widetilde\eta^k\right)=0.
\end{split}
\end{equation}
%The symmetry of $u$ gives
%\[
%u_{ijk}v_{\ell_1}^iv_{\ell_1}^jv_{\ell_1}^k=
%%2u_{ijk}w^{i}v_{\ell_1}^{j}w^k=
% 2u_{ijk}v_{\ell_1}^{i}w^{j}w^k.
%\]
%By setting $w=tv_{\ell_2}$ for $t>0$ and $v_{\ell_2}\neq v_{\ell_1}$  we obtain
%\[
%u_{ijk}v_{\ell_1}^iv_{\ell_1}^jv_{\ell_1}^k=2t^2u_{ijk}v_{\ell_2}^{i}v_{\ell_1}^{j}v_{\ell_2}^k.
%\]
%{\color{red}[In equation \eqref{eq:u_in_perturbed_basis}, $v+\epsilon w$ and $w-\epsilon v$ need to be orthogonal. This requires that $|v|=|w|$]}
%\textbf{Ok I see this, sorry for long discussion! I incorporate the old proof.}
%As the left hand side is independent of $t$ we must have $u_{ijk}v_{\ell_2}^{i}v_{\ell_1}^{j}v_{\ell_2}^k=0$. Thus %\eqref{eq:u_in_perturbed_basis_2} and 
%the symmetry of $u$ implies
%\[
%u_{ijk}v_{\ell_1}^iv_{\ell_1}^jv_{\ell_1}^k=
%u_{ijk}v_{\ell_2}^{i}v_{\ell_1}^{j}v_{\ell_2}^k=
% u_{ijk}v_{\ell_1}^{i}v_{\ell_2}^{j}v_{\ell_2}^k=0,
%\]
%for $v_{\ell_1},v_{\ell_2}\in B$ that are disjoint. This completes the verification of \eqref{eq:u_in_basis_2}.
%Thus we have shown
%\[
%u_{ijk}v_{\ell_1}^iv_{\ell_1}^jv_{\ell_1}^k=u_{ijk}v_{\ell_2}^{i}v_{\ell_1}^{j}v_{\ell_2}^k=0.
%\]
Also we have
$
u_{ijk}\eta^i\eta^j(\xi+\epsilon\widetilde \eta)^k=0,
%\quad\quad u_{ijk}\eta^i\eta^j(\xi+\epsilon\tilde{\eta})^k=0,
$
from which we derive
\begin{equation}
\label{eq:u_in_perturbed_basis_3}
\begin{split}
%u_{ijk}\tilde{\eta}_0^i\tilde{\eta}_0^j\xi_0^k=0,\quad\quad u_{ijk}\eta_0^i\eta_0^j\xi_0^k=0,\\
%u_{ijk}\tilde{\eta}^i\tilde{\eta}^j\eta^k=0, 
%\quad\quad 
u_{ijk}\eta^i\eta^j\tilde{\eta}^k=0.
\end{split}
\end{equation}
%Obliviously we can replace $\eta$ by $\widetilde \eta$ in \eqref{eq:u_in_perturbed_basis}--\eqref{eq:u_in_perturbed_basis_3}. 
Next, we note 
\[
u_{ijk}(\tilde{\eta}+\eta+\epsilon\xi)^i(\tilde{\eta}+\eta+\epsilon\xi)^j(\xi-\epsilon\eta)^k=0,
\]
and the roles of $\eta$ and $\widetilde \eta$ can be interchanged.
%this reduces to
%\[
%u_{ijk}\tilde{\eta}^i(\eta+\epsilon\xi)^j(\xi-\epsilon\eta)^k=0.
%\]
Collecting the coefficients for $1,\epsilon,\epsilon^2$ terms we get by \eqref{eq:u_in_polari}--\eqref{eq:u_in_perturbed_basis_3} 
\begin{equation}
\label{eq:u_in_perturbed_basis_4}
\begin{split}
u_{ijk}\tilde{\eta}^i\xi^j\xi^k-u_{ijk}\tilde{\eta}^i\eta^j\eta^k=0,\quad \quad 
u_{ijk}\tilde{\eta}^i\xi^j\eta^k=0
\\
u_{ijk}\eta^i\xi^j\xi^k-u_{ijk}\eta^i\tilde \eta^j\tilde \eta^k=0,\quad \quad 
u_{ijk}\eta^i\xi^j\tilde \eta^k=0.
\end{split}
\end{equation}
%Switching $\eta_0$ and $\tilde{\eta}_0$, we also have
%\[
%\begin{split}
%u_{ijk}(x_0)\eta_0^i\xi_0^j\xi_0^k-u_{ijk}(x_0)\eta_0^i\tilde{\eta}_0^j\tilde{\eta}_0^k=0,\\
%u_{ijk}(x_0)\eta_0^i\xi_0^j\tilde{\eta}_0^k=0.
%\end{split}
%\]

%So far we have proved that the following terms vanish
%\[
%\begin{split}
%u_{ijk}\eta^{i}\eta^{j}\xi^k, \quad u_{ijk}\eta^i\widetilde \eta^j\xi^k, \quad u_{ijk}\xi^{i}\xi^{j}\eta^k, \quad u_{ijk}\eta^i\eta^j\tilde{\eta}^k, \quad u_{ijk}\tilde{\eta}^i\xi^j\eta^k,
%\end{split}
%\]
%and obliviously we can replace $\eta$ by $\widetilde \eta$ and vice versa in the previous collection.
%Now, we know the following terms vanish
%\[
%\begin{split}
%u_{ijk}(x_0)\eta_0^i\xi_0^j\tilde{\eta}_0^k, \quad u_{ijk}(x_0)\tilde{\eta}_0^i\xi_0^j\eta_0^k, \quad u_{ijk}(x_0)\eta_0^{i}\tilde{\eta}_0^{j}\xi_0^k,\\
% u_{ijk}(x_0)\eta_0^{i}\eta_0^{j}\xi_0^k, \quad u_{ijk}(x_0)\tilde{\eta}_0^{i}\tilde{\eta}_0^{j}\xi_0^k, \quad u_{ijk}(x_0)\xi_0^{i}\xi_0^{j}\eta_0^k,\\
%  u_{ijk}(x_0)\xi_0^{i}\xi_0^{j}\tilde{\eta}_0^k, \quad u_{ijk}(x_0)\tilde{\eta}_0^i\tilde{\eta}_0^j\eta_0^k,  \quad u_{ijk}(x_0)\eta_0^i\eta_0^j\tilde{\eta}_0^k.
%\end{split}
%\]
To continue we note that since $B$ is an orthogonal basis, it follows that 
\[
(\mu u)_j(x_0)=\delta^{ik}u_{ijk}=\sum_{k=1}^3u_{kjk}=\sum_{k=1}^3u_{jkk}=0, \quad j \in \{1,2,3\},
\]
or equivalently
\begin{equation}
\label{eq:last_system}
\left\lbrace
\begin{array}{l}
u_{ijk}\eta^i\xi^j\xi^k+u_{ijk}\eta^i\eta^j\eta^k+u_{ijk}\eta^i\tilde{\eta}^j\tilde{\eta}^k=0,
\\
u_{ijk}\tilde{\eta}^i\xi^j\xi^k+u_{ijk}\tilde{\eta}^i\eta^j\eta^k+u_{ijk}\tilde{\eta}^i\tilde{\eta}^j\tilde{\eta}^k=0,
\\
u_{ijk}\xi^i\xi^j\xi^k+u_{ijk}\xi^i\eta^j\eta^k+u_{ijk}\xi^i\tilde{\eta}^j\tilde{\eta}^k=0.
\end{array}\right.
\end{equation}
It remains to show that each term in \eqref{eq:last_system} vanishes. As \eqref{eq:u_in_perturbed_basis_2} and \eqref{eq:u_in_perturbed_basis_4}  give 6 additional equations, the following linear systems hold true:
\begin{equation}\label{eq:3systems}
\begin{split}
&
\begin{cases}
u_{ijk}\xi^i\xi^j\xi^k-2u_{ijk}\xi^i\eta^j\eta^k=0,
\\
u_{ijk}\xi^i\xi^j\xi^k-2u_{ijk}\xi^i\tilde{\eta}^j\tilde{\eta}^k=0,
\\
u_{ijk}\xi^i\xi^j\xi^k+u_{ijk}\xi^i\eta^j\eta^k+u_{ijk}\xi^i\tilde{\eta}^j\tilde{\eta}^k=0;\\
\end{cases}
\\
&
\begin{cases}
2u_{ijk}\xi^{i}\eta^{j}\xi^k-u_{ijk}\eta^i\eta^j\eta^k=0,
\\
u_{ijk}\eta^i\xi^j\xi^k+u_{ijk}\eta^i\eta^j\eta^k+u_{ijk}\eta^i\tilde{\eta}^j\tilde{\eta}^k=0,
\\
u_{ijk}\eta^i\xi^j\xi^k-u_{ijk}\eta^i\tilde{\eta}^j\tilde{\eta}^k=0;
\end{cases}
\\
&
\begin{cases}
2u_{ijk}\xi^{i}\tilde{\eta}^{j}\xi^k-u_{ijk}\tilde{\eta}^i\tilde{\eta}^j\tilde{\eta}^k=0,
\\
u_{ijk}\tilde{\eta}^i\xi^j\xi^k+u_{ijk}\tilde{\eta}^i\eta^j\eta^k+u_{ijk}\tilde{\eta}^i\tilde{\eta}^j\tilde{\eta}^k=0,
\\
u_{ijk}\tilde{\eta}^i\xi^j\xi^k-u_{ijk}\tilde{\eta}^i\eta^j\eta^k=0.
\end{cases}
\end{split}
\end{equation}
Each system consists of three linearly independent equations for three variables, and thus can only have trivial solutions.
% Then the following terms also vanish
%\[
%\begin{split}
%u_{ijk}(x_0)\eta_0^i\xi_0^j\xi_0^k,\quad u_{ijk}(x_0)\eta_0^i\eta_0^j\eta_0^k,\quad u_{ijk}(x_0)\eta_0^i\tilde{\eta}_0^j\tilde{\eta}_0^k,\\
%u_{ijk}(x_0)\tilde{\eta}_0^i\xi_0^j\xi_0^k,\quad u_{ijk}(x_0)\tilde{\eta}_0^i\eta_0^j\eta_0^k ,\quad u_{ijk}(x_0)\tilde{\eta}_0^i\tilde{\eta}_0^j\tilde{\eta}_0^k,\\
%u_{ijk}(x_0)\xi_0^i\xi_0^j\xi_0^k ,\quad u_{ijk}(x_0)\xi_0^i\eta_0^j\eta_0^k ,\quad u_{ijk}(x_0)\xi_0^i\tilde{\eta}_0^j\tilde{\eta}_0^k.
%\end{split}
%\]
%Therefore we have proven \eqref{eq:u_in_basis}.

\medskip

Finally, we show that the boundary value problem \eqref{eqnu1} has a trivial cokernel. If $f \in H^{-1}$ we can choose a series of $f_j \in L^2, \: j\in \N$ that converges to $f$ in $H^{-1}$ sense. For the existence of such sequence see for instance \cite[Section 3]{adams2003sobolev}.
%Recall that the $H^{-1}$-norm is the operator norm
%\[
%\|u\|_{H^{-1}}=\inf \{\langle u, v\rangle: v \in H^1_0, \: \|v\|_{H^1}=1\}.
%\]
Let $\phi_{j_k}$ be a sequence of smooth tensor fields approximating $f_j$ in $L^2$. Let $\epsilon>0$ and $k,j\in \N$ be so large that
\[
\|f-f_k\|_{H^{-1}}< \epsilon/2, \quad \|f_k-\phi_{k_j}\|_{L^2}< \epsilon/2
\]
%Recall that $H^{-1}$ is an operator norm on $H^1$. 
%Then using the H\"older's inequality for $L^2$-functions we have $\|\cdot\|_{H^{-1}}\leq  \|\cdot\|_{L^2}$. 
Then
\[
\|f-\phi_{k_j}\|_{H^{-1}}\leq \|f-f_k\|_{H^{-1}}+\|f_k-\phi_{k_j}\|_{L^2}< \epsilon.
\]
Thus smooth tensor fields are dense in $H^{-1}$. 

%\color{blue}
Suppose then that
\[
(f,h)\in C^{\infty}(S^2\tau'_M\otimes^\mathcal{B} S^{1} \tau'_M)\times C^{\infty}(S^2\tau'_{\p M}\otimes^\mathcal{B} S^{1} \tau'_{\p M})
\]
is in the co-kernel of the operator $(\LaB, \hbox{Tr})$. We show that $(f,h)\equiv (0,0)$. To verify this we first note that the choice of $(f,h)$ implies in particularly that
\begin{equation}
\label{eq:co-kernel}
\int_M \langle\LaB u,f\rangle_g\mathrm{d}V=0, \quad \hbox{ for any } u\in C^\infty(S^2\tau'_M\otimes^\mathcal{B} S^{1} \tau'_M), \: u|_{\p M}=0.
\end{equation}
To show that \eqref{eq:co-kernel} implies $f\equiv 0$ is very similar to the proof of an analogous claim in \cite[Theorem 3.3.2]{Shara}, and thus omitted here. The second claim $h\equiv 0$ follows from the fact that the trace map is onto.  Finally due to denseness of smooth tensor fields we conclude that $(\LaB, \hbox{Tr})$ has a trivial co-kernel.
\color{black}
%
% We conclude that for any $f\in H^{-1}$ which satisfies
%%Suppose then that 
%\[
%\langle f, \LaB u\rangle_{H^{-1}}=0, \quad \hbox{ for all } u\in C^\infty(S^2\tau'_M\otimes^\mathcal{B} S^{1} \tau'_M), \: u|_{\p M}=0,
%\] 
%the two last steps imply that $f$ must be zero.
%%for some $f \in H^{-1}$. The Riesz representation theorem implies the existence of $v \in H^1_0(S^2\tau'_M\otimes^\mathcal{B} S^{1} \tau'_M)$ such that
%%\[
%%0=\langle f, \LaB u\rangle_{H^{-1}}=\langle v, \LaB u\rangle_{H^{1}}=\int_M \langle\LaB u,v\rangle_g+\langle d'\LaB u,d'v\rangle_g\mathrm{d}V
%%%=\int_M \langle\LaB u,v\rangle_g+\langle \LaB u, \LaB v\rangle_g\mathrm{d}V
%%\]
\end{proof}

\section{The normal operator of mixed ray transform of $1+1$ tensors}
\label{Se:normal_op}
In this section, we 
%\color{blue}
show that the $L^2$-normal operator $\mathcal{N}_{L}$ of the mixed ray transform $L=L_{1,1}$ is an integral operator and find its Schwartz kernel.  
\color{black}
We also show that $\mathcal{N}_L$  is a pseudo-differential operator ($\Psi$DO) of order $-1$ and give a representation for the principal symbol.  In order to do this we will assume without loss of generality that $M\subset \R^3$ is a smooth domain and the metric tensors $g$ extends to $\R^3$ in such a way that any geodesic exiting $M$ never returns to $M$. We make a standing assumption, for the rest of this paper, 
%\textbf{[JZ: what does ``standing assumption" means?], (TS)[just to underline, that this assumption is valid through out rest of the paper.]} 
that any tensor field, excluding the metric $g$, defined in $M$ is extended to any larger domain with zero extension.

%\textbf{[Good point the notations used to be mixed between two formalism. Let us give the definition as Sharafutndinov gave it (TS)][JZ: I would suggest we give the rigorous definition of mixed ray transform acting on trace-free tensors. For the longitudinal ray transform, people work only on symmetric tensors because antisymmetric tensors are the boring kernel. The main point of our tensor decomposition is to tell people we can actually treat the $\mathrm{Im}\lambda$ as the boring kernel for the mixed ray transform. I kind of wish people will directly consider trace free tensors after our paper. Since we choose to give the definition here, we should give a definition that we hope people to use.]
%[Sure that is a very good point. I will do the changes (TS)]
%}
%}
%\color{red}
We begin this section by giving a formal definition of the mixed ray transform, equivalent to \eqref{eq:mixed_ray_trans_act}, on the trace-free tensors.
 For a vector
$(x,v)\in TM$, we define
%\color{blue}
the \textit{contraction map}
\color{black}
\begin{equation}
\label{eq:Lambda}
\Lambda_v\colon S^k T'_xM \otimes^\mathcal{B} S^\ell  T'_xM\rightarrow S^k T'_xM, \quad
(\Lambda_v f)_{i_1\ldots i_m}=f_{i_1\ldots  i_mj_1\ldots j_\ell}v^{j_1}\cdots v^{j_\ell}.
\end{equation}
%\color{blue}
Let
\[
p_v\colon T_xM \to v^\perp, \quad (p_v)_i^k\xi^i
%=\xi -\frac{\langle v ,\xi\rangle}{|v|_g^2} v
=\left(\delta_i^k-\frac{v_iv^k}{|v|_g^2}\right)\xi^i, \quad k \in \{1,2,3\},
\]
be the \textit{projection map} onto the orthocomplement of the vector $v$. The second linear operator is the \textit{restriction map} 
\begin{equation}
\label{eq:P_v}
P_v\colon S^mT'_xM\rightarrow S^mT'_xM, \quad 
(P_v f)_{i_1\ldots i_m}=f_{j_1\ldots j_m}(p_v)_{i_1}^{j_1}\ldots(p_v)_{i_m}^{j_m}.
\end{equation}

%\begin{equation}
%\label{eq:P_v}
%P_v\colon S^mT'_xM\rightarrow S^mT'_xM, \quad 
%(P_v f)_{i_1\ldots i_m}=\left(\delta_{i_1}^{j_1}-\frac{1}{|v|_g^2}v_{i_1}v^{j_1}\right)\cdots \left(\delta_{i_m}^{j_m}-\frac{1}{|v|_g^2}v_{i_m}v^{j_m}\right)f_{j_1\ldots j_m}.
%\end{equation}

\color{black}
Then we can give the following definition of the mixed ray transform by 
%\color{blue}
the following ``distributional" representation (see for instance  \cite[Chapter 7.2]{Shara} )
%\color{red}
%\[
%L_{k,\ell} \colon C^{\infty}(S^k \tau'_M \otimes S^\ell  \tau'_M )\rightarrow C^{\infty}(\beta_k(\partial_+(SM))),
%\]
\begin{equation}
\label{eq:mixed_ray_trans}
L_{k,\ell}f(x,\xi)=\int_0^{\tau(x,\xi)}\mathcal{T}_\gamma^{0,t}(P_{\dot{\gamma}(t)}\Lambda_{\dot{\gamma}(t)}f_{\gamma(t)})\mathrm{d}t, \quad f \in C^\infty(S^k \tau'_M \otimes^\mathcal{B} S^\ell\tau'_M), \: (x,v)\in \p_+SM.
\end{equation}
Here we used the notation
$\mathcal{T}_\gamma^{t,s}$ for the parallel translation along $\gamma$ from the point $\gamma(s)$ to $\gamma(t)$. Symbol $\tau(x,\xi)$ stands for the exit time of the geodesic $\gamma=\gamma_{x,\xi}$ and $f_{\gamma(t)}$ is the evaluation of the tensor field $f$ at $\gamma(t)$.
%, which acts on vectors of the fiber $T_{\gamma(t)}M$.
%\color{blue} 
Let us still clarify the action of the mixed ray transform. 
%We use notions and conventions of distribution theory.
Let $f \in C^\infty(S^k \tau'_M \otimes^\mathcal{B} S^\ell \tau'_M)$ and $(x,\xi) \in \p_+SM$.  We set the action of $L_{k,\ell}f(x,\xi)$ on vector $v=a\xi+\eta \in T_xM$, where $\eta \perp \xi, \: a \in \R$ and $\eta(t)$ is the parallel transport of $\eta$, to be given by integrating the following quantity over the interval $[0,\tau(x,\xi)]$
\[
\begin{split}
\langle \mathcal{T}_\gamma^{0,t}P_{\dot{\gamma}(t)}\Lambda_{\dot{\gamma}(t)}f_{\gamma(t)}, v^k \rangle=
& \: \langle P_{\dot{\gamma}(t)}\Lambda_{\dot{\gamma}(t)}f_{\gamma(t)}, (\mathcal{T}_\gamma^{t,0}v)^k \rangle
\\
=& \: f_{i_1\ldots i_kj_1\ldots  j_\ell}(\gamma(t))\eta(t)^{i_1} \cdots\eta(t)^{i_k} \dot\gamma(t)^{j_1}\cdots\dot\gamma(t)^{j_\ell},
\end{split}
\] 
where $v^k=(\underbrace{v,\ldots, v}_{k})$ and in the last equation we used the formulas \eqref{eq:Lambda} and \eqref{eq:P_v}. This implies the equivalence for  \eqref{eq:mixed_ray_trans_act} and \eqref{eq:mixed_ray_trans}.

Finally we define the target space of the mixed ray transform. Let $\pi\colon\partial_+(SM) \to M$ be the restriction of the natural projection map from tangent bundle to the base manifold. Using the pullback map $\pi^\ast$ we construct the symmetric pullback bundle of $k$-sensors on $\partial_+(SM)$, and reserve the notation
$\beta_k(\partial_+(SM))$ for the sections of this bundle. These sections act as follows: Let $f\in \beta_k(\partial_+(SM))$ and $(x,\xi)\in \partial_+(SM)$,
%\textbf{[JZ: can you rephrase this sentence?]}
% Since the differential $\mathrm{d}\pi\colon T_{(x,\xi)}\partial_+(SM) \to T_xM$ equals to the projection onto the horizontal vectors we see that $f(x,\xi)$, for $f\in \beta_k(\partial_+(SM))$,  as a symmetric $k$-tensor acting on the vectors of the fiber $T_x M$ or in the other words the following  map, 
then, using the pairing notation, the following map is symmetric and $k$-linear
\[
\langle f(x,\xi); v_1,\ldots, v_k \rangle\in \R, \quad v_1\ldots v_k \in T_xM.
\]
Since the exit-time function $\tau$ is smooth in $(x,\xi)\in\p_+SM$ by \cite[Lemma 4.1.1.]{Shara}, the formula \eqref{eq:mixed_ray_trans} implies that the mixed ray transform maps the elements of $C^\infty(S^k \tau'_M \otimes^\mathcal{B} S^\ell \tau'_M)$ into $C^\infty(\beta_k(\partial_+(SM)))$. 
\color{black}

\subsection{The normal operator of $L_{1,1}$ is an integral operator}
For now on we denote $L_{1,1}= L$ for brevity and work only in the space of $S\tau'_M\otimes^\mathcal{B} S\tau'_M$. 

%\color{red}
%\textbf{This used to be part of the Section 1 (TS)}

Let us now describe the measure we use on the target space of the mixed ray transform. For the measure $\mathrm{d}\sigma$ we mean the Riemannian volume of $\p SM$, however we choose  a more suitable measure for $\partial_+(SM)$,  
\[
\mathrm{d}\mu(z,\omega)=|\langle \omega,\nu(z)\rangle_g|\mathrm{d}\sigma=|\langle \omega, \nu(z) \rangle_g|\mathrm{d}S_z\mathrm{d}S_\omega,
\] 
where $\mathrm{d}S_z$ is the surface measure of $\p M$ and $\mathrm{d}S_\omega$ is the surface measure of $S_zM$. 
That is if $(z',z^3)$ is a boundary coordinate system we have
\[
\mathrm{d}S_z=(\det g(z))^{1/2}\mathrm{d}z^1\mathrm{d}z^{2} \quad \hbox{ and } \mathrm{d}S_\omega=(\det g(z))^{1/2}\mathrm{d}S_{\omega_0},
\]
where $\mathrm{d}S_{\omega_0}$ 
%\textbf{[JZ: I have used the notation as in \cite{stefanov2004stability}]} 
is the Euclidean measure of the unit sphere $\mathbb S^{2}\subset \R^3$. 
The $L^2$-inner product on $\beta_1(\partial_+(SM))$ is given by
\[
\int_{\partial_+(SM)} f_{i}(z,\omega)\overline{h}_{j}(z,\omega)g^{ij} \mathrm{d}\mu(z,\omega).
\]
%
%\[
%\int_{\partial_+(SM)} f_{i_1\cdots i_m}(z,\omega)\overline{h}_{j_1\cdots j_m}(z,\omega)g^{i_1j_1}(z)\cdots g^{i_mj_m}(z) \mathrm{d}\mu(z,\omega).
%\]
%
It is shown in \cite[Chapter 7]{Shara} that $L$, originally defined on smooth tensor fields, has a bounded extension
\[
L\colon H^m(S\tau'_M\otimes^\mathcal{B} S\tau'_M)\to H^m(\beta_1(\partial_+(SM))), \quad m\geq 0.
\]
In this section we consider $L$ as an operator
\[
L\colon L^2(S\tau'_M\otimes^\mathcal{B} S\tau'_M)\rightarrow L^2(\beta_1(\partial_+(SM));\mathrm{d}\mu),
\]
and compute its normal operator. 
\color{black}
\begin{remark}
Since $L$ is a bounded operator, its adjoint 
\[
L^\ast\colon L^2(\beta_1(\partial_+(SM));\mathrm{d}\mu)\to L^2(S\tau'_M\otimes^\mathcal{B} S\tau'_M)
\]
exists and is bounded. Thus the normal operator
$
\mathcal{N}_L:=L^\ast L
$
is bounded on $ L^2(S\tau'_M\otimes^\mathcal{B} S\tau'_M)$.
\end{remark}
%\color{blue}
In the following we denote by $(x(t),\omega(t))\in SM, \: t \in \R$, the lift of the geodesic $x(t)$ in $SM$, that is issued from $(z,\omega)\in \partial_+(SM)$.
\color{black}
Let  $f,h \in S\tau'_M\otimes^\mathcal{B} S\tau'_M$. 
%Since $Lf\in \beta_1(\p_+(SM))$ are sections of the vertical co-bundle we can naturally raise the index of $(Lf)_a$, which can be written as
%\[
%(Lf)^b(z,\omega)=\!\!\int_0^{\tau(z,\omega)}\!\!(\mathcal{T}_{\gamma}^{0,t})_a^u(P_{\omega(t)})_u^if_{ij}(x(t))\omega^j(t)\mathrm{d}t \; g^{ab}(z)=\!\!\int_0^{\tau(z,\omega)}\!\!\!g^{bu}(x(t))(P_{\omega(t)})_u^if_{ij}(x(t))\omega^j(t)\mathrm{d}t.
%\] 
Then we have
 \[
 \begin{split}
 \langle Lf,Lh\rangle_{L^2(\beta_1(\partial_+(SM)))} \:\:\:\:&\\
 =&\int_{\partial_+(SM)}\left(\int_0^{\tau(z,\omega)}(\mathcal{T}_{\gamma}^{0,t})_o^u(P_{\omega(t)})_u^if_{ij}(x(t))\omega^j(t)\mathrm{d}t\right)\\
 &\quad\quad\quad \left(\int_0^{\tau(z,\omega)}(\mathcal{T}_{\gamma}^{0,s})_{o'}^{u'}(P_{\omega(s)})_{u'}^{i'}\bar{h}_{i'j'}(x(s))\omega^{j'}(s)\mathrm{d}s\right)
 g^{oo'}(z)\mathrm{d}\mu(z,\omega)
 \\
=:&I_++I_-.
 \end{split}
 \]
Here we wrote
\[
 \begin{split}
 I_\pm=
 \int_{\partial_+(SM)}\int_{\R} &\int_0^\infty g^{ou'}(x(s))(P_{\omega(s)})_{u'}^{i'}\bar{h}_{i'j'}(x(s))\omega^{j'}(s) (\mathcal{T}_{\gamma}^{s,s\pm t})_o^{u}\\
 &\quad\quad\quad\quad\quad\quad (P_{\omega(s\pm t)})_u^if_{ij}(x(s\pm t))\omega^j(s\pm t)\; \mathrm{d}t \mathrm{d}s\mathrm{d}\mathrm{\mu}(z,\omega),
 \end{split}
\]
and used the fact that $g^{-1}$ is parallel.

We introduce new variables $x:=x(s,z,\omega)\in M$,
%\color{blue}
that is the point obtained by following the geodesic $x(s)$ given by the initial conditions $(z,\omega) \in \p_+SM$ until the time $s$,
\color{black}
and $\xi:=\xi(s,t,z,\omega):=t\omega(s)\in T_xM$,
%\color{blue}
which is the scaling of the velocity $\omega(s)$, of the geodesic $x(s)$, by the positive factor $t$. 
%\color{black}
%
%\textbf{(TS)[Let's keep the explanation about the variables $(x,\xi)$, as this was one of the question that the referee asked us.]}
%
Since $(M,g)$ is simple, the map
\[
\R \times (0,\infty)  \times \p_{+}(SM) \ni(s,t,(z,\omega)) \mapsto (x,\xi) \in TM
\] 
can be restricted to a diffeomorphism
%\color{blue}
onto a set  $U\subset TM \setminus (0 \cup T\p M)$, where $U=\bigsqcup_{x \in M} U_x$ and $U_x:=(\exp_x)^{-1}M \subset T_xM$.
\color{black}
Moreover since the geodesic flow preserves the measure $(\det g)\mathrm{d}\xi \mathrm{d}x$ of $TM$, we get
\[
\mathrm{d}t \mathrm{d}s\mathrm{d}\mathrm{\mu}(z,\omega)
%=|\xi|_{g(x)}^{1-n}dV(x,\xi_\pm)
=(- 1)^3|\xi|_g^{2}(\det g)\mathrm{d}\xi \mathrm{d}x.
%\quad \xi:=t\omega(s).
\]
%we use change of variables
%\[
%\R \times (0,\infty)  \times \p_{+}(SM) \ni(s, t, z,\omega)\mapsto (x,\xi)\in T^*M,
%\]
%where $x=x(s)$, $\xi=\pm t\omega(s)$. Then $t=|\xi|$.
For more details about this change of coordinates we refer to \cite{holman2009generic}.
%Since $(M,g)$ is simple we can assume that $(x,\xi)$ are global coordinates of $T^*M$.

We denote
\[
y=\exp_{x}\xi, \: x=x(s,z,\omega), \: \xi=\xi(s,\pm t, z,\omega) \quad \hbox{ and } \quad \widehat{\xi}=\frac{\xi}{|\xi|_g} .
\]
It straightforward to see that
\[
\omega^j(s)= \widehat{\xi}^j, \quad  \hbox{and} \quad \omega^j(s\pm t)= (\hbox{grad}^g_{y} \rho(x,y))^j=g^{ij}(y)\frac{\p\rho}{\p y ^i},
\]
where $\rho(x,y)$ is the Riemannian distance function of $g$ on $M\times M$. Thus
\[
 \begin{split}
 I_\pm=
& \int_M\int_{U_x} g^{ou'}(x)(P_{\widehat{\xi}})_{u'}^{i'}\bar{h}_{i'j'}(x)\widehat{\xi}^{j'}
%\\
%\times &
(\mathcal{T}_{\gamma_{x,\widehat{\xi}}} ^{0,|\xi|_g})_o^{u}(P_{\partial_{y}\rho})_u^if_{ij}(y)g^{kj}(y)\frac{\p\rho}{\p y^k}\; \frac{1}{|\xi|^2_g}(\det g)\mathrm{d}\xi \mathrm{d}x
 \end{split},
\]
and we get
\[
\begin{split}
I_+ = I_-
 \end{split}.
\]

Therefore
\begin{equation}
\label{eq:normal_1}
\begin{split}
\langle \mathcal N_L f,h&\rangle_{L^2(S\tau'_M\otimes S\tau'_M)}=
\\
&2 \int_M\int_{U_x} g^{ou'}(x)(P_{\widehat{\xi}})_{u'}^{i'}\bar{h}_{i'j'}(x)\widehat{\xi}^{j'}
(\mathcal{T}_{\gamma_{x,\widehat{\xi}}} ^{0,|\xi|_g})_o^{u}(P_{\partial_{y}\rho})_u^if_{ij}(y)g^{kj}(y)\frac{\p\rho}{\p y^k}\; \frac{1}{|\xi|^{2}_g}(\det g)\mathrm{d}\xi \mathrm{d}x.
%& 2\int   \bigg( \int f_{ij}(y)g^{kj}(y)\frac{\partial\rho(x,y)}{\partial y^k}(P_{\partial_y\rho})^i_u\widehat{\xi}^{j'}(P_{\hat{\xi}})^{i'}_{u'}\left(\mathcal{T}_{\gamma_{x,\hat{\xi}}}^{0,|\xi|}\right)^{u}_a\frac{\sqrt{\det g(x)}}{|\xi|^{2}_g}(\det g)\mathrm{d}\xi\bigg) \overline{h}_{i'j'}(x)\sqrt{\det g(x)} \mathrm{d}x.
\end{split}
\end{equation}
Since $(M,g)$ is simple the map $(T_xM\setminus \{0\})\ni \xi \mapsto \exp_x \xi=:y$ is a diffeomorphism with inverse
\[
\xi^i=-\frac{1}{2}(\hbox{grad}_x^g(\rho(x,y))^2)^i=-\frac{1}{2}g^{ij}(x)\frac{\p \rho(x,y)^2}{\p x^j},
\]
and moreover
\[
|\xi|_g=\rho(x,y), \quad \widehat{\xi}_m=-\frac{\p \rho(x,y)}{\p x^m} \quad \hbox{and} \quad \mathrm{d}\xi=(\det g^{-1})\bigg|\det\frac{\p^2 \rho(x,y)^2/2}{\p x\p y} \bigg|\mathrm{d}y.
\]
To simplify the integral \eqref{eq:normal_1} even more we compute
\[
(P_{\widehat{\xi}})_{u'}^{i'}\bar{h}_{i'j'}(x)=\bar{h}^{i'}_{j'}(x)\bigg(g_{i'u'}(y)-\frac{\p \rho}{\p x^{u'}}\frac{\p \rho}{\p x^{i'}}\bigg)
\]
and 
\[
(P_{\partial_{y}\rho})_u^if_{ij}(y)g^{kj}(y)=(P_{\partial_{y}\rho})_u^if_{i}^k(y)=f^{ik}(y)\bigg(g_{iu}(y)-\frac{\p \rho}{\p y^u}\frac{\p \rho}{\p y^i}\bigg).
\]

Finally we have 
\begin{equation}
\label{eq:action_of_normal_op}
\begin{split}
&\langle \mathcal N_L f,h\rangle_{L^2(S\tau'_M\otimes S\tau'_M)}=
\\
&\quad\quad\quad-2\int_M \bar h^{k\ell}(x)  \frac{1}{\sqrt{\det g(x)}} \int_M\bigg(g_{ku'}(x)-\frac{\p \rho}{\p x^{u'}}\frac{\p \rho}{\p x^{k}}\bigg)\frac{\p \rho}{\p x^\ell}g^{ou'}(x)\left(\mathcal{T}_{\gamma_{x,-\hbox{\tiny{grad}}_x^g\rho} } ^{0,\rho(x,y)}\right)_o^{u} 
\\
&\quad\quad\quad
\quad\times \frac{f^{ij}(y)}{\rho(x,y)^{2}} \bigg(g_{iu}(y)-\frac{\p \rho}{\p y^u}\frac{\p \rho}{\p y^i}\bigg)\frac{\p\rho}{\p y^j}\;  \bigg|\det\frac{\p^2 \rho(x,y)^2/2}{\p x\p y} \bigg|\mathrm{d}y \sqrt{\det g(x)} \mathrm{d}x.
\end{split}
\end{equation}
%\[
%\begin{split}
%&\langle \mathcal N f,h\rangle_{L^2(S\tau'_M\otimes S\tau'_M)}=
%\\
%&\quad\quad\quad 2\int \bar h_{i'm}(x)  \int\bigg(\delta^{i'u}-\frac{(x-y)^{u}(x-y)^{i'}}{|x-y|^2}\bigg)\frac{(x-y)^{m}}{|x-y|}
%\\
%&\quad\quad\quad
%\quad\times \frac{f_{ik}(y)}{|x-y|^2} \bigg(\delta^i_{u}-\frac{(x-y)_{u}(x-y)^{i}}{|x-y|^2}\bigg)\frac{(x-y)^{k}}{|x-y|}\;  \mathrm{d}y\mathrm{d}x
%\\
%=&
%\quad\quad\quad 2\int    \int  h_{i'm}(x) f_{ik}(x-z)\bigg(\delta^{i'}_{u'}-\frac{z_{u'}z^{i'}}{|z|^2}\bigg)\delta^{u'u}
%\bigg(\delta^i_{u}-\frac{z_{u}z^{i}}{|z|^2}\bigg)\frac{z^{k}}{|z|} \frac{z^{m}}{|z|} \frac{1}{|z|^2}\;  \mathrm{d}z\mathrm{d}x
%\end{split}
%\]
Therefore the normal operator of mixed ray transform can be written as 

\[
\begin{split}
(\mathcal N_L f)_{k\ell}(x)=
 \frac{-2 }{\sqrt{\det g(x)}}   &\int_M \bigg(g_{ku'}(x)-\frac{\p \rho}{\p x^{u'}}\frac{\p \rho}{\p x^{k}}\bigg)A^{uu'}(x,y)
\bigg(g_{iu}(y)-\frac{\p \rho}{\p y^u}\frac{\p \rho}{\p y^i}\bigg)\frac{f^{ij}(y)}{\rho(x,y)^{2}} \frac{\p\rho}{\p y^j}
\\
&\quad\quad\quad
\times \frac{\p \rho}{\p x^\ell}\; \bigg|\det\frac{\p^2 \rho(x,y)^2/2}{\p x\p y} \bigg|\mathrm{d}y.
\end{split}
\]
Here 
\begin{equation}
\label{Auu}
A(x,y)^{uu'}:=g^{ou'}(x)\left(\mathcal{T}_{\gamma_{x,-\hbox{\tiny{grad}}_x^g\rho(x,y)}}^{0,\rho(x,y)}\right)^{u}_o,
\end{equation}
with $A^{uu'}(x,x)=g^{uu'}(x)$. We note here that $\mathcal{T}_{\gamma_{x,-\hbox{\tiny{grad}}_x^g\rho(x,y)}}^{0,\rho(x,y)}$ is the parallel transport along the geodesic connecting $y$ to $x$. Thus $\mathcal{N}_L$ is an integral operator with an integral kernel 
\begin{equation}
\label{eq:Kernel}
\begin{split}
K_{ijk\ell}(x,y)=
& \frac{-2  \bigg|\det\frac{\p^2 \rho(x,y)^2/2}{\p x\p y} \bigg|}{\sqrt{\det g(x)}\rho(x,y)^{2}} \bigg(g_{ku'}(x)-\frac{\p \rho}{\p x^{u'}}\frac{\p \rho}{\p x^{k}}\bigg)A^{uu'}(x,y)
\bigg(g_{iu}(y)-\frac{\p \rho}{\p y^u}\frac{\p \rho}{\p y^i}\bigg) \frac{\p\rho}{\p y^j}
\frac{\p \rho}{\p x^\ell}.
\end{split}
\end{equation}
%\color{blue}
\subsection{Normal operator as a $\Psi$DO}

From now on, we rely on the fact that $M\subset \R^3$ and $g$ is extended to whole $\R^3$ to apply the theory of pseudo-differential operators. In this subsection we  show that the normal operator is a $\Psi$DO of order $-1$ and find its principal symbol. Since $M$ is closed we consider certain open neighborhoods of it. 

%\color{red}
Since $(M,g) \subset \R^3$ is simple and $g$ is extended to whole $\R^3$, we can find open domains $M_1,M_2 \subset \R^3$ such that $M\subset M_1\subset\subset M_2$ and $(\overline{M_i},g), \: \{1,2\}$ is simple (see \cite[page 454]{stefanov2004stability}). We need an open extension of $M$ in order to show that $\mathcal{N}_L$ is a $\Psi$DO. In the next section we explain why do we need the aforementioned double extension.  We note that by this extension the normal operator $\mathcal{N}_L$  is defined for $1+1$ tensor fields over $M_2$, and $\mathcal{N}_Lf(x)$ remains the same for $x\in M$ if $\supp f\subset M$.
%\color{blue}

First we find more convenient representation for the kernel $K$ near the diagonal of $M_2\times M_2$.  To do so we use the following relations introduced in \cite[Lemma 1]{stefanov2004stability}:

\begin{lemma}
\label{Le:Gs}
There exists $\delta>0$ such that in $U:= \{(x,y)\in M_2\times M_2: |x-y|_e<\delta\}$ the following hold
\begin{equation}
\label{eq:Gs}
\begin{split}
\rho^2(x,y)&=G^{(1)}_{ij}(x,y)(x-y)^i(x-y)^j,\\
\frac{\partial \rho^2(x,y)}{\partial x^j}&=2G^{(2)}_{ij}(x,y)(x-y)^i,\\
\frac{\partial^2\rho^2(x,y)}{\partial x^i\partial y^j}&=-2G_{ij}^{(3)}(x,y),
\end{split}
\end{equation}
where $G_{ij}^{(1)}(x,y),\,G_{ij}^{(2)}(x,y),\,G_{ij}^{(3)}(x,y)$ are smooth and on the diagonal
\begin{equation}
\label{eq:G_on_diagonal}
G_{ij}^{(1)}(x,x)=G_{ij}^{(2)}(x,x)=G_{ij}^{(3)}(x,x)=g_{ij}(x).
\end{equation}
\end{lemma}
\begin{proof}
See \cite[Lemma 3.1]{SU}.
\end{proof}
\noindent
As $G^{(m)}(x,y)$ is a matrix depending on the points $(x,y)\in U$ we use the short-hand notations $G^{(m)}:=G_{ij}^{(m)}(x,y)$, $\tilde{G}_{ij}^{(2)}:=G_{ij}^{(2)}(y,x)$, $G^{(m)}z:=G_{ij}^{(m)}(x,y)(x-y)^i,\:z:=x-y$. These imply
\begin{equation}
\label{eq:Gs_2}
\rho^{2} (x,y)=G^{(1)}z\cdot z,
\quad \!\!
\frac{\p \rho}{\p x^j}=\frac{\bigg[G^{(2)}z\bigg]_j}{(G^{(1)}z \cdot z)^{1/2}}, \quad \!\! \frac{\p \rho}{\p y^j}=\frac{\bigg[\widetilde G^{(2)}z\bigg]_j}{(G^{(1)}z \cdot z)^{1/2}}.
\end{equation}

%
%\begin{equation}
%\label{eq:Gs_2}
%\rho(x,y)^{2}=G^{(1)}(x-y)\cdot (x-y),
%\quad \!\!
%\frac{\p \rho}{\p x^j}=\frac{\bigg[G^{(2)}(x-y)\bigg]_j}{(G^{(1)}(x-y) \cdot (x-y))^{1/2}}, \quad \!\! \frac{\p \rho}{\p y^j}=\frac{\bigg[\widetilde G^{(2)}(x-y)\bigg]_j}{(G^{(1)}(x-y) \cdot (x-y))^{1/2}}.
%\end{equation}
\color{black}
%
%Thus the kernel of $\mathcal{N}_L$ is
%\[
%\begin{split}
%K_{ijk\ell}(x,y)=\frac{-2|\det G^{(3)}|}{\sqrt{\det g(x)}\left(G^{(1)}(x-y)\cdot(x-y)\right)^2}  \left(g_{iu}(y)-\frac{\left[\tilde{G}^{(2)}(x-y)\right]_i\left[\tilde{G}^{(2)}(x-y)\right]_u}{G^{(1)}(x-y)\cdot(x-y)}\right)A^{uu'}(x,y)
%\\
%\times\left(g_{ku'}(x)-\frac{\left[G^{(2)}(x-y)\right]_k\left[G^{(2)}(x-y)\right]_{u'}}{G^{(1)}(x-y)\cdot(x-y)}\right) \left[\widetilde G^{(2)}(x-y)\right]_j\left[{G}^{(2)}(x-y)\right]_\ell.
%\end{split}
%\]
%
%with $\check{M}$ is the inverse Fourier transform of $M$ w.r.t. $\xi$ where
%
%\begin{lemma}
%The integral kernel $K_{ijk\ell}$ is very regular in the sense of \cite[Chapter 1, Section 2]{treves1} and in $L^1_{loc}$ with respect to either variable $x,y$.
%\end{lemma}
%\begin{proof}
%\textbf{To be added.}
%\end{proof}
%
\noindent
Thus the following formula holds for the integral kernel of $\mathcal{N}_L$
%\color{blue} 
on the neighborhood $U$ of the diagonal of the extension $M_2 \times M_2$.
\color{black}
\[
\begin{split}
K_{ijk\ell}(x,y)=&-2 \left(g_{iu}(y)-\frac{\left[\tilde{G}^{(2)}z\right]_i\left[\tilde{G}^{(2)}z\right]_u}{G^{(1)}z\cdot z}\right)\frac{A^{uu'}(x,y)}{\left(G^{(1)}z\cdot z\right)^2} 
\\
&\times\left(g_{ku'}(x)-\frac{\left[G^{(2)}z\right]_k\left[G^{(2)}z\right]_{u'}}{G^{(1)}z\cdot z}\right) \left[\widetilde G^{(2)}z\right]_j\left[{G}^{(2)}z\right]_\ell\frac{ \big|\det G^{(3)} \big|}{\sqrt{\det g(x)}}.
\end{split}
\]
%\color{blue}
From this and \eqref{eq:Kernel} we see that the integral kernel $K_{ijk\ell}$ is smooth in $M_2 \times M_2$ outside the diagonal, at which it has a singularity of the type $|x-y|_e^{-2}$.

Let $\chi \in C^\infty_0(U)$ equal to $1$ near the diagonal of $M_2\times M_2$. We write
\begin{equation}
\label{eq:spilt_of_K}
K_{ijk\ell}=\chi K_{ijk\ell}+(1-\chi K_{ijk\ell})=:K_{ijk\ell}^1+K_{ijk\ell}^2.
\end{equation}
Since $K_{ijk\ell}^2 \in C^\infty(M_2 \times M_2)$ the corresponding integral operator is a $\Psi$DO of order $-\infty$, with an amplitude of order $-\infty$. 
\color{black}

\begin{lemma}
\label{Le:M_tilde}
Let set $U$ be as in lemma \ref{Le:Gs}. For any $((x,y),z)\in U \times \R^3$ we define 
\begin{equation}
\label{eq:tildeM}
\begin{split}
\widetilde M_{ijk\ell}(x,y,z):=
&-2
\left(g_{iu}(y)-\frac{\left[\tilde{G}^{(2)}z\right]_i\left[\tilde{G}^{(2)}z\right]_u}{G^{(1)}z\cdot z}\right)\frac{A^{uu'}(x,y)}{\left(G^{(1)}z\cdot z\right)^2} 
\\
&\times\left(g_{ku'}(x)-\frac{\left[G^{(2)}z\right]_k\left[G^{(2)}z\right]_{u'}}{G^{(1)}z\cdot z}\right) \left[\widetilde G^{(2)}z\right]_j\left[{G}^{(2)}z\right]_\ell\frac{ \big|\det G^{(3)} \big|}{\sqrt{\det g(x)}}.
\end{split}
\end{equation}
The distribution $\widetilde M_{ijk\ell}$ belongs to $L^1_{loc}(\R^3)$ with respect to $z$ variable and  is positively homogeneous of order $-2$. Moreover $\widetilde M_{ijk\ell}$ is smooth in $U\times (\R^3\setminus\{0\})$.
\end{lemma}
\begin{proof}
We note that the equation \eqref{eq:Gs} imply that $G^{(2)}z$ and $\widetilde G^{(2)}z$ are 1-homogeneous with respect to $z$. Let $K_r$ be the compact set that is the image of the closed ball $B_r(0), \: r>0$ under the diffeomorphism $z'=H^{-1}z$, where $H$ is the square root of $G^{(1)}(x,y)$.
%This implies
%\[
%G^{(1)}(H^{-1}z')\cdot (H^{-1} z')=|z'|_e^2, \quad \hbox{ and } \quad  \mathrm{d}z=|\det H|\mathrm{d}z'.
%\]
By a change to spherical coordinates we obtain
\[
\begin{split}
\int_{K_r}|\widetilde M_{ijk\ell}(x,y,z')|\: \mathrm{d}z'=&\int_{B_r}|\widetilde M_{ijk\ell}(x,y,H^{-1}z)|\: |\det H^{-1}| \mathrm{d}z
\\
%\\
%=&\int_{B_r}\bigg(g_{ku'}(x)-\frac{\bigg[G^{(2)}H^{-1} z'\bigg]_{u'}\bigg[G^{(2)}H^{-1} z'\bigg]_{k}}{|z'|_e^2}\bigg)\frac{A^{uu'}(x,y)}{|z'|_e^4}
%\\
%\times & \bigg(g_{iu}(y)-\frac{\bigg[\widetilde G^{(2)}H^{-1} z'\bigg]_u\bigg[\widetilde G^{(2)}H^{-1} z'\bigg]_i}{|z'|_e^2}\bigg)
%\\
%\times &
%\bigg[\widetilde G^{(2)}H^{-1} z'\bigg]_j\bigg[G^{(2)}H^{-1} z'\bigg]_\ell\frac{ \big|\det G^{(3)} \big|}{\sqrt{\det g(x)}}\: |\det H^{-1}|\mathrm{d}z'
%\\
%=&\int_{S^2}\int_0^r\bigg(g_{ku'}(x)-\frac{t^2\bigg[G^{(2)}H^{-1}\omega \bigg]_{u'}\bigg[G^{(2)}H^{-1} \omega\bigg]_{k}}{t^2}\bigg)\frac{A^{uu'}(x,y)}{t^4}
%\\
%\times & \bigg(g_{iu}(y)-\frac{t^2\bigg[\widetilde G^{(2)}H^{-1} \omega\bigg]_u\bigg[\widetilde G^{(2)}H^{-1} \omega\bigg]_i}{t^2}\bigg)
%\\
%\times &
%t^2\bigg[\widetilde G^{(2)}H^{-1} \omega\bigg]_j\bigg[G^{(2)}H^{-1} \omega\bigg]_\ell\frac{ \big|\det G^{(3)} \big|}{\sqrt{\det g(x)}}\: |\det H^{-1}|t^2\mathrm{d}t\mathrm{d}\omega
%\\
=& 
rC  \int_{S^2}\bigg|\bigg(g_{iu}(y)-\bigg[\widetilde G^{(2)}H^{-1} \omega\bigg]_u\bigg[\widetilde G^{(2)}H^{-1} \omega\bigg]_i\bigg) A^{uu'}(x,y)
\\
\times &\bigg(g_{ku'}(x)-\bigg[G^{(2)}H^{-1}\omega \bigg]_{u'}\bigg[G^{(2)}H^{-1} \omega\bigg]_{k}\bigg) 
%\\
%\times &
\bigg[\widetilde G^{(2)}H^{-1} \omega\bigg]_j\bigg[G^{(2)}H^{-1} \omega\bigg]_\ell \bigg| \:\mathrm{d}\omega,
\end{split}
\]
where $C=2|\det H^{-1}|\frac{ \big|\det G^{(3)} \big|}{\sqrt{\det g(x)}} $. Since the last integrand is continuous we have proven the first claim.

\medskip 
The second claim follows since $\widetilde{M}_{ijk\ell}(x,y,H^{-1}tz)=t^{-2}\widetilde{M}_{ijk\ell}(x,y,H^{-1}z), \: t>0$ implies
\[
\int_{\R^3}\widetilde M_{ijk\ell}(x,y,z) \varphi(z)\: \mathrm{d}z=t\int_{\R^3}\widetilde M_{ijk\ell}(x,y,z) \varphi(tz)\: \mathrm{d}z
\]
for every test function $\varphi$ and $t>0$. 
%To verify this we compute
%
%\[
%\begin{split}
%\int_{\R^3}\widetilde M_{ijk\ell}(x,y,z) \varphi(z)\: \mathrm{d}z=&|\det H^{-1}|\int_{\R^3}\widetilde M_{ijk\ell}(x,y,H^{-1}z) \varphi(H^{-1}z)\: \mathrm{d}z
%\\
%=&t|\det H^{-1}|\int_{\R^3}\widetilde M_{ijk\ell}(x,y,H^{-1}z) \varphi(tH^{-1}z)\: \mathrm{d}z
%\\
%=&t\int_{\R^3}\widetilde M_{ijk\ell}(x,y,z) \varphi(tz)\: \mathrm{d}z,
%%t^{4}\int_{\R^3}\widetilde M_{ijk\ell}(x,y,z) \varphi(tz)\: \mathrm{d}z
%\end{split}
%\]
%This holds since $M_{ijk\ell}(x,y,H^{-1}tz)=t^{-2}M_{ijk\ell}(x,y,H^{-1}z)$.
\end{proof}

%\color{blue}
Due to the previous lemma and \eqref{eq:spilt_of_K} we can write for $(x,y)\in M_2\times M_2$ that 
\[
\begin{split}
K^1_{ijk\ell}(x,y)=\chi(x,y) \widetilde M_{ijk\ell}(x,y,x-y)=\mathcal{F}^{-1}_\xi\left(M_{ijk\ell}(x,y,\xi)\right)\bigg|_{x-y}
\end{split}
\]
%\begin{equation}
%\label{eq:operator_of_M}
%\begin{split}
%(\mathcal{N}_Lf)_{k\ell}(x)=
%& 
%\int K_{ijk\ell}(x,y)f^{ij}(y) \mathrm{d}y
%=
%%\int e^{i(x-y)\cdot \xi }M_{ijk\ell}(x,y, \xi)f^{ij}(y) \mathrm{d}\xi \mathrm{d}y 
%%\\
%%=&
%(2\pi)^{-3}\int e^{i(x-y)\cdot \xi }M_{ijk\ell}(x,y, \xi)f^{ij}(y)  \mathrm{d}y \mathrm{d}\xi,
%\end{split}
%\end{equation}
where
\begin{equation}\label{eq:Msymbol}
\begin{split}
M_{ijk\ell}(x,y,\xi)&=\chi(x,y)\int e^{-i \xi\cdot z}  \widetilde M_{ijk\ell}(x,y,z)\mathrm{d}z.
\end{split}
\end{equation}
Therefore $M_{ijk\ell}$ is homogeneous of order 
$
%-(-2)-3=
-1$ in $\xi$.  Since $M_{ijk\ell}$ is smooth in $M_2\times M_2 \times (\R^3 \setminus \{0\})$ and homogeneous of order $-1$, it is an amplitude of order $-1$. Thus the decomposition \eqref{eq:spilt_of_K} implies that $\mathcal{N}_L$ is a $\Psi$DO of order $-1$. 
\color{black}
To conclude the study of the normal operator, we state the main result of this section. 
\begin{proposition}
\label{Pr:normal_is_pseudo}
The normal operator $\mathcal{N}_L$ of mixed ray transform is a 
%properly supported 
pseudo-differential operator of order $-1$ in $M_2$ with the
%an amplitude $M_{ijk\ell}(x,y,\xi)$ and 
principal symbol 
\begin{equation}
\label{eq:Msymbol_2}
\begin{split}
\sigma_p(\mathcal{N}_L)^{ijk\ell}(x,\xi)=&
%-2\sqrt{\det g(x)}\int e^{-i \xi\cdot z}\bigg(\delta^k_{u}-\frac{[gz]_{u}z^{k}}{[gz\cdot z]}\bigg)
% \bigg(\delta^u_{i}-\frac{[gz]_iz^{u}}{[gz\cdot z]}\bigg)\frac{[gz]_jz^\ell}{[gz\cdot z]^{2}} \;\mathrm{d}z
%\\
%=&
%-2\sqrt{\det g(x)}\int_{\R^3} e^{-i \xi\cdot z}\bigg(g_{ki}(x)-\frac{[gz]_{i}[gz]_{k}}{[gz\cdot z]}\bigg)
%\frac{[gz]_j[gz]_\ell}{[gz\cdot z]^{2}} \;\mathrm{d}z.
%%=&
-2\sqrt{\det g(x)}\int_{\R^3} e^{-i \xi\cdot z}\bigg(g^{ki}(x)-\frac{z^{i}z^{k}}{|z|_g^2}\bigg)
\frac{z^jz^\ell}{|z|_g^4} \;\mathrm{d}z
\\
\\
=&-\frac{\sqrt{\pi}}{3\cdot 2^{5/2}} \bigg(-12g^{ik}|\xi|_g^{-1}(g^{jl}-|\xi|_g^{-2}\xi^j\xi^l)\\
&+9|\xi|_g^{-5}\xi^i\xi^j\xi^k\xi^l +3|\xi|_g^{-1}(g^{jk}g^{il}+g^{ik}g^{jl} + g^{kl}g^{ij}\color{black} )
\\
&-3|\xi|_g^{-3}(g^{ij}\xi^k\xi^l+g^{ik}\xi^j\xi^l+g^{il}\xi^j\xi^k+
g^{jk}\xi^i\xi^l+g^{jl}\xi^i\xi^k  +g^{kl}\xi^i\xi^j \color{black} )\bigg).
\end{split}
\end{equation}
Here 
$z_i=g_{ij}(x)z^j$ and 
$|z|^2_g=g_{ij}(x)z^iz^j$.
\end{proposition}

\begin{remark}
%Traditionally the theory of $\Psi$DO's is developed on the co-tangent side. We chose to write the principal symbol on the tangent side as we will later apply it to tensor fields. 
We note that terms $ g^{kl}g^{ij}$ and $g^{kl}\xi^i\xi^j$ in \eqref{eq:Msymbol_2} do not contribute to the action of the symbol as we are working on the kernel of the map $\mu$. Thus we ignore those two terms in the following calculations.
% {\color{red}Thus these terms do not affect the computations appearing later in this paper.}
\end{remark}

\begin{proof}
We emphasize that $\mathcal{N}_L$ is not properly supported. However there exists properly supported $\Psi$DO $A_{ijk\ell}$ of order $-1$ such that $(\mathcal{N}_L)_{ijk\ell}-A_{ijk\ell}$ is smoothing. Neglecting this technicality we obtain the principal symbol of $\mathcal{N}_L$ by setting $x=y$ in \eqref{eq:Msymbol}. Then formula \eqref{eq:G_on_diagonal} implies
\[
G^{(2)}_{ij}(x,x)z^i=z_j, \quad \hbox{and} \quad G^{(1)}_{ij}(x,x)z^iz^j=|z|^2_g.
\]
Therefore after raising indices formulas  \eqref{eq:tildeM} and  \eqref{eq:Msymbol} imply the first equation of
%we get
%\[
%\begin{split}
%\sigma_p(\mathcal{N}_L)_{ijk\ell}(x,\xi)=&
%-2\sqrt{\det g(x)}\int e^{-i \xi\cdot z}\bigg(g_{ku'}(x)-\frac{[gz]_{u'}[gz]_{k}}{[gz\cdot z]}\bigg)\frac{g^{uu'}(x)}{[gz\cdot z]^{2}}
%\\
%&\:\times \bigg(g_{iu}(x)-\frac{[gz]_u[gz]_i}{[gz\cdot z]}\bigg)[gz]_j[gz]_\ell \;\mathrm{d}z.
%\end{split}
%\]
%Therefore
%\[
%\sigma_p(\mathcal{N}_L)_{ij}^{k\ell}(x,\xi):=\sigma_p(\mathcal{N}_L)_{ijhm}(x,\xi)g^{hk}(x)g^{m\ell}(x)
%\]
\eqref{eq:Msymbol_2}.

%Let us first study the presentation of the principal symbol in the case where Riemannian metric has constant coefficients, that is globally we have $g_{ij}(x)=g_{ij}$. We return back to the equation \eqref{eq:action_of_normal_op} and note that in the constant metric case:
%\begin{itemize}
%\item Geodesics are straight lines.
%\item Parallel transport is an identity map
%\item $\rho(x,y)=\sqrt{g_{ij}(x-y)^i(x-y)^j}:=|x-y|_g$
%\item $\frac{\p}{\p x_k}\rho(x,y)=g_{kj}(x-y)^j|x-y|_g^{-1}=-\frac{\p}{\p y_k}\rho(x,y)$
%\item $\bigg|\det\frac{\p^2 \rho(x,y)^2/2}{\p x\p y} \bigg|=|\det g|$
%\end{itemize}
%
% after doing change of variables $z=x-y$ this equation has the following form for any $f,h \in S\tau'_M\otimes^{\mathcal{B}_0} S\tau'_M$
%\[
%\langle \mathcal N f,h\rangle_{L^2(S\tau'_M\otimes S\tau'_M)}=2\int\int f_{kl}(x-z)\overline{h}_{ij}(x) \left(\delta_a^i-\frac{z_az^i}{|z|^2}\right)\delta^{ba}\left(\delta_b^k-\frac{z_bz^k}{|z|^2}\right)\frac{z^j}{|z|}\frac{z^l}{|z|}\frac{1}{|z|^{2}}\mathrm{d}z\mathrm{d}x,
%\]
%where we used that Euclidean parallel transport is an identity operator. Then we have
%\[
%(\mathcal{N}_Lf)^{ij}(x)=2\int\left(g^{ik}-\frac{z^iz^k}{|z|^2}\right)\frac{z^jz^l}{|z|^{4}}f_{kl}(x-z)\mathrm{d}z,
%\]
%which we write as a convolution 
%\begin{equation}
%\label{eq:Euclidean_norm_op_as_conv}
%(\mathcal{N}_Lf)^{ij}(x)=2f_{kl}*\left(\frac{x^jx^l\delta^{ik}}{|x|^{4}}-\frac{x^ix^jx^kx^l}{|x|^{6}}\right).
%\end{equation}

We proceed on to compute the Fourier transform in  \eqref{eq:Msymbol_2}. We recall the following formula for the $n-$dimensional Fourier transform of the powers of a norm given by a positive definite matrix $g$:
\[
\mathcal{F}[|x|_g^{\alpha}](\xi)=c_{n,\alpha}\frac{1}{\sqrt{\det g}}|\xi|_g^{-\alpha-n}, \quad  c_{n,\alpha}=2^{n/2-\alpha}\frac{\Gamma(\frac{n-\alpha}{2})}{\Gamma(\frac{\alpha}{2})},\quad \alpha \neq n+2k, \: k \in \Z.
\]
Here $\Gamma$ is the Gamma function. In dimension $3$ we have
\[
\sqrt{\det g}\;\mathcal{F}_x[|x|_g^{-6}](\xi)=2^{-9/2}\frac{\Gamma(-\frac{3}{2})}{\Gamma(3)}|\xi|_g^3=\frac{\sqrt{\pi}}{3\cdot 2^{7/2}}|\xi|_g^3,
\]
\[
\sqrt{\det g}\;\mathcal{F}_x[|x|_g^{-4}](\xi)=2^{-5/2}\frac{\Gamma(-\frac{1}{2})}{\Gamma(2)}|\xi|_g=-\frac{\sqrt{\pi}}{2^{3/2}}|\xi|_g.
\]
%In dimension three ($n=3$),
Thus the Fourier transform in \eqref{eq:Msymbol_2} equals to
\[
%\mathcal{F}(\mathcal{N}_Lf)^{ij}=C\hat{f}_{kl}
C\left(
-12g^{ik}\frac{\partial^2|\xi|_g}{\partial\xi^j\partial\xi^l}+\frac{\partial^4|\xi|_g^3}{\partial\xi^i\partial\xi^j\partial\xi^k\partial\xi^l}\right)
=:CN^{ijkl}
%\hat{f}_{kl}
, \quad  C=-\frac{\sqrt{\pi}}{3\cdot 2^{5/2}}.
\]
Finally we compute the derivatives in the formula above to find that $CN^{ijkl}$ is the right hand side of the equation \eqref{eq:Msymbol_2}.
%\color{red}
%\[
%\begin{split}
%N^{ijkl}(\xi)=&-12g^{ik}|\xi|_g^{-1}(g^{jl}-|\xi|_g^{-2}\xi^j\xi^l)\\
%&+9|\xi|_g^{-5}\xi^i\xi^j\xi^k\xi^l +3|\xi|_g^{-1}(g^{jk}g^{il}+g^{ik}g^{jl} + g^{kl}g^{ij} )
%\\
%&-3|\xi|_g^{-3}(g^{ij}\xi^k\xi^l+g^{ik}\xi^j\xi^l+g^{il}\xi^j\xi^k+
%g^{jk}\xi^i\xi^l+g^{jl}\xi^i\xi^k  +g^{kl}\xi^i\xi^j  ).
%\end{split}
%\]
%\color{black}
\end{proof}
\begin{remark}
If $g$ is a constant coefficient metric, then \eqref{eq:Msymbol_2} gives the full symbol of the normal operator. The proof is similar to \cite[Section 4]{stefanov2004stability}.
\end{remark}

\section{Reconstruction formulas and stability estimates}
\label{Se:Parametrix}
%\color{blue}
For this section we set the assumption that the metric $g$ is only $C^k$-smooth in $\R^3$ for some $k\in \N$ that is large enough. Nevertheless we can still assume that the closed set $M\subset \R^3$ is extended to simple open domains $(M_1,g)$ and $(M_2,g)$ such that $M\subset \subset M_1\subset\subset M_2$. 

We begin this section by recalling some basic theory of $\Psi$DOs whose amplitudes are only finitely smooth. We also show that the solution operator of the boundary value problem \eqref{eqnu1} depends continuously on $g$ with respect to $C^k$-topology.
% \textbf{[JZ: not $C^k$-continuous, various places below. What we use is the continuity of the map w.r.t. the $C^k$-topology for the metrics. $C^k$-continuous means the map is $k$-th continuously differentiable.]} of the metric. 
In the second part of this section we show that the normal operator $\mathcal{N}_L$ is elliptic on the subspace of solenoidal tensors on $M_1$ and construct a parametrix with respect to this subspace. In order to do this we need to have the second extension $M_2$ as we will use the projection operator $\mathcal{P}_{M_2}=\dB(\LaB_{\!\!\!\!M_2})^{-1} \deB$ onto the potential tensors, with respect to the largest domain $M_2$, in the construction of the parametrix. In addition we define the projection operator $\mathcal{S}_{M_2}:=\Id-\mathcal{P}_{M_2}$ onto solenoidal tensor fields on $M_2$.  In the final part of the section we study the stability of the normal operator. We also provide a reconstruction formula for the solenoidal component, with respect to $M$.

\subsection{Pseudo-differential operators with finitely smooth amplitudes}
Since the metric is assumed to pose only finite smoothness we need to set some finite regularity conditions also for the  amplitudes of the $\Psi$DOs we are interested in.  We use the notation $A^m, \: m \in \R$ for the space of $C^k$-smooth amplitudes of order $m$ in $M_2$. That is an amplitude $a(x,y,\xi), \: (x,y,\xi)\in M_2 \times M_2 \times \R^3$ in $A^m$ are set to satisfy only a finite amount of seminorm estimates: 
\[
|D_\xi^\alpha D_x^\beta D_y^\gamma a(x,y,\xi)|\leq C_{\alpha,\beta,\gamma}(\mathcal K)\left(1+|\xi|_e \right)^{m-|\alpha|}, \quad \hbox{$(x,y) \in \mathcal K \subset M_2$ is compact,} \quad C_{\alpha,\beta,\gamma}(\mathcal K)>0,
\] 
where $\alpha,\beta,\gamma \in \N^3$ are multi-indices that satisfy
$
|\alpha|,|\beta|,|\gamma|\leq k.
$
We repeat the proof of \cite[Theorem 2.1]{treves1} to observe that for any $m_0,\:s_0 >0$ there exists $k \in \N$ such that for any $|m|\leq m_0$ and $|s|\leq s_0$ the linear operator
\[
\mathcal A\colon H^s(\overline {M_1}) \to H^{s-m}(\overline {M_1})
\]
is bounded if $\mathcal A$ is a $\Psi$DO in $M_2$ with an amplitude in $A^m$ having the regularity $k$. We also note that the operators with amplitudes in $A^m$, for any $m \in \R$ are\textit{ finitely pseudo-local} in the following sense:  If $U \subset M_2$ is open and $u \in \mathcal E'(M_2)$, then for any $k' \in \N$ and $m \in \R$ there exists $k \in N$ such that if $\mathcal A$ is $\Psi$DO with a $C^k$-smooth amplitude in $A^m$ then the following holds:
\[
\mathcal Au \in C^{k'}(U), \quad \hbox{ if }u \in C^{k'}(U).
\]
This follows from the proof of \cite[Theorem 2.2.]{treves1} after minor modifications. 
%If $\mathcal A,\mathcal B$ are $\Psi$DOs with $C^k$-regular amplitudes in $A^m$ and $A^{-m}$ respecfully, we say that $\mathcal B$ is a \textit{parametrix} of $\mathcal A$ if  there exists $k'\geq 1$ such that 
%\[
%\mathcal A\mathcal B,\: \mathcal B\mathcal A \colon \mathcal E'(M_2)\to C^{k'}(M_2).
%\]

Using the aforementioned machinery for finitely smooth  amplitudes we note that the integral kernel $K_{ijk\ell}$ of $\mathcal{N}_L$, as given in \eqref{eq:Kernel}, is well defined and depends continuously of the metric $g$ in $C^k$-topology, if $k$ is large enough. Thus the normal operator $\mathcal{N}_L$ depends continuously of the metric $g$ and moreover the same holds also for the functions $G^{(m)}$ in Lemma \ref{Le:Gs}. Thus the claim of Proposition \ref{Pr:normal_is_pseudo} is unchanged if $g$ is regular enough and the formula \eqref{eq:Msymbol_2} shows that also the principal symbol $\sigma(\mathcal{N}_L)$ depends continuously about $g$ in $C^k$-topology.

% \textbf{[We do the first extension, since the theory of $\Psi$DOs is the easiest on open domains of $\R^3$. However we will use the solution operator for an elliptic boundary value problem, similar to \eqref{eqnu1}, thus we also need the second extension $(M_2,g)$.][JZ: need to rewrite.]}
\color{black}

\medskip
%\color{red}
We note that the volume form $\mathrm{d}V_g$ depends on the metric tensor $g$. However  for any $t\in \R$ there exists $k \in \N$ such that if we fix a simple $C^k$-metric $g_0$ of $M$, then there exists a $C^k$-neighborhood $U$ of $g_0$, consisting of simple metrics, on which we can choose a uniform bi-Lipschitz constant for the norms $\|\cdot\|_{H^t_{g_{i}}}, \: i \in \{1,2\}$ for any $g_1, g_2 \in U$. Recall that we are working with $1+1$-tensors that are in the kernel $S\tau'_M\otimes^\mathcal{B} S\tau'_M$ of the trace operator $\mu$ related to a metric tensor $g$. %This space depends on $g$ and actually for any $f \in S\tau'_M\otimes^\mathcal{B} S\tau'_M$ we can construct a metric $\widetilde g$, that is arbitrarily close to $g$ with respect to $L^\infty$-norm and $\widetilde g^{ij}f_{ij}\neq 0$. Therfore $C^k$-norms \textbf{DO NOT} ``stabilize" $S\tau'_M\otimes^\mathcal{B} S\tau'_M$. This is an issue as we need a following type of lemma to prove the generic s-injectivity:
To avoid this issue, we introduce operator
\[
\kappa_g^\sharp\colon  S\tau'_M\otimes^{\mathcal{B}} S\tau'_M\rightarrow S\tau_M\otimes^{\mathcal{B}} S\tau'_M,
\]
such that
\[
(\kappa_g^\sharp f)^i_j=g^{ai}f_{aj}.
\]
Then the following subspace of $1$-covariant $1$-contravariant tensor fields 
\[
S\tau_M\otimes^\mathcal{B} S\tau'_M=\{f\in S\tau_M\otimes S\tau'_M, f^i_i=0\}
\]
coincides with the image of $\kappa_g^\sharp$, but is defined independent of the metric $g$. We let $\kappa_g^\flat$ be the inverse of $\kappa_g^\sharp$. The continuity of the maps $\kappa_g^\flat, \kappa_g^\sharp$ with respect to metric $g$ is evident.

\begin{lemma}
Let  $g_0 \in C^1(M)$. There exists a neighborhood $U \subset C^1(M)$ of $g_0$, such that the map
\[
\kappa_g^\sharp\circ(\LaB)_{M,g}^{-1}\circ\kappa_g^\flat\colon  H^{-1}(S\tau_M)\rightarrow H^{1}_0(S\tau_M),
\]
where $(\LaB)_{M,g}^{-1}$ is the solution operator of \eqref{eqnu1}, with vanishing boundary value, and the projections
 \[
 \kappa_g^\sharp\circ\mathcal{P}_M\circ\kappa_g^\flat,\: \kappa_g^\sharp\circ\mathcal{S}_M\circ\kappa_g^\flat\colon  L^2(S\tau_M\otimes^\mathcal{B} S\tau'_M)\rightarrow  L^2(S\tau_M\otimes^\mathcal{B} S\tau'_M)
 \] 
 depend continuously on $g \in U$.
\end{lemma}
%\begin{remark}
%In the above lemma, the definitions of the norms $L^2$, $H^1$ and $H^{-1}$ depend on the metric $g$. But the norms are equivalent for any metrics, and thus induce the same topology. Therefore we can discuss the continuity of all the maps.
%\end{remark}
\begin{proof}
We consider first a smooth  metric $g_0$, and note that Lemma \ref{existence2} implies that the solution operator $(\LaB)_{M,g_0}^{-1}\colon  H^{-1}(S\tau'_M)\rightarrow H^{1}_0(S\tau'_M)$ is bounded. Let $\epsilon>0$ and $g$ be any smooth metric such that $\|g_0-g\|_{C^1}<\epsilon$. We write
\begin{equation}\label{difference_lap}
\begin{split}
&\kappa_g^\sharp\circ(\LaB)_{M,g}^{-1}\circ\kappa_g^\flat-\kappa_{g_0}^\sharp\circ(\LaB)_{M,g_0}^{-1}\circ\kappa_{g_0}^\flat\\
=&\left(\kappa_g^\sharp\circ(\LaB)_{M,g}^{-1}\circ\kappa_g^\flat\right)\left(\kappa_{g_0}^\sharp\circ\LaB_{M,g_0}\circ\kappa_{g_0}^\flat-\kappa_g^\sharp\circ\LaB_{M,g}\circ\kappa_g^\flat\right)\left(\kappa_{g_0}^\sharp\circ(\LaB)_{M,g_0}^{-1}\circ\kappa_{g_0}^\flat\right),
\end{split}
\end{equation}
and choose $u,v\in H^1_0(S\tau_M)$. 
%with respect to $g_0$. Then $u,v\in H^1_0(S\tau_M)$ also with respect to $g$. 
Then by Cauchy-Schwarz inequality we obtain
\begin{equation}
\label{eq:action_of_Laplacian_diff}
\begin{split}
\left|\Big\langle (\kappa_{g_0}^\sharp\circ\LaB_{M,g_0}\circ\kappa_{g_0}^\flat-\kappa_g^\sharp\circ\LaB_{M,g}\circ\kappa_g^\flat) u,v\Big\rangle\right|
%\\
=&\left|\langle\mathrm{d}_{g_0}^\mathcal{B}\kappa_{g_0}^\flat u,\mathrm{d}^\mathcal{B}_{g_0}\kappa_{g_0}^\flat v\rangle_{L^2_{g_0}}\!-\!\langle\mathrm{d}_{g}^\mathcal{B}\kappa_{g}^\flat u,\mathrm{d}^\mathcal{B}_{g}\kappa_{g}^\flat v\rangle_{L^2_{g}}\right|
%\!+\!
%\left|\langle\kappa_g^\sharp\circ\LaB_{M,g}\circ\kappa_g^\flat) u,v\rangle_{L^2_{g_0}}\!-\! \langle\kappa_g^\sharp\circ\LaB_{M,g}\circ\kappa_g^\flat) u,v\rangle_{L^2_{g}}\right|
\\
\leq &\:  C\|g-g_0\|_{C^1}(\|g\|_{C^1}+\|g_0\|_{C^1})\|u\|_{H^1_{g_0}}\|v\|_{H^1_{g_0}}.
%\\
%\leq &C\|g-g_0\|_{C^k}(\|g\|_{C^1}+\|g_0\|_{C^k})\|u\|_{H^1_{g_0}}\|v\|_{H^1_{g_0}}+C\|g-g_0\|_{C^1}\|\LaB_{M,g}(\kappa_g^\flat u)\|_{H^{-1}_{g_0}}\|v\|_{H^1_{g_0}}
%\\
%\leq &C\|g-g_0\|_{C^1}(\|g\|_{C^k}+\|g_0\|_{C^1})\|u\|_{H^1_{g_0}}\|v\|_{H^1_{g_0}}+C\|g-g_0\|_{C^1}\|u\|_{H^1_{g_0}}\|v\|_{H^1_{g_0}}.
\end{split}
\end{equation}
In the last inequality we used the uniform Lipschitz equivalence of the $H^1$-norms.

If $\varepsilon\ll 1$, then $\|g-g_0\|_{C^1}(\|g\|_{C^1}+\|g_0\|_{C^1})\leq C\varepsilon$, for some $C>0$, which can be chosen uniformly when ever $g$ is close enough to $g_0$. Consequently \eqref{eq:action_of_Laplacian_diff} implies
\[
\|\kappa_{g_0}^\sharp\circ\LaB_{M,g_0}\circ\kappa_{g_0}^\flat-\kappa_g^\sharp\circ\LaB_{M,g}\circ\kappa_g^\flat\|_{H^1\rightarrow H^{-1}}\leq C\varepsilon,
\]
from which after using \eqref{difference_lap} it follows that
\[
\|(\LaB)_{M,g}^{-1}\|_{H^{-1}\rightarrow H^1}\leq \|(\LaB)_{M,g_0}^{-1}\|_{H^{-1}\rightarrow H^1}\left(1+\varepsilon \|(\LaB)_{M,g}^{-1}\|_{H^{-1}\rightarrow H^1}\right),
\]
and moreover
\[
\|(\LaB)_{M,g}^{-1}\|_{H^{-1}\rightarrow H^1}\leq \left(1-C\varepsilon \|(\LaB)_{M,g_0}^{-1}\|_{H^{-1}\rightarrow H^1}\right)^{-1}\|(\LaB)_{M,g_0}^{-1}\|_{H^{-1}\rightarrow H^1}.
\]
Finally we use \eqref{difference_lap} again, to conclude
\begin{equation}
\label{eq:bounded_inverse_of_lap}
\|\kappa_g^\sharp\circ(\LaB)_{M,g}^{-1}\circ\kappa_g^\flat-\kappa_{g_0}^\sharp\circ(\LaB)_{M,g_0}^{-1}\circ\kappa_{g_0}^\flat\|_{H^{-1}\rightarrow H^1}\leq C_0\|g-g_0\|_{C^1},
\end{equation}
where  $C_0>0$  can be chosen uniformly in some small $C^1$-neighborhood $U$ of $g_0$. Since smooth metrics are dense in $C^1(M)$, the Lipschitz estimate \eqref{eq:bounded_inverse_of_lap} implies that we can define the solution operator $(\LaB)_{M,g}^{-1}$ as bounded map $H^{-1}(S\tau'_M)\rightarrow H^{1}_0(S\tau'_M)$ for $g \in C^1(M)$ with the estimate \eqref{eq:bounded_inverse_of_lap} still valid. Thus the first claim of the Lemma is proven. We note that the continuity of operators $ \kappa_g^\sharp\circ\mathcal{P}\circ\kappa_g^\flat,\, \kappa_g^\sharp\circ\mathcal{S}\circ\kappa_g^\flat$ follows from this. 
\end{proof}
%\color{blue}

The previous lemma has the following straightforward generalization. 
\begin{corollary}
\label{Co:continuity_of_the_ops}
Let $t>0$. There exists $k_0 \in \N$, such that for any $k\geq k_0$ and  $g_0 \in C^k(M)$ there exists a neighborhood $U \subset C^k(M)$ of $g_0$, such that the solution operator
\[
\kappa_g^\sharp\circ(\LaB)_{M,g}^{-1}\circ\kappa_g^\flat\colon  H^{t-1}(S\tau_M)\rightarrow (H^{t+1}\cap H^1_0)(S\tau_M),
\]
and the projections
 \[
 \kappa_g^\sharp\circ\mathcal{P}_M\circ\kappa_g^\flat,\: \kappa_g^\sharp\circ\mathcal{S}_M\circ\kappa_g^\flat\colon  H^t(S\tau_M\otimes^\mathcal{B} S\tau'_M)\rightarrow  H^t(S\tau_M\otimes^\mathcal{B} S\tau'_M)
 \] depend continuously on $g \in U$.
\end{corollary}
%\begin{proof}
%\textbf{(TS)[I'll need to think through the proof and see if something should be mentioned about it.]}
%\end{proof}
\color{black}

%\begin{lemma}
%There exists $k_0 \in \N$ such that for any $k\geq k_0$ and for any simple $g_0 \in C^k(M)$, there exists a neighborhood $U\subset C^k(M)$ of $g_0$ that consist of simple Riemannian metircs and $C(g_0)>0$ such that
%\[
%\|\mathcal R_g-\mathcal R_{g_0}\|_{X\to Y}<C(g_0)\|g-g_0\|_{C^k}, \quad g \in U,
%\] 
%where $X=L^2(\p M)$, $Y = H^1(M)$ and $\mathcal R_g$ is the solution operator of the boundary value problem 
%\[
%\LaB\!\!_g u=0, \hbox{ in } M, \quad \hbox{Tr}u=w.
%\]
%\end{lemma}
\subsection{Reconstruction formula}
\label{Se:Reconstruction formula}
%\color{blue}
We fix a simple metric $g_0\in C^k(M)$, where $k \in \N$  and consider a simple metric $g \in C^k(M)$ in a small neighborhood of $g_0$ with respect to $C^k$-topology. We recall that any tensor field defined in $M$ is extended to any larger domain with the zero extension. Moreover we note that by conjugating all the operators, that are to be used in this section, with $ \kappa_g^\sharp$ from left and $\kappa_g^\flat$ from right, we can work in the space of  trace-free tensor fields $S\tau_M\otimes^\mathcal{B} S\tau'_M$, that is invariant of any metric structure. To reduce the notations we omit the conjugation from here onwards. 

Let $|D|_g$ be a $\Psi$DO with the full symbol $|\xi|_g$. We begin by constructing a parametrix for the $\Psi$DO 
\begin{equation}
\label{eq:operator_M}
%\mathcal{M}=
%\left(\begin{array}{c}
%|D|_g\mathcal{N}_L
%\\
%\mathcal{P}_{M_2}
%\end{array}\right)
%, \quad
\mathcal{M}f=
\left(\begin{array}{c}
|D|_g\mathcal{N}_Lf
\\
\mathcal{P}_{M_2}f
\end{array}\right), \quad \hbox{ in $M_1$.}
\end{equation}
%$\chi$ is a smooth cut off function that equals to one in $M_1$ and vanishes outside $M_2$, 
We note that the right hand side of \eqref{eq:operator_M} extends $\mathcal{M}$ in  $M_2$, 
%\color{blue}
and due to Corollary \ref{Co:continuity_of_the_ops}, the operator $\mathcal M$ depends continuously on $g$.
%with respect to the boundary value problem
%\[
%\LaB v=h, \: \hbox{ in } M_2, \quad v|_{\p M_2}=0.
%\]
%We will extend $g$ near $\partial M_2$, such that $\mathcal{N}_L$ and $\mathcal{M}$ are pseudodifferential operators inside $M_2$. \textbf{Could this be done by introducing smooth cut of function $\chi$ that equals one in $M_1$ and vanishes outside $M_2$?}
%\begin{remark}
%We note here that since $\mathcal{N}_L$ is not properly supported, strictly speaking we would need to substitute $\mathcal{N}_L$ in \eqref{eq:operator_M} with properly supported $\Psi$DO, $A=\mathcal{N}_L+\mathcal{K}_0$, where $\mathcal{K}_0$ is a smoothing operator in $M_2$. As we aim to find a parametrix for $\mathcal M$, this substitution is unnecssary.
%\end{remark}
In the following we use the notation $\circ$ for a finite asymptotic expansion for the symbol of a product of two $\Psi$DOs.
\color{black}
The principal symbol $\sigma(\mathcal{M})$ of $\mathcal{M}$ is given by
\[
\sigma(\mathcal{M})=
\left(\begin{array}{c}
|\xi|_g \circ\sigma(\mathcal{N}_L)\\
\sigma(\mathcal{P}_{M_2})\\
\end{array}\right).
\]
We show now that $\sigma(\mathcal{M})$ is elliptic near $M_1$. To do so we lower the $(ij)$-indices in \eqref{eq:Msymbol_2} to get:
\[
\begin{split}
\sigma(\mathcal{N}_L)^{kl}_{ij}=CN_{ij}^{kl}&=
C\bigg(-12|\xi|_g^{-1}\delta_i^{k}\delta^{l}_j+12|\xi|_g^{-2}\delta_i^{k}\xi_j\xi^l)\\
&+9|\xi|_g^{-5}\xi_i\xi_j\xi^k\xi^l +3|\xi|_g^{-1}(\delta^{k}_j\delta^{l}_i+\delta^{k}_i\delta^{l}_j 
%+g^{kl}g_{ij} 
)
\\
&-3|\xi|_g^{-3}(g_{ij}\xi^k\xi^l+\delta^{k}_i\xi_j\xi^l+\delta^{l}_i\xi_j\xi^k+
\delta^{k}_j\xi_i\xi^l+\delta^{l}_j\xi_i\xi^k  
%+g^{kl}\xi_i\xi_j 
)\bigg).
\end{split}
\]
\begin{remark}
\label{Re:parametrix_symbol}
A straightforward calculation shows that 
\[
g^{ij}N_{ij}^{kl}=0,  \quad \xi^{j}N_{ij}^{kl}=0, \quad \xi^{i}N_{ij}^{kl}=12|\xi|_g^{-1}\left(|\xi|_g^{-2}\xi^k\xi_j\xi^l-\xi^k\delta^{l}_j\right)
\]
and 
\[
N^{kl}_{ji}-N^{kl}_{ij}=12\left(|\xi|_g^{-1}(\delta^{k}_i\delta^{l}_j-\delta^{k}_j\delta^{l}_i)-|\xi|_g^{-3}(\delta^{k}_i\xi_j\xi^l-\delta^{k}_j\xi_i\xi^l)\right).
\]
Therefore
\[
\begin{split}
&|\xi|_g\left[N_{ij}^{kl}+\frac{1}{4}\left(N_{ji}^{kl}-N_{ij}^{kl}-|\xi|_g^{-2}\xi_j\xi^{j'}N_{j'i}^{kl}-|\xi|_g^{-2}\xi_i\xi^{i'}N_{i'j}^{kl}\right)\right]
\\
=&6|\xi|_g^{-2}\delta_i^{k}\xi_j\xi^l -6\delta_i^{k}\delta^{l}_j-3|\xi|_g^{-2}g_{ij}\xi^k\xi^l+3|\xi|_g^{-4}\xi_i\xi_j\xi^k\xi^l.
\\
%=&6\delta^{k}_i|\xi|^{-2}\xi_j\xi^l-6\delta^{k}_i\delta^{l}_j-3|\xi|^{-2}g_{ij}\xi^k\xi^l+3|\xi|^{-4}\xi_i\xi_j\xi^k\xi^l.
\end{split}
\]
\end{remark}
\noindent Motivated by Remark \ref{Re:parametrix_symbol} we define
\[
E_{ij}^{mn}=-\frac{1}{6C}\left(\delta_i^m\delta_j^n+\frac{1}{4}\left(\delta_i^n\delta_j^m-\delta_i^m\delta_j^n-|\xi|_g^{-2}\xi_j\xi^m\delta_i^n-|\xi|_g^{-2}\xi_i\xi^m\delta_j^n\right)\right).
\]
Therefore
\begin{equation}\label{Econstruction}
E_{ij}^{mn}|\xi|_g\sigma(\mathcal{N}_L)_{mn}^{kl} =\delta_i^k\delta_j^l-\delta_i^k|\xi|_g^{-2}\xi_j\xi^l+\frac{1}{2}|\xi|_g^{-2}g_{ij}\xi^k\xi^l-\frac{1}{2}|\xi|_g^{-4}\xi_i\xi_j\xi^k\xi^l.
\end{equation}

Thus we need to characterize the remainder term in \eqref{Econstruction}. Since $-\LaB$ is elliptic near  $M_1$,
with principal symbol 
\[
-\sigma(\Delta^\mathcal{B})
%=j^\mathcal{B}_\xi i^\mathcal{B}_\xi
%=\xi^j(\xi_j\delta_\ell^i-\frac{1}{3}(\lambda\mu)\xi_\ell)
=|\xi|_g^2\delta_\ell^i-\frac{1}{3}\xi_\ell\xi^i,
\]
it has a parametrix $(\LaB)^{-1}$ with principal symbol
\[
-\sigma((\LaB)^{-1})_{i}^k=\frac{1}{2}|\xi|_g^{-4}\xi_i\xi^k+|\xi|_g^{-2}\delta_i^k.
\]
We note that near $M_1$ the parametrix $(\LaB)^{-1}$ and the solution operator $(\Delta^\mathcal{B}_{M_2})^{-1}$ of \eqref{eqnu1}, with zero boundary value, coincide 
%\color{blue}
up to a finitely smoothing operator.
\color{black}
%
%\color{blue}
%up to an operator of order $-m$. 
%\color{black}
This implies
\[
\mathrm{i}\sigma((\Delta^\mathcal{B}_{M_2})^{-1}\delta^\mathcal{B})_{i}^{kl}
=
\frac{1}{2}|\xi|_g^{-4}\xi_i\xi^k\xi^l+\delta_i^k|\xi|_g^{-2}\xi^l,
\]
and
\[
\sigma(\mathcal{P}_{M_2})=\sigma(\mathrm{d}^\mathcal{B}(\Delta^\mathcal{B}_{M_2})^{-1}\delta^\mathcal{B})_{ij}^{kl}
%=&\frac{1}{2}|\xi|_g^{-4}\xi_i\xi_j\xi^k\xi^\ell+\delta_i^k|\xi|_g^{-2}\xi_j\xi^l-\frac{1}{6}g_{ij}|\xi|_g^{-2}\xi^k\xi^l-\frac{1}{3}g_{ij}|\xi|_g^{-2}\xi^k\xi^\ell\\
=\frac{1}{2}|\xi|_g^{-4}\xi_i\xi_j\xi^k\xi^\ell+\delta_i^k|\xi|_g^{-2}\xi_j\xi^l-\frac{1}{2}g_{ij}|\xi|_g^{-2}\xi^k\xi^l.
\]
Therefore
\[
\left(\begin{array}{cc}E&\mathrm{Id}\end{array}\right)\circ\left(\begin{array}{c}|\xi|_g\circ\sigma(\mathcal{N}_L)\\
\sigma(\mathcal{P}_{M_2})
\end{array}\right)=\mathrm{Id},
\]
and we have shown that $\mathcal{M}$ is elliptic near  $M_1$.

\medskip
From now on we study the mapping properties of $\mathcal{M}$. We use the notation $\tilde \sigma$ for the full symbol of a $\Psi$DO and $S^m$ for the space of $C^k$-regular symbols $a(x,\xi)$ of order $m$. 
%\color{blue}
Let $m >0$ be given. We choose $k \in \N$ to be so large that there exists a  $\Psi$DO $\mathcal{A}$ of order $-2$, that is given by a finite asymptotic expansion of $(\LaB)^{-1}$ with homogeneous symbols in $\xi$-variable, and satisfies
\color{black} 
$\tilde \sigma( \mathcal A)\circ \tilde{\sigma}(\Delta^\mathcal{B})=\mathrm{Id}~\text{mod}~ S^{-m}$ near $M_1$. 
%\color{blue}
From here onwards we increase the regularity of the  $C^k$-smooth metric $g$ whenever needed without further mention. 
\color{black}
We get 
\[
\begin{split}
\tilde{\sigma}(\mathcal{P}_{M_2})&=\tilde{\sigma}(\mathrm{d}^\mathcal{B})\circ \tilde \sigma( \mathcal A)\circ \tilde{\sigma}(\delta^\mathcal{B})\quad\text{mod}~S^{-m},\\
\tilde{\sigma}(\mathcal{S}_{M_2})&=\mathrm{Id}-\tilde{\sigma}(\mathrm{d}^\mathcal{B})\circ \tilde \sigma( \mathcal A) \circ \tilde{\sigma}(\delta^\mathcal{B})\quad \text{mod}~ S^{-m}.
\end{split}
\]
%\color{blue}
Since $\LaB$ and  $(\LaB)^{-1}$ depend continuously on $g$ we can also choose $\mathcal A$ that depends continuously on $g$, with respect to   $C^k$-topology, if $g$ is near to $g_0$.
\color{black}
Since $\mathcal{M}$ is elliptic near $M_1$ there exists a pseudo-differential operator $\mathcal{L}=(\mathcal{L}_1,\mathcal{L}_2)$ of order $0$, with principle symbol $(E,\,\mathrm{Id})$, such that
\begin{equation}
\label{eq:inverse_of_order_m}
(\tilde{\sigma}(\mathcal{L}_1),\,\tilde{\sigma}(\mathcal{L}_2))\circ\left(\begin{array}{c}|\xi|_g \circ \tilde \sigma(\mathcal{N}_L)\\
\tilde{\sigma}(\mathcal{P}_{M_2})
%\tilde{\sigma}(\mathrm{d}^\mathcal{B})\circ A\circ \tilde{\sigma}(\delta^\mathcal{B})
\end{array}\right)=\tilde{\sigma}(\mathcal{L}_1)\circ|\xi|_g\circ \tilde{\sigma}(\mathcal{N}_L)+\tilde{\sigma}(\mathcal{L}_2)\circ\tilde{\sigma}(\mathcal{P}_{M_2})=\mathrm{Id}\quad\text{mod}~S^{-m},
\end{equation}
near $M_1$. We set two operators
\[
\Lambda=\mathrm{Id}-\dB \mathcal{A}  \deB, \quad \hbox{ and } \quad \mathcal{T}_1=\Lambda \mathcal L_1 |D|_g\Lambda, \quad \hbox{ in } M_2, 
\]
and note
\[
\mathcal{S}_{M_2}\mathcal{N}_L=\mathcal{N}_L\mathcal{S}_{M_2}=\mathcal{N}_L,\quad \mathcal{S}_{M_2}\mathcal{P}_{M_2}=\mathcal{P}_{M_2}\mathcal{S}_{M_2}=0. 
\]
Then we apply $\tilde \sigma(\Lambda)$ from right and left to \eqref{eq:inverse_of_order_m} to obtain
\[
\begin{split}
\tilde \sigma(\Lambda)=&\tilde\sigma(\Lambda)^2=\tilde\sigma(\Lambda)\circ\tilde{\sigma}(\mathcal{L}_1)\circ|\xi|_g\circ \tilde{\sigma}(\mathcal{N}_L)\circ \tilde\sigma(\Lambda)=\tilde\sigma(\Lambda)\circ\tilde{\sigma}(\mathcal{L}_1)\circ|\xi|_g\circ \tilde\sigma(\Lambda)\circ \tilde{\sigma}(\mathcal{N}_L)
\\
=&\tilde \sigma (\mathcal T_1)\circ \tilde{\sigma}(\mathcal{N}_L)\quad\text{mod}~S^{-m},
\end{split}
\]

%\color{blue}
We choose  $t>0$ and note that 
\color{black}
we have shown the existence of a bounded operator $\mathcal{K}_1\colon L^2(S\tau'_M\otimes^\mathcal{B} S\tau'_M)\to H^{t}(S\tau'_{M_1}\otimes^\mathcal{B} S\tau'_{M_1})$,  which satisfies
\begin{equation}
\label{eq:first_recons_formula}
 \mathcal{T}_1 \mathcal{N}_Lf=f-\mathrm{d}^\mathcal{B}\mathcal{A}\delta^\mathcal{B}f+\mathcal{K}_1f=f^s_{M_1}-\dB w+\mathcal{K}_1f
%,\quad\text{in}~M_1
, \quad \hbox{ in } M_1 \quad f \in L^2(S\tau'_M\otimes^\mathcal{B} S\tau'_M),
\end{equation}
%\color{blue}
where $w:= \left((\Delta^\mathcal{B}_{M_1})^{-1}-\mathcal{A}\right)\deB f$. Since $\mathcal{T}_1,\: \mathcal{N}_L$ and $\mathcal A$ depend continuously on $g$, the formula $\eqref{eq:first_recons_formula}$ implies that also $\mathcal K_1$ depends continuously on $g$.

We conclude this subsection by finding a reconstruction formula for the solenoidal part $f^s_{M_1}$ modulo $H^t$-regular fields in $M_1$. To do this we show that the linear map $ L^2(S\tau'_M\otimes^\mathcal{B} S\tau'_M)\ni f \mapsto  \dB w \in H^{t+1}(S\tau'_M\otimes^\mathcal{B} S\tau'_M)$ is bounded and depends continuously in $g$. 
% In the light of \eqref{eq:first_recons_formula} we only need to compare the operators $\mathcal A$ and the solution operator $(\Delta^\mathcal{B}_{M_1})^{-1}$ of the boundary value problem
%\[
%\LaB u=f, \: \hbox{in } M_1, \quad u|_{\p M_1}=h,
%\] 
%and show that their difference is an operator of an order $t+1$. 
First we note  that the map $ L^2(S\tau'_M\otimes^\mathcal{B} S\tau'_M)\ni f \mapsto  w \in H^1(S\tau'_{M_1}\otimes^\mathcal{B} S\tau'_{M_1})$ is bounded since $\mathcal{A}$ is an operator of order $-2$. Since $f$ vanishes outside $M$ we have
%, after increasing the smoothness of $g$ once more, 
due to the finite pseudo-local property that the distribution $-\mathcal{A}\deB f$ near $\p M_1$ is of regularity $t+2$. Thus the map
\[
L^2(S\tau'_{\p M}\otimes^\mathcal{B} S\tau'_{\p M}) \ni f \to w|_{\p M_1}=-\mathcal{A}\deB f|_{\p M_1}\in H^{t+\frac{3}{2}}(S\tau'_{\p M_1}\otimes^\mathcal{B} S\tau'_{\p M_1})
\]
is bounded and depends continuously about the metric $g$ in $C^k$-topology. 
%
%$w\vert_{\partial M_1}=\mathcal{A}\deB f\vert_{\partial M_1}$ depends continuously from $f$ and is smooth due to pseudolocal property of $\Psi$DO's. 
This implies that $w$ solves the boundary value problem:
\[
\LaB w=\left(\mathrm{Id}-\LaB\mathcal{A}\right) \deB f,\: \hbox{in } M_1,\quad w\vert_{\partial M_1}=-\mathcal{A}\delta f\vert_{\partial M_1}.
\]
As 
%\[
%\tilde \sigma(\LaB) \circ \tilde \sigma(\mathcal{A})=Id, \quad \text{mod}~S^{-m} 
%\] 
the symbol of the $\Psi$DO $\left(\mathrm{Id}-\LaB\mathcal{A}\right) \deB$ is in $S^{-m}$ for $m>0$ and $f$ is $L^2$-regular we have that $\left(\mathrm{Id}-\LaB\mathcal{A}\right) \deB f \in H^{t}(S\tau'_{M_1}\otimes^\mathcal{B} S\tau'_{M_1})$. %for $g$ smooth enough. 
The Corollary \ref{Co:continuity_of_the_ops} implies that the map $L^2(S\tau'_M\otimes^\mathcal{B} S\tau'_M) \ni f\mapsto w\in H^{t+2}(S\tau'_{M_1}\otimes^\mathcal{B} S\tau'_{M_1})$ is bounded. 

Therefore we have verified that the map 
\[
L^2(S\tau'_M\otimes^\mathcal{B} S\tau'_M) \ni f\mapsto \dB w\in H^{t+1}(S\tau'_{M_1}\otimes^\mathcal{B} S\tau'_{M_1})
\] 
is bounded and depends continuously on $g$ with respect to $C^k$-topology if $k\in \N$ is large enough and $g$ is close to $g_0$. After setting
\[
\mathcal K_2f:=-\dB w+\mathcal K_1f, \quad f \in L^2(S\tau'_M\otimes^\mathcal{B} S\tau'_M),
\]
the equation \eqref{eq:first_recons_formula} implies the main result of this subsection:
\color{black}
\begin{proposition}
\label{Pr:1_reconst_result}
%\color{blue}
Let $t>0$. 
\color{black}
There exists $k_0\in \N$ such that for any simple metric $g\in C^k(M), \: k\geq k_0$ there exists a first order $\Psi$DO, $\mathcal{T}_1$ in $M_2$ and a bounded operator 
\[
\mathcal{K}_2\colon L^2(S\tau'_M\otimes^\mathcal{B} S\tau'_M)\to H^{t}(S\tau'_{M_1}\otimes^\mathcal{B} S\tau'_{M_1}), 
\]
such that the first reconstruction formula for the solenoidal part is valid:
\begin{equation}\label{seq1}
\mathcal{T}_1\mathcal{N}_Lf=f^s_{M_1}+\mathcal{K}_2f\quad\text{in}~M_1, \quad \hbox{ for } f \in L^2(S\tau'_M\otimes^\mathcal{B} S\tau'_M).
\end{equation}
Moreover if we fix a simple metric $g_0\in C^k(M)$, the operators $\mathcal T_1$ and $\mathcal K_2$ depend continuously about $g$ in some neighborhood of $g_0$  with respect to $C^k$-topology. 

If $g$ is infinitely smooth, there exists a smoothing operator  
\[
\widetilde{ \mathcal{K}}_2\colon L^2(S\tau'_M\otimes^\mathcal{B} S\tau'_M)\to C^\infty(S\tau'_{M_1}\otimes^\mathcal{B} S\tau'_{M_1}),
\]
such that
\[
\mathcal{T}_1\mathcal{N}_Lf=f^s_{M_1}+\widetilde{\mathcal{K}}_2f\quad\text{in}~M_1,
\]
for any $f \in L^2(S\tau'_M\otimes^\mathcal{B} S\tau'_M)$.
\end{proposition}
%\begin{remark}
%By the previous result we have constructed a parametrix of $\mathcal{N}_L$ in 
%\[
%L^2(S\tau'_M\otimes^\mathcal{B} S\tau'_M)\cap \mathcal S(S\tau'_{M_1}\otimes^\mathcal{B} S\tau'_{M_1}).
%\]
%\textbf{[The definition?]}
%\end{remark}

\subsection{Stability estimates for the normal operator} In the previous section we found a reconstruction formula for the solenoidal part $f^s_{M_1}$ with respect to the extended domain $M_1$. In this section we prove a stability estimate for the normal operator and find a reconstruction formula for the solenoidal part $f^s_{M}$. However as it turns out we need higher regularity for $f$ to do so. 

%For now we assume that $f\in C^\infty_0(S\tau'_{M}\otimes^\mathcal{B} S\tau'_{M})$. 
%\color{blue}
Let $g \in C^k(M)$ be a simple metric and $f \in L^2(S\tau'_M\otimes^\mathcal{B} S\tau'_M)$. 
%%First we need to carefully compare $f^s_{M_1}$ and $f^s_{M}$. 
We write $f=f^s_{M_1}+\dB v_{M_1}$, where $v_{M_1}$ solves a boundary value problem  \eqref{eq_Blaplacian}, on $M_1$.
\color{black}
Since $f=0$ on $M_1\setminus M$, the finite pseudo-local property of $(\Delta^\mathcal{B}_{M_1})^{-1}\deB$ yields $v_{M_1}\in C^{1}(M_1\setminus M)$. 
%\textbf{(TS)[This needs to be verified]}
Moreover we have by (\ref{seq1}) that
\begin{equation}
\label{eq:potential_as_smoothing_op}
-\mathrm{d}^\mathcal{B}v_{M_1}=\mathcal{T}_1\mathcal{N}_Lf-\mathcal{K}_2f\quad\text{in}~M_1\setminus M.
\end{equation}

In the following we will find a $L^2$-estimate for $v_{M_1}$ on $M_1 \setminus M$. Let $x_0\in \p M$. Then for any $x\in M_1\setminus M$ in a small neighborhood $U$ of $x_0$, we choose a unit vector $\xi$ such that the geodesic $\gamma(t)=\gamma_{x,\xi}(t)$ in $M_1\setminus M$ issued from $x$ meets $\partial M_1$ before it meets $\partial M$. We use the notation $\tau=\tau(x,\xi)>0$ for the time this geodesic hits $\p M_1$. Since $v_{M_1}$ vanishes at $\p M_1$ we have as in the proof of Lemma \ref{existence2} that
\begin{equation}\label{integral1}
[v_{M_1}(x)]_i\eta^i=-\int_0^{\tau}[\mathrm{d}^\mathcal{B}v_{M_1}(\gamma(t))]_{ij}\eta^i(t)\dot{\gamma}^j(t)\mathrm{d}t,
\end{equation}
where $\eta(t)$ is a unit length vector field, parallel along $\gamma$ and $\eta(0)=\eta$ is perpendicular to $\xi$. 
%Here we used
%\begin{equation}\label{dd}
%\frac{\mathrm{d}}{\mathrm{d}t}[v_{i}(\gamma(t))\eta^i(t)]=\left(\frac{D v}{\mathrm{d}t}\right)_{i}\eta^i(t)=v_{i;j}(\gamma(t))\eta^i(t)\dot{\gamma}^j(t),
%\end{equation}
%and $g_{ij}\eta^i(t)\dot{\gamma}^j(t)=0$. 
The substitution \eqref{eq:potential_as_smoothing_op} and the continuity of the integrand give
\[
\left|[v_{M_1}(x)]_i\eta^i\right| \leq \int_0^{\tau(x,\xi)}\left|[\mathcal{T}_1\mathcal{N}_Lf-\mathcal{K}_2f]_{ij}\eta^i(t)\dot{\gamma}^j(t)\right| \mathrm{d}t  \leq C \left|(\mathcal{T}_1\mathcal{N}_Lf-\mathcal{K}_2f)(x)\right|_{g},
\]
where $C$ depends only from the distance to $\p M_1$. Perturbing the initial direction $\xi$, we see that the inequality above holds for linearly independent $\{\eta_{(k)}\}_{k=1}^3$. As $|v_{M_1}(x)|_g^2$ can be estimated by $C\sum_{k=1}^3\left|[v_{M_1}(x)]_i\eta^i_{(k)}\right|^2$,  where the constant $C$ is uniform in a neighborhood of $x_0$, we get
\[
\|v_{M_1}\|_{L^2(M_1\setminus M)}\leq C\|\mathcal{T}_1\mathcal{N}_Lf-\mathcal{K}_2f\|_{L^2(M_1\setminus M)},
\]
first in $U$ and then globally by using a finite covering for the pre-compact set $M_1 \setminus M$.

\medskip
Next we estimate the $H^1$-norm of $v_{M_1}$ in $M_1\setminus M$. As we can again estimate $|\nabla v_{M_1}|^2_g$ by
\[
C\sum_{k,\ell=1}^3\bigg|\alpha_{(k)}^j\nabla_j[v_{M_1}]_i\alpha_{(\ell)}^i\bigg| ^2, \quad \{\alpha_{(k)}\}_{k=1}^3 \hbox{  orthonormal },
\] 
it is enough to estimate $\alpha_{(k)}^j\nabla_j[v_{M_1}]_i\alpha_{(\ell)}^i$. Recall that
\[
\xi^j\nabla_j[v_{M_1}]_i\eta^i=\xi^j\left[\mathrm{d}^\mathcal{B}v_{M_1}\right]_{ij}\eta^i,
\]
for any $\eta,\xi$ perpendicular to each other. Then \eqref{eq:potential_as_smoothing_op} implies
\begin{equation}
\label{eq:mixed_estimate}
\left| \xi^j\nabla_j[v_{M_1}]_i\eta^i \right| \leq C \left| \mathcal{T}_1\mathcal{N}_Lf-\mathcal{K}_2f \right|_g,
\end{equation}
and it remains to estimate $\alpha_{(k)}^i\nabla_j[v_{M_1}]_i\alpha_{(k)}^j$. 

%\textbf{I don't see why we cannot just use the unit exterior normal instead of $\xi_0$? The proof does not seem to require perturbation of the initial direction. The original proof needed it}

We choose $x_0\in\partial M$, and local coordinates $x'$ on $\partial M$ near $x_0$. 
%For now we fix a unit length vectorfield $\xi_0$ which is close to the unit exterior normal field of $\partial M$ near $x_0$. 
Let $(x',x_3)$ be the boundary normal coordinates given in a neighborhood $U \subset M_1\setminus M$ of $x_0$. That is each point $(x',x_3)=x\in U$ is uniquely expressed as $x=\gamma_{(x',0),\nu}(t)$,  where $\nu$ is the exterior unit normal to $\p M$ and we have chosen $x_3=t$ as the third coordinate. We denote $\xi=\dot \gamma_{(x',0),\nu}(t)$. For $\eta \in T_xM_1$,  that is of unit length and perpendicular to $\xi $, the formula $(\ref{integral1})$, in the given coordinates, has the form
\begin{equation}
\label{eq:to_be_deri}
[v_{M_1}(x)]_i\eta^i=-\int_{x_3}^\infty [\mathrm{d}^\mathcal{B}v_{M_1}(\gamma(t))]_{i3}\eta^i(t)\mathrm{d}t.
\end{equation}
We can replace the exit time by $\infty$ in upper bound of integrationin \eqref{eq:to_be_deri} since $v_{M_1}$ has a line integrable zero extension outside $M_1$.

We denote the coordinate vector fields with respect to $x'$ variables as $\{X_{(k)}\}_{k=1}^2$. We note that these fields are orthogonal to the third coordinate frame $\frac{\mathrm{d}}{\mathrm{d}t}=\dot \gamma_{(x',0),\nu}(x_3)$ and $\eta$ can be given by a linear combination of $\{X_{(k)}\}_{k=1}^2$. We extend $\eta$ near $\gamma$ in such a way that $\nabla_{X_{(k)}}\eta=0$ at $\gamma(t)$. This can be done for instance with parallel transport using Fermi coordinates given by the coordinate frame  $\{\frac{\mathrm{d}}{\mathrm{d}t}, X_{(1)}, X_{(2)} \}$ along $\gamma$. Then we apply $X_{(k)}$ to both sides of \eqref{eq:to_be_deri}. Since $v_{M_1}\in H^1_0(M_1)$ we obtain
\begin{equation}
\label{eq:derivated_eq}
\begin{split}
X_{(k)}^j\nabla_j[v_{M_1}(x',x_3)]_i\eta^i
=&-\int_{x_3}^\infty X_{(k)}\left([\mathrm{d}^\mathcal{B}v_{M_1}(\gamma(t))]_{i3}\eta^i\right)\mathrm{d}t.
%\\
%+&X_{(k)}(x_3)[\mathrm{d}^\mathcal{B}v_{M_1}(\gamma(t))]_{i3}\eta^i(x_3)+X_{(k)}(\tau)[\mathrm{d}^\mathcal{B}v_{M_1}(\gamma(t))]_{i3}\eta^i(\tau)
\end{split}
\end{equation}

Let $\chi$ be a smooth cut-off function such that $\chi=1$ near $\partial M$ and $\chi=0$ near $\partial M_1$ and outside $M_1$. Then $\mathcal{K}_3\colon f\mapsto (1-\chi)X_{(k)}\mathrm{d}^\mathcal{B}v_{M_1}$ is finitely smoothing operator by the fact $f=0$ in $M_1\setminus M$ and the finite pseudo local property of the operator
$
X_{(k)}\dB (\LaB_{\!\!\!\!\!M_1})^{-1}\deB
$. 
Equations \eqref{eq:potential_as_smoothing_op} and \eqref{eq:derivated_eq}  imply
\[
X_{(k)}^j\nabla_j[v_{M_1}(x',x_3)]_i\eta^i
%=\int_{0}^\infty\left[\chi(\mathcal{T}_1\mathcal{N}_Lf+\mathcal{K}_2f)+\mathcal{K}_3\right]_{i3}\eta^i(t)\mathrm{d}t
=\int_{x_3}^\infty
%{\gamma_{(x',0),\xi_0}(x_3, \infty)}
\chi X_{(k)}\left([\mathcal{T}_1\mathcal{N}_Lf]_{i3}\eta^i\right)\;\mathrm{d}t+\mathcal{K}_4f,
\]
for some 
%\color{blue}
$\mathcal{K}_4\colon L^2(M_1)\rightarrow H^t(M_1)$, where $t>0$ is as in Proposition \ref{Pr:1_reconst_result}.  
\color{black}
%Since $\xi_0$ is close to the outer unit normal of $\p M$ it holds that $\eta$ is close to the subspace spanned by $\{X_{(1)}, X_{(2)}\}$. 
Therefore the continuity implies the existence of $C$, depending only on the distance to $\p M_1$, such that the following pointwise estimate holds:
\begin{equation}
\label{eq:estimate_for_derivative_of_the_elliptic_solution}
\left|\eta^j\nabla_j[v_{M_1}]_i\eta^i\right|\leq C\left( \sum_{k=1}^2|\chi\nabla_{X_{(k)}}(\mathcal{T}_1\mathcal{N}_Lf)|_g
%+|\chi\mathcal{T}_1\mathcal{N}_Lf|_g
+|\mathcal{K}_4f|_g\right).
\end{equation}

%
%\textbf{If we choose $X_{k}$ to be given by the coordinate vector fields we have to be more careful with this equation, as $\eta$ does not belong to the linear space they span! By standard Sobolev imbedding theorems we need to have $\nabla v_{M_1}\in H^2\Rightarrow f\in H^2$ to guarantee the continuity for $\nabla v_{M_1}$.}

It remains to estimate $\xi^j\nabla_j[v_{M_1}(x,x_3)]_i\xi^i$. We take $\tilde{\eta}$ such that $\{\eta,\tilde{\eta},\xi\}$ form an orthonormal parallel basis along $\gamma$. In this basis we can write
\[
\mu(\mathrm{d}'v_{M_1})=
%g^{ij}\nabla_j u_i=
\eta^j\nabla_j[v_{M_1}]_i\eta^i+\tilde{\eta}^j\nabla_j[v_{M_1}]_i\tilde{\eta}^i+\xi^j\nabla_j[v_{M_1}]_i\xi^i.
\]
Therefore we have
\[
\begin{split}
&\xi^j\left(-\mathrm{d}^\mathcal{B}v_{M_1}\right)\xi^i\\\
=&\xi^j\left(-\nabla_j[v_{M_1}]_i+\frac{1}{3}\mu(\mathrm{d}'v_{M_1})g_{ij}\right)\xi^i\\
=&-\frac{2}{3}\xi^j\nabla_j[v_{M_1}]_i\xi^i+\frac{1}{3}\left(\eta^j\nabla_j[v_{M_1}]_i\eta^i+\tilde{\eta}^j\nabla_j[v_{M_1}]_i\tilde{\eta}^i\right).
\end{split}
\]
%This implies a similar estimate to \eqref{eq:estimate_for_derivative_of_the_elliptic_solution} for $\xi^j\nabla_j[v_{M_1}(x,x_3)]_i\xi^i$.
%and
%\[
%\xi^j\nabla_j[v_{M_1}]_i\xi^i=\frac{3}{2}\xi^j\left(\mathrm{d}^\mathcal{B}v_{M_1}\right)\xi^i-\frac{1}{2}\left(\eta^j\nabla_j[v_{M_1}]_i\eta^i+\tilde{\eta}^j\nabla_j[v_{M_1}]_i\tilde{\eta}^i\right)
%\]and bBy \eqref{eq:estimate_for_derivative_of_the_elliptic_solution}, 
%we have the estimates
%\[
%\left|\eta^j\nabla_j[v_{M_1}(x',x_3)]_i\eta^i\right|\leq C\left(\sum_{k=1}^2\|\nabla_{X_{(k)}}\mathcal{T}_1\mathcal{N}_Lf\|_{L^2(M_1\setminus M)}+\|\mathcal{T}_1\mathcal{N}_Lf\|_{L^2(M_1\setminus M)}+\|\mathcal{K}_4f\|\right)
%\]
%and
%\[
%\left|\xi^j\nabla_j[v_{M_1}(x,x_3)]_i\xi^i\right|\leq C\left(\sum_{k=1}^2\|\nabla_{X_{(k)}}\mathcal{T}_1\mathcal{N}_Lf\|_{L^2(M_1\setminus M)}+\|\mathcal{T}_1\mathcal{N}_Lf\|_{L^2(M_1\setminus M)}+\|\mathcal{K}_4f\|\right).
%\]
%we perturb the initial vector field $\xi_0$, and 
By \eqref{eq:mixed_estimate} and \eqref{eq:estimate_for_derivative_of_the_elliptic_solution} we have proved
\begin{equation}
\label{eq:H1_estimate_for_the_solution}
\|v_{M_1}\|_{H^1(U)}\leq C \left(\sum_{k=1}^2\|\chi\nabla_{X_{(k)}}(\mathcal{T}_1\mathcal{N}_Lf)\|_{L^2(U)}+\|\mathcal{T}_1\mathcal{N}_Lf\|_{L^2(U)}+\|\mathcal{K}_4f\|_{H^{t}(U)}\right).
\end{equation}

\medskip 
To conclude this section we introduce a norm $\tilde{H}^2(M_1)$ originally given in \cite{SU,stefanov2004stability} to be implemented in the main result of this section. By shrinking $M_1$ if necessary we choose a finite open cover $(U_j)_{j=1}^J$ for $M_1 \setminus M$, such that in $U_j$ we have boundary normal coordinates $(x_j',x_j^3)$, as above. Let $(\chi_j)_{j=1}^J$ be a collection of functions that satisfy $\chi_j\in C_0^\infty(U_j)$, $\chi:=\sum_{j}^J\chi_j$ equals to $1$ near $\partial M$ and each $\chi_j$ vanishes near $\p M_1$. We set
\[
\|h\|^2_{\tilde{H}^1(M_1)}=\int_{M_1}\sum_{j=1}^J\chi_j\left(\sum_{i=1}^2|\nabla_{X_j^{(k)}}	h|_g^2+|x_j^3\nabla_{V_j}h|_g^2\right)+|h|_g^2\; \mathrm{d}x,
\]
where $V_j$ is the tangent vector to $\gamma_{(x'_j,0),\nu}(x^3_j)$. We note that here $x_3>0$ in $M_1\setminus M$. The norm $\tilde{H}^2(M_1)$ is then defined by
\begin{equation}
\label{eq:tilde_H_2_norm}
\|h\|_{\tilde{H}^2(M_1)}=\sum_{i=1}^3\|\nabla_{X_{(k)}}h\|_{\tilde{H}^1(M_1)}+\|h\|_{H^1(M_1)}.
\end{equation}
The equations \eqref{eq:potential_as_smoothing_op} and  \eqref{eq:H1_estimate_for_the_solution} imply the first estimate
\begin{equation}
\label{eq:H1_estimate_for_the_solution_2}
\|v_{M_1}\|_{H^1(M_1\setminus M)}\leq C \left(\|\mathcal{T}_1\mathcal{N}_Lf\|_{\widetilde H^1(M_1)}+\|\mathcal{K}_4f\|_{H^{t}(M_1)}\right).
\end{equation}

Finally we are ready to estimate the solenoidal part $f_{M}^s$. We write
\begin{equation*}
\begin{split}
f=f_{M_1}^s+\mathrm{d}^\mathcal{B}v_{M_1}\quad\quad\text{in}~M_1,
 \quad 
f=f_{M}^s+\mathrm{d}^\mathcal{B}v_{M}\quad\quad\text{in}~M,
\end{split}
\end{equation*}
and denote $u=v_{M_1}-v_{M}$. The construction of the potential parts implies
\begin{equation}
\label{eq_u_M}
\LaB u=0\quad\text{in}~M^{int}, \quad u\vert_{\partial M}=v_{M_1}\vert_{\partial M}.
\end{equation}
By Corollary \ref{Co:continuity_of_the_ops}, equations \eqref{eq:potential_as_smoothing_op}, \eqref{eq:H1_estimate_for_the_solution_2} and the trace theorem we obtain the second estimate
\[
\|v_{M_1}-v_{M}\|_{H^1(M)}\leq C \|v_{M_1}\|_{H^{1/2}(\p M)}\leq C \|v_{M_1}\|_{H^{1}(M_1\setminus M)} \leq  C \left( \|\mathcal{T}_1\mathcal{N}_Lf\|_{\widetilde H^1(M_1)}+\|\mathcal{K}_4f\|_{H^{t}(M_1)}\right).
\]
Then we use $f^s_M=f^s_{M_1}+\dB(v_{M_1}-v_M)$, \eqref{eq:first_recons_formula}, and $\eqref{eq:H1_estimate_for_the_solution}$ to establish our main estimate
\begin{equation}
\label{eq:main_estimate}
\begin{split}
\|f^s_M\|_{L^2(M)}\leq
&\;  \|\mathcal{T}_1\mathcal{N}_Lf-\mathcal{K}_2f\|_{L^2(M)}+ \|\dB(v_{M_1}-v_M)\|_{L^2(M)}
\\
\leq
&\;  \|\mathcal{T}_1\mathcal{N}_Lf-\mathcal{K}_2f\|_{L^2(M)}+ \|v_{M_1}-v_M\|_{H^1(M)}+ \|(v_{M_1})\|_{H^1(M_1\setminus M)}
\\
\leq &\;  C\bigg(\|\mathcal{T}_1\mathcal{N}_Lf\|_{\widetilde H^1(M_1)}+ \|\mathcal{T}_1\mathcal{N}_Lf\|_{L^2(M_1)}+\|\mathcal{K}_4f\|_{H^{t}(M_1)} \bigg)
\\
\leq & \; C\bigg(\|\mathcal{N}_Lf\|_{\widetilde H^2(M_1)}+\|\mathcal{K}_4f\|_{H^{t}(M_1)} \bigg).
\end{split}
\end{equation}
We note that the last estimate is valid since $\mathcal{T}_1$ is an operator of order $1$. 
%So far we have assumed that $f\in C^\infty_0(S\tau'_M\otimes^\mathcal{B} S\tau'_M)$. By $L^2$-density of smooth tensor fields and \eqref{eq:estimates_for_decomp} the estimate above holds for any $f\in L^2(S\tau'_M\otimes^\mathcal{B} S\tau'_M)$. 
The following lemma guarantees that $H^1$-regularity for $f$ implies the finiteness of  $\|\mathcal{N}_Lf\|_{\widetilde H^2(M_1)}$. This result has been presented earlier in \cite{stefanov2004stability}.

\begin{lemma}
If $f \in H^1(S\tau'_M\otimes^\mathcal{B} S\tau'_M)$, then $\|\mathcal{N}_Lf\|_{\widetilde H^2(M_1)}$ is finite.
\end{lemma}

In the following we use the notation $\mathcal{S}$ for the solenoidal projection on $M$. The main theorem of this section is:

\begin{theorem}
\label{th:stability}
Let $t>0$. 
\color{black}
There exists $k_0\in \N$ such that for any simple metric $g\in C^k(M), \: k\geq k_0$, the following claims hold:

\begin{enumerate}
\item There exists a bounded linear operator $\widetilde{\mathcal{K}}\colon L^2(S\tau'_M\otimes^\mathcal{B} S\tau'_M) \to H^{t}(S\tau'_{M_1}\otimes^\mathcal{B} S\tau'_{M_1})$ such that
\begin{equation}
\label{eq:estimate_for_solenoidal_part}
\|f^s_M\|_{L^2(M)}
\leq  \; C\bigg(\|\mathcal{N}_Lf\|_{\widetilde H^2(M_1)}+\|\widetilde{\mathcal{K}}f\|_{H^{t}(M_1)} \bigg), \quad \hbox{ if } 
f \in H^{1}(S\tau'_M\otimes^\mathcal{B} S\tau'_M).
\end{equation}
for some $C>0$.
\item 
%\color{blue}
There exist  bounded linear operators
\[
\mathcal{Q}\colon \widetilde H^2(S\tau'_{M_1}\otimes^\mathcal{B} S\tau'_{M_1})
%\tilde{H}^2(M_1)
\to \mathcal{S}\left(L^2(S\tau'_M\otimes^\mathcal{B} S\tau'_M)\right), \]
\[
 \mathcal{K}\colon L^2(S\tau'_{M}\otimes^\mathcal{B} S\tau'_M) \to  \mathcal{S}\left(H^{t}(S\tau'_{M_1}\otimes^\mathcal{B} S\tau'_{M_1})\right)
\]
such that 
\begin{equation}
\label{parametrix}
\mathcal{Q}\mathcal{N}_Lf=f^s_M+ \mathcal{K}f,
\quad\hbox{ if } f \in H^{1}(S\tau'_M\otimes^\mathcal{B} S\tau'_M).
\end{equation}
Moreover for any simple metric $g_0\in C^k(M)$, there exists a neighborhood $U \subset C^k(M)$, consisting of simple metrics, such that the operators  $\mathcal Q$ and $\mathcal{K}$, in \eqref{parametrix}, depend continuously on $g \in U$. 
%\textbf{[JZ: not $C^k$-continuous.]}
\color{black}
\item  If $g \in C^\infty(M)$, then the vector space 
\[
\ker L \cap  \mathcal{S}\left(L^2(S\tau'_M\otimes^\mathcal{B} S\tau'_M)\right) \subset C^\infty(S\tau'_{M}\otimes^\mathcal{B} S\tau'_{M}),
\]
is finite dimensional.
\item 
\label{it:stability}
If $g \in C^\infty(M)$ and $L$ is $s$-injective then 
\[
\|f^s_M\|_{L^2(M)}
\leq  \; C\|\mathcal{N}_Lf\|_{\widetilde H^2(M_1)}, \quad \hbox{ for } 
f \in H^{1}(S\tau'_M\otimes^\mathcal{B} S\tau'_M),
\]
for some $C>0$.
\end{enumerate}
\end{theorem}
\begin{proof}
%\color{blue}
We note that \eqref{eq:estimate_for_solenoidal_part} is the same inequality as \eqref{eq:main_estimate}, if we  set $\widetilde{\mathcal{K}}= \mathcal{K}_4$. The first claim follows from the construction done before this theorem.
\color{black}

\medskip
Let $ f \in H^{1}(S\tau'_M\otimes^\mathcal{B} S\tau'_M)$. To prove the reconstruction formula \eqref{parametrix} for $f^s_M$, we proceed as in the proof of  \cite[Proposition 5.1]{SU}. 
%\color{blue}
We fix a simple metric $g_0 \in C^k(M)$ and a neighborhood $U$ of $g_0$ that consists of simple metrics. During the proof we are implicitly shrinking $U$ and increasing $k$ without further mention. Let $g \in U$.
\color{black}
 If we vary the initial direction $\xi$ in \eqref{integral1}, we find three linearly independent  $\eta_i\in T_{x}M, \: i \in \{1,2,3\}$ such that the right hand side of \eqref{integral1} gives $v_{M_1}(x)$, for $x \in M_1\setminus M$. 
%\color{blue}
Due to the finite pseudo-local property $v_{M_1}$ can be assumed to be $C^1$-smooth in $M_1\setminus M$ and contained in $H^t(S\tau'_{M_1\setminus M})$. Moreover the map 
$H^1(S\tau'_{M}\otimes^\mathcal B S\tau'_{M})\ni f \mapsto v_{M_1} \in H^t(S\tau'_{M_1\setminus M})$ is bounded and due to Corollary \ref{Co:continuity_of_the_ops}  it depends continuously on $g$ in $C^k$-topology if $k$ is large enough. On the other hand we can use formula \eqref{integral1} and the trace theorem to define a linear operator 
\[
\mathcal{T}_2\colon H^{t-1}(S\tau'_{M_1\setminus M} \otimes^\mathcal BS\tau'_{M_1\setminus M}) \to H^{t-\frac{1}{2}}(S\tau'_{\p  M}),
\]
%where $B\subset \partial_-SM$ is given such that any geodesic emitted there hits $\p M_1$ before $M$.
which depends continuously on $g$, and satisfies 
%\textbf{[JZ: this sentence needs to be rephrased.], (TS)[How about this?, Please feel free to edit this more if you like.]}
\begin{equation}
\label{eq:trace_of_v_M_1}
\hbox{Tr}_{M}\:v_{M_1}= \mathcal T_2(\mathrm{d}^\mathcal{B}v_{M_1})  =\mathcal{T}_2\left(\mathcal{T}_1\mathcal{N}_L-\mathcal{K}_2\right)f.
\end{equation}
In the last equation we used the substitution \eqref{eq:potential_as_smoothing_op}.
Thus the right hand side of \eqref{eq:trace_of_v_M_1} is a bounded map from $H^1(S\tau'_{M}\otimes^\mathcal B S\tau'_{M})$ into $H^{t-\frac{1}{2}}(S\tau'_{\p M})$.
%Let $\mathcal{T}_2$ be the operator defined in Remark \ref{re:operator_T_2}, which satisfies
%\[
%\hbox{Tr}_M(v_{M_1})=\mathcal{T}_2(\mathcal{T}_1\mathcal{N}_L-\mathcal{K}_2)f.
%\]
%Due to trace theorem $\mathcal{T}_2(\mathcal{T}_1\mathcal{N}_L-\mathcal{K}_2): H^{m}(M) \to H^{m-1/2}(\partial M)$ is continuous. 
After converting the problem $ \eqref{eq_u_M}$ into an elliptic problem with zero boundary value, the Corollary \ref{Co:continuity_of_the_ops} implies that the solution operator $\mathcal{R}$ of the boundary value problem \eqref{eq_u_M}, is a bounded operator $\mathcal{R}\colon H^{t-\frac{1}{2}}(S\tau'_{\p M}) \to H^{t}(S\tau'_{M})$, depending continuously on $g$ in some neighborhood of $g_0$ with respect to $C^k$-topology. 
\color{black} 
We get
\[
u:=v_{M_1}-v_M=\mathcal{R}(\hbox{Tr}_M(v_{M_1}))=\mathcal{R}\mathcal{T}_2(\mathcal{T}_1\mathcal{N}_L-\mathcal{K}_2)f.
\]
Then the first reconstruction formula \eqref{seq1} implies
\[
\begin{split}
f^s_M=&f^s_{M_1}+\mathrm{d}^\mathcal{B}u\\
=&(\mathcal{T}_1\mathcal{N}_L-\mathcal{K}_2)f+\mathrm{d}^\mathcal{B}\mathcal{R}\mathcal{T}_2(\mathcal{T}_1\mathcal{N}_L-\mathcal{K}_2)f\\
=&(\mathrm{Id}+\mathrm{d}^\mathcal{B}\mathcal{R}\mathcal{T}_2)\mathcal{T}_1\mathcal{N}_Lf+\mathcal{K}_5f,
\end{split}
\]
where $\mathcal{K}_5=-(\mathrm{Id}+\mathrm{d}^\mathcal{B}\mathcal{R}\mathcal{T}_2)\mathcal{K}_2$. We conclude the proof of \eqref{parametrix} by setting $\mathcal{Q}:=\mathcal{S}(\mathrm{Id}+\mathrm{d}^\mathcal{B}\mathcal{R}\mathcal{T}_2)\mathcal{T}_1$ and $\mathcal{K}:=\mathcal{S}\mathcal{K}_5$. We emphasize that by Theorem \ref{Th:tensor_decomp} the solenoidal projection $\mathcal{S}\colon H^t(S\tau'_M\otimes^\mathcal{B} S\tau'_M) \to H^t(S\tau'_M\otimes^\mathcal{B} S\tau'_M)$ is bounded, and due to Corollary \ref{Co:continuity_of_the_ops} the operators $\mathcal{Q}$ and $\mathcal K$ depend continuously about the metric $g$ in  some small neighborhood of a fixed simple metric $g_0 \in C^k(M)$ in $C^k$-topology, for $k \in \N$ large enough.   

%\medskip
%Then we proceed to prove the kernel of the mixed ray transform is finite dimensional. 
%We note that \eqref{parametrix} can be written as
%\begin{equation}
%\label{eq:parametrix_2}
%\mathcal{Q}\mathcal{N}_L=\mathcal{S}_M+\mathcal{K}, \quad \hbox{ on } L^2(M).
%\end{equation}
%This holds since the right hand side of this equation has a bounded extension $(\mathcal{S}_M+\mathcal{K}) \colon L^2(M) \to L^2(M_1)$. Therefore $\mathcal{Q}\mathcal{N}_L$ has also a bounded extension to $L^2(M_1)$.
%
%For $L^2$-regular $f \in \left(\ker L \cap \mathcal{S}(S\tau'_M\otimes^\mathcal{B} S\tau'_M)\right)$ we have
%\[
%f=f^s_M\quad \hbox{ and } \quad \mathcal N_L f=0.
%\] 
%Thus $\eqref{eq:parametrix_2}$ implies that $f$ is smooth and moreover \eqref{eq:estimate_for_solenoidal_part} implies  
%\[
%\|f\|_{L^2(M)}
%\leq  \; C\|\widetilde{\mathcal{K}}f\|_{L^2(M_1)}.
%\]
%Therefore $\widetilde{\mathcal{K}}f\colon \ker L \cap \mathcal{S}( L^2(S\tau'_M\otimes^\mathcal{B} S\tau'_M)) \to \mathcal{S} (L^2(S\tau'_{M_1}\otimes^\mathcal{B} S\tau'_{M_1}))$ is compact and has a bounded left inverse $\mathcal{H}$. Since $\mathcal{H}\widetilde{\mathcal{K}}$ is a compact indentity map on $L^2(\ker L \cap \mathcal{S}(S\tau'_M\otimes^\mathcal{B} S\tau'_M))$, this space has to be finite dimensional.

\medskip 
The remaining parts of the theorem can be proven as in \cite[Theorem 2]{stefanov2004stability}.
\end{proof}

\section{S-injectivity for analytic metrics}
\label{Se:analytic_s_injective}
\subsection{The analytic parametrix}
%By (real) analysity of a function or tensor field $f$ in a domain $\Omega \subset \R^3$ we mean that $\Omega$ can be  extended to a domain $\Omega^\C \subset \C^3$ such that $f$ has a holomorphic extension in $\Omega^\C$.  
In this section we assume that $(M,g)$ is a simple manifold with a (real) analytic metric $g$ 
on $M$ up to the boundary.
As in Section \ref{Se:Parametrix} we extend $M$ and $g$ to simple open domains $(M_1,g)$,  $(M_2,g)$ and $M\subset\subset M_1\subset\subset M_2 \subset \R^3$. We note that this can be done in such a way that, $g$ is analytic in a neighborhood of $M_2$ and $M_i, \: i\in \{1,2\}$ have analytic boundaries. Since analytic functions are dense this does not require the original boundary $\p M$ to be analytic (see \cite[Section 3]{SU}).

We construct an analytic parametrix for operator $\mathcal{M}$. We denote the set of analytic tensor fields on $M$ by $\mathbb{A}(M)$. That is every $f \in \mathbb{A}(M)$ has an analytic extension to some open domain containing $M$. For the basic theory of analytic $\Psi$DO we refer to \cite[Chapter V]{treves1}.
Recall that a continuous linear operator from $\mathcal{E}'(M)$ to $\mathcal{D}'(M)$ is analytic regularizing, if its range is contained in $\mathbb{A}(M)$.

%
%\begin{definition}{\cite[Chapter V, Definition 3.1]{treves1}}
%\label{De:analytic}
%Let $D\subset \R^3$ be an open domain. Integral operator $T$ on $D$ is an analytic $\Psi$DO, if for any relatively compact open set $\Omega \subset D$ there exists a pseudoanalytic amplitude $a$ on $\Omega \times \Omega \times \R^3$, (see \cite[Chapter V, Definition 2.1]{treves1}) such that $T-A$ is analytic regularizing in $\Omega$. Here $A$ is the operator of $A$ given similarly to formula \eqref{eq:operator_of_M}.
%\end{definition}

Our first result in this section is a re-formulation of Proposition \ref{Pr:normal_is_pseudo} in the analytic setting.

\begin{proposition}
\label{pr:analytic_PsiDOs}
The operators $\mathcal{N}_L$ and $\mathcal{M}$, from \eqref{eq:operator_M}, are analytic $\Psi$DOs in $M_2$.
\end{proposition}
\begin{proof}
Our proof follows the proof of \cite[Proposition 3.2]{SU}.

Since $g$ is  analytic in $M_2$, there exits $\delta>0$ such that the operator $A$ defined by \eqref{Auu} and functions $G^{(m)}, \: m\in\{1,2,3\}$ from \eqref{eq:Gs} are analytic in $U=\{(x,y)\in M_2\times M_2;\, |x-y|_e< \delta\}$. Let $V$ be an open set such that $V\times V\subset U$. Then $\widetilde M_{ijk\ell}$, given in  \eqref{eq:tildeM},  is analytic in $V \times V \times (\R^3\setminus \{0\}) $, and due to Lemma \ref{Le:M_tilde} distribution $\widetilde M_{ijk\ell}$ is positively homogeneous of order $-2$ in $z$ variable. Here we use the fact that $A$ is analytic in $U$, since the solution to an ODE with analytic coefficients is analytic.

 Thus $M_{ijk\ell}(x,y,\xi)$ is  analytic  in $V\times V\times (\R^3\setminus \{0\}) $ as a Fourier transform of $\widetilde M_{ijk\ell}$ in  $z$ variable. To see this, one only need to notice that $\widetilde M(x,y,z)$ is even in $z$, and \cite[Theorem 7.1.24]{hormander1} implies
\begin{equation}
\label{eq:M}
\begin{split}
M_{ijk\ell}(x,y,\xi)=&\int_{\R^3} e^{-i\xi \cdot z}\widetilde M_{ijk\ell}(x,y,z)\: \mathrm{d z}=\pi \int_{\mathbb{S}^2} \widetilde M_{ijk\ell}(x,y,\omega) \delta_0(\omega \cdot \xi)\; \mathrm{d \omega},
%=\pi \int_{0}^{2\pi } \int_0^\pi\widetilde M_{ijk\ell}(x,y,\theta,\phi) \delta_0(\omega \cdot \xi)\sin(\theta)\; \mathrm{d \theta} \mathrm{d \phi}.
\end{split}
\end{equation}
where $\widetilde M$, in the last integrand, is an analytic function of all its variables.
%Let $A_\xi$ be an orthogonal matrix that rotates $\xi \mapsto |\xi|_ee_3$. Then 
%\[
%\begin{split}
%M_{ijk\ell}(x,y,\xi)
%%=&\pi \int_{S^2} \widetilde M_{ijk\ell}(x,y,A_{\xi}^T\omega) \delta_0(\omega \cdot e_3)\; \mathrm{d \omega}
%%\\
%=&\pi \int_{0}^{2\pi }\widetilde M_{ijk\ell}(x,y,A_{\xi}^T\omega(\phi)) \; \mathrm{d \phi},
%\end{split}
%\]
%where $\phi$ is the azimuthal angle. The analyticity of $M_{ijk\ell}(x,y,\xi)$ follows from this since  $\widetilde M_{ijk\ell}(x,y,\omega)$, $ \omega \in \mathbb{S}^2$ is analytic and $\xi \mapsto A_\xi$ is analytic. Also composition of analytic functions is analytic.

\medskip
To prove that $M(x,y,\xi)$ is an analytic 
%\color{blue}
amplitude 
\color{black}
(see \cite[V, Definition 2.2]{treves1}) in $V \times V$ we proceed as follows. Let $\phi \in C^{\infty}_0(\mathbf{R}^3)$. We write
\[
M_{ijk\ell}(x,y,\xi)=\phi(\xi)M_{ijk\ell}(x,y,\xi)+(1-\phi(\xi))M_{ijk\ell}(x,y,\xi),
\]
and show that the first term is an amplitude of analytic regularizing operator and the second one is an analytic amplitude. This shows that $\mathcal{N}_L$ is an analytic $\Psi$DO in $V$. 

%\color{blue}
%We have just verified that $\phi(\xi)M(x,y,\xi)$ is analytic with respect to $(x,y)$ and due to choice of $\phi$ it is compactly supported in $\xi$. Thus \cite[Chapter V, Corollary 2.1]{treves1} implies that the operator of  $\phi(\xi)M(x,y,\xi)$ analytic regularizing.
%\color{black}
To prove that the operator of $\phi(\xi)M_{ijk\ell}(x,y,\xi)$ is analytic regularizing, we need to show that the corresponding integral kernel 
\[
\mathcal H_{ijk\ell}(x,y)=(2\pi)^{-3}\int_{\R^3} e^{i\xi \cdot (x-y)}\phi(\xi)M_{ijk\ell}(x,y,\xi) \; \mathrm{d}\xi, \quad (x,y) \in V \times V,
\]
is analytic. To do this we use the fact that $M_{ijk\ell}(x,y,\xi)$ is positively homogeneous of order $-1$ in $\xi$. Then a change to spherical coordinates gives 
\[
\mathcal H_{ijk\ell}(x,y)=(2\pi)^{-3} \int_{\mathbb S^2}\int_0^\infty e^{i(r\omega) \cdot (x-y)}\phi(r \omega)M_{ijk\ell}(x,y,\omega)r \; \mathrm{d}r\mathrm{d}\omega.
\]
Since a product of analytic functions is analytic and $\phi$ is compactly supported,
%As the indegrand on the right hand side is analytic in $(x,y)$ and since $\phi$ is compactly supported we can pass the $(x,y)$-derivatives inside the integral to see that for any compact $K \subset \R^3$ there exists constants $C_0,C_1>0$ such that
%\[
%|\p^\alpha_{x,y}\mathcal H^1_{ijk\ell}(x,y)|\leq C_0C^{|\alpha|}_1\alpha!, \quad (x,y)\in K. 
%\]
this proves that $\mathcal H_{ijk\ell}$ is analytic.

%We need to prove that for any compact set $K\subset \C^3\times \C^3$ there are constants $C,R>0$ such that for any multi-index $\alpha$ the following estimate holds true
%\[
%|D^\alpha_\xi F(x,y,\xi)|\leq C^{|\alpha|+1}\alpha!|\xi|_e^{-1 -|\alpha|}, \quad \hbox{ when ever } |\xi|_e\geq R\max\{1,|\alpha|\}, \: (x,y) \in K. 
%\]
Let $R_0>0$ be a radius of a ball containing the support of $\phi$. 
%\color{blue}
Since $F(x,y,\xi):=(1-\phi(\xi))M_{ijk\ell}(x,y,\xi)$ equals  to $M_{ijk\ell}(x,y,\xi)$, when $|\xi|_e>R_0$ and $M_{ijk\ell}(x,y,\xi)$ is homogeneous of order $-1$ in $\xi$, we can write
\begin{equation}
\label{eq:analytic_ext}
F(x,y,\xi)=M_{ijk\ell}\left(x,y,\frac{\xi}{|\xi|_e}\right)|\xi|^{-1}_e, \quad |\xi|_e>R_0.
\end{equation}
Therefore we can use the right hand side of \eqref{eq:analytic_ext} to extend $F$ analytically on $V^\C \times V^\C \times (\C^3\setminus \overline B(0,R_0))$. Here $V^\C$ is an extension of $V$ to $\C^3$.
\color{black}
%\color{blue}
%We set $f(x,y,\omega):=M_{ijk\ell}(x,y,\omega), \: |\omega|_e=1$, which is analytic. 
\color{black}
%Then it follows from the uniqueness of the analytic continuation that the analytic extension of $F$ to $V^\C\times V^\C \times (\C^3\setminus B_{R_0}(0))$ is 
%\[
%F(x,y,z)=M_{ijk\ell}\left(x,y,\frac{\xi}{|\xi|_e}\right)|\xi|^{-1}_e.
%\]
This implies that for any compact $K\subset V^\C\times  V^\C$ there exist $C>0$ 
%and $R_1>0$ 
such that       
\[
|F(x,y,z)|\leq C |\xi|^{-1}_e, \quad  \quad (x,y)\in K, \: |\xi|_e>R_0.
%z \in D_{R_1}=\{z \in \C^3: 1+|\Im z|< R_1^{-1} |\Re z|\}.
\]       
%\textbf{Why does $R_1$ exists? This might be a real issue!!!!! }
%\begin{remark}
% Let multi index $\alpha=(k,k,k)$ for any $k \in \N$. Then 
%\[
%D_\xi^\alpha F=\mathcal{F}_z(z^\alpha \widetilde M), \quad |\xi|_e>R_0.
%\]
%and $z^\alpha \widetilde M$ is homogeneous of order $|\alpha|-2$. Thus $D_\xi^\alpha F$ is homogeneous of order
%$
%%-3-(3k-2)=-3k-1=
%-|\alpha|-1,
%$
%when ever $|\xi|_{e}>R_0$. 
%\end{remark}

We choose $R>0$ so large that $\widetilde B_R(\xi):=\prod_{i=1}^3B_{i,R}(\xi) \subset (\C^3 \setminus B_{R_0}(0))$ if $\xi \in \R^3, \: |\xi|_e> R$ and $B_{i,R}(\xi):=\{z \in \C: \: |z-\xi_i|_e\leq \frac{1}{2R}|\xi|_e\}$. Then we apply Cauchy's integral formula on $\widetilde B_R(\xi)$  to find
\[
\begin{split}
|D^\alpha_\xi F (x,y,\xi)|
%\leq & \: \alpha! (2\pi)^{-3}\int_{\p B_{1,R}}\cdots\int_{\p B_{3,R}}\frac{|F(x,y,z)|}{|\xi_1-z_1|^{\alpha_1+1}\cdots|\xi_3-z_3|^{\alpha_3+1}}\mathrm{d}z_3\cdots \mathrm{d}z_1
%\\
%\leq &\: \alpha! (2\pi)^{-3}\left(\frac{1}{2R}\right)^{|\alpha|+3} |\xi|_e^{-|\alpha|-3}  \sup_{z \in B} |a(z,y,z)|  \int_{\p B_{1,R}} \cdots \int_{\p B_{3,R}} 1 \mathrm{d}z_3\cdots \mathrm{d}z_1
%\\
%\leq &\: \alpha! \left(\frac{1}{2R}\right)^{|\alpha|} |\xi|_e^{-|\alpha|}  \sup_{z \in B} |a(z,y,z)|  
%\\
\leq &\: \alpha! \left(\frac{C}{2R}\right)^{|\alpha|}  |\xi|_e^{-|\alpha|-1} , \quad (x,y)\in K, \: |\xi|_e>R.
\end{split}
\]
Therefore $F$ is an analytic amplitude. 
%\color{blue}
Since $V \subset U$ was arbitrarily chosen, we have proven that for any $x_0 \in M_2^{int}$ there exists a neighborhood $V_{x_0}$ contained in $M_2^{int}$ in which $\mathcal{N}_L$ is an analytic $\Psi$DO.  
\color{black}
%\textbf{[Let us still double check this](TS)}
From here we can follow the lines in the proof of \cite[Proposition 3.2]{SU} to conclude that $\mathcal{N}_L$ is actually an analytic $\Psi$DO in the whole $M_2$.

\medskip
Next we proceed to prove that $\mathcal{P}_{M_2}=\deB (\Delta^\mathcal{B}_{M_2})^{-1} \dB$ is an analytic $\Psi$DO. Formula \eqref{eq:dB} implies that $\dB$ is an analytic operator.  Also $\deB$ is analytic. Therefore $\LaB$ is an elliptic analytic operator in the domain $\Omega$. Thus there exists an analytic parametrix $T$ of $\LaB $ in an open set $M'$ containing $\overline{M_2}$. We need to show that $(\Delta^\mathcal{B}_{M_2})^{-1}-T$ is analytic regularizing on $M_2$. Let the distribution $f$ be supported in $M_2$. We set $w:=((\Delta^\mathcal{B}_{M_2})^{-1}-T)f$. Then $\LaB w \in \mathbb{A}(M_2)$ and  the interior analytic regularity  implies $ w \in \mathbb{A}(M_2)$, and we have proven that $(\Delta^\mathcal{B}_{M_2})^{-1}$ is an analytic $\Psi$DO in $M_2$.
% in the sense of Definition \ref{De:analytic}. 

\medskip 
Since $|D|_g$ is analytic, we have proven that operator $\mathcal{M}$ in \eqref{eq:operator_M} is analytic.
\end{proof}

In the following lemma we extend the result of Lemma \ref{existence2} to an analytic case. The proof is similar to \cite[Lemma 3.3]{SU}.
\begin{lemma}\label{analytic_lemma}
Let $x_0\in\partial M$, and assume that the metric $g$ and the tensor fields $u$ and $v_0$ are analytic in a (two-sided) neighborhood of $x_0$ and that $\partial M$ is analytic near $x_0$. Let tensor field $v$ solve
\begin{equation}
\LaB v=u\quad\text{in}~M, \quad v\vert_{\partial M}=v_0.
\end{equation}
Then $v$ extends as an analytic function in some (two-sided) neighborhood of $x_0$.
\end{lemma}

The main result of this subsection is the following.
\begin{proposition}\label{analytic_parametrix_thm}
There exists a bounded operator $\mathcal{W}:H^{1}(S\tau'_{M_1}\otimes^\mathcal{B} S\tau'_{M_1})\to L^2(S\tau'_{M_1}\otimes^\mathcal{B} S\tau'_{M_1})$ such that for any $2$-tensor $f\in L^2(S\tau'_{M}\otimes^\mathcal{B} S\tau'_{M})$ we have
\[
f_{M_1}^s=\mathcal{W}\mathcal{N}_Lf+\mathcal{K}f
\] 
with $\mathcal{K}f$ analytic in $M_1$.
\end{proposition}
\begin{proof}
Since $\mathcal M$ is an elliptic analytic $\Psi$DO in $M_1$ we can construct a parametrix $\mathcal{L}=\{\mathcal{L}_1,\mathcal{L}_2\}$ of $\mathcal{M}$ in $M_1$ such that
\begin{equation}\label{analyticParametrix}
\mathcal{L}\mathcal{M}=\mathrm{Id}+\mathcal{K}_1,
\end{equation}
with $\mathcal{L}$ an analytic $\Psi$DO of order $0$ in a neighborhood of $M_1$, and $\mathcal{K}_1$ analytically regularizing in $M_1$. Apply $S_{M_2}$ to the left and right, to \eqref{analyticParametrix} and notice $\mathcal{N}_L\mathcal{S}_{M_2}=\mathcal{S}_{M_2}\mathcal{N}_L=\mathcal{N}_L$, $\mathcal{P}_{M_2}\mathcal{S}_{M_2}=0$. 
We denote 
%\[
%\mathcal{S}_{M_2}\mathcal{L}_1|D|\mathcal{N}_L=\mathcal{S}_{M_2}+\mathcal{K}_2,\quad\text{in}~M_1.
%\]
%Denote 
$\mathcal{W}=\mathcal{S}_{M_2}\mathcal{L}_1|D|_g$, and have
\[
\mathcal{W}\mathcal{N}_L=S_{M_2}+\mathcal{K}_2\quad\text{in}~M_1.
\]
Here $\mathcal{K}_2$ is analytic regularizing in $M_1$.

%By replacing $M,M_1,M_2$ with $M_1,M_2,M_3\supset\supset M_2$ in the above discussion, there exists $\mathcal{P}$ such that
%\[
%\mathcal{P}\mathcal{N}_L=\mathcal{S}_{M_2}+\mathcal{K}_2\quad\text{in}~M_1.
%\]
%Here $\mathcal{K}_2$ maps $L^2(M_1)$ into $\mathbb{A}(\overline{M}_1)$, but not into $\mathbb{A}(\overline{M}_2)$.

We need to compare $f^s_{M_1}$ and $f^s_{M_2}$ for $f\in L^2(S\tau'_{M}\otimes^\mathcal{B} S\tau'_{M})$. We write $f^s_{M_i}=f-\mathrm{d}^\mathcal{B}v_{M_i}$. As in the previous section, we have $f^s_{M_1}=f^s_{M_2}+\mathrm{d}^\mathcal{B}u$ in $M_1$ for $u=v_{M_2}-v_{M_1}$. Then $u\in H^1(M_1)$ and  solves
\[
\Delta^\mathcal{B}u
%=\delta'\mathrm{d}'u-\frac{1}{3}\mathrm{d}(\delta u)
=0\quad\text{in}~M_1, \quad u\vert_{\partial M_1}=v_{M_2}.
\]
We note that since $\p M_1$ is analytic we have $\hbox{Tr}_{M_1}h\in \mathbb{A}(\p M_1)$ for any $h$ analytic near $\p M_1$. As $\mathrm{supp} f$ is disjoint from $\partial M_1$, analytic pseudo-locality yields $v_{M_2}\in\mathbb{A}(\partial M_1)$.
%\textbf{Seems to me that in addition to the analytic pseudolocal property of $ (\LaB)^{-1}_{M_2}\deB $, we need also the fact that the trace map preserves the analyticity. Don't we need to assume that the boundary of $M_1$ is analytic?}
By Lemma \ref{analytic_lemma}, $u\in\mathbb{A}(M_1)$; thus $f\mapsto \mathrm{d}^\mathcal{B}u$ is a linear operator mapping $L^2(M)$ into $\mathbb{A}(M_1)$. Then we use the relation
\[
f^s_{M_1}=f^s_{M_2}+\mathrm{d}^\mathcal{B}u=\mathcal{WN}f-\mathcal{K}_2f+\mathrm{d}^\mathcal{B}u
\]
to complete the proof.
\end{proof}
\subsection{$S$-injectivity of $1+1$ tensors for analytic metrics}
In this subsection we will prove $s$-injectivity for an analytic simple metric $g$.
\begin{lemma}\label{lemm42}
Let $g$ be a smooth, simple metric in $M$ and let $f\in C^\infty(S\tau'_M\otimes^\mathcal{B} S\tau'_M)$. If $Lf=0$, then there exits a tensor field $v$ vanishing on $\partial M$ such that for $\tilde{f}=f-\mathrm{d}^\mathcal{B}v$ we have $\mathrm{Tr}_M (\partial^m\tilde{f})=0$ for any multi-index $m \in \N^3$.

Moreover, if $g$ and $f$ are analytic in a (two-sided) neighborhood of $\partial M$, and $\partial M$ is also analytic, then $v$ can be chosen so that $\tilde{f}=0$ near $\partial M$.
\end{lemma}

\begin{proof}
We fix $x_0\in\partial M$ and take boundary normal coordinates $x=(x',x^3)$ in a neighborhood $U\subset M$ of $x_0$. Then in $U$ we have $g_{i3}=\delta_{i3}$ for any $i=1,2,3$. We aim first to find a tensor field $v$, vanishing on $\partial M$, such that for $\tilde{f}:=f-\mathrm{d}^\mathcal{B}v$ we have 
\begin{equation}\label{zero1}
\tilde{f}_{i3}=0, \quad \hbox{ in some open neighborhood }\widetilde U \subset U \hbox{ of } x_0.
%=f_{i3}-\left(\nabla_3v_i-\frac{1}{3}\delta_{i3} (g^{k\ell} \nabla_k v_\ell)\right).
\end{equation}
Due to \eqref{eq:dB}, this is equivalent to
\begin{equation}\label{eq_ftilde}
\begin{split}
&f_{13}-\nabla_3v_1=0, \quad f_{23}-\nabla_3v_2=0, \quad  f_{33}-\nabla_3v_3+\frac{1}{3}(g^{k\ell} \nabla_k v_\ell)=0, \quad \hbox{ in } U
\\
&v\vert_{x^3=0}=0.
\end{split}
\end{equation}
%
%\textbf{Note \color{red}$(\nabla_1v_1+\nabla_2v_2+\nabla_3v_3)$ $\color{black}= g^{k\ell} \nabla_k v_\ell \hbox{ if and only if } g_{ij}=\delta_{ij} \hbox{ for any } i,j$}

Remember that $\nabla_jv_i=\partial_jv_i-\Gamma_{ji}^kv_k$, and the Christoffel symbols  in the boundary normal coordinates, satisfy $\Gamma_{33}^k=\Gamma_{k3}^3=\Gamma_{3k}^3=0$. We first solve the  system of initial value problems
\[
\left\lbrace
\begin{array}{l}
\partial_3v_1-\Gamma_{31}^1 v_1-\Gamma_{31}^2 v_2=\nabla_3v_1=f_{13}, 
\\
\partial_3v_2-\Gamma_{32}^1 v_1-\Gamma_{32}^2 v_2=\nabla_3v_2=f_{23},  
\\
v_1(x',0)=0, \quad v_2(x',0)=0, 
\end{array}
\right.
\]
for $v_1$ and $v_2$, which are given along boundary normal geodesics $\gamma_{(x',0),\nu}(x_3)$. 
Then using $g^{i3}=\delta^{i3}$ we write the last equation of \eqref{eq_ftilde} in a form of the following initial value problem
%\[
%\begin{split}
%0=&f_{33}-\nabla_3v_3+\frac{1}{3}(g^{ij} \nabla_i v_j)
%\\
%=&f_{33}-\nabla_3v_3+\frac{1}{3}\left(\delta^{i3}\nabla_i v_3 \right)+G=0
%\\
%=&f_{33}-\frac{2}{3}\left(\partial_3v_3-\Gamma_{33}^kv_k\right)+G
%\end{split}
%\]

\[
\partial_3v_3
%-\frac{1}{2}\sum_{i=1}^2 \partial_iv_3
=\frac{3}{2}\left(f_{33}-G\right),\quad   v_3(x',0)=0,
\]
where  $G$ depends only on $v_i, \: \p_j v_i,\: g^{jk}, \: \Gamma^k_{j\ell}$ for $i \in \{1,2\}$. We have found $v$ near the boundary. Clearly, if $g$ and $f$ are analytic near $\partial M$, so is $v$.

We note that the convexity of the boundary implies that for $(x,\xi)\in \p_+ (SM)$ where $x\in\partial M\cap U$, is close to $x_0$, $|\xi|_g=1$ and the normal component of $\xi$ is small enough, the geodesic issued from $(x,\xi)$ hits the boundary again in $U$. Then the boundary value of $v$ and $Lf=0$ imply $L\tilde{f}(x,\xi)=0$.  We choose the boundary coordinates $x'$ such that $g_{ij}\vert_{x=x_0}=\delta_{ij}$.  We claim
\begin{equation}\label{zero2}
\tilde{f}_{3\alpha}\vert_{x=x_0}=0,\quad\tilde{f}_{\alpha\beta}\vert_{x=x_0}=0,\quad(\tilde{f}_{11}-\tilde{f}_{22})\vert_{x=x_0}=0,
\end{equation}
for $\alpha=1,2$, $\beta=1,2$ and $\alpha\neq\beta$. If this is true, then $(\ref{zero1})$ and 
$
\sum_{i=1}^3\widetilde f_{ii}(x_0)=\mu \widetilde f(x_0) =0,
$
give $\tilde{f}(x_0)=0$.
To prove \eqref{zero2} we let $\xi\in T_{x_0}\partial M$, $|\xi|_g=1$, and take a curve $\delta:(-\epsilon,\epsilon)\rightarrow\partial M$ adapted to $(x_0,\xi)$. Let $\gamma=\gamma_\epsilon\colon [0,1]\rightarrow M$ be the shortest geodesic of the metric $g$ joining the points $x_0$ and $\delta(\epsilon)$, i.e., $\gamma(0)=x_0$ and $\gamma(1)=\delta(\epsilon)$. Let $\eta\in T_{x_0}M$ be perpendicular to $\xi$, and $\eta_\epsilon$ be  the orthogonal projection of $\eta$ to $\dot \gamma_\epsilon(0)$. We also set $\eta(t)=\eta_\epsilon(t)$ to be the parallel transport of $\eta_\epsilon$ along $\gamma_\epsilon$. Since the points $(\gamma(t),\frac{\dot{\gamma}(t)}{|\dot{\gamma}(t)|_g})$ and $(\gamma(t),\eta(t))$ tend  to $(x_0,\xi)$ and $(x_0,\eta)$, respectively, uniformly for $t\in [0,1]$ as $\epsilon\rightarrow 0$, we have
\[
\eta^k
%\left(\delta_i^k-(\eta_{0})_{i}\xi^k\right)
\tilde{f}_{kj}\xi^j=\lim_{\epsilon \to 0}\int_0^1 \eta^k(t)
%\left(\delta_i^k-\frac{\dot{\gamma}(t)_i\dot{\gamma}(t)^k}{|\dot{\gamma}(t)|_g^2}\right)
\tilde{f}_{kj}(\gamma(t))\frac{\dot{\gamma}^j(t)}{|\dot{\gamma}(t)|_g}\mathrm{d}t=\lim_{\epsilon \to 0}\frac{1}{|\dot{\gamma}_\epsilon (0)|_g} \left\langle L\widetilde f\left(x_0,\dot{\gamma}_\epsilon(0)\right), \eta_\epsilon \right\rangle=0.
\]
%Thus $\left(\delta_i^k-(\xi_{0})_{i}\xi_0^k\right)\tilde{f}_{kj}\xi_0^j=0$. 
We set $e_3=\nu(x_0)$, and $e_\alpha=\frac{\p}{\p x'^\alpha}|_{x_0}$ for $\alpha=1,2$. The previous equation implies
\[
\tilde{f}_{m\alpha}=
%e_m^k\tilde{f}_{kj}e_\alpha^j=e_m^i\left(\delta_i^k-(e_\alpha)_{i}e_\alpha^k\right)\tilde{f}_{kj}e_\alpha^j=
0, \quad m \in \{1,2,3\}, \: \alpha \in\{1,2\}, \: m\neq \alpha.
\]
%and
%Take $\xi_0=e_\alpha$, $\eta_0=e_\beta$, $\alpha,\beta=1,2$, we obtain
To obtain the last equation in \eqref{zero2} we set $\xi=\frac{1}{\sqrt{2}}(e_1+e_2)$, $\eta=\frac{1}{\sqrt{2}}(e_1-e_2)$, thus
\[
\frac{1}{2}(f_{11}-f_{22})=\xi^i\tilde{f}_{ij}\eta^j
%=\xi^i_0\left(\delta_i^k-(\eta_{0})_{i}\eta_0^k\right)\tilde{f}_{kj}\eta_0^j
=0.
\]
This completes the proof of \eqref{zero2}. Since $x_0$ was an arbitrary point in $\p M \cap U$ we have shown that $\widetilde f$ vanishes at $\p M \cap U$. It remains to show
\begin{equation}
\label{zero4}
\partial^m_{x_3}\tilde f_{ij}\vert_{x=x_0}=0, \quad m=1,2,\cdots. 
\end{equation}
We do not prove this directly but note that, if 
\begin{equation}\label{zero3}
\partial^m_{x_3}\tilde{f}_{3\alpha}\vert_{x=x_0}=0,\quad\partial^m_{x_3}\tilde{f}_{\alpha\beta}\vert_{x=x_0}=0,\quad\partial^m_{x_3}(\tilde{f}_{11}-\tilde{f}_{22})\vert_{x=x_0}=0, \quad \alpha, \beta \in\{1,2\}, \: \alpha \neq \beta.
\end{equation}
holds, then due to 
\[
%0=\partial_{x_3}\left(\mu f\right)=\partial_{x_3}\left( g^{ij}\tilde f_{ij}\right)=g^{ij} \left(\partial_{x_3} \tilde f_{ij}\right)=
\sum_{i=1}^3\partial_{x_3} \tilde f_{ii}(x_0)=\mu \left( \partial_{x_3} \tilde f\right)\bigg|_{x=x_0}=(\partial_{x_3} \mu \tilde f)\bigg|_{x=x_0}=0
\]
\eqref{zero4} also holds. The equation above holds since the trace and the covariant derivative commute and Christoffel symbols vanish at $x_0$.

The proof of \eqref{zero3} is similar to the proof of \cite[Theorem 2.1]{lassas2003semiglobal}. We give it here for the sake of completeness. Let $m>0$ be the smallest integer for which \eqref{zero3} does not hold. 
We consider a $2$-tensor 
$
h_{ij}:=\p^m_{x_3}\widetilde f_{ij}|_{x=x_0}
$
acting on $T_{x_0}M$. 
Since \eqref{zero3} does not hold for $m$, there exists 
%Then \eqref{zero1} implies $h_{i3}=0$ for $i\in \{1,2,3\}$. Since \eqref{zero3} holds for any $m<k$ we have
%\[
%\p^\ell_3\widetilde f_{ij}|_{x=x_0}=\nabla^\ell_3\widetilde f_{ij}|_{x=x_0}, \quad \ell \leq k
%\]
%due to choice of our coordinates. Therefore  
%\[
%h_{11}+h_{22}=\mu h=\mu (\nabla^k_3 \widetilde f)=\nabla^k_3(\mu \widetilde f)=0,
%\] 
%since contractions and covariant derivatives commute.
%
%Let $e_\alpha$ and $e_m$ be two of the coordinate vector fields at $T_{x_0}M$ where $e_\alpha,$ are tangential to the boundary. Then
%\[
%e^i_m h_{ij}e^j_m=h_{m \alpha}, \quad 
%\]
$\xi_0\in T_{x_0}M$ of unit length, tangent to $\partial M$, and $\eta_0\in T_{x_0}M$ that is perpendicular to $\xi_0$, such that $\eta^i_0 h_{ij}\xi_0^j\neq 0$. Then the Taylor expansion of $\tilde{f}$ implies that  $\eta^i\tilde{f}_{ij}\xi^j$ is either (strictly) positive or negative for $x^3>0$ and $|x'-x_0'|_e$ both sufficiently small and $(\xi,\eta)$ close to $(\xi_0,\eta_0)$. Therefore, $\langle L\tilde{f}(x,\xi),\eta\rangle$ is either (strictly) positive or negative for all $(x,\xi)\in\partial_+(SM)$ close enough to $(x_0,\xi_0)$ and $\eta\perp\xi$  close to $\eta_0$. This is a contradiction.

We have completed the construction of $v$ near $x_0$. As in the proof of \cite[Lemma 4.1]{SU}, we can extend the construction of $v$ anywhere near $\partial M$.

If $g$ and $f$ are analytic, then $v$ is analytic near $\partial M$. Then $\tilde{f}$ is analytic and thus $\tilde{f}=0$ near $\partial M$.
\end{proof}

In the following Theorem we use global semi-geodesic coordinates for simple manifold $(M_1,g)$, introduced in \cite[Lemma 4.2]{SU}, under which the metric $g$ has the global representation.
\[
g_{i3}=\delta_{i3},\quad i=1,\,2,\,3.
\]
We use the notations $e_i, \: i \in \{1,2,3\}$ for the corresponding  coordinate vector fields.

\begin{theorem}\label{sinjectivity_analytic}
Let $g$ be a simple metric in $M$, that has an analytic extension. Then $L$ is s-injective.
\end{theorem}
\begin{proof}
%We work in the semi-geodesic coordinates for $M_1$. 
Assume that $g\in\mathbb{A}(M)$, and the mixed ray transform of $f\in L^2(S\tau'_M\otimes^\mathcal{B} S\tau'_M)$ vanishes. Then, by Proposition \ref{analytic_parametrix_thm}, we have $f^s_{M_1}\in\mathbb{A}({M}_1)$. Clearly, $Lf^s_{M_1}=0$ as well.

By Lemma \ref{lemm42}, there exists a smooth tensor field $w$ that is analytic near $\p M$ and more over $\tilde{f}:=f^s_{M_1}-\mathrm{d}^\mathcal{B}w$ vanishes near $\partial M_1$. We denote the set of all points $x \in \p M_1$, for which the coordinate vector field $e_3(x) \in \p_{\pm}(SM_1)$, by $(\partial M_1)_\pm$. We aim first to find a second tensor field $v$ that satisfies the following global equation
\[
(\widetilde f-\dB v)_{i3}=0, \quad i\in \{1,2,3\}, \quad  v\vert_{(\partial M_1)_+}=0.
\] 
That is we solve the equations similar to $(\ref{eq_ftilde})$:
\begin{equation}\label{eq_ftilde1}
\begin{split}
&\tilde{f}_{13}-\nabla_3v_1=0, \quad \tilde{f}_{23}-\nabla_3v_2=0, \quad \tilde{f}_{33}-\nabla_3v_3+\frac{1}{3}(g^{k\ell} \nabla_k v_\ell)=0, \quad \hbox{ in } M_1
\\
&v\vert_{(\partial M_1)_+}=0.
\end{split}
\end{equation}
Since $(M,g)$ is simple it follows from the definition of the semi-geodesic coordinates that each point in $M$ can be reached by a geodesic parallel to $e_3$ from a unique point of $(\p M)_+$. Therefore the system \eqref{eq_ftilde1} can be used to define $v$ globally.
%We note here that we are solving the above system in the semi-geodesic coordinates instead of boundary normal coordinates, so $v$ satisfies $(\ref{eq_ftilde1})$ everywhere in $M_1$. 
As before, we first determine $v_1$ and $v_2$ from the system of linear boundary value problems
\[
\left\lbrace \begin{array}{l}
\partial_3v_1-\Gamma_{31}^1 v_1-\Gamma_{31}^2 v_2=\tilde{f}_{13}, \quad \partial_3v_2-\Gamma_{32}^1 v_1-\Gamma_{32}^2 v_2=\tilde{f}_{23}
\\
v_1|_{(\p M_1)_+}=v_2|_{(\p M_1)_+}=0.
\end{array}
\right.
\]
We note that this system has a unique global solution since $M$ is compact.
% and therefore $\widetilde f$ and $\Gamma_{ij}^k$ are bounded and have bounded gradients. 
Due to analyticity, $v_1,\,v_2$ vanish in a neighborhood $U$ of $\overline{(\partial M_1)_+}$. The last equation of \eqref{eq_ftilde1} takes the form of the following boundary value problem
\[
\begin{split}
%\tilde{f}_{33}-\nabla_3v_3+\frac{1}{3}(\delta^{k\ell} \nabla_k v_\ell)=
\p_3v_3=\frac{3}{2}\left(\widetilde f_{33}-G\right), \quad v_3\vert_{(\partial M_1)_+}=0, 
\end{split}
\]
where  $G$ depends only on $v_i, \: \p_j v_i,\: g^{jk}, \: \Gamma^k_{j\ell}$ for $i \in \{1,2\}$. Thus we have found $v$ and shown that it vanishes in $U$.

\medskip
Now we define $f^\sharp:=f^s_{M_1}-\mathrm{d}^\mathcal{B}w-\mathrm{d}^\mathcal{B}v$. We have $f^\sharp=0$ in $U$ and $f^\sharp_{i3}=0$ in ${M}_1$, $i=1,\,2,\,3$. Moreover, $w+v=0$ on $(\partial M_1)_+$. On the other hand, there is a unique $v^\sharp\in C({M}_1)$ that solves $(\ref{eq_ftilde1})$ with $\tilde{f}$ replaced by $f^s_{M_1}$, and $v^\sharp=0$ on $(\partial M_1)_+$. Therefore $f^\sharp=f^s_{M_1}-\mathrm{d}^\mathcal{B}v^\sharp$, with $v^\sharp=w+v$. Since the coefficients in the system $(\ref{eq_ftilde1})$ are analytic, and so are $f^s_{M_1}$ and $\partial M_1$, tensor field $v^\sharp$ is analytic in ${M}_1\setminus E$, where $E \subset \p M_1$ is the set where $e_3$ is tangential to $\p M$. Thus $f^\sharp$ is analytic in ${M}_1\setminus E$. Due to the fact that $f^\sharp=0$ in $U$ containing $E$, and by analytic continuation, $f^\sharp=0$ in ${M}_1$.

\medskip
We have proven $f^s_{M_1}=\mathrm{d}^\mathcal{B}v^\sharp$ in ${M}_1$, and $v^\sharp=0$ on $(\partial M_1)_+$.  Soon we show that $v^\sharp=0$ also on the complement of $(\partial M_1)_{+}$. If this holds, then we have
\[
\LaB v^\sharp=\deB f^s_{M_1}=0, \quad \hbox{ in } M_1, \quad v^\sharp|_{\p M_1}=0.
\]
Thus Lemma \ref{existence2} implies $v^\sharp=0$, and $f^s_{M_1}=0$.
To prove that $v^\sharp=0$ on $\partial M_1$ we proceed as follows: Let $x\in (\partial M_1)_+$  and $y \in (\partial M_1\setminus (\partial M_1)_+)$. Since $(M_1,g)$ is simple there exists a unique geodesic $\gamma$ connecting $y$ to $x$. Let $\eta \in T_y M_1$ be perpendicular to $\dot \gamma(0)$. Since  $v^\sharp=0$ in $(\partial M_1)_+$ and $Lf^s_{M_1}=0$ we  have the following equation by \eqref{integral1}
\[
v_i^\sharp\eta^i=\langle L \dB v^\sharp, \eta \rangle=\langle Lf^s_{M_1}, \eta \rangle=0.
\]
Perturbing $x$ in the open set $(\partial M_1)_+\subset \p M$ we can show that the previous equation holds for any $\eta$ in the linearly independent set $(\eta_i)_{i=1}^3$. Thus $v^\sharp(y)=0$.

\medskip
So far we have shown $f=\dB v_{M_1}$. Since $\mathrm{supp}\,f\subset{M}$, we have that $\mathrm{supp}\,v_{M_1}\subset M$. Therefore $v_{M_1}=v_{M}$, by the following lemma that is proven analogously to \cite[Proposition 4.3]{SU}. This gives $f^s_M=0$ and completes the proof of the theorem.
\end{proof}

%This result is similar to \cite[Proposition 4.3]{SU}. 
\begin{lemma}
Let $f=\mathrm{d}^\mathcal{B}v$, $v\vert_{\partial M}=0$, and $v\in C^1(M)$. Then $v(y)=0$ for any $y$ such that $f(y)=0$, and $y$ can be connected to a point on $\partial M$ by a path that does not intersect $\mathrm{supp}\,f$.
\end{lemma}
%\begin{proof}
%Let $y$ be as above. There exists a polygon $p=\gamma_1\cup\gamma_2\cup\dots\cup\gamma_m \subset (M^{int}\setminus \supp f)$, where each segment $[0,1]\ni t\mapsto \gamma_j$, $j=1,\cdots,m$ is a geodesic, and $p$ connects $y$ to some $z\in\partial M$. Using the condition $v\vert_{\partial M}=0$ and \eqref{integral1}, we get $v_i(\gamma_1(1))\eta^i(1)=0$. By perturbing the initial point $z=\gamma_1(0)$ of $\gamma_1$ a little bit, and using the simplicity assumption, we get $v(\gamma_1(1))=0$. Similarly, we can get $v=0$ near $\gamma_1(1)$. We then repeat the same argument for $\gamma_2$, $\gamma_3$ and so on, and get $v(y)=0$.
%\end{proof}

\section{Generic s-injectivity}
\label{Se:generic_s_inj}
In this section we prove Theorem \ref{maintheorem} for $1+1$ tensor fields using the Fredholm property \eqref{parametrix} of the normal operator, and the s-injectivity result for analytic metrics.  We note that by possibly conjugating all the operators with $ \kappa_g^\sharp$ from left and $\kappa_g^\flat$ from right, we can work with the space $S\tau_M\otimes^\mathcal{B} S\tau'_M$, of trace-free $(1,1)$-tensor fields that is defined independent of any metric structure. 

We are ready to present the proof of Theorem \ref{maintheorem}.
\begin{proof}[Proof of Theorem \ref{maintheorem} ]
%\textbf{Should we just cite \cite[Theorem 1.5]{SU} and skip the proof here?}
Let $g \in C^m(M)$ be a simple metric. By formula \eqref{parametrix}, in Theorem \ref{th:stability} we have
\[
\mathcal{QN}=\mathcal{S}+\mathcal{K}, \quad \quad \mathcal N:=\mathcal{N}_L, \: \mathcal S:=\mathcal{S}_{M_2},
\]
with $\mathcal{SQ}=\mathcal{Q}$, $\mathcal{NS}=\mathcal{N}$. After applying $\mathcal{S}$ from the left to the above identity, we have
\[
\mathcal{QN}
%=\mathcal{SQN}
=\mathcal{S}+\mathcal{SK}.
\]
Thus $\mathcal{K}=\mathcal{SK}$ and similarly $\mathcal{KS}=\mathcal{K}$. As $\mathcal{S}$ is self adjoint we also have $\mathcal{K}^*=\mathcal{S}\mathcal{K}^*=\mathcal{K}^*\mathcal{S}$. If we set $\widetilde{\mathcal{Q}}:=\mathcal{S}(\mathrm{Id}+\mathcal{K}^*)\mathcal{Q}$, then previous observations yield
\[
\begin{split}
\widetilde{\mathcal{Q}}\mathcal{N}=&\mathcal{S}(\mathrm{Id}+\mathcal{K}^*)\mathcal{QN}
=\mathcal{S}(\mathrm{Id}+\mathcal{K}^*)(\mathrm{Id}+\mathcal{K})
%=&\mathcal{S}+\mathcal{S}\mathcal{K}^*+\mathcal{SK}+\mathcal{S}\mathcal{K}^*\mathcal{K}
%\\
=\mathcal{S}+\mathcal{K}^*+\mathcal{K}+\mathcal{K}^*\mathcal{K}
=\mathcal{S}+\mathcal{\widetilde{K}},
\end{split}
\]
where $\mathcal{\widetilde{K}}=\mathcal{K}^*+\mathcal{K}+\mathcal{K}^*\mathcal{K}$ is a compact self-adjoint operator $L^2(S\tau'_M\otimes^\mathcal{B} S\tau'_M)\to\mathcal{S}L^2(S\tau'_M\otimes^\mathcal{B} S\tau'_M)$. This implies
\begin{equation}
\label{eq:fredhoml_op}
\widetilde{\mathcal{Q}}\mathcal{N}+\mathcal{P}=\mathrm{Id}+\mathcal{\widetilde{K}}\quad\text{ on }L^2(S\tau'_M\otimes^\mathcal{B} S\tau'_M).
\end{equation}
%We note that by possibly conjugating all the above operators with $ \kappa_g^\sharp$ from left and $\kappa_g^\flat$ from right, we can work with the space $S\tau_M\otimes^\mathcal{B} S\tau'_M$. 

%\textbf{(TS)[I think that the rest of the proof requires still some more details. I can work on those]}
%{\color{blue}
We are ready to show that the set of $s$-injective metrics is open in $C^m$-topology, for any $m\in \N$ that is large enough. In the following we will indicate the dependence on $g$ by placing the subscript $g$ on the operators $\mathcal{N}$, $\mathcal{S}$, etc.. Suppose that $L_{g_0}$ is $s$-injective for some simple metric $g_0\in C^m({M})$. Then ${\mathcal{N}}_{g_0}$ is $s$-injective as well, and moreover the operator on right hand side of \eqref{eq:fredhoml_op} has a finite dimensional kernel on the space of solenoidal tensor fields. By using the $s$-injectivity of $\mathcal{N}_{g_0}$, as in the proof of \cite[Theorem 1.5]{SU}, we can construct a finite rank operator $\mathcal{Q}_0:L^2(S\tau'_{M_1}\otimes^\mathcal{B} S\tau'_{M_1})\rightarrow L^2(S\tau'_M\otimes^\mathcal{B} S\tau'_M)$ such that
\begin{equation}
\label{eq:fredhoml_op_2}
\mathrm{Id}+\mathcal{K}_{g_0}^\sharp
%=(\mathrm{Id}+\widetilde{\mathcal{K}}_{g_0})+\mathcal{Q}_0{\mathcal{N}_L}_{g_0}
=(\widetilde{\mathcal{Q}}_{g_0}+\mathcal{Q}_0){\mathcal{N}}_{g_0}+\mathcal{P}_{g_0} \quad \hbox{ on } L^2(S\tau'_M\otimes^\mathcal{B} S\tau'_M),
\end{equation}
is one-to-one, where $\mathcal{K}_{g_0}^\sharp:=\widetilde{\mathcal{K}}_{g_0}+\mathcal{Q}_0{\mathcal{N}}_{g_0}$ is compact. Thus according to Fredholm alternative  $(\mathrm{Id}+\widetilde{\mathcal{K}}_{g_0})^{-1}$ is bounded. We choose $f\in H^1(S\tau'_{M_1}\otimes^\mathcal{B} S\tau'_{M_1})$ and apply the operator $\mathrm{Id}+\mathcal{K}_{g_0}^\sharp$ to the solenoidal part $f^s_{M,g_0}$ of $f$ to obtain
\[
\begin{split}
\|f^s_{M,g_0}\|_{L^2(M)}\leq &C\left(\|\widetilde{\mathcal{Q}}_{g_0}{\mathcal{N}}_{g_0}f\|_{\tilde{H}^1(M_1)}+\|\mathcal{Q}_0{\mathcal{N}}_{g_0}f\|_{L^2(M_1)}\right) \leq  C\|{\mathcal{N}}_{g_0}f\|_{\tilde{H}^2(M_1)}.
%,\quad \hbox{ for any }f\in H^1(S\tau_M\otimes^\mathcal{B} S\tau'_M).
\end{split}
\]
This is the stability estimate of Theorem \ref{maintheorem} for $g=g_0$. Next verify the same estimate, with uniform $C$ for $g\in C^m(M)$ that is close enough to $g_0$ with respect to $C^m$-topology, for any $m$ large enough. To do this we first write analogously
\begin{equation}\label{eq:fredholm}
(\widetilde{\mathcal{Q}}_{g}+\mathcal{Q}_0){\mathcal{N}}_{g}+\mathcal{P}_{g}=\mathrm{Id}+\mathcal{K}_{g}^\sharp.
\end{equation}
We note here that the finite rank operator $Q_0$ is the same as in \eqref{eq:fredhoml_op_2}, and the  compact operator $\mathcal{K}_{g}^\sharp:=\widetilde{\mathcal{K}}_{g}+\mathcal{Q}_0{\mathcal{N}}_{g}$, as an operator in $L^2(S\tau'_M\otimes^\mathcal{B} S\tau'_M)$, depends continuously on $g$. Therefore if $k$ is large enough, it holds that the operator $\mathrm{Id}+\mathcal{K}_{g}^\sharp$ remains invertible, with a uniform bound for its inverse, whenever $g$ is close enough to $g_0$ in $C^m$-topology. After applying \eqref{eq:fredholm} to $f=f^s_{M,g}\in H^1(S\tau_{M_1}\otimes^\mathcal{B} S\tau'_{M_1})$, we have
\[
\begin{split}
\|f^s_{M,g}\|_{L^2(M)}
%\leq & C\left(\|\widetilde{\mathcal{Q}}_{g} {\mathcal{N}}_{g}f\|_{L^2(M)}+\|\mathcal{Q}_0 {\mathcal{N}}_{g}f\|_{L^2(M)}\right)
%\\
\leq C \left(\|{\mathcal{N}}_{g}f\|_{\tilde{H}^2(M_1)}+\|{\mathcal{N}}_{g}f\|_{L^2(M_1)}\right)
%\\
 \leq C\|{\mathcal{N}}_{g}f\|_{\tilde{H}^2(M_1)},
\end{split}
\]
with $C>0$ independent of $g$ in a small neighborhood of $g_0$ in $C^m$-topology. This implies that also $g$ is $s$-injective.

The proof of the theorem can be completed by using $s$-injectivity of $L_g$ for analytic metric $g$ (Theorem \ref{sinjectivity_analytic}), and the fact that analytic metrics are dense in $C^m({M})$.
\end{proof} 

\section{Ellipticity of the normal operator and adaptation of the proofs  for $2+2$ tensor fields}
\label{Se:2+2_case}
In this section, we will first show the ellipticity of the normal operator for $2+2$ tensors (restricted to the subspace of solenoidal tensors). To be more specific, we will show that the operator $\mathcal M=(|D|_g\mathcal{N}_L,\,\mathcal{P}_{M_2})^T$ is elliptic. 
Then we will sketch adaptions needed to prove theorems \ref{maintheorem}, \ref{th:stability} and \ref{sinjectivity_analytic} for $2+2$ tensor fields.

\subsection{Parametrix of the normal operator for solenoidal $2+2$ tensor fields}
In the following we study the action of the principal symbol $\sigma(\mathcal{N}_L)$. We note that in $1+1$ case, the principal symbol \eqref{eq:Msymbol_2} can be written as
\[
\begin{split}
\sigma(\mathcal{N}_L)^{ijk\ell}(x,\xi)
=&
%M^{ijk\ell}(x,x,\xi)=
-2\sqrt{\det g(x)}\int_{\R^3} e^{-i \xi\cdot z}\bigg(\delta^k_{u}-\frac{z_{u}z^{k}}{|z|_g}\bigg)g^{uu'}(x)
 \bigg(\delta_{u'}^{i}-\frac{z^iz_{u'}}{|z|_g^2}\bigg)\frac{z^jz^\ell}{|z|_g^{4}} \;\mathrm{d}z
\\
%=&-2\int_{\R^3}e^{-iH\xi \cdot y}\bigg(\delta^k_{u}-\frac{z_{u}z^{k}}{|y|_e}\bigg)
% \bigg(\delta^u_{i}-\frac{z_iz^{u}}{|y|_e^2}\bigg)\frac{z_jz^\ell}{|y|_e^{4}} \mathrm{d}y
%\\
%=&-2\pi\int_{S^2}\bigg(\delta^k_{u}-z_{u}z^{k}\bigg)
% \bigg(\delta^u_{i}-z_iz^{u}\bigg)z_jz^{\ell} \delta(y \cdot H\xi ) \mathrm{d}(y)
%% \\
%%=&-2\pi\int_{S^2}\bigg(\delta^k_{u}-z_{u}z^{k}\bigg)
%% \bigg(\delta^u_{i}-z_iz^{u}\bigg)z_jz^{\ell} \delta(y \cdot H\xi ) \mathrm{dS^2}(y).
%% % \quad \xi_0=\frac{\xi}{|\xi|_g}.
%%\end{split}
%%\] 
%%%Here we used the fact $\int_{S^{n-1}}f(\omega)\delta(\xi \cdot \omega)\mathrm{d}\omega=\frac{1}{|\xi|_e}\int_{S^{n-1}}f(\omega)\delta(\frac{\xi}{|\xi|_e} \cdot \omega)\mathrm{d}\omega$. 
%%
%%\[
%%\begin{split}
%\\
%=&-2\pi\int_{S_xM}\bigg(\delta^k_{u}-\omega_{u}\omega^{k}\bigg)
% \bigg(\delta^u_{i}-\omega_i\omega^{u}\bigg)\omega_j\omega^{\ell} \delta(\langle \omega,\xi \rangle_g)\mathrm{d}\omega
%\\
=&\frac{-2\pi}{|\xi|_g}\int_{S_xM\cap \xi^\perp}\left(\delta_u^k-\omega^k\omega_u\right)g^{uu'}(x)\left(\delta^i_{u'}-\omega^i\omega_{u'}\right)\omega^j\omega^\ell\mathrm{d}\omega.
\end{split}
\]
%Where $\sqrt{\det H}\; \mathrm{d}\mathbb S^2=\mathrm{d}S_xM:=\mathrm{d}\omega$. 
%To obtain the last equation we used \eqref{eq:tildeM}, \eqref{eq:M}, 
%\\
%$\int_{\mathbb S^{n-1}}\delta(x\cdot \xi)\mathrm{d}x=\frac{\hbox{Vol}(\mathbb S^{n-2})}{|\xi|_e}$ and change of coordinates $z^j=\sqrt{g^{ij}(x)}y_i$. 
We recall the notation $(P_{\omega})^i_j:=\delta_j^i-\omega^i\omega_j$. Thus for $f\in T_x'M\otimes^\mathcal{B} T_x'M$ we have
\[
\begin{split}
\langle \sigma(\mathcal{N}_L)(x,\xi) f, f\rangle_g
%=&\frac{-2\pi}{|\xi|}\int_{S_xM\cap \xi^\perp}g^{uu'}(P_\omega)_{u'}^{i'}\omega^{j'}(P_\omega)_{u}^{i}\omega^{j}f_{ij}\bar{f}_{i'j'}\mathrm{d}\omega
%\\
=&\frac{-2\pi}{|\xi|_g}\int_{S_xM\cap \xi^\perp}|(P_\omega)_{u}^{i}\omega^{j}f_{ij}|_g^2\mathrm{d}\omega.
\end{split}
\]
We do not derive an explicit formula for the principal symbol of the normal operator in the case of $2+2$ tensors, but sketch the main steps to conclude that  for any $f\in S^2T_x'M\otimes^\mathcal{B} S^2T_x'M$, we have analogously to the $1+1$ case,
\begin{equation}
\label{eq:2+2_prin_sym}
\langle \sigma(\mathcal{N}_L)(x,\xi) f, f\rangle_g=\frac{2\pi}{|\xi|_g}\int_{S_xM\cap \xi^\perp}|(P_\omega)_{a}^{i}(P_\omega)_{b}^{j}\omega^{k}\omega^{l}f_{ijkl}|_g^2\mathrm{d}\omega.
\end{equation}

%In the following we denote the geodesic, issued from $(z,\omega)\in \partial_+(S^M)$, by $(x(t),\omega(t))$.
Let  $f,h \in S^2\tau'_M\otimes^\mathcal{B} S^2\tau'_M$. Then, for the geodesic $\gamma$ with initial conditions $(z,\omega )\in \p_+SM$, we have
\[
\begin{split}
&\langle Lf,Lh\rangle_{L^2(\beta_2(\partial_+(SM)))}
\\
=&\int_{\partial_+(SM)}\left(\int_0^{\tau(z,\omega)}(\mathcal{T}_{\gamma}^{0,t})_{ab}^{uv}(P_{\omega(t)})_v^i(P_{\omega(t)})_u^jf_{ijk\ell}(x(t))\omega^k(t)\omega^\ell(t)\mathrm{d}t\right)
\\
&\quad\quad\quad \left(\int_0^{\tau(z,\omega)}(\mathcal{T}_{\gamma}^{0,s})_{a'b'}^{u'v'}(P_{\omega(s)})_{v'}^{i'}(P_{\omega(s)})_{u'}^{j'}\bar{h}_{i'j'k'\ell'}(x(s))\omega^{k'}(s)\omega^{\ell'}(s)\mathrm{d}s\right)
g^{aa'}(z)g^{bb'}(z)\mathrm{d}\mu(z,\omega).
 \end{split}
 \]
%This inner product decomposes to the sum quantities
%\[
% \begin{split}
% I_\pm=
% &\int_{\partial_+(SM)} \int_{\R} \int_0^\infty g^{au'}(x(s))g^{bv'}(x(s))(P_{\omega(s)})_{v'}^{i'}(P_{\omega(s)})_{u'}^{j'}\bar{h}_{i'j'k'\ell'}(x(s))\omega^{k'}(s)\omega^{\ell'}(s)
%\\
% &\quad\quad\quad\quad\quad\quad \quad (\mathcal{T}_{\gamma}^{0,s\pm t})_{ab}^{uv}(P_{\omega(s\pm t)})_v^i(P_{\omega(s\pm t)})_u^jf_{ijk\ell}(x(s\pm t))\omega^k(s\pm t)\omega^\ell(s\pm t)\; \mathrm{d}t \mathrm{d}s
% \mathrm{d}\mu(z,\omega).
% \end{split}
%\]
By an analogous argument to one in Section \ref{Se:normal_op} we show that $\mathcal{N}_L$ is an integral operator,
%\color{blue}
whose Schwartz kernel near the diagonal can be written as
\color{black} 
\begin{equation}
\label{eq:2+2_kernel}
\begin{split}
K_{ii'jj'kk'\ell'}(x,y)=&\frac{2A^{uu'vv'}(x,y)}{\left(G^{(1)}z\cdot z)\right)^3}\left[\tilde{G}^{(2)}z\right]_\ell\left[\tilde{G}^{(2)}z\right]_{\ell'}\left[G^{(2)}z\right]_j\left[G^{(2)}z\right]_{j'}\frac{|\det G^{(3)}|}{\sqrt{\det g(x)}}
\\
\times& \left(g_{ku}(y)-\frac{\left[\tilde{G}^{(2)}z\right]_k\left[\tilde{G}^{(2)}z\right]_u}{G^{(1)}z\cdot z}\right)\left(g_{k'v}(y)-\frac{\left[\tilde{G}^{(2)}z\right]_{k'}\left[\tilde{G}^{(2)}z\right]_{v}}{G^{(1)}z\cdot z}\right)
\\
\times &\left(g_{iu'}(x)-\frac{\left[G^{(2)}z\right]_i\left[G^{(2)}z\right]_{u'}}{G^{(1)}z\cdot z}\right)\left(g_{i'v'}(x)-\frac{\left[G^{(2)}z\right]_{i'}\left[G^{(2)}z\right]_{v'}}{G^{(1)}z\cdot z}\right), \quad z=x-y,
\end{split}
\end{equation}
where 
\[
A^{uu'vv'}(x,y)=g^{au'}(x)g^{bv'}(x)\left(\mathcal{T}_{\gamma_{x,-\nabla^g_x\rho(x,y)}}^{0,\rho(x,y)}\right)^{uv}_{ab}.
\]
Therefore $\mathcal{N}_L$ is a $\Psi$DO of order $-1$, and formula \eqref{eq:2+2_prin_sym} is valid.

For now on we use the short hand notation $\mathcal{P}=\mathcal{P}_{M_2}$ and aim to show that 
\[
\sigma(\mathcal M) f:=\left(\begin{array}{c}
|\xi|_g\circ \sigma(\mathcal{N}_L) f
\\
\sigma(\mathcal{P})f
\end{array}\right)
=0, \quad \hbox{} (x,\xi)\in T^\ast M_2\setminus\{0\},
\]
implies $f=0$, which proves that the zeroth order operator $\mathcal{M}=(|D|_g\mathcal{N}_L,\,\mathcal{P})^T $ is elliptic.

Let $\xi\in T_xM$. We choose $\omega,\,\tilde{\omega}\in S_xM$ such that $\{\hat{\xi}:=\frac{\xi}{|\xi|_g},\omega,\,\tilde{\omega}\}$ is an orthonormal basis of $T_xM$. We also simplify
\[
Q^{ijkl}_{ab}(\omega):=(P_\omega)_{a}^{i}(P_\omega)_{b}^{j}\omega^{k}\omega^{l}.
\]
If $(|\xi|_g\circ \sigma(\mathcal{N}_L)(x,\xi),\sigma(\mathcal{P}))^T f=0$ , then
\[
j^\mathcal B_\xi \sigma(\mathcal{P})=j^\mathcal B_\xi i^\mathcal B_\xi \sigma((\LaB)^{-1})j^\mathcal B_\xi =j^\mathcal B_\xi,
\]
where $j^\mathcal B_\xi, i^\mathcal B_\xi $ are given in \eqref{eq:i_xi} and \eqref{eq:j_xi}, implies 
\begin{equation}
\label{eq:hat_xi_f}
\hat{\xi}^\ell f_{ijk\ell}=\hat{\xi}^kf_{ijk\ell}=0.
\end{equation}
Thus it suffices to prove that $ f_{ijk\ell}\omega^\ell=0$ and $f_{ijk\ell}\widetilde \omega^\ell=0$. 

Next we note that \eqref{eq:2+2_prin_sym} gives
\[
Q(\omega)f:=Q^{ijkl}_{ab}(\omega)f_{ijkl}=Q^{ijkl}_{ab}(\tilde{\omega})f_{ijkl}=Q^{ijkl}_{ab}\left(\frac{\tilde{\omega}+\omega}{\sqrt{2}}\right)f_{ijkl}=0.
\]
Therefore we have
\[
\hat{\xi}^jf_{ijkl}\omega^k\omega^l=\hat{\xi} Q(\omega)f=0,\quad\quad \hat{\xi}^jf_{ijkl}\tilde{\omega}^k\tilde{\omega}^l=\hat{\xi} Q(\tilde{\omega})f=0,
\]
\[
\hat{\xi}^jf_{ijkl}\omega^k\tilde{\omega}^l=\hat{\xi} Q\left(\frac{\omega+\tilde{\omega}}{\sqrt{2}}\right)f-\frac{1}{2}\hat{\xi}^jf_{ijkl}\omega^k\omega^l-\frac{1}{2}\hat{\xi}^jf_{ijkl}\tilde{\omega}^k\tilde{\omega}^l=0,
\]
and
\begin{equation}
\label{eq:omega_Q_tilde_omega}
\tilde{\omega}^jf_{ijkl}\omega^k\omega^l=\tilde{\omega}Q(\omega)f=0,\quad \omega^jf_{ijkl}\tilde{\omega}^k\tilde{\omega}^l=\omega Q(\tilde{\omega})f=0.
\end{equation}
Since we assumed $f \in S^2T_x'M\otimes^\mathcal{B} S^2T_x'M$, the trace-free condition $\mu f=0$
%\[
%0=\mu f= \sum_{j=1}^3f_{ijkj}=\sum_{i=1}^3f_{iji\ell}
%\]
and \eqref{eq:hat_xi_f}  yield
\[
f_{ijkl}\omega^i\omega^k=f_{ijkl}\omega^j\omega^l=-f_{ijkl}\tilde{\omega}^j\tilde{\omega}^l=-f_{ijkl}\tilde{\omega}^i\tilde{\omega}^k.
%=f_{ijkl}\tilde{\omega}^i\tilde{\omega}^j\tilde{\omega}^k\tilde{\omega}^l.
\]
After applying \eqref{eq:omega_Q_tilde_omega} to previous equation we get
\[
f_{ijkl}\omega^i\omega^j\omega^k\tilde{\omega}^l=f_{ijkl}\omega^i\tilde{\omega}^j\tilde{\omega}^k\tilde{\omega}^l=
f_{ijkl}\tilde{\omega}^i\tilde{\omega}^j\omega^k\tilde{\omega}^l=f_{ijkl}\tilde{\omega}^i\omega^j\tilde{\omega}^k\omega^l=0.
\]
Then we compute
\[
\begin{split}
& 2(\omega-\tilde{\omega})(\omega-\tilde{\omega})Q\left(\frac{\omega+\tilde{\omega}}{\sqrt{2}}\right)f\\
=&f_{ijkl}(\omega-\tilde{\omega})^i(\omega-\tilde{\omega})^j(\omega+\tilde{\omega})^k(\omega+\tilde{\omega})^l\\
=&f_{ijkl}\omega^i\omega^j\omega^k\omega^l+f_{ijkl}\tilde{\omega}^i\tilde{\omega}^j\tilde{\omega}^k\tilde{\omega}^l-4f_{ijkl}\omega^i\tilde{\omega}^j\omega^k\tilde{\omega}^l\\
=&f_{ijkl}\omega^i\omega^j\omega^k\omega^l+5f_{ijkl}\tilde{\omega}^i\tilde{\omega}^j\tilde{\omega}^k\tilde{\omega}^l\\
=&6f_{ijkl}\omega^i\omega^j\omega^k\omega^l\\
=&0.
\end{split}
\]
Therefore we can conclude that
%\[
%f_{ijkl}\omega^i\tilde{\omega}^j\omega^k\tilde{\omega}^l=f_{ijkl}\omega^i\omega^j\omega^k\omega^l=f_{ijkl}\tilde{\omega}^i\tilde{\omega}^j\tilde{\omega}^k\tilde{\omega}^l=0.
%\]
%Now we can conclude
$f=0$.

The rest of the proof of Proposition \ref{Pr:1_reconst_result} for $2+2$ tensor fields is as presented earlier.

\subsection{A sketch of proof for Theorem \ref{th:stability} in $2+2$ case}
%For now we assume that $f\in C^\infty_0(S^2\tau'_{M}\otimes^\mathcal{B} S^2\tau'_{M})$. 
For $x\in M_1\setminus M$, choose $\xi$ such that the geodesic $\gamma=\gamma_{x,\xi}$ hits the boundary $\partial M_1$ before $\partial M$ and minimizes the distance between $x$ and $\partial M_1$.
%\textbf{This needs to be verified!}
 As we have proven Proposition \ref{Pr:1_reconst_result} for $2+2$ tensor fields, formula \eqref{eq:potential_as_smoothing_op} holds and  \eqref{integral1} is to be replaced by
\begin{equation}
\label{eq:T_2_new}
[v_{M_1}(x)]_{ijk}\eta^i\eta^j\xi^k=-
\int^{\tau}_0[\mathrm{d}^\mathcal{B}v_{M_1}(\gamma(t))]_{ijk\ell}\eta^i(t)\eta^j(t)\dot{\gamma}^k(t)\dot{\gamma}^\ell(t)\mathrm{d}t,
\end{equation}
where $\eta\perp\xi$. Choose $\eta,\widetilde{\eta}$ such that $B=\{\eta,\widetilde{\eta},\xi\}$ form an orthonormal basis of $T_x(M_1\setminus M)$.

Therefore we have
\[
|[v_{M_1}(x)]_{ijk}\eta^i\eta^j\xi^k|\leq C \left|(\mathcal{T}_1\mathcal{N}_Lf-\mathcal{K}_2f)(x)\right|_{g}.
\]
 We need to show that there exists $C>0$, uniform for any $x\in M_1 \setminus M^{int}$, such that
\begin{equation}
\label{eq:component_estimate}
|[v_{M_1}(x)]_{ijk}w_{m_1}^iw_{m_2}^jw_{m_3}^k| \leq C \left|(\mathcal{T}_1\mathcal{N}_Lf-\mathcal{K}_2f)(x)\right|_{g}, \quad w_{m_h} \in B.
\end{equation}
As $|v_{M_1}(x)|_g^2$ can be estimated by the sum of all the terms $\left|[v_{M_1}(x)]_{ijk}w_{m_1}^iw_{m_2}^jw_{m_3}^k\right|^2$, the following $L^2$-estimate holds
\[
\|v_{M_1}\|_{L^2(M_1\setminus M)}\leq C\|\mathcal{T}_1\mathcal{N}_Lf-\mathcal{K}_2f\|_{L^2(M_1\setminus M)}.
\]
%if estimate \eqref{eq:component_estimate} is valid. 
To prove \eqref{eq:component_estimate} we need to repeat the steps between \eqref{eq:u_in_polari} and \eqref{eq:3systems}. First, we have
\[
[v_{M_1}(x)]_{ijk}\left((\eta+\widetilde \eta)^i(\eta+\widetilde \eta)^j\xi^k-\eta^i\eta^j\xi^k-\widetilde \eta^i\widetilde \eta^j\xi^k\right)=2[v_{M_1}(x)]_{ijk}\eta^i \widetilde \eta^j\xi^k,
\]
Then
\[
|[v_{M_1}(x)]_{ijk}\eta^i\widetilde{\eta}^j\xi^k|\leq C \left|(\mathcal{T}_1\mathcal{N}_Lf-\mathcal{K}_2f)(x)\right|_{g}.
\]

For $x$ in a neighborhood of $x_0\in\partial M$, there exists $\epsilon_0>0$ such that $\gamma_{x,\xi-\epsilon\eta}$ meets $\partial M_1$ before meeting $\partial M$ for any $\epsilon<\epsilon_0$. Then we can obtain
\[
|[v_{M_1}(x)]_{ijk}(\eta+\epsilon\xi)^i(\eta+\epsilon\xi)^j(\xi-\epsilon\eta)^k|\leq C \left|(\mathcal{T}_1\mathcal{N}_Lf-\mathcal{K}_2f)(x)\right|_{g}.
\]

%\textbf{I like this way of writing. It is much more compact that what I wrote. However I don't quite understand how do you use the four different $\epsilon$ to prove the estimate below. Could you please explain more? In the eralier case we showed that a polynomial of order $3$ in variable $\epsilon$ has infinite amount of roots thus all the multipliers have to vanish. }

Choosing four distinct real numbers $0<\epsilon_1,\epsilon_2,\epsilon_3,\epsilon_4<\epsilon_0$, by invertibility of the Vandermonde matrix
\[
\left(\begin{array}{cccc}
1 &\epsilon_1 &\epsilon_1^2 &\epsilon_1^3\\
1 &\epsilon_2 &\epsilon_2^2 &\epsilon_2^3\\
1 &\epsilon_3 &\epsilon_3^2 &\epsilon_3^3\\
1 &\epsilon_4 &\epsilon_4^2 &\epsilon_4^3\\
\end{array}\right),
\]
we have the estimates
\[
|u_{ijk}\xi^{i}\xi^{j}\eta^k|,
\left|u_{ijk}\left(\xi^i\xi^j\xi^k-2\eta^{i}\xi^{j}\eta^k\right)\right|,
\left|u_{ijk}\left(2\eta^i\xi^j\xi^k-\eta^i\eta^j\eta^k\right)\right|\leq C\left|(\mathcal{T}_1\mathcal{N}_Lf-\mathcal{K}_2f)(x)\right|_{g}.
\]
Here, the constant $C$ depends on $\epsilon_1,\epsilon_2,\epsilon_3,\epsilon_4$, which could be chosen such that $C$ is uniform in a neighborhood of $x_0$. One can then just continue the steps and get the estimates
\eqref{eq:component_estimate} with $C$ uniform in a neighborhood of $x_0$. We omit the details here. Finally, by a compactness argument, we have \eqref{eq:component_estimate}  with $C$ uniform in $M_1\setminus M$.

\medskip
Next we estimate the $H^1$-norm of $v_{M_1}$ in $M_1\setminus M$. As earlier we can estimate $|\nabla v_{M_1}|^2_g$ by the sum of all terms
\[
\left|w_{m_4}^\ell\nabla_\ell[v_{M_1}(x)]_{ijk}w_{m_1}^iw_{m_2}^jw_{m_3}^k\right|^2, \quad w_{m_h}\in B.
\]  Recall that we have
\[
w^\ell \nabla_\ell[v_{M_1}]_{ijk}w_{m_1}^iw_{m_2}^jw^k= \left[\mathrm{d}^\mathcal{B}v_{M_1}\right]_{ijk\ell}w^i_{m_1}w_{m_2}^jw^kw^\ell,
\]
if $w\neq w_{m_1},w_{m_2}$.
We only need to estimate the terms
%\color{blue}
\begin{equation}
\label{eq:missing_terms_1}
\widetilde w^\ell\nabla_\ell[v_{M_1}(x)]_{ijk}w_{m_1}^iw_{m_2}^jw^k, \quad  w \neq \widetilde w, \: w\neq  w_{m_h}
\end{equation}
and 
\begin{equation}
\label{eq:missing_terms_2}
\widetilde w^\ell\nabla_\ell[v_{M_1}(x)]_{ijk}\widehat{w}^iw^jw^k.
\end{equation}
\color{black}
We start with the term \eqref{eq:missing_terms_1}  and as earlier we work in boundary normal coordinates $(x',x_3)$ of $M$ near some fixed boundary point $x_0\in \p M$. 

%We assume first that $w=\dot\gamma_{(x',0),\nu}(x_3)$. Then \eqref{eq:to_be_deri} writes 
%\[
%[v_{M_1}(x)]_{ij3}w_{m_1}^iw_{m_2}^j=-\int_{x_3}^\infty [\mathrm{d}^\mathcal{B}v_{M_1}(\gamma(t))]_{ij3}w_{m_1}^i(t)w_{m_2}^j(t)\mathrm{d}t,
%\]
%and \eqref{eq:derivated_eq} has the form
%\[
%\begin{split}
%X_{(k)}^j\nabla_j[v_{M_1}(x)]_{ij3}w_{m_1}^iw_{m_2}^j
%=&-\int_{x_3}^\infty X_{(k)}\left( [\mathrm{d}^\mathcal{B}v_{M_1}(\gamma(t))]_{ij3}w_{m_1}^i(t)w_{m_2}^j(t)\right)\mathrm{d}t.
%\end{split}
%\]

%If $w=\eta$ is perpendicular to $\dot\gamma_{(x',0),\nu}(x_3)=\xi$, then
We have the following identity analogous to \eqref{eq:to_be_deri}:
\[
[v_{M_1}(x)]_{ij3}w_{m_1}^iw_{m_2}^j=-\int_{x_3}^{\infty} [\mathrm{d}^\mathcal{B}v_{M_1}(\gamma_{x,\xi}(t))]_{ij33}w_{m_1}^i(t)w_{m_2}^j(t)\mathrm{d}t, \quad w_{m_h}\in \{\eta,\widetilde \eta\},
\]
and \eqref{eq:derivated_eq} has become
\begin{equation}
\label{eq:derivated_eq_2}
\begin{split}
X_{(k)}^\ell \nabla_\ell[v_{M_1}(x)]_{ij3}w_{m_1}^iw_{m_2}^j
=&-\int_{x_3}^{\infty} X_{(k)}\left( [\mathrm{d}^\mathcal{B}v_{M_1}(\gamma_{x,\xi}(t))]_{ij33}w_{m_1}^i(t)w_{m_2}^j(t)\right)\mathrm{d}t.
%\\
%-&\left[X_{(k)}\tau(x,w)\right]\int_{0}^{\tau(x,\eta)} [\mathrm{d}^\mathcal{B}v_{M_1}(\gamma_{x,w}(t))]_{ijk}w_{m_1}^i(t)w_{m_2}^j(t)w^k(t)\mathrm{d}t.
\end{split}
\end{equation}
That is we have estimated \eqref{eq:missing_terms_1}  when $w =\xi$ and $\widetilde w\in \{\eta,\widetilde \eta\}$. To estimate for the remaining case of \eqref{eq:missing_terms_1}  we denote $\widetilde w=\xi$ and $w=\eta$. Then it must hold that $w_{m_h}\in \{\xi,\widetilde \eta\}$ and 
\[
[\mu \nabla  v_{M_1}(x)]_{jk}=
%g^{pm}\nabla_m[v_{M_1}(x)]_{pjk}=
\xi^m\nabla_m[v_{M_1}(x)]_{pjk}\xi^p+\eta^m\nabla_m[v_{M_1}(x)]_{pjk}\eta^p+\widetilde \eta^m\nabla_m[v_{M_1}(x)]_{pjk}\widetilde \eta^p.
\]
%\[
%[\delta v_{M_1}(x)]_{jk}= \xi^m\nabla_m [v_{M_1}(x)]_{jpk}\xi^p+\eta^m\nabla_m [v_{M_1}(x)]_{jpk}\eta^p+\widetilde\eta^m\nabla_m [v_{M_1}(x)]_{jpk}\widetilde \eta^p.
%\]
Straightforward computation yields
\[
\begin{split}
\left[\mathrm{d}^\mathcal{B}v_{M_1}\right]_{ijk\ell}w_{m_1}^iw_{m_2}^j\eta^k \xi^\ell
%=&\bigg(\frac{1}{2}(\nabla_\ell [v_{M_1}(x)]_{ijk}+\nabla_k [v_{M_1}(x)]_{ij\ell})
%-\frac{1}{8 a_1}\bigg( g_{i\ell}[\mu \nabla  v_{M_1}(x)]_{jk}
%\\
%+& g_{ik}[\mu \nabla  v_{M_1}(x)]_{j\ell}
%+ g_{j\ell}[\mu \nabla  v_{M_1}(x)]_{ik} +g_{jk}[\mu \nabla  v_{M_1}(x)]_{i\ell} \bigg)\bigg)w_{m_1}^iw_{m_2}^j\eta^k \xi^\ell
%\\
%=&\bigg(\frac{1}{2}(\nabla_\ell [v_{M_1}(x)]_{ijk}+\nabla_k [v_{M_1}(x)]_{ij\ell})\bigg)w_{m_1}^iw_{m_2}^j\eta^k \xi^\ell
%\\
%-&
%\frac{1}{8 a_1}\bigg( \delta_{i\ell}[\mu \nabla  v_{M_1}(x)]_{jk}w_{m_2}^j\eta^k +
%\delta_{j\ell}[\mu \nabla  v_{M_1}(x)]_{ik}w_{m_1}^i\eta^k\bigg)
%\\
=&\bigg(\frac{1}{2}(\nabla_\ell [v_{M_1}(x)]_{ijk}+\nabla_k [v_{M_1}(x)]_{ij\ell})\bigg)w_{m_1}^iw_{m_2}^j\eta^k \xi^\ell
\\
&-\frac{1}{10}\bigg( \delta_{i\ell}\left(\xi^m\nabla_m[v_{M_1}(x)]_{pjk}\xi^p+\eta^m\nabla_m[v_{M_1}(x)]_{pjk}\eta^p\right.
\\
&+\widetilde \eta^m\nabla_m[v_{M_1}(x)]_{pjk}\widetilde \eta^p \bigg)w_{m_1}^iw_{m_2}^j\eta^k \xi^\ell
\\
&+ 
\delta_{j\ell}\bigg(\xi^m\nabla_m[v_{M_1}(x)]_{pik}\xi^p+\eta^m\nabla_m[v_{M_1}(x)]_{pik}\eta^p
\\
&+\widetilde \eta^m\nabla_m[v_{M_1}(x)]_{pik}\widetilde \eta^p\bigg)w_{m_1}^iw_{m_2}^j\eta^k \xi^\ell \bigg).
\end{split}
\]
Taking $w_1,w_2\in \{\xi,\widetilde{\eta}\}$ in the above formula, we can get desired estimates for
\\
$\xi^\ell\nabla_\ell[v_{M_1}(x)]_{ijk}w_{m_1}^iw_{m_2}^j\eta^k$.

It remains to estimate the terms appearing in \eqref{eq:missing_terms_2}. Set $\widetilde w=\xi $ in \eqref{eq:missing_terms_2} and write 

\begin{equation}
\label{eq:equation_for_dB}
\begin{split}
&\left[\mathrm{d}^\mathcal{B}v_{M_1}\right]_{ijk\ell} \widehat{w}^iw^jw^k\xi^\ell \\
=&\bigg(\frac{1}{2}(\nabla_\ell [v_{M_1}(x)]_{ijk}+
\nabla_k [v_{M_1}(x)]_{ij\ell})
\bigg) \widehat{w}^iw^jw^k\xi^\ell 
\\
&-
\frac{1}{10}\bigg( \delta_{i\ell}\left(\xi^m\nabla_m[v_{M_1}(x)]_{pjk}\xi^p+\eta^m\nabla_m[v_{M_1}(x)]_{pjk}\eta^p+\widetilde \eta^m\nabla_m[v_{M_1}(x)]_{pjk}\widetilde \eta^p\right) \widehat{w}^iw^jw^k\xi^\ell
\\
&+ 
\delta_{j\ell}\left(\xi^m\nabla_m[v_{M_1}(x)]_{pik}\xi^p+\eta^m\nabla_m[v_{M_1}(x)]_{pik}\eta^p+\widetilde \eta^m\nabla_m[v_{M_1}(x)]_{pik}\widetilde \eta^p\right) \widehat{w}^iw^jw^k\xi^\ell
\\
&+ 
\delta_{ik}\left(\xi^m\nabla_m [v_{M_1}(x)]_{pj\ell}\xi^p
+\eta^m\nabla_m [v_{M_1}(x)]_{pj\ell}\eta^p+\widetilde\eta^m\nabla_m [v_{M_1}(x)]_{pj\ell}\widetilde \eta^p
\right) \widehat{w}^iw^jw^k\xi^\ell
\\
&+
\left(\xi^m\nabla_m [v_{M_1}(x)]_{pi\ell}\xi^p
+\eta^m\nabla_m [v_{M_1}(x)]_{pi\ell}\eta^p+\widetilde\eta^m\nabla_m [v_{M_1}(x)]_{pi\ell}\widetilde \eta^p
\right) \widehat{w}^i\xi^\ell  \bigg).
\end{split}
\end{equation}

We drop out all the terms in the right hand side of \eqref{eq:equation_for_dB} that do not have the normal derivative $\xi^m\nabla_m$, to obtain
\[
\begin{split}
&\frac{1}{2}\nabla_\ell [v_{M_1}(x)]_{ijk}\widehat{w}^iw^jw^k\xi^\ell 
-\frac{1}{10}\bigg( \delta_{i\ell}\xi^m\nabla_m [v_{M_1}(x)]_{pjk}\xi^p\widehat{w}^iw^jw^k\xi^\ell
\\
+ &
\delta_{j\ell}\xi^m\nabla_m [v_{M_1}(x)]_{pik}\xi^p\widehat{w}^iw^jw^k\xi^\ell
+
\delta_{ik}\xi^m\nabla_m [v_{M_1}(x)]_{pj\ell}\xi^p \widehat{w}^iw^jw^k\xi^\ell
+
\xi^m\nabla_m [v_{M_1}(x)]_{pi\ell}\xi^p \widehat{w}^i\xi^\ell  \bigg).
\end{split}
\]
If $w=\xi$ the simplified version of the right hand side of \eqref{eq:equation_for_dB} is
\[
\begin{split}
&\frac{1}{2}\nabla_\ell [v_{M_1}(x)]_{ijk}\widehat{w}^i\xi^j\xi^k\xi^\ell 
-\frac{1}{10}\bigg( \delta_{i\ell}\xi^m\nabla_m [v_{M_1}(x)]_{pjk}\xi^p\xi^j\xi^k\widehat{w}^i\xi^\ell
\\
+ &
\xi^m\nabla_m [v_{M_1}(x)]_{pik}\xi^p\widehat{w}^i\xi^k
+
\delta_{ik}\xi^m\nabla_m [v_{M_1}(x)]_{pj\ell}\widehat{w}^i\xi^p \xi^j\xi^\ell\xi^k
+
\xi^m\nabla_m [v_{M_1}(x)]_{pi\ell}\xi^p \widehat{w}^i\xi^\ell  \bigg).
\end{split}
\]
Which is always a nonzero multiple of $\nabla_\ell [v_{M_1}(x)]_{ijk}\widehat{w}^i\xi^j\xi^k\xi^\ell$, and thus can be estimated.

For $w=\eta $ or $w=\widetilde \eta$, the situations are analogous, and we only consider the first case. We get
\[
\begin{split}
\frac{1}{2}\nabla_\ell [v_{M_1}(x)]_{ijk}\widehat{w}^i\eta^j\eta^k\xi^\ell -&\frac{1}{10}\bigg( \delta_{i\ell}\xi^m\nabla_m [v_{M_1}(x)]_{pjk}\xi^p\widehat{w}^i\eta^j\eta^k\xi^\ell
\\
&+ 
\delta_{ik}\xi^m\nabla_m [v_{M_1}(x)]_{pj\ell}\xi^p\widehat{w}^i \eta^j\eta^k\xi^\ell
+
\xi^m\nabla_m [v_{M_1}(x)]_{pi\ell}\xi^p \widehat{w}^i\xi^\ell  \bigg),
\end{split}
\]
for the simplified version of the right hand side of \eqref{eq:equation_for_dB}. Here the last two terms have already been estimated and the first term vanishes if $\widehat{w}\neq \xi$. Therefore we have also found a formula for $\nabla_\ell [v_{M_1}(x)]_{ijk}\widehat{w}^i\eta^j\eta^k\xi^\ell$ that contains the  only tangential derivatives and $\mathrm{d}^\mathcal{B}v_{M_1}$. 
%We conclude that we have proven that any term in \eqref{eq:missing_terms} can be written as a linear combination of $\dB v_{M_1}$ and terms appearing in \eqref{eq:derivated_eq_2}.

As earlier we can find $C>0$ depending only on the distance to $\p M_1$, which satisfies the following pointwise estimate:
\[
\left| w_{m_4}^\ell\nabla_\ell[v_{M_1}(x)]_{ijk}w_{m_1}^iw_{m_2}^jw_{m_3}^k\right|\leq C\left( \sum_{k=1}^2|\chi\nabla_{X_{(k)}}(\mathcal{T}_1\mathcal{N}_Lf)|_g
%+|\chi\mathcal{T}_1\mathcal{N}_Lf|_g
+|\mathcal{K}_4f|_g\right), \quad  w_{m_h} \in B.
\]

\medskip
To complete the proof of the second claim of Theorem \ref{th:stability} we refine operator $\mathcal{T}_2$ using equation \eqref{eq:T_2_new}. 
%Th obtain
%\[
%\hbox{Tr}_{M}\:v_{M_1}=\mathcal{T}_2\left(\mathcal{T}_1\mathcal{N}_L-\mathcal{K}_2\right)
%\]
The rest of the proof of Theorem \ref{th:stability} is analogous to what we did earlier.

\subsection{A sketch of proof for Theorem \ref{sinjectivity_analytic} in $2+2$ case}
We sketch here the required changes needed for the proofs of Theorems  \ref{sinjectivity_analytic} and \ref{maintheorem} for $2+2$ tensors fields.

First we note that the formula \eqref{eq:2+2_kernel} implies the claim of Proposition \ref{pr:analytic_PsiDOs} for $2+2$ tensor fields. We note that Proposition \ref{analytic_parametrix_thm} is analogous to $1+1$ case, since we have proved Theorem \ref{th:stability} for $2+2$ tensor fields. Then we arrive at Lemma \ref{lemm42}, which requires some modifications.

\begin{proof}[Proof of Lemma \ref{lemm42} in $2+2$ case]
We fix $x_0\in\partial M$ and take the boundary normal coordinates $x=(x',x^3)$ in a neighborhood $U\subset M$ of $x_0$. In these coordinates we have $g_{i3}=\delta_{i3}$, in $U$, for any $i=1,2,3$. We aim to find a (trace-free) $3$-tensor field $v$, vanishing on $\partial M$, such that for $\tilde{f}:=f-\mathrm{d}^\mathcal{B}v$ we have 
\begin{equation}\label{eq:2+2_ODE_sysem0}
\tilde{f}_{ijk3}=0, \quad \hbox{ in some open }\tilde U \subset  U, \hbox{ that contains $x_0$},
\end{equation}
which is, due to \eqref{eq:dB}, equivalent to
\begin{equation}\label{eq:2+2_ODE_sysem1}
\begin{split}
&f_{ijk3}-\frac{1}{2}\nabla_3 v_{ijk}-\frac{1}{2}\nabla_k v_{ij3}
+\frac{1}{10}\bigg( \delta_{j3}g^{mn}\nabla_nv_{imk}+\delta_{i3}g^{mn}\nabla_nv_{jmk}\\
&\quad\quad\quad\quad+ g_{jk}g^{mn}\nabla_nv_{im3}+g_{ik}g^{mn}\nabla_nv_{jm3} \bigg)=0, \quad \hbox{ in } \tilde U.
%\\
%&f_{ijk3}-\frac{1}{2}\left(\nabla_3 v_{ijk}+\nabla_kv_{ij3}\right)
%\\
%&+\frac{1}{10}\bigg( g_{ik}g^{pm}\nabla_pv_{mj3}
%+g_{jk}g^{pm}\nabla_pv_{mi3}+\delta_{i3}g^{pm}\nabla_pv_{mjk}
%+\delta_{j3}g^{pm}\nabla_pv_{mik} \bigg)=0, \quad \hbox{ in } U.
\end{split}
\end{equation}
%\[
%\tilde{f}_{ij33}=0, \quad \hbox{ in } U,
%%=f_{i3}-\left(\nabla_3v_i-\frac{1}{3}\delta_{i3} (g^{k\ell} \nabla_k v_\ell)\right),
%\]
%which is, due to \eqref{eq:dB}, equivalent to
%\begin{equation}
%\label{eq:2+2_ODE_sysem}
%\begin{split}
%&f_{ij33}-\nabla_3 v_{ij3}
%-\frac{1}{5}\bigg( \delta_{i3}g^{km}\nabla_kv_{mj3}
%+\delta_{j3}g^{km}\nabla_kv_{mi3} \bigg)=0
%\end{split}, \quad \hbox{ in } U.
%\end{equation}

%\textbf{I really like this proof, that is also what I tried. I was worried that there are some terms with index $3$ in a wrong place or some misplaced $\p_\alpha$. However it seems that not.}

The order for determining the components of $v$ is quite similar to what is outlined in the proof of Lemma \ref{existence2}. Let us first consider the case $k=3$. Then the above equation becomes
\begin{equation}\label{eq:2+2_ODE_sysem}
f_{ij33}-\nabla_3 v_{ij3}
+\frac{1}{5}\bigg( \delta_{j3}g^{mn}\nabla_nv_{im3}+\delta_{i3}g^{mn}\nabla_nv_{jm3} \bigg)=0, \quad \hbox{ in } \tilde U.
\end{equation}
Remember that 
%\[
%\nabla_\ell v_{ijk}=\partial_\ell v_{ijk}-\left(\Gamma_{i\ell}^mv_{mjk}+\Gamma_{j\ell}^mv_{imk}+\Gamma_{k\ell}^mv_{ijm}\right),
%\]

\[
\nabla_\ell v_{ijk}=\partial_\ell v_{ijk}-\left(\Gamma_{i\ell}^mv_{mjk}+\Gamma_{j\ell}^mv_{imk}+\Gamma_{k\ell}^mv_{ijm}\right),
\]
and the Christoffel symbols, in the boundary normal coordinates, satisfy $\Gamma_{33}^k=\Gamma_{k3}^3=\Gamma_{3k}^3=0$. 
%
%First we set $k=3$ and solve $v_{ij3}$. In this case \eqref{eq:2+2_ODE_sysem}  has the form
%\[
%\begin{split}
%f_{ij33}-\nabla_3 v_{ij3}
%+\frac{1}{5}\bigg( \delta_{i3}g^{km}\nabla_kv_{mj3}
%+\delta_{j3}g^{km}\nabla_kv_{mi3} \bigg)=0
%\end{split}, \quad \hbox{ in } U.
%\]
If $i,j\neq 3$ we can write \eqref{eq:2+2_ODE_sysem} as an ODE system, with respect to the travel-time variable $x_3$, for the unknowns $v_{\alpha\beta3}$, $\alpha,\beta\neq 3$, which can thus be determined.

If $i=3$ and $j\neq 3$ we write \eqref{eq:2+2_ODE_sysem} in the form 
\[
\begin{split}
&
%f_{3j33}-\nabla_3 v_{j3}-\frac{1}{5}g^{km}\nabla_kv_{mj3}=
f_{3j33}-\frac{6}{5}\nabla_3 v_{3j3}
+\frac{1}{5}g^{\alpha\beta}\nabla_\alpha v_{\beta j3}=0
\end{split}, \quad \hbox{ in } U, \quad  \alpha,\beta\in \{1,2\}
\]
Thus $v_{3j3}$ can be found by solving the corresponding initial value problems. Finally we set $i=j=3$ and the system \eqref{eq:2+2_ODE_sysem} takes the form
\[
\begin{split}
&
%f_{3333}-\nabla_3 v_{333}-\frac{2}{5}g^{km}\nabla_kv_{m33}=
f_{3333}-\frac{7}{5}\nabla_3 v_{333}
+\frac{2}{5}g^{\alpha\beta}\nabla_\alpha v_{\beta 33}=0
\end{split}, \quad \hbox{ in } U, \quad  \alpha,\beta\in \{1,2\}.
\]
%Then we have found $v_{ij3}$. 
%
%\textbf{How to find $v_{ijk},\: k \neq 3$? Or can we just set them to be zero?}
%
%Next we set $k\in \{1,2\}$. Let $i=j=k=1$. Then
%
%\[
%\begin{split}
%&f_{1113}-\frac{1}{2}\left(\nabla_3 v_{111}+\nabla_1v_{113}\right)+\frac{1}{5} \left(g^{ph}\nabla_p v_{h13}\right)
%\\
%&f_{1113}-\frac{1}{2}\left(\partial_3 v_{111}-\Gamma_{13}^m \left(2v_{m11}+ v_{11m}\right)+\partial_1 v_{113}- \left(\Gamma_{11}^m2v_{m13}+ \Gamma_{13}^mv_{11m}\right)\right)
%\\
%&
%+\frac{1}{5} \left(g^{ph} \left(  \partial_p v_{h 13}-\left(\Gamma_{ph}^mv_{m13}+\Gamma_{1p}^mv_{h m3}+\Gamma_{3p}^m v_{h 1m}\right)\right)\right)
%\\
%&=0, \quad \hbox{ in } U.
%\end{split}
%\]
%
%
%
%First we study the case $i,j\neq 3$, and $i,j\neq k$. Therefore we obtain a set of equations equations 
%\[
%\begin{split}
%&f_{1123}-\frac{1}{2}\left(\nabla_3 v_{112}+\nabla_2v_{113}\right)
%\\
%&
%=f_{1123}-\frac{1}{2}\left(\partial_3 v_{112}-\left(2\Gamma_{13}^\alpha v_{\alpha 12}+\Gamma_{23}^\alpha v_{11\alpha}\right)+\partial_2 v_{113}-\left(2\Gamma_{12}^mv_{1m3}+\Gamma_{32}^mv_{11m}\right)\right)
%\\
%&=0, \quad \hbox{ in } U, \quad \alpha\in \{1,2\}.
%\end{split}
%\]
%
%\[
%\begin{split}
%&f_{2213}-\frac{1}{2}\left(\nabla_3 v_{221}+\nabla_1v_{223}\right)
%\\
%&
%=f_{2213}-\frac{1}{2}\left(\partial_3 v_{221}-\left(2\Gamma_{23}^\alpha v_{\alpha 21}+\Gamma_{13}^\alpha v_{22\alpha}\right)+\partial_1 v_{223}-\left(2\Gamma_{21}^mv_{2m3}+\Gamma_{31}^mv_{22m}\right)\right)
%\\
%&=0, \quad \hbox{ in } U, \quad \alpha\in \{1,2\}.
%\end{split}
%\]
%

Now we have determined $v_{ij3}$, next we consider the case $k\neq 3$. For $i,j\neq 3$, the equation \eqref{eq:2+2_ODE_sysem1} gives an ODE system for $v_{\alpha\beta k}$, $\alpha,\beta\in \{1,2\}$. Then take $i=3$ and $j\neq 3$, we get a system for $v_{3\alpha k}$, $\alpha\neq 3$. Finally, take $i=j=3$, we get a system for $v_{33k}$.

 Thus we have found a tensor field $v$ that vanishes at the boundary and solves \eqref{eq:2+2_ODE_sysem0}. We claim that constructed $v$ is trace-free. To see this, we multiply $g^{jk}$ to both sides of \eqref{eq:2+2_ODE_sysem1} and get
 \[
 -\frac{1}{2}\nabla_3(v_{ijk}g^{jk})+\frac{1}{10}\delta_{i3}\nabla^m (v_{mjk}g^{jk})=0.
 \]
 First take $i\neq 3$, we have 
 \[
 \nabla_3(v_{ijk}g^{jk})=\partial_3(\mu v)_i-\Gamma_{3i}^k(\mu v)_k=0.
 \]
 Since $\Gamma_{3i}^k=0$ for $k=3$, the above identity gives an ODE system for $(v_{1jk}g^{jl},v_{2jk}g^{jl})$. Consequently, $v_{ijk}g^{jk}=0$ for $i\neq 3$.
 Then we take $i=3$ and conclude that $v_{3jk}g^{jk}=0$. The claim is proved.

It is easy to see that if $f$ and $g$ are analytic near $\p M$, so is $v$.\\

Similar to the proof of Lemma \ref{lemm42}, we can show that
\begin{equation}
\label{eq:zeros1}
\eta^i\eta^j
\tilde{f}_{ijk\ell}\xi^k\xi^\ell
%=\lim_{\epsilon \to 0} \left\langle L\widetilde f\left(x_0,\frac{\dot{\gamma}_\epsilon(0)}{|\dot{\gamma}_\epsilon (0)|_g}\right), \eta_\epsilon \right\rangle
=0,
\end{equation}
at the chosen boundary point $x_0$, whenever $\xi$ is tangential to the boundary and $\eta \perp \xi$. We set $e_3=\nu(x_0)$, and $e_\alpha=\frac{\p}{\p x'^\alpha}|_{x_0}$ for $\alpha=1,2$. Setting $\xi=e_\alpha$ and $\eta=e_i$, $\alpha\in\{1,2\}$, $i\in\{1,2,3\}$, $i\neq \alpha$, then the previous equation implies
\[
\tilde{f}_{ii\alpha\alpha}=
0, \quad i\in \{1,2,3\}, \: \alpha \in\{1,2\}, \: i\neq \alpha.
\]
This means that the following terms vanish
\[
\tilde f_{1122},\, \tilde f_{2211},\, \tilde f_{3311},\, \tilde f_{3322}.
\]
Then take $\xi=e_1$, $\eta=\frac{1}{\sqrt{2}}(e_2+e_3)$, we can conclude that $\tilde f_{2311}=0$. Similarly $\tilde f_{1322}=0$. Let us summarize what we have right now:
\[
\tilde{f}_{ij\alpha\alpha}=
0, \quad i,j\in \{1,2,3\}, \: \alpha \in\{1,2\}, \: i,j\neq \alpha.
\]
Take $\xi=\frac{1}{\sqrt{2}}(e_1+e_2)$ and $\eta=e_3$, we obtain that $f_{3312}=0$.

Let $\epsilon>0$. If we set $\eta=e_1+\epsilon e_2$ and $\xi=e_2-\epsilon e_1$, then equation \eqref{eq:zeros1} implies that the coefficients of the powers of the $\epsilon$ satisfy
%\[
%\begin{split}
%0=&\tilde f_{1122}+2\epsilon \left(\tilde f_{1222}-\tilde f_{1121} \right)+\epsilon^2 \left( \tilde f_{1111} -4\tilde f_{1221} +\tilde f_{2222} \right) +2\epsilon^3 \left(\tilde f_{1211}-\tilde f_{2212}\right)+\epsilon^4\tilde f_{2211}
%\end{split}
%\]
\begin{equation}
\label{eq:zeros3}
\tilde f_{\alpha\alpha\alpha\alpha} -4\tilde f_{\alpha\beta\alpha\beta} +\tilde f_{\beta\beta\beta\beta}=0, \quad \hbox{and} \quad \tilde f_{\beta\alpha\alpha\alpha}-\tilde f_{\beta\beta\beta\alpha}=0, \quad \alpha\neq \beta, \: \alpha,\beta \in\{1,2\}.
\end{equation}
%From now on we choose $\epsilon >0$ to be so small that any geodesic $\gamma_{x_0,\xi}$ exits $M$ on $\tilde U$ if $|\langle \nu(x_0),\xi\rangle_g|<\epsilon$. Let $\delta>0$. After choosing smaller $\epsilon$ we can assume that 
%\[
%\tau(x_0,\xi)<\delta, \quad \hbox{ if } \xi \in \{\xi\in S_{x_0} M: \: 0 \leq \langle \nu(x_0),\xi\rangle_g <\epsilon\}.
%\]
%We set $\eta=e_3- \epsilon e_i$, $\xi=e_i+\epsilon e_3$, $i=1,2$ and give $\tilde f$ a zero extension outside $M$. For any $\epsilon$ small enough we find $\delta$ such that
%\[
%\langle L(\tilde f, \xi_{}), \eta_{} \rangle=\int_0^\delta \tilde f_{ijk\ell}(t)\eta_{}^i(t)\eta_{}^j(t)\xi_{}^k(t)\xi_{}^\ell(t) \: \mathrm{d}t.
%\]
%Since $\eta(t) \to e_3, \: \xi(t) \to e_i$ uniformly in $[0,\delta]$ as $\epsilon\to 0$, there exists $\epsilon'>0$ such that for any $0<\epsilon<\epsilon'$ the following holds
%\[
%\left |\eta^i\eta^j\tilde{f}_{ijk\ell}\xi^k\xi^\ell \right| =\frac{1}{\delta} \left |\int_0^\delta\eta^i\eta^j\tilde{f}_{ijk\ell}\xi^k\xi^\ell \: \mathrm{d}t\right| < \frac{2}{\delta} \left |\langle L(\tilde f, \xi), \eta \rangle \right |.
%\]
%
%Therefore $L\tilde f=0$ implies $\eta^i\eta^j\tilde{f}_{ijk\ell}\xi^k\xi^\ell=0$, and 
By \eqref{eq:2+2_ODE_sysem0} and the trace-free condition, we have $\tilde{f}_{\alpha\alpha\alpha\alpha}=-\tilde f_{\alpha\beta\alpha\beta}=\tilde f_{\beta\beta\beta\beta}$ and $\tilde f_{\beta\alpha\alpha\alpha}=-\tilde f_{\beta\beta\beta\alpha}$. Together with \eqref{eq:zeros3}, we have
\[
\tilde{f}_{\alpha\alpha\alpha\alpha}=\tilde f_{\alpha\beta\alpha\beta}=\tilde f_{\beta\alpha\alpha\alpha}=0\quad \alpha\neq \beta, \: \alpha,\beta \in\{1,2\}.
\]
Taking $\xi=e_1+\epsilon e_2$ and $\eta=e_3+e_2-\epsilon e_1$ in \eqref{eq:zeros1} and collecting coefficients of $1,\epsilon,\epsilon^2,\epsilon^3$, we have
\[
\begin{split}
-\tilde f_{3222}+2\tilde f_{3112}=0,\\
-\tilde f_{3111}+2\tilde f_{3221}=0.\\
\end{split}
\]
Together with the relation resulted from trace-free condition $\tilde f_{3222}+\tilde f_{3112}=\tilde f_{3111}+\tilde f_{3221}=0$, we obtain $\tilde f_{3222}=\tilde f_{3112}=\tilde f_{3111}=\tilde f_{3221}=0$. Therefore we can conclude that $\tilde f$ vanishes at $x_0$. Since $x_0$ was an arbitrary point in $\p M \cap U$ we have shown that $\widetilde f$ vanishes at $\p M \cap U$. 

Similar to the proof of Lemma \ref{lemm42}, we can prove
\begin{equation}
\label{eq:normal_der}
\partial^p_{x_3}\tilde{f}_{ijk\ell}\vert_{x=x_0}=0, \quad p \in \N, \quad i,j,k,\ell\in \{1,2,3\},
\end{equation}
and conclude the proof.
\end{proof}

The adaptations needed for the proof of Theorem \ref{sinjectivity_analytic} in the $2+2$ case are straightforward and therefore omitted.
The rest of the proof for Theorem \ref{maintheorem} is analogous to $1+1$ case.

\subsection*{Acknowledgements}

MVdH gratefully acknowledges support from the Simons Foundation under the MATH + X program, the National Science Foundation under grant DMS-1815143, and the corporate members of the Geo-Mathematical Imaging Group at Rice University. TS gratefully acknowledges support from the Simons Foundation under the MATH + X program and the corporate members of the Geo-Mathematical Imaging Group at Rice University. Part of this work was carried out during TS's visit to University of Washington, and he is grateful for hospitality and support. GU was partially supported by NSF, a Walker Professorship at UW and a Si-Yuan Professorship at IAS,
HKUST.

%\color{blue}
The authors want to express their gratitude to the referees for their invaluable comments which have improved this paper significantly. 
\color{black}
\appendix

\section{Linearization of anisotropic elastic travel tomography}
\label{Se:app_1}

%Seismic waves are modelled by the anisotropic elastic wave equation in
%$\R^{1+3}$. 
%This equation can be microlocally decoupled into 3different polarizations \cite{stolk2002microlocal}. 
%\color{blue}
In this appendix, we effectively study linearized travel-time tomography problems for polarized elastic waves. 
For our purposes this means the determination of some elastic parameters by measuring  
the travel times of \textit{qS}-polarized waves -- see the definition below.
% travel-times of different types of polarized waves.
\color{black}
We use the typical notations and terminologies of the seismological
literature, see for instance \cite{cerveny2005seismic}. We let
$\mathbf{C}=C_{ijkl}(x)$ be a smooth stiffness tensor on $\R^3$ which
satisfies the symmetry
\begin{equation}
\label{eq:symmetry_of_elastic_tensor}
C_{ijk\ell}(x)=C_{jik\ell}(x)=C_{k\ell ij}(x), \quad x\in \R^3.
\end{equation}
We also assume that the density of mass $\rho(x)$ is a smooth function of $x$ and define density--normalized elastic moduli
\[
\mathbf{A}=A_{ijk\ell}(x)=\frac{C_{ijk\ell}(x)}{\rho(x)}.
\]
The elastic wave operator $P$, associated with the elastic moduli $\mathbf{A}$, is
%\color{blue}
a matrix valued second order partial differential operator
\color{black}
given by
\[
P_{ik}=\delta_{ik}\frac{\p^2}{\p t^2}-\sum_{j,k}\left(A_{ijk\ell}(x)\frac{\p}{\p x^j}\frac{\p}{\p x^\ell}\right)+\hbox{lower order terms.}
\]
For every $(x,p)\in T^\ast\R^3$ we define a square matrix $\Gamma(x,p)$, by
\begin{equation}
\label{eq:Chirstoffel_matrix}
\Gamma_{ik}(x,p):=
\sum_{j,\ell}A_{ijk\ell}(x)p_jp_\ell
%a_{ijk\ell}(x)p^\ell p^j
.
\end{equation}
This is known as the \textit{Christoffel matrix}. Due to \eqref{eq:symmetry_of_elastic_tensor} the matrix $\Gamma(x,p)$ is symmetric. We also assume that $\Gamma(x,p)$ is positive definite for every $(x,p)\in T^\ast\R^3 \setminus \{0\}$. 
%\color{blue}
%which implies that $P$ is hyperbolic. 
%\color{black}

The principal symbol $\sigma(t,x,\omega, p)$  of the operator $P$ is then
%\color{blue}
a matrix-valued map
\color{black}
given by
\[
\sigma(t, x,\omega, p)=\omega^2I-\Gamma(x,p), \quad (t, x,\omega, p)\in T^\ast\R^{1+3}. 
\] 
Since the matrix $\Gamma(x,p)$ is positive definite and symmetric, it has three positive eigenvalues $G^m(x,p),$ $\: m \in \{1,2,3\}$, which are homogeneous of degree $2$ in the momentum variable $p$.

%\color{blue}
%It was shown in \cite{stolk2002microlocal} that the anisotropic elastic wave equation $P_{ik}u^k=f_i$ can be always microlocally decoupled into three independent pseudo differential equations ($\Psi$DE) for scalar fields on $\R^{1+3}$ such that the principal symbols of the corresponding pseudo differential operators are the scalar function $(\omega^2-G^m(x,p))_{m=1}^3$ on $\R^{1+3}\times \R^{1+3}$. The solutions of these scalar $\Psi$DEs are called $qP$, $qS_1$ and $qS_2$ polarized waves. 
%%We note that if the medium is isotropic then the aforementioned $\Psi$DEs are just regular scalar wave equations where the spatial components are given by the corresponding Lam\'e parameters, that encode the elastic properties. 
%\color{black}

We assume that 
\begin{equation}
\label{eq:nonvanishing_lamda_der}
G^1(x,p) > G^{m}(x,p), \quad  m \in \{2,3\}\hbox{, $(x,p) \in T^\ast\R^3\setminus \{0\}$}.
\end{equation}
It was shown in \cite{de2019inverse} that $\sqrt{G^1}$ is a Legendre
transform of some Finsler metric $F$. Thus the bicharacteristic curves of $\omega^2-G^1(x,p)$ are given by the co-geodesic flow of $F$.  
%\color{blue}
We recall that a bicharacteristic curve is a smooth curve on $T^\ast\R^{1+3}$, on which $\omega^2-G^1(x,p)$ vanishes, that solves the Hamilton's equation for the Hamiltonian $\omega^2-G^1(x,p)$. We consider a second order pseudo-differential operator $\square_P:=\frac{\p^2}{\p t^2}-G^1(x,D), \:D:= \mathrm i(\p_{x_1}, \p_{x_2}, \p_{x_3})$. Since $G^1$ is related to a Finsler metric the operator $\square_P$ is of real principal type (the bi-characteristic curves exit any compact set). 
The solutions $u$ of the corresponding scalar PDE $\square_P u=f$ represent $qP$-waves (quasi-pressure waves) and moreover the wavefront set  of $u(t,\cdot)$ propagates along the bicharacteristics of $\omega^2-G^1(x,p)$ \cite{Hormander2, GrUh}. 

Next, we describe the propagation of the slower $qS_1$ and $qS_2$
waves (quasi-shear waves), that are given as solutions to the scalar
equations with the operators $\square_{S_m} := \frac{\p^2}{\p
  t^2}-G^m(x,D), \: m\in\{2,3\}$.  In the anisotropic case the unit
level sets $(G^m)^{-1}\{1\} \subset T^\ast\R^3, \: m\in \{2,3\}$, also
referred to as slowness surfaces, 
%(in dimension $3$) 
typically will have points in common. See \cite{crampin1981shear} for a study of different types of intersections.
%A revised proof will be presented in \cite{slowness}. It is based on a contradiction argument, using that a tangent line bundle on $S^2$ does not exist.} 
The size and codimension of their intersection set depends on the additional symmetries that the stiffness tensor may have. Thus in general the smaller eigenvalues $G^2,\: G^3$ are only continuous. We denote by $D_c=(G^2)^{-1}\{1\}
\cap (G^3)^{-1}\{1\}$ the set of degenerate eigenvalues, and note that
outside this set $G \in \{G^2,G^3\}$ yields a smooth Hamiltonian
$H(x,p) = \frac{1}{2}G(x,p)$.  Let $U \subset (T^\ast\R^3 \setminus
\{0\})\setminus D_c$ be an open set, then a local Hamiltonian flow
$\theta\colon \mathcal{D} \to U$ of $H$ exists, where $\mathcal D$ is
the maximal flow domain of $\theta$ that satisfies
\begin{equation}
\label{eq:flow}
\theta(t,(x,p))\in U, \quad (x,p)\in (U \cap G^{-1}\{1\}). \quad 
\end{equation}
%\color{blue}
In general it is possible that
the \textit{momentum gradient}
 $D_p H$ vanishes at some point $(x,p) \in U$, which would cause problems in translating between Hamiltonian and Lagrangian formalisms.  For this reason the operators $\square_{S_m}$ may not be of real principal type in $U$. See for instance \cite[Section 1.2]{oksanen2020inverse} for the connection between different definitions for real principal type operators.
%It is straightforward to show that in this case the spatial projection of the characteristic curve would be the constant $x$.
We make a standing assumption that $D_p H$ does not vanish in $U$. In other words we exclude the occurrence of inflection points. 
%Then due the predescribed microlocal diagonalization  propagates singularities along  the bicharacteristic curves, starting in $U$, of the principal symbols $\omega-G(x,p)$ for  short times \cite{Hormander2, GrUh}. 
\color{black}
We choose $(x_0,p_0) \in U\cap G^{-1}\{1\}$ and say that the
\textit{elastic travel-time} $\tau_c$ from $(x_0,p_0)$ to $(x,p) \in
\theta(\R_+,(x_0,p_0))\cap U$ is the smallest $t$ for which
$\theta(t,(x_0,p_0))=(x,p)$.

Under these two assumptions for the Hamiltonian $H$ in $U$ we are ready to set an inverse problem for anisotropic elastic travel-times. We suppose that there exists an open set $M\subset \R^3$ and open sets
$\Sigma, \Sigma' \subset \p M$ such that for any $x\in \Sigma, x'
\in\Sigma$ there is a unique characteristic curve of $H$ contained
in $T^\ast M \cap U$ whose spatial projection  $\gamma$ connects $x$ to $x'$, where $T^\ast M$ is
the cotangent bundle of $M$. Thus for any $x\in \Sigma, x' \in\Sigma$
there exists a unique triplet
\[
(\tau_c;(x,p);(x',p'))\in \R_+\times U \times U \quad \hbox{which satisfies } \quad \theta(\tau_c,(x,p))=(x',p')\in G^{-1}\{1\}.
\]
Therefore $\tau_c$ is the elastic travel-time from $(x,p)$ to $(x',p')$ and we call $d_G(x,x'):=\tau_c$ the \textit{elastic distance between} $x$ and $x'$. We arrive in an inverse problem of \textit{anisotropic  elastic travel-time tomography}:

\begin{problem}
\label{Ip_local}
What can one infer about $G$ in $T^\ast M$ when \textit{boundary
  distance data}
\begin{equation}
\label{eq:BDF}
   \{d_G(x,x') \in \R_+: x\in \Sigma, \: x' \in\Sigma'\}
\end{equation}
is given?
\end{problem}
\noindent We note that in general the sets $\Sigma$ and $\Sigma'$ can be very small, and \eqref{eq:BDF} may not contain any information about $G$ in some open set of $T^\ast M$. This is illustrated in the Figure 2.
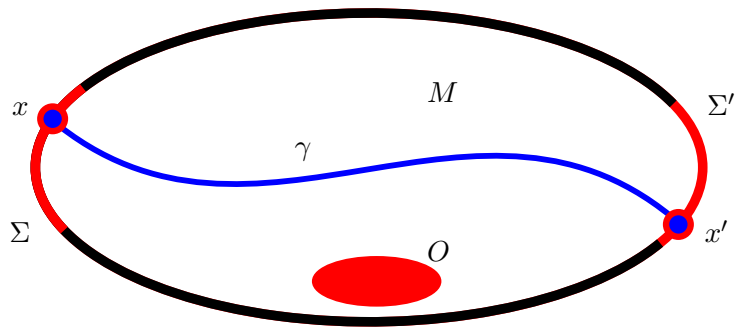
\begin{figure}[h]
\begin{picture}(250,125)
\put(0,5){\includegraphics[width=9cm]{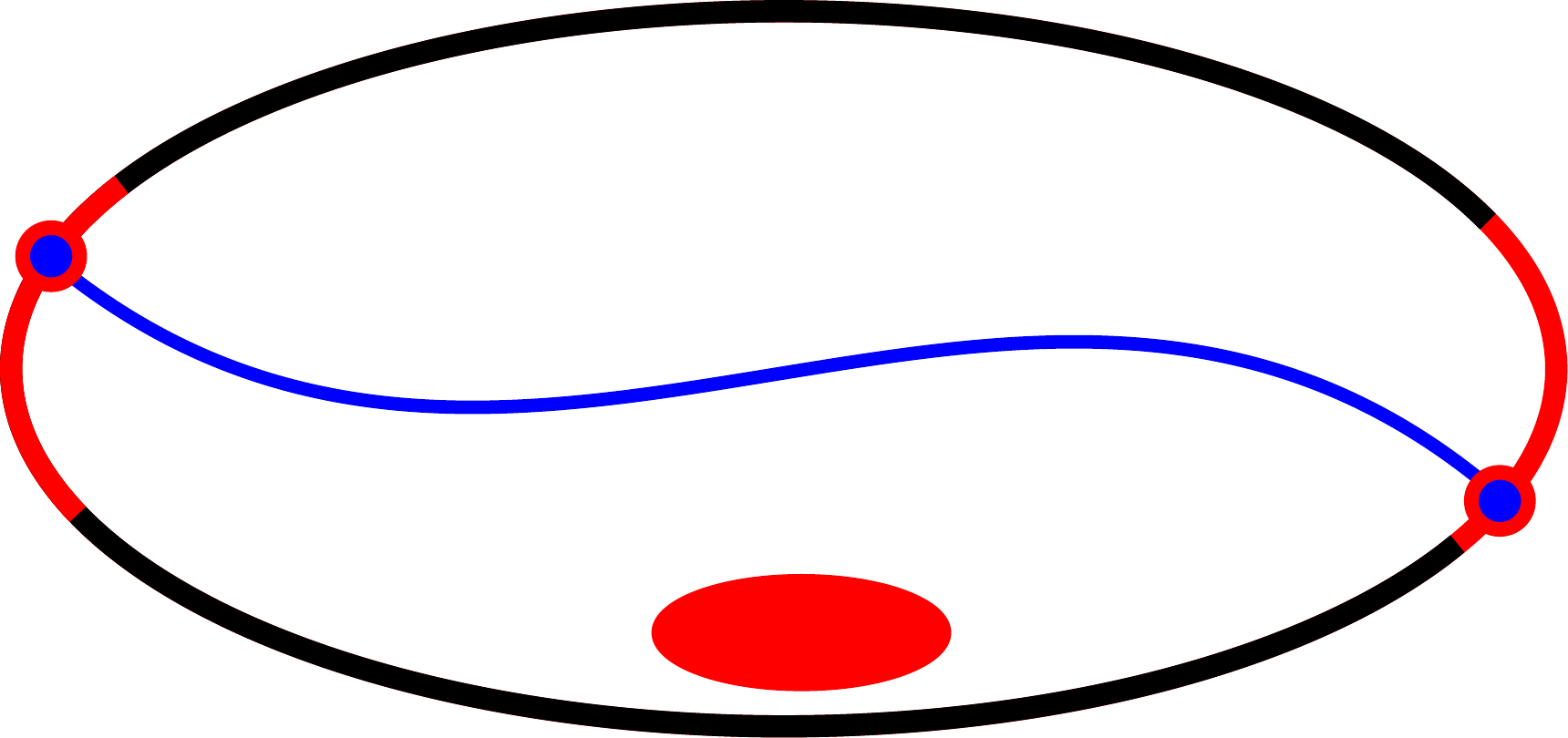}}
\put(-8,37){$\Sigma$}
\put(100,70){$\gamma$}
\put(-7,85){$x$}
\put(255,37){$x'$}
\put(256,85){$\Sigma'$}
\put(150,30){$O$}
\put(150,90){$M$}
 \end{picture}
 \caption{If there exists a set $O \subset M$ such that all characteristic curves whose terminal points are contained in $\Sigma \cup \Sigma'$ avoid the set $T^\ast O$, then data \eqref{eq:BDF} does not provide information about $G$ on $T^\ast O$.}
\end{figure}

%We state also a global version of the problem. Let $M$ be a smooth
%compact manifold with boundary and $H$ be a smooth two-homogeneous
%Hamiltonian on $T^\ast M\setminus \{0\}$ without inflection points. We
%say that the set $M$ is $H$-convex, if for any $x,x' \in M$ there
%exists a unique characteristic curve of $H$ contained in $T^\ast M$
%whose projection to base manifold $M$ connects $x$ to $x'$.  The global version of inverse problem
%\ref{Ip_local} is the following
%
%\begin{problem}
%\label{Ip_global}
%Determine $H$ in $T^\ast M$ from the \textit{global boundary distance data}
%\[
%\{d_H(x,x') \in \R_+: x , x' \in\p M\}.
%\]
%Here $d_H(\cdot, \cdot)$ is defined as $d_G$ earlier. 
%\end{problem}
\noindent Problem \ref{Ip_local} 
%and \ref{Ip_global} are
is highly non-linear and perturbations to anisotropic elasticity have been largely unresolved. In the following we consider linearizations of this problem. Since we have assumed that on the set $U$ the eigenvalues $G^2$ and $G^3$ are distinct, they and the corresponding unit length
%{\color{blue}\textbf{[I am not sure ``unitary" is the right word.]}} 
eigenvector fields  are smooth on $U$ (see for instance \cite[Chapter 11, Theorem 2]{evans2010partial}). In the following we recall a coordinate representation of the Hamilton's equation that the characteristic curve $\theta(t,(x_0,p_0))$, $(x_0,p_0) \in U\cap G^{-1}\{1\}$ satisfies.  Let $q=q(x,p)$ be a polarization vector of  $G$  on $U$. In the other words it is the 
unit (with respect to Euclidean metric) eigenvector of the Christoffel matrix associated with the eigenvalue $G$. 
%\textbf{[JZ: Is $q$ parallel (maybe after renormalization)?], [I don't understand the question, we do not have Riemannian structure yet](TS)} 
Then we can  write the eigenvalue $G$ as
\[
G(x,p)=\Gamma_{ik}q^iq^k=A_{ijk\ell}q^iq^kp^jp^\ell, \quad (x,p)\in U.
\]
%Above we used the summing convention for the repeated indices in lower and upper positions. 
From now on we denote the characteristic curve $\theta(t,(x_0,p_0))$ by $(x(t),p(t))\in U$. Therefore the polarization vector $q=q(x(t),p(t))$ is also seen as a function of the time variable $t$ implicitly. The following Hamilton's equation holds true on $(x(t),p(t))$
\begin{equation}
\label{eq:Hamilton}
\dot x_m =\frac{\p}{\p p_m}H(x,p)=A_{ijkm}q^iq^kp^j \quad \hbox{and} \quad \dot{p}_m=-\frac{\p}{\p x_m}H(x,p)=-\frac{1}{2}\frac{\p A_{ijk\ell}}{\p x_m}q^iq^jp^jp^\ell.
\end{equation}
%\color{blue}
%We note that in the last equation we used the the assumption $|q|_e=1$, which implies
%\[
%A_{ijk\ell}(\p_{x_m}q^i)q^kp^jp^\ell=G\left(q \cdot  \p_{x_m}q\right)=0.
%\]
\color{black}
Next, we recall the linearization scheme for the elastic distance $d_G$, that leads to an integral geometric problem of $4$-tensor fields. This procedure has
been introduced earlier in geophysical literature (see for instance
\cite{chapman1992traveltime}) using Fermat's principle. It is well
known that characteristic curves of Hamiltonian flow satisfy this
principle (see for instance \cite{babich2018elastic}). For the
convenience of the reader we give the proof and the exact claim below.

Since $D_p H$ does not vanish on $U$, the Legendre transform, that maps
a co-vector to a vector, is well defined. Using the inverse of this
map we define a Lagrangian function $L$ on the image of $U$ under this
transform. That is
\[
L(x,y):=p(x,y)\cdot y-H(x,p(x,y)).
\]
If $(x(t),y(t))=(x(t),\dot x(t))$ is the image of a characteristic
curve of $H$, under the Legendre transform, then on this curve $L\equiv 1/2$ and the following
Euler-Lagrange equations hold true
\[
\frac{\p}{\p x_i}L(x,y)-\frac{\mathrm{d}}{\mathrm{d} t}\left( \frac{\p}{\p y_i} L(x,y)\right)=0, \quad \hbox{ for every } i\in\{1,2,3\}.
\]
Let $x(t)$ be the base projection of a characteristic curve and
$x_s(t)$ any smooth one-parameter variation of this curve that fixes
the start point $x=x_s(0)$ and end point $x'=x_s(\tau_c)$. We choose the
notation $V$ for the variation field $V(t)=\frac{\p}{\p  s}x_s(t)|_{s=0}$ of $x_s(t)$. Then using integration by parts, we
obtain
\[
\frac{\mathrm{d}}{\mathrm{d} s}\int_0^{\tau_c}L(x_s(t),\dot x_s(t)) \mathrm{d}t \bigg|_{s=0}=\int_0^{\tau_c}\left(D_xL(x,\dot x)-\frac{\mathrm{d}}{\mathrm{d} t}\left(D_yL(x,\dot x)\right)\right)\cdot  V \mathrm{d}t \bigg|_{s=0}.
\]
Thus Euler-Lagrange equations imply that characteristic curves are the critical points of the energy functional 
\[
\mathcal{L}(\gamma)=\int_0^{\tau_c}L(\gamma(t),\dot \gamma(t)) \mathrm{d}t,
\]   
where $\gamma$ is any $C^1$-smooth curve. This is the version of
Fermat's principle we need. We note that the characteristic curves are
not necessarily local minimizers of the energy functional.

We now have the framework to linearize the travel-time tomography
problem \ref{Ip_local}. Let $\mathbf{A}_s$ be a smooth one-parameter family of
elastic moduli. Suppose that we can choose sets $U, M, \Sigma,
\Sigma'$ such that the discussion above holds for any $s$, if $\mathbf{A}$ is
replaced by $\mathbf{A}_s$. Since the Legendre transforms depend on the parameter $s$ we obtain a family of Lagrangians
\[
L_s(x,y)=p_s(x,y)\cdot y-H_s(x,p_s(x,y)), 
\]
which satisfy
\begin{equation}
\label{eq:der_of_Lag}
\frac{\mathrm{d}}{\mathrm{d} s}L_s(x,y)=\frac{\mathrm{d}}{\mathrm{d} s}p_s\cdot y-\frac{\mathrm{d}}{\mathrm{d} s}H_s-D_pH_s\cdot \frac{\mathrm{d}}{\mathrm{d} s}p_s
\end{equation}
%{\color{red} where $D_pH$ is the gradient of $H$ with respect to $p$-variable.}
If the data \eqref{eq:BDF} is independent of $s$, then due to Fermat's principle, equations \eqref{eq:Hamilton} \eqref{eq:der_of_Lag}, 
% \eqref{eq:der_of_q} 
and the assumption that $q_s$ is a unit vector on $U$, we obtain 

%\textbf{[JZ: Not clear how to obtain this. Looks different than the form given in \cite[eq. (18)]{chapman1992traveltime}], [(TS) The computation that leads to this should be just an application of the chain rule and the tips given above. Note  that in  \cite[eq. (18)]{chapman1992traveltime} $v=|p|^{-{1}}$ and $\hat{p}=vp$ so $\frac{\hat p_i \hat p_\ell}{v^2}=p_i p_\ell$ and $\hat q$ is a unitary polarization. If I remember correctly the whole idea why I wanted  to do Chapman calculation again was to avoid the derivatives of $v$. Thus I needed the Fermat's principle in the form written above.]}
\begin{equation}
\label{eq:linearization}
\begin{split}
0=& \frac{\mathrm{d}}{\mathrm{d} s} d_G(x,x')\bigg|_{s=0}
=\frac{\mathrm{d}}{\mathrm{d} s}\int_0^{\tau_c}L_s(x_s(t),\dot x_s(t)) \mathrm{d}t \bigg|_{s=0}
=\int_0^{\tau_c}\frac{\mathrm{d}}{\mathrm{d} s} L_s(x(t),\dot x(t)) \mathrm{d}t \bigg|_{s=0}
\\
=&-\int_0^{\tau_c}\frac{\mathrm{d}}{\mathrm{d} s}H_s(x,p) \mathrm{d}t \bigg|_{s=0}
=-\int_0^{\tau_c}\frac{1}{2}\frac{\mathrm{d}}{\mathrm{d} s}\left[(A_s)_{ijk\ell}q_s^iq_s^k\right]p^jp^\ell \mathrm{d}t \bigg|_{s=0}
=-\int_0^{\tau_c}\frac{A'_{ijk\ell}}{2} q^iq^kp^jp^\ell \mathrm{d}t, 
\end{split}
\end{equation}
where $A'_{ijk\ell} =\frac{\mathrm{d}}{\mathrm{d} s}(A_s)_{ijk\ell}|_{s=0},$ and $(x,p)=(x(t),p(t))$ is the characteristic curve of the reference model. Thus the linearization of ansitropic elastic travel-time $d_G$ leads to an integral problem for $4$-tensor fields.

In the following we will see that a similar linearization scheme applies in the case when
$\mathbf{A}_s$ is isotropic for any $s$. Recall that in isotropic medium the
elastic moduli can be written as
\begin{equation}
\label{eq:iso_elastic_mod}
\mathbf{C}^0=C^0_{ijk\ell}=\lambda\delta_{ij}\delta_{k\ell}+\mu\left(\delta_{ik}\delta_{j\ell} +\delta_{i\ell}\delta_{jk}\right).
\end{equation}
The functions $\lambda(x), \: \mu(x)>0$ are known as the Lam\'e parameters.  

If the mass density $\rho$ is given it follows from \eqref{eq:iso_elastic_mod} that the isotropic Christoffel matrix is
\begin{equation}
\label{eq:iso_Christ_mat_2}
\frac{\lambda}{\rho}p_ip_k+\frac{\mu}{\rho}\left(\delta_{ik}|p|_e^2+p_ip_k\right),
\end{equation}
whose eigenvectors are $p$ and any unit vector $q$ that is perpendicular to $p$. Moreover the corresponding eigenvalues are
\begin{equation}
\label{eq:iso_Christ_mat}
G^1=\Gamma_{ik}\hat p^i \hat p^k=\left(\frac{\lambda+2\mu}{\rho}\right)|p|_e^2, \: \hat{p}=\frac{p}{|p|_e}\quad \hbox{and} \quad G:=G^2=G^3=\Gamma_{ik}q^iq^k=\frac{\mu}{\rho}|p|_e^2.
\end{equation}
In particular $G^1$ or $G$ are both smooth and do not have inflection points. Thus any smooth one-parameter family $\mathbf{A}_s$ of isotropic elastic moduli satisfies all the additional assumptions we had to impose earlier for the general anisotropic case.

We recall that for isotropic elasticity, there are two different wave-speeds, namely,  \textit{P}-wave (Pressure, longitudinal wave) speed $c_P=\sqrt{\frac{\lambda+2\mu}{\rho}}$ and \textit{S}-wave  (Shear, transverse wave) speed $c_S=\sqrt{\frac{\mu}{\rho}}$. Therefore we can consider $M$, in Problem \ref{Ip_local}, as a Riemannian manifold with conformally Euclidean metric $g_P=c_P^{-2}\mathrm{d}s^2$ or $g_S=c_S^{-2}\mathrm{d}s^2$, where $\mathrm{d}s^2$ is the Euclidean metric. 

We repeat the earlier linearization scheme with respect to smaller eigenvalue $G$ and obtain
\[
\begin{split}
0=& \frac{\mathrm{d}}{\mathrm{d} s}d_G(x,x')\bigg|_{s=0}
%=-\int_0^{\tau_c}\frac{\mathrm{d}}{\mathrm{d} s}H_s(x,p) \mathrm{d}t \bigg|_{s=0}
%\\
%=&-\int_0^{\tau_c}\left(\frac{\mu}{\rho}\right)'|p|_e^2\mathrm{d}t
%=-\int_0^{\tau_c}\left(c_S^2\right)'|p|_e^2\mathrm{d}t
%=-\int_0^{\tau_c}\left(c_S^2\right)'c_S^{-2}\mathrm{d}t
%\\
=-\frac12\int_0^{\tau_c}\left(\log\left(c_S^{2}\right)\right)'\mathrm{d}t.
\end{split}
\]
%\textbf{[JZ: let us do not use the $\log$.]}
%\textbf{[JZ: Where is $\sqrt{\frac{\rho}{\lambda+\mu}}$ from?] [Sorry there was a mistake, now it should be clear and correct]}
This equation
follows from \eqref{eq:iso_Christ_mat} and from the initial condition $(x,p)\in G^{-1}\{1\}$. Moreover in this case $\tau_c$ is the
Riemannian distance between $x,x' \in \p M$ with respect to the metric $g_S$. Therefore we have shown that the linearization of elastic travel-times in isotropic case leads to an integral geometry problem on functions.

In our final example we consider the linearization of the averaged
quasi-shear wave travel-times in \textit{weakly anisotropic
  medium}. Suppose that we are given a family $\mathbf{A}_s :=
\frac{\mathbf{C}^0}{\rho} + s \frac{\mathbf{C}}{\rho}$ of elastic
moduli on some open and precompact domain of
  $\R^3$. Here $\mathbf{C}^0$ is an isotropic stiffness tensor having
the form \eqref{eq:iso_elastic_mod}, $\mathbf{C}$ is  an arbitrary
anisotropic stiffness tensor, which satisfies the symmetry \eqref{eq:symmetry_of_elastic_tensor}, and $s$ is a real parameter close to
zero. We note that for $|s|$ small enough the largest eigenvalue
$G_s^1$ of the Christoffel matrix $\Gamma_s$ of $\mathbf{A}_s$ is
always distinct from the smaller ones. Therefore, $G_s^1(x,p)$ and the
corresponding $g_S$-unit eigenvector field $q_s(x,p)$ are smooth in
all variables $(x,p,s)$. 
As the elastic moduli $\mathbf A_s$ is isotropic at $s = 0$, we have the degeneracy of eigenvalues,
  $G_0^2(x,p) = G_0^3(x,p)$ and, hence, $G_s^2(x,p)$ and $G_s^3(x,p)$
  may not be smooth when $s$ tends to zero. However, $G_s^2(x,p) +
  G_s^3(x,p) = \hbox{Tr}\left(\Gamma_s\right)(x,p) - G^1(x,p)$ is
  smooth in $(x,p,s)$, and strictly positive for $p \neq
  0$. Therefore, we introduce the smooth one-parameter family of averaged $qS$-Hamiltonians,
\begin{equation} \label{eq:average_Ham}
   H_s := \frac{1}{4} \left(G_s^2 + G_s^3\right) ,
\end{equation}
inheriting the homogeneity of order $2$ in $p$-variable. Since $H_0$
is conformally Euclidean, we note that, for $|s|$ sufficiently small,
$\sqrt{H_s}$ is, in fact, a smooth family of co-Finsler metrics, with $s$-uniformly lower bounded injectivity radii 
while $H_s$ does not
have inflection points. We briefly analyze the Hamiltonian flow
associated with $H_s$ -- which may be thought of as describing the
propagation of singularities of an artificial wave -- see below. In the following we will show that, up to the first order, the travel time
along this flow can be identified with the average of travel times associated with the two \textit{qS}-waves. This identification has been originally proposed in \cite{vcerveny1982linearized}.

First we note  that it follows from \cite[equations (24) and (26)]{vcerveny1982linearized} that for $x,x' \in \R^3$ that are close enough we can write the average $qS$-travel time as 
\begin{equation}
\label{eq:average_travel_tike}
\begin{split}
\frac{d_{G_s^2}(x,x')+d_{G_s^3}(x,x')}{2}=& \:d_{g_S}(x,x')-s\int_{x_0(t)}\frac 14 \sum_{ijk\ell} \left(\delta_{ik}-c_S^2 p_ip_k \right)p_jp_\ell\frac{C_{ijk\ell}}{\rho} \:\mathrm{d}t + \mathcal O(s^2),
%\\
%=&\: d_{g_S}(x,x')-s\int_{x(t)}\frac 14 \sum_{ijk\ell} \left(c_S^2\delta^{ik}- \dot{x}^i \dot{x}^k \right)\dot{x}^j\dot{x}^\ell\frac{C_{ijk\ell}}{c_S^6\rho} \:\mathrm{d}t
\end{split}
\end{equation}
where $x_0(t)$ is the $g_S$-geodesic connecting $x$ to $x'$ and $p=p(t)$ is the momentum of $x_0(t)$. We recall that the term $\delta_{ik}-c_S^2 p_ip_k$, in \eqref{eq:average_travel_tike}, is the projection onto the orthocomplement of $p$. 

Next we study the linearization of $H_s$-travel times $d_{H_s}(x,x')$. Let $p_0(t) \in T_{x_{0}(t)}^\ast\R^3$ be the momentum of the
$g_S$-geodesic connecting $x$ to $x'$. Thus $(x_0(t),p_0(t))$ is a
characteristic curve of $H_0$, with
initial value $(x,p) \in T^\ast\R^3, \: |p|_{g_S}=1$. In addition we
denote $d_{H_s}(x,x')=\tau_c$ that is the $g_S$-distance between
$x$ and $x'$. We note that due to the uniform lower bound for the injectivity radii of $\sqrt{H_s}$, for every $s \in (-\epsilon,\epsilon)$ there exists a $\sqrt{H_s}$-distance minimizing geodesic $x_s(t)$ from $x$ to $x'$, possibly after choosing $x'$ closer to $x$ and choosing $\epsilon > 0$ small enough.
%Moreover the initial condition $p_s \in \{v \in T^\ast_x\R^3: |v|_{g_S}=1\}$ of $x_s$ depends smoothly on $s$, and we write $(x_s(t),p_s(t))$ for the corresponding characteristic curve of $H_s$.

%neighborhood $ V\subset \{v \in T_x^{\ast}\R^3\ :\ |v|_{g_S} = 1\}$ of
%the initial value $p \in T_x^{\ast}\R^3$ and $\delta > 0$ such that
%for any $s \in (-\epsilon,\epsilon)$ the following map
%is a smooth embedding, 
%\begin{multline*}
%   \Phi_s \colon V \times (0, \tau_c+\delta)
%   \ni(v,t) \mapsto x_{v,s}(t)\in \R^3,
%\\
%   x_{v,s}(t)\hbox{ is the geodesic associated with $\sqrt{H_s}$,
%     with initial condition $v \in V$}.
%\end{multline*}
%We note that the range of $\Phi_s$ contains some neighborhood $U\subset \R^3$ of $\{x(t) \in \R^3: \: t
%\in (0, \tau_c]\}$ that does not depend on $s$.
%}

We consider the following system of linear ODEs 
\begin{equation} \label{eq:ODE_for_parallel_frame}
\begin{split}
&\frac{D}{\mathrm{d} t}e_s^I(t)=-\left\langle e_s^I, \frac{D}{\mathrm{d} t}q_s \right\rangle_{g_S} q_s(t), \quad e_s^I(0)=\eta_s^I \in T_x\R^3,
% \quad a_s\in C(\R) ,
% 
\\
&  \langle \eta_s^I, \eta_s^J \rangle_{g_S}=\delta^{IJ}, \quad  \hbox{ and } \quad  \langle \eta_s^I,q_s(0) \rangle_{g_S}=0,  \quad I,J \in \{1,2\}.
\end{split}
\end{equation}
Here  $q_s(t)=q_s(x_s(t),p_s(t))$ and $\frac{D}{\mathrm{d} t}$ is the covariant derivative with respect to Riemannian metric $g_S$ 
along the $\sqrt{H_s}$-geodesic $x_s(t)$ from $x$ to $x'$.
Both solutions $e_s^I$ of \eqref{eq:ODE_for_parallel_frame} satisfy  $\left\langle e_s^I, q_s \right\rangle_{g_S}\equiv 0,$ due to the assumption $|q_s|_{g_S}\equiv 1$, hence it also holds that $\left\langle e_s^I, e_s^J \right\rangle_{g_S}\equiv \delta^{IJ}$ along $x_s$. Therefore we have shown that for any $s$ the vector fields $\{q_s, e^1_s, e^2_s\}$ form a  $g_S$-orthonormal frame moving along $x_s$. 
%Since $q_s(0)=q_s(x,p)$ depend smoothly on $s$ and $(x,p)$ we can always choose initial conditions $\eta^I_s$ that also depend smoothly on $(s,x,p)$. Thus also the solutions $e^I_s$ of the system \eqref{eq:ODE_for_parallel_frame} depend smoothly on $(s,x,p)$ and $t$.  By perturbing the initial values of the characteristic curves in $V$ we see that $\{q_s, e^1_s, e^2_s\}$ yields a smooth local $g_S$-orthonormal frame near $x_s(t), \: t>0$ for any $s$. 
%
With respect to this basis we can write 
in $T^\ast_{x_s(t)}\R^3$
\[
G_s^2+G_s^3=(c_S)^{-2}\hbox{Tr}\left(\widehat{\Gamma}_s\right), \quad\quad  \hbox{ for } \widehat{\Gamma}_s
=\left(\begin{array}{cc}
(\Gamma_{s}e^1_s)\cdot e^1_s & (\Gamma_{s}e^1_s)\cdot e^2_s
\\
&
\\
(\Gamma_{s}e^2_s)\cdot e^1_s & (\Gamma_{s}e^2_s)\cdot e^2_s
\end{array}
\right).
\]
Since $e^1_0(t)$ and $e^2_0(t)$ are orthogonal to $\dot{x}_0(t)$, and we assumed that they are $g_S$-normalized, we obtain
\[
\frac{\mathrm{d}}{\mathrm{d} s}\left[(\Gamma_{s}e^J_s)\cdot e^J_s\right]\bigg|_{s=0}=\frac{C_{ijk\ell}}{\rho}p_0^jp_0^\ell \left(e_0^J\right)^i\left(e_0^J\right)^k, \quad J \in \{1,2\}.
\]
Finally we assume that $H_s$-travel-time $d_{H_s}(x,x')$ from $x$ to $x'$ is the constant $\tau_c$. Then we run through the same linearization process as earlier and obtain
\begin{equation}
\label{eq:lin_averaged_tt}
\begin{split}
0=&\frac{\mathrm{d}}{\mathrm{d} s}d_{H_s}(x,x')= -\int_0^{\tau_c}\frac{\mathrm{d}}{\mathrm{d} s}H_s(x_0(t),p_0(t)) \mathrm{d}t \bigg|_{s=0}
% - \int_0^{\tau_c}\frac{A_{ijk\ell}}{c_S^{2}\rho}p^jp^\ell (e^1+e^2)^i(e^1+e^2)^k \mathrm{d}t
 \\
=& - \int_0^{\tau_c}\frac{C_{ijk\ell}}{4c_S^{6}\rho}\dot x^j \dot x^\ell (e^1)^i(e^1)^k \mathrm{d}t-\int_0^{\tau_c}\frac{C_{ijk\ell}}{4c_S^{6}\rho}\dot x^j \dot x^\ell (e^2)^i(e^2)^k \mathrm{d}t, \quad \dot x=\dot{x}_0(t), \: e^I:=e^I_0(t), \: I \in \{1,2\}.
\end{split}
\end{equation}
%\textbf{[JZ: Shouldn't we obtain $ \int_0^{\tau_c}\frac{C_{ijk\ell}}{c_S^{6}\rho}\dot x^j \dot x^\ell (e^1)^i(e^1)^k \mathrm{d}t+\int_0^{\tau_c}\frac{C_{ijk\ell}}{c_S^{6}\rho}\dot x^j \dot x^\ell (e^2)^i(e^2)^k \mathrm{d}t$? I think the vectors $e^1,e^2$ can be chosen to form any orthonormal basis for $\xi^\perp$. The parallelism of $e^j$, $j=1,2$ is completely irrelevant.] (TS)[Yeah you are right, I made a mistake in the bilinearity by forgetting the crossterm. The ODE framework of \eqref{eq:ODE_for_parallel_frame} is just one way to define $e^1$ and $e^2$, but natural to me, as geophysicist use similar constructions for instance in \cite[Section 4.1.1.]{cerveny2005seismic} for the ray centered coordinates.]}

\noindent In the last equation we also transformed the momentum variable into the velocity variable. 

%\textbf{(Maybe we should drop this part as  \cite{vcerveny1982linearized} does not specify $\tau_S$ in their formula (21)). ``We note that an analogous linearization for the average value of the travel-time for \textit{qS}-waves has been presented earlier in  \cite[(26)]{vcerveny1982linearized}." (We prove the first order equality between $H_s$- and the average of $qS$-travel times [TS].) }
If we write the right hand side of \eqref{eq:average_travel_tike} using the basis $\{\dot x_0, e_0^1, e_0^2 \}$ we notice that the first order term equals to the right hand side of \eqref{eq:lin_averaged_tt}. Therefore we have verified that $H_s$-travel times and average of travel times associated with the two \textit{qS}-waves coincide up to the first order. 
We also note that
up to a constant
the integrands in the right hand side of \eqref{eq:lin_averaged_tt} are the same as in \cite[Problem 7.1.1]{Shara}.
%\textbf{(TS)[This was a good reference.]  }
%we obtain by the assumption $|\eta_s|^2_{g_S}=c_S^{-2}|\eta_s|_e^2\equiv 1$ and formula  \eqref{eq:iso_Christ_mat_2} that
%\begin{equation}
%\label{eq:int_of_A}
%\begin{split}
%   0 =& -2\int_0^{\tau_c}\frac{\mathrm{d}}{\mathrm{d} s}H_s(x,p) \mathrm{d}t \bigg|_{s=0}=
%  - \int_0^{\tau_c}\frac{\mathrm{d}}{\mathrm{d} s}\left[\frac{1}{|q_s|^2_e}\left(\frac{C^0_{ijk\ell}}{\rho}+s\frac{A_{ijk\ell}}{\rho}\right)\eta_s^i\eta_s^k\right]p^jp^\ell \mathrm{d}t \bigg|_{s=0}
%  \\  
% =&
% %-\int_0^{\tau_c}\frac{1}{|q|^2_e}A_{ijk\ell} q^iq^k p^jp^\ell\mathrm{d}t=
% %-\int_0^{\tau_c}c_S\frac{A_{ijk\ell}}{\rho} q_iq_k p_jp_\ell\mathrm{d}t= 
% -\int_0^{\tau_c}\frac{A_{ijk\ell}}{\rho c_S^6} \eta^i\eta^j\dot{x}^j\dot{x}^\ell\mathrm{d}t.
%\end{split}
%\end{equation}
%Above $q=q(x(t),(p(t))$ is given by the parallel transport of $q\perp
%p(0)$ with respect to the reference structure $g$
%\cite[Section 7.1.5]{Shara}. 

Finally we note that the polarization vector $q_0(t)$, coincides with the velocity field $\dot{x}_0(t)$ of the geodesic $x_0(t)$ in the reference medium $g_S$. Thus we have, due to \eqref{eq:ODE_for_parallel_frame}, that $e_0^I(t)$ is given by a parallel translation of $\eta_0^I \in T_x\R^3$ along the reference ray $x_0(t)$. 
We have shown that an anisotropic perturbation of an 
averaged
isotropic shear wave travel-time leads to an integral geometry problem on the $4$-tensor field $f_{ijk\ell}:=\frac{1}{2}\frac{C_{ki\ell j}+C_{kj \ell i}}{\rho c_S^6}$ which is related to the mixed ray transform $L_{2,2}f$ for the metric $g_S$. However we want to emphasize that the travel-time $d_{H_s}(x,x)$, of the aforementioned artificial wave, is given only by the averaged $qS$-Hamiltonian $H_s$, in \eqref{eq:average_Ham}, that is independent of the choice of the initial values $\eta^I_s$ in \eqref{eq:ODE_for_parallel_frame}, which yield the shear wave polarization vector fields $e_0^I(t)$  in formula \eqref{eq:lin_averaged_tt}. 

To conclude the first part of the appendix we note that \eqref{eq:lin_averaged_tt} implies that the travel-time data alone only gives us partial information about the mixed ray transform. However, if in addition we include the measurement of the shear wave amplitude, the complete mixed ray transform can be obtained. In the Appendix \ref{dnmap}, we will show how one can recover the mixed ray transform from the linearization of the \text{Dirichlet-to-Neumann} map of an elastic wave equation on $M$, by probing with Gaussian beams. We also refer to \cite[Chapter 7]{Shara} for an alternative derivation of the mixed ray transform.

\section{The Relation of the MRT and the Dirichlet-to-Neumann map}\label{dnmap}

%\textbf{(TS)[This looks now good to me. Let's still read it trough for typos.]}
%\textbf{(Ts)[This new version is very nice, I think we should include these both Appendices. For this reason let us try to globalize and condense the notations. Could you please change the index $l$ to $\ell$ below? I will check the geo reference that Maarten gave Tomorrow May 27th.]}

%\color{blue}
In this section, we give another derivation of the mixed ray transform from the inverse boundary value problem for elastic wave equations. We let $M\subset\mathbb{R}^3$ be a bounded domain with smooth
boundary $\partial M$ and $x=(x^1,x^2,x^3)$ be the Cartesian coordinates. The system of equations describing elastic waves can be written as
\begin{equation}\label{EQ no1}
\begin{split}
&\rho\frac{\partial^2u}{\partial t^2}-\operatorname{div}(\mathbf{C}\varepsilon(u))=0,\quad (t,x)\in (0,T)\times M,\\
&u=h,\quad \text{on } (0,T)\times\partial M,
%\\
%&
\quad u(0,x)=\frac{\partial}{\partial t}u(0,x)=0,\quad x\in M.
\end{split}
\end{equation}
Here, $u$ denotes the displacement vector and
\[\varepsilon(u)=(\varepsilon_{ij}(u))=\frac{1}{2}\left(\frac{\partial u_i}{\partial x^j}+\frac{\partial u_j}{\partial x^i}\right)\]
is the linear strain tensor. Furthermore, $\mathbf{C}=(C_{ijk\ell})=(C_{ijk\ell}(x))$ is the stiffness
tensor and $\rho=\rho(x)$ is the density of mass. As in Appendix \ref{Se:app_1} we assume that  $C_{ijk\ell}$ and $\rho$ are all smooth functions, and 
%\textbf{(TS)[Please check the equation. Should $C^0$ be $C\epsilon$?]}
the elastic tensor $\mathbf{C}$ is assumed to have the symmetries as in \eqref{eq:symmetry_of_elastic_tensor}.
%\[C_{ijk\ell}=C_{jik\ell}=C_{k\ell ij}.\]
In addition we assume  the operator $-\operatorname{div} (\mathbf{C}\varepsilon(\cdot))$ to be elliptic, in the following sense: There exists $\delta>0$ such that for any 
$3\times 3$ real-valued symmetric matrix $(\varepsilon_{ij})$,
\[
   \sum_{i,j,k,\ell=1}^3C_{ijk\ell}\varepsilon_{ij}\varepsilon_{k\ell}\geq\delta\sum_{i,j=1}^3\varepsilon_{ij}^2.
\]
%Also we assume that the density $\rho$ is fixed but unknown. 

Under these assumptions we let $\Lambda_{\mathbf{C}}$ to be the Dirichlet-to-Neumann map for the elastic wave equation \eqref{EQ no1} (see for instance \cite{de2019unique}), given by 
\[
\Lambda_{\mathbf{C}}:C^2([0,T];H^{1/2}(\partial M))\ni h\mapsto \mathbf{C}\varepsilon(u)\nu \vert_{(0,T)\times\partial M}\in L^2([0,T];H^{-1/2}(\partial M)).
\]
Where $T>0$ is large enough. The following inverse problem is of fundamental importance in seismology : 
\begin{problem}
\label{Ip_DtoN}
Can we reconstruct the elastic tensor $\mathbf{C}$ and the density $\rho$ from the Dirichlet-to-Neuman map  $\Lambda_{\mathbf{C}}$?
\end{problem}

We note that this problem is open for a general anisotropic $\mathbf{C}$. For isotropic medium, the uniqueness is shown under certain geometrical assumptions \cite{bhattacharyya2018local, hansen2003propagation, rachele2000inverse, rachele2003uniqueness,stefanov2017local}. The uniqueness of transversely isotropic tensors is proved under piecewise analytic assumption in \cite{de2019unique}, as well as fully anisotropic tensors under piecewise homogeneous assumption.
%\textbf{(TS) [Maybe we could list here some known cases like transversaly isotropic one that you considered in  \cite{de2019unique}.]}
In contrast to the elastic problem, the corresponding inverse problem for scalar wave equation has been solved in \cite{belishev1987approach, belishev1992reconstruction}. In this second appendix, instead of studying Inverse Problem \ref{Ip_DtoN} for general anisotropic elastic tensors, we consider a linearization of the problem around isotropic elasticity. We will see that the linearization leads to a family of ray transforms on four tensors.

%\textbf{(TS)[It would be nice to give a citation for the existence of the DN-map. Naturally this follows from the hyperbolicity. Could you please also mention the spaces over which $\Lambda$ is a bounded map? Like $\Lambda\colon X \to Y$ is a bounded map]}
%\textbf{(TS)[Let us not recall the definition of the isotropic elastic tensor.])}
%
%The elastic tensor $\mathbf{C}^0$ is said to be isotropic if
%\begin{equation}\label{symmetry}
%C_{ijkl}^0=\lambda\delta_{ij}\delta_{kl}+\mu(\delta_{ik}\delta_{jl}+\delta_{il}\delta_{jk}),
%\end{equation}
%where $\lambda,\mu$ are  called the Lam\'{e} parameters. 
%\textbf{(TS)[Can we please avoid the $\delta C$ notation? If $\delta C$ is just an other elastic tensor lets just write it that way. If $\delta C$ is a one parameter family of tensors, let us specify that]}

From here we consider a one parameter family of anisotropic perturbations $s\mathbf{C}$ around the isotropic elasticity $\mathbf{C}^0$ of the form \eqref{eq:iso_elastic_mod}, that is we study an elastic tensor $\mathbf{C}_s=\mathbf{C}^0+s\mathbf{C}$.
%
%\textbf{(TS)[Let us hide the componentwise representation of $\mathbf{C}_s$. I also changed the anisotropic tensor $\mathbf{A} \to \mathbf{C}$, since we are using $\mathbf{A}$ for the density normalised elastic moduli earlier.]}
%
%, or componentwisely,
%\[
%(C_s){ijk\ell}=C_{ijk\ell}^0+s C_{ijk\ell}=\lambda\delta_{ij}\delta_{k\ell}+\mu(\delta_{ik}\delta_{j\ell}+\delta_{i\ell}\delta_{jk})+sC_{ijk\ell}.
%\]
%\textbf{(TS)[If we choose to include this approach the derivation of $\dot \Lambda$ should be done carefully. That is to define the space for $C$s and the range of $C\mapsto \Lambda_C$. Also the action in formula \eqref{integral_identity} should be justified.}
We note that the map $\mathbf{C}_s \mapsto \Lambda_{\mathbf{C}_s}$ is Frech\'{e}t differentiable at $\mathbf{C}^0$, and the Frech\'{e}t derivative is
$
\dot{\Lambda}_{\mathbf{C}^0}: \mathbf{C} \mapsto \dot{\Lambda}_{\mathbf{C}^0}(\mathbf{C} ):=\lim_{s\rightarrow 0}\frac{1}{s}(\Lambda_{\mathbf{C}^0+s\mathbf{C}}-\Lambda_{\mathbf{C}^0}).
$
We will study the injectivity of the linear map $\dot{\Lambda}_{\mathbf{C}^0}$, whose action is given by
\begin{equation}\label{integral_identity}
\langle\dot{\Lambda}_{\mathbf{C}^0}(\mathbf{C} )h_1,h_2\rangle_{(0,T)\times\partial M}=\int_{(0,T)\times M}C_{ijk\ell}(x)\partial_{x_i} w_j(x,t)\partial_{x_k} v_\ell(x,t)\mathrm{d}x\mathrm{d}t,
\end{equation}
and $w$ $(v)$ solves the elastic wave equation (backward one) with the isotropic elastic tensor $\mathbf{C}^0$,
\begin{equation}\label{elastic_eq_linearized}
\begin{cases}
&\rho\frac{\partial^2w}{\partial t^2}-\operatorname{div} (\mathbf{C}^0w)=0,\quad \text{in } (0,T)\times M,\\
&w=h_1,\quad \text{on } (0,T)\times\partial M,\\
&w(0,x)=\frac{\partial}{\partial t}w(0,x)=0,\quad x\in M,
\end{cases}
\begin{cases}
&\rho\frac{\partial^2v}{\partial t^2}-\operatorname{div} (\mathbf{C}^0v)=0,\quad\text{in }  (0,T)\times M,\\
&v=h_2,\quad \text{on } (0,T)\times\partial M,\\
&v(T,x)=\frac{\partial}{\partial t}v(T,x)=0,\quad x\in M.
\end{cases}
\end{equation}
%\begin{equation}\label{elastic_eq_linearized_backward}
%\begin{cases}
%&\rho\frac{\partial^2v}{\partial t^2}-\operatorname{div} (\mathbf{C}^0v)=0,\quad (t,x)\in (0,T)\times M,\\
%&v=h,\quad \text{on } (0,T)\times\partial M,\\
%&v(T,x)=\frac{\partial}{\partial t}v(T,x)=0,\quad x\in M.
%\end{cases}
%\end{equation}
A similar linearization for the time-harmonic elastic wave equation can be found in \cite{yang2019unique}.

%We can extract mixed ray transform from the integral identity \eqref{integral_identity}.
%\textbf{(TS) [This is a great approach, but let us give the reader more details about the required steps. I understand that for now this was just a suggestion for how to link MRT to DN-map]}

%Correspondingly, we can view \textit{P} waves traveling along geodesics in Riemannian manifold $(M,g_P)$, and \textit{S} waves traveling along geodesics in $(M,g_S)$.

Next we summarize the construction of Gaussian beam solutions to \eqref{elastic_eq_linearized} used in \cite[Section 3]{uhlmann2019inverse}. We also refer to \cite{feizmohammadi2019recovery} for more discussions on Gaussian beams solutions. Assume that  $M \subset \subset \widetilde M \subset \R^3$, where $\widetilde M$ is open and bounded, the Riemannian metric $g_{P/S}$ with respect to $\mathbf{C}^0$ is known on $\widetilde M$ and the Riemannian manifold $(M,g_{P/S})$ is simple. We choose a maximal unit-speed geodesic $\gamma$ in $(M,g_{P/S})$, and extend it to $\widetilde M$ assuming that once leaving $M$ it will not return back to it. Then $\vartheta(t)=(t+\alpha,\gamma(t))$ is a null-geodesic in the Lorentzian manifold $((0,T)\times \widetilde M,-\mathrm{d}t^2+g_{P/S})$ joining two points on $(0,T)\times \partial M$, as long as for $\alpha>0$ and $T$ large enough.
Let us first take an asymptotic solution $M_\varrho$ to the elastic wave equation on $(0,T)\times\widetilde{M}$,
\[
\rho\frac{\partial^2M_\varrho}{\partial t^2}-\operatorname{div} (\mathbf{C}^0M_\varrho)=\mathcal{O}(\varrho^{-N}),
\]
 representing \textit{S}-waves, of the form
\begin{equation*}
M_\varrho=\chi \left(\sum_{j=0}^{N+1}\varrho^{-j}\mathbf{a}_j\right) e^{\mathrm{i}\varrho\varphi},
%\quad 
%\begin{cases}
%&\rho\frac{\partial^2w}{\partial t^2}-\operatorname{div} (\mathbf{C}^0w)=0,\quad \text{in } (0,T)\times M,\\
%&w(0,x)=\frac{\partial}{\partial t}w(0,x)=0,\quad x\in M,
%\end{cases},
\end{equation*}
where $\varrho$ is a large parameter and all the vector fields $(\mathbf{a}_j)_{j=0}^{N+1}$ and the phase function $\varphi$ depend  on time $t$ and on location $x$. Here $\chi$ is a real valued cut-off function that is compactly supported and equal to $1$ in a neighborhood of $\vartheta$.  The phase function $\varphi$ satisfies
%\textbf{(TS) [Does $\Im(D^2\varphi)>0$ mean that $\Im(D^2\varphi)$ is pos. def.?]}
$
D\varphi\vert_{\vartheta(t)}=\dot{\gamma}(t),
$
where $D$ is the gradient with respect to the Euclidean metric on $M$. The imaginary part of the spatial Hessian of the phase function $\varphi$ is positive definite, i.e. $\Im(D^2\varphi)>0$.
% In \eqref{solution_u}, $M_\varrho$ is an asymptotic solution to the equation on the right hand side of \eqref{solution_u}, and the remainder $R_\varrho$ satisfies zero boundary and initial conditions on $(0,T) \times M$.
%\textbf{(TS)[Could you please indicate which functions above depend on time $t$ and which on location $x$?][JZ: They all depend on both $t$ and $x$.]}
%\textbf{(TS)[Is $D=D_x$  the Euclidean gradient w.r.t. the spatial variables only?]}
In addition we have
\begin{equation}
\label{eq:a_0}
\mathbf{a}_0(\vartheta(t))=A_S(\vartheta(t)) \mathbf{e}(\vartheta(t)),
\end{equation}
where $\mathbf{e}(\vartheta(t))=\eta(t)$ is an arbitrary parallel vector field along $\gamma(t)$, perpendicular to $\dot{\gamma}(t)$, that is 
%$\nabla_{\dot{\gamma}(t)}\eta(t)=0$ (here $\nabla$ is the Levi-Civita connection w.r.t. the Riemannian metric $g_S$)
$\frac{D}{\mathrm{d} t} \eta(t)=0$, $\eta(t)\perp\dot{\gamma}(t)$,
%\textbf{{TS}[Is here a freedom to choose $\eta(t) \in \dot{\gamma}(t)^\perp$? This could be emphasized.]} 
and the amplitude $A_S$ can be chosen such that
\begin{equation}
\label{eq_amplitude}
A_S\vert_{\vartheta}=\det(Y_S)^{-1/2}c_S^{-1/2}\rho^{-1/2},
\end{equation}
where $Y_S(x,t)$ is well defined on $\vartheta$ and is given as a solution of a second order ODE. 
%(see for instance  \cite{uhlmann2019inverse}).
%\textbf{(TS) [Could you please define $Y_S$ heuristically and / or give a citations for it?]}
Furthermore, we have
\begin{equation}\label{d2YS}
\det(\Im(D^2\varphi))|\det(Y_S)|^2\equiv c_0
\end{equation}
on $\vartheta$ with $c_0$ a constant.  Let $h_1=M_\varrho\vert_{(0,T)\times\partial M}$, then one can determine the remainder $R_\varrho$ satisfying zero boundary and initial conditions, such that
\begin{equation}\label{solution_u}
w=M_\varrho+R_\varrho
\end{equation}
is a solution to the first equation in \eqref{elastic_eq_linearized}. For any $m \in \N$ we can choose large enough $N$ such that the remainder term $R_\varrho$ satisfies the estimate $\|R_\varrho\|_{H^1(M\times (0,T))}=\mathcal{O}(\varrho^{-m})$.
We also take 
\begin{equation}\label{solution_v}
v= \overline{M_\varrho}+R_\varrho'=\chi \left(\sum_{j=0}^{N+1}\varrho^{-j}\overline{\mathbf{a}_j}\right) e^{-\mathrm{i}\varrho\overline{\varphi}}+R_\varrho',
%\quad 
%\begin{cases}
%&\rho\frac{\partial^2v}{\partial t^2}-\operatorname{div} (\mathbf{C}^0v)=0,\quad\text{in }  (0,T)\times M,\\
%&v(T,x)=\frac{\partial}{\partial t}v(T,x)=0,\quad x\in M.
%\end{cases}
\end{equation}
for a solution of the backward elastic wave equation in \eqref{elastic_eq_linearized} with $h_2=\overline{M_\varrho}\vert_{(0,T)\times\partial M}$. 

We multiply the identity \eqref{integral_identity} by $\varrho^{-\frac{1}{2}}$
%, and note that if we choose the boundary sources 
%\[
%h_1:=w\vert_{(0,T)\times\partial M}=M_\varrho\vert_{(0,T)\times\partial M} \quad \hbox{and} \quad h_2:=v\vert_{(0,T)\times\partial M}=\overline{M_\varrho}\vert_{(0,T)\times\partial M},
%\]
and use the representations \eqref{solution_u} for $w$ and \eqref{solution_v} for $v$, then due to the estimate \cite[equation (3.33)]{feizmohammadi2019recovery} and the substitution $u_0:= \chi^2 \partial_{x_i}\varphi\,[\mathbf{a}_0]_j\overline{\partial_{x_k}\varphi\,[\mathbf{a}_0]_\ell} $ we obtain
\begin{equation}\label{integral_uv}
\begin{split}
\varrho^{-\frac{1}{2}}\langle\dot{\Lambda}_{\mathbf{C}^0}(\mathbf{C} )h_1,h_2\rangle_{(0,T)\times\partial M}=&
%\varrho^{-\frac{1}{2}} \int_{0}^T\int_MC_{ijk\ell}(x)\partial_{x_i} w_j(x,t)\partial_{x_k} v_\ell(x,t)\mathrm{d}x\mathrm{d}t
%\\
%=&
%\varrho^{-\frac{1}{2}}\int_{0}^T\int_M\varrho^{2} C_{ijk\ell}e^{-2\varrho\Im\varphi} 
%\left[
%\chi^2 \partial_{x_i}\varphi\,[\mathbf{a}_0]_j\overline{\partial_{x_k}\varphi\,[\mathbf{a}_0]_\ell} \;\mathrm{d}x\mathrm{d}t+\mathcal{O}(\varrho^{-1})
%\right.
%\\
%&+\varrho^{1}\left( \partial_{x_i}[\chi\mathbf{a}_0]_j\overline{\chi\partial_{x_k}\varphi\,[\mathbf{a}_0]_\ell}+ \partial_{x_i}\varphi\chi[\mathbf{a}_0]_j\overline{\partial_{x_k}\,[\chi\mathbf{a}_0]_\ell}\right)
%\\
%&+\left( \partial_{x_i}[\chi\mathbf{a}_1]_j\overline{\chi\partial_{x_k}\varphi\,[\mathbf{a}_0]_\ell}+ \partial_{x_i}\varphi\chi[\mathbf{a}_0]_j\overline{\partial_{x_k}\,[\chi\mathbf{a}_1]_\ell}\right.
%\\
%&+\left.\left.\chi^2 \partial_{x_i}\varphi\,[\mathbf{a}_1]_j\overline{\partial_{x_k}\varphi\,[\mathbf{a}_1]_\ell}\right)\right]\mathrm{d}x\mathrm{d}t +\mathcal{O}(\varrho^{-3}), 
%\\
%=&:\int_{0}^T\int_M e^{-2\varrho\Im\varphi}\left(\varrho^{\frac{3}{2}}u_0+\varrho^{\frac{1}{2}} u_1+\varrho^{-\frac{1}{2}}u_2 \right)\mathrm{d}x\mathrm{d}t+\mathcal{O}(\varrho^{-3}),\\
%=&
\varrho^{\frac{3}{2}}\int_{0}^T\int_M e^{-2\varrho\Im\varphi}u_0\,\mathrm{d}x\mathrm{d}t+\mathcal{O}(\varrho^{-1}), \quad \varrho \to \infty.
\end{split}
\end{equation}
Note that one can use the Fermi coordinates $(\tau,x')$, as constructed in \cite{uhlmann2019inverse}, under which the Euclidean volume form is $\mathrm{d}x\mathrm{d}t=c_S^3\mathrm{d}\tau\wedge\mathrm{d}x'$, moreover $\tau=\sqrt{2}t$ and $x'=0$ on $\vartheta$. Notice that $D\Im\varphi=0$ on $\vartheta$. Then after using the method of stationary phase to the integral 
\[
\int e^{-2\varrho\Im\varphi}u_0c_S^3\mathrm{d}x'
\] 
with the phase function $f:=\mathrm{i}2\Im\varphi$ and amplitude $u:=u_0c_S^3$ as in \cite[Theorem 7.5.5.]{hormander1}, we can write the right hand side of \eqref{integral_uv} into the form
\[
\begin{split}
%&\int_{0}^{\sqrt{2}T}\left(2\pi\right)^{\frac{3}{2}}\left| \det D^2 f(\vartheta(\tau))\right|^{-\frac{1}{2}}
%%e^{\frac{\mathrm{i}\pi}{4}\sgn( D^2 f(\vartheta(\tau))}
%u_0(\vartheta(\tau))c_S^3\mathrm{d}\tau+\mathcal{O}(\varrho^{-1}), 
%\\
%=&\int_{0}^{\sqrt{2}T}\left(2\pi\right)^{\frac{3}{2}}(\mathrm{i}2)^{-\frac{3}{2}}\left| \det D^2 \Im\varphi(\vartheta(\tau))\right|^{-\frac{1}{2}}
%e^{\frac{\mathrm{i}\pi}{4}\sgn( \mathrm{i}2\Im D^2 \varphi(\vartheta(\tau))}
%e^{-\varrho 2\Im\varphi(\vartheta(\tau))}u_0(\vartheta(\tau))c_S^3\mathrm{d}\tau+\mathcal{O}(\varrho^{-1}), \quad \varrho \to \infty
\left(-\mathrm{i}\pi\right)^{\frac{3}{2}}\int_{\vartheta}\left| \det D^2 \Im\varphi(\vartheta(\tau))\right|^{-\frac{1}{2}}
u_0(\vartheta(\tau))c_S^3(\vartheta(\tau))\mathrm{d}\tau+\mathcal{O}(\varrho^{-1}), \quad \varrho \to \infty.
\end{split}
\]
%\textbf{(TS)[I think we should write down this formula which is obtained after using the stationary phase method and to make it clear why $\varrho^{-\frac{3}{2}}$ is canceled, and  what happens to the oscillating part  $e^{\mathrm i\varrho f(\vartheta(\tau))}$ in the limit $\varrho \to \infty$?]}
Next we use the  properties \eqref{eq_amplitude} and \eqref{d2YS} to observe that
\[
\begin{split}
u_0(\vartheta(\tau))=& |A_S|^2\p_{x_i}\varphi\,\mathbf{e}_j\overline{\partial_{x_k}\varphi\,\mathbf{e}_\ell}\Big\vert_{\vartheta(t)}=|\det(Y_S)|^{-1}c_S^{-1}\rho^{-1}\dot{\gamma}_i(t)\eta_j(t)\dot{\gamma}_k(t)\eta_\ell(t)
\\
=&\frac{\left| \det D^2 \Im\varphi(\vartheta(\tau))\right|^{\frac{1}{2}}}{\sqrt{c_0}}\rho^{-1}\dot{\gamma}_i(t)\eta_j(t)\dot{\gamma}_k(t)\eta_\ell(t).
\end{split}
\]
Since $c_0$ is a known constant, we have verified that
\begin{equation}
\label{eq:limit_of_int_indent}
\begin{split}
0=&\lim_{\varrho \to \infty}\varrho^{-\frac{1}{2}}\langle\dot{\Lambda}_{\mathbf{C}^0}(\mathbf{C} )h_1,h_2\rangle_{(0,T)\times\partial M}=\int_{\vartheta}\frac{C_{ijk\ell}}{\rho}c_S^2\dot{\gamma}_i(t)\eta_j(t)\dot{\gamma}_k(t)\eta_\ell(t)\mathrm{d}t.
%\\
%&h_1:=w\vert_{(0,T)\times\partial M}=M_\varrho\vert_{(0,T)\times\partial M}, \quad h_2:=v\vert_{(0,T)\times\partial M}=\overline{M_\varrho}\vert_{(0,T)\times\partial M}.
\end{split}
\end{equation}
%We refer to \cite[Section 4]{feizmohammadi2019recovery} for more details about an analogous construction.
%\textbf{(TS)[We should note what happens to $c_0$ after applying \eqref{d2YS}. [JZ: the constant $c_0$ is known.] I did not see the connection of their method to ours clearly enough. [I just used their method] If I have troubles then also the referee and readers will likely have, we should be more transparent in our derivation.]}
%
% \textbf{(TS)[Could you please add here few more details, including that we can we use the method of stationary phase as we had the assumption $D\varphi\vert_{\vartheta(t)}=\dot{\gamma}(t)$, which is real, thus $D\Im \varphi=0$. Also it seems to me that the result follows from \cite[Thm 7.5.5.]{hormander1}, as you suggested in \cite{uhlmann2019inverse}, after setting $k=1$ in (eq. 7.7.12.). Could you site this theorem and mention that the powers of $\varrho$ in front of \eqref{integral_uv} and what you get from $\det (\rho \mathrm D^2 \varphi)^{-\frac{1}{2}}$ cancel since we assumed that $\dim M =3$? Shouldn't we see some power of $c_0$ from  \eqref{d2YS} in this formula, this should be also commented. What happens for the oscillating part $e^{-2 \mathrm{i}\varrho \Im\varphi}$ as $\varrho \to \infty?$]}
%

Finally we use the notation $v^i$ for raising the indices of a co-vector $v_i$ under the metric $g_S$, that is we have $v_i=c_S^{-2}v^i$. Then due to formula \eqref{eq:limit_of_int_indent} we have recovered the mixed ray transform of the tensor field $f_{ijk\ell}:=\frac{1}{2}\frac{C_{ki\ell j}+C_{kj\ell i}}{\rho c_S^6}\in S^2 \tau'_M \otimes S^2 \tau'_M$, from $\dot{\Lambda}_{\mathbf{C}^0}(\mathbf{C})$  along the arbitrarily chosen geodesic $\gamma$ with respect to the metric $g_S$ for any parallel vector field $\eta$ along $\gamma$ that is perpendicular to $\dot \gamma$.
\color{black}
%\[
%\int_{\gamma}\frac{C_{ijk\ell}}{\rho c_S^6}\dot{\gamma}^i(t)\eta^j(t)\dot{\gamma}^k(t)\eta^\ell(t)\mathrm{d}t.
%\]
%Finally after setting $f_{ijk\ell}:=\frac{1}{2}\frac{C_{ki\ell j}+C_{kj\ell i}}{\rho c_S^6}\in S^2 \tau'_M \otimes S^2 \tau'_M$, we now have recovered $L_{2,2}f$ from $\dot{\Lambda}_{\mathbf{C}^0}(\mathbf{C})$.

%\begin{remark}
%In contrast to the travel-time tomography discussed in Appendix \ref{Se:app_1}, for which only the information in the phase is used, here we have also used the information contained in the amplitude.
%\textbf{(TS)[This is a nice comparison, could you please give some more details how the phase is related to the Hamiltonian? I guess the connection is that the velocity of the characteristic curve is the spatial derivative of the phase?]}
%\end{remark}

%\textbf{(TS)[is a very nice derivation of MRT. Sharafutdinov has the same multiplier $\frac{1}{\rho c^6_S}$ in his book. I changed the first part of the appendix to match with this final form.]}

For \textit{P}-waves, we can construct solutions concentrating near a null geodesic $\vartheta(t)=(t+\alpha,\gamma(t))$ in the Lorentzian manifold $((0,T)\times M, -\mathrm{d}t^2+g_P)$. For the solutions $w,v$ constructed as \eqref{solution_u}, \eqref{solution_v}, we can take
\[
\mathbf{a}_0=A_PD\varphi,
\]
where the $P$-wave amplitude satisfies
\[
A_P\vert_{\vartheta}=\det(Y_P)^{-1/2}c_P^{-1/2}\rho^{-1/2}.
\]
Similar as above, we end up with the (longitudinal) ray transform
\[
0=\int_{\gamma}\frac{C_{ijk\ell}}{\rho c_P^6}\dot{\gamma}^i(t)\dot{\gamma}^j(t)\dot{\gamma}^k(t)\dot{\gamma}^\ell(t)\mathrm{d}t.
\]
\bibliographystyle{abbrv}
\bibliography{biblio}
\end{document}